\documentclass[12pt]{book}
\pdfoutput=1 

\ifx\pdfoutput\undefined\else
\fi
\ifx\pdfoutput\undefined
   \RequirePackage[dvips]{graphicx}
   \DeclareGraphicsExtensions{.eps,.ps}
   \RequirePackage{color}
\else%
   \ifnum\pdfoutput=0
      \RequirePackage[dvips]{graphicx}
      \DeclareGraphicsExtensions{.eps,.ps}
      \RequirePackage{color}
   \else
      \RequirePackage[pdftex]{graphicx}
      \DeclareGraphicsExtensions{.pdf,.png,.jpg}
      \RequirePackage[pdftex]{color}
   \fi
\fi
\RequirePackage{amsmath}
\RequirePackage{amssymb}
\RequirePackage{amsthm}
\RequirePackage{microtype}
\RequirePackage{txfonts} 
\RequirePackage{fancyhdr}
\RequirePackage[utf8]{inputenc}
\RequirePackage[T1]{fontenc}
\RequirePackage{url}
\RequirePackage{ifthen}
\RequirePackage{array}
\RequirePackage{multicol}
\RequirePackage[breaklinks]{hyperref}

\DeclareRobustCommand*{\arxiv}[1]{\href{http://www.arxiv.org/abs/#1}{\nolinkurl{arXiv:#1}}}
\DeclareRobustCommand*{\doi}[1]{\href{http://dx.doi.org/#1}{\nolinkurl{doi: #1}}}
\DeclareRobustCommand*{\surname}[1]{\textsc{#1}}

\makeatletter

\def\town#1{\def\@town{#1}}               \def\@town{???}
\def\refereea#1{\def\@refa{#1}}           \def\@refa{???}
\def\refereeb#1{\def\@refb{#1}}           \def\@refb{???}
\def\submitteddate#1{\def\@subdate{#1}}   \def\@subdate{???}
\def\submittedyear#1{\def\@subyear{#1}}   \def\@subyear{\number\year}
\def\examinationdate#1{\def\@exdate{#1}}  \def\@exdate{???}
\def\programm#1{\def\@prog{#1}}           \def\@prog{???}
\def\publicationyear#1{\def\@pubyear{#1}} \def\@pubyear{\number\year}
\def\isbn#1{\def\@isbn{#1}}               \def\@isbn{???}

\renewcommand*{\maketitle}{%
\begin{titlepage}
   \vspace*{20mm}
   \begin{center}
   {\Huge\bfseries\@title\par}
   \vspace{30mm}
   {\large Dissertation\\[1mm]
    zur Erlangung des mathematisch-naturwissenschaftlichen Doktorgrades\\[1mm]
    ,,Doctor rerum naturalium``\\[1mm]
    der Georg-August-Universit\"at G\"ottingen\\[5mm]
    im Promotionsprogramm \@prog\\[1mm]
    der Georg-August University School of Science (GAUSS)
    \vfill
    vorgelegt von\\[2mm]
    {\Large\bfseries\@author}\\[2mm]
    aus \@town
    \par\vspace{30mm}
    G\"ottingen,\space\@subyear}
    \end{center}
\end{titlepage}
\newpage
\thispagestyle{empty}
\vspace*{100mm}
    \noindent {\large
    Betreuungsausschuss\\[2mm]
    \@refa{} (Mathematisches Institut)\\[1mm]
    \@refb{} (Mathematisches Institut)\\[5mm]
    Mitglieder der Prüfungskommission\\[2mm]
    Referentin: \@refa\\[1mm]
    Korreferent: \@refb\\[2mm]
    Weitere Mitglieder der Prüfungskommission:\\[2mm]
    Dr. Christian \surname{Blohmann} (Max-Planck-Institut f\"ur Mathematik)\\[1mm]
    Prof. Dr. Karl-Henning \surname{Rehren} (Institut f\"ur Theoretische Physik)\\[1mm]
    Prof. Dr. Thomas \surname{Schick} (Mathematisches Institut)\\[1mm]
    Prof. Dr. Max \surname{Wardetzky} (Institut f\"ur Num. und Angew. Mathematik)\\[5mm]
    Tag der m\"undlichen Pr\"ufung: \@exdate\par}
}
\makeatother

\fancypagestyle{plain}
{\fancyhf{}\fancyfoot[LE,RO]{\thepage}%
   
}
\makeatletter
\def\cleardoublepage{\clearpage\if@twoside\ifodd\c@page\else
   \hbox{}\thispagestyle{empty}\newpage
   \if@twocolumn\hbox{}\newpage\fi\fi\fi}
\makeatother

\makeatletter
\def\@makechapterhead#1{\vspace*{50\p@}{\parindent\z@\raggedright
   \normalfont\interlinepenalty\@M\Huge\bfseries{%
   \@hangfrom{\hskip\z@\relax\thechapter\hskip2ex}#1\@@par}
   \par\nobreak\vskip40\p@}}
\makeatother

\RequirePackage{mathrsfs}   
\RequirePackage{mathtools}
\mathtoolsset{centercolon}

\RequirePackage[all,poly,web,knot]{xy}
\xyoption{pdf} 
\RequirePackage{tikz}
\usetikzlibrary{matrix,arrows,calc,decorations.pathmorphing}
\tikzset{snake arrow/.style={->,decorate,decoration={snake,amplitude=.4mm,segment length=2mm,post length=1mm}}}
\tikzset{snake back/.style={<-,decorate,decoration={snake,amplitude=.4mm,segment length=2mm,pre length=1mm}}}
\tikzset{mydot/.style={draw,circle, inner sep=1.5pt}}

\numberwithin{equation}{chapter}
\theoremstyle{plain}
\newtheorem{theorem}{Theorem}[chapter]
\newtheorem{proposition}[theorem]{Proposition}
\newtheorem{lemma}[theorem]{Lemma}
\newtheorem{corollary}[theorem]{Corollary}
\theoremstyle{definition}
\newtheorem{definition}[theorem]{Definition}
\newtheorem{example}[theorem]{Example}
\newtheorem{assumption}[theorem]{Assumption}
\theoremstyle{remark}
\newtheorem{remark}[theorem]{Remark}

\RequirePackage{enumitem}
\setlist[itemize]{
  align=left,     
  labelsep=*,     
  leftmargin=*,   
  topsep=1mm,     
  itemsep=0mm     
}
\setlist[enumerate]{
  align=left,       
  labelsep=*,       
  leftmargin=1.7em,   
  topsep=1mm,       
  itemsep=0mm       
}
\setlist[enumerate,1]{label=\textup{(\roman*)}}
\setlist[enumerate,2]{label=\textup{(\alph*)}}

\RequirePackage{xcolor}
\RequirePackage[us]{datetime}

\hypersetup{%
  bookmarksnumbered,
  pdfstartview=FitH,
  pdftitle={Higher Groupoid Actions, Bibundles, and Differentiation},
  pdfauthor={Li Du},
  pdfkeywords={higher groupoids, higher Lie groupoids, higher groupoid actions, higher groupoid bibundles, differentiation of higher Lie groupoids},
  pdfsubject={Mathematics},
}

\AtBeginDocument{%
  \let\oldref\ref
  \renewcommand{\ref}[1]{\textup{\oldref{#1}}}%
}

\RequirePackage[%
  bibencoding=utf8,
  backend=bibtex,
  hyperref=true,
  backref=true,
  backrefstyle=three,
  style=numeric-comp,
  sortcites,
  sorting=nyt
]{biblatex}

\renewbibmacro{in:}{%
  \ifentrytype{article}{}{\printtext{\bibstring{in}\intitlepunct}}}

\renewcommand{\mkbibbrackets}[1]{\textup{[}#1\textup{]}}

\DeclareCiteCommand{\cite}[\mkbibbrackets]
  {\usebibmacro{cite:init}%
   \usebibmacro{prenote}}
  {\textup{\usebibmacro{citeindex}%
   \usebibmacro{cite:comp}}}
  {}
  {\usebibmacro{cite:dump}%
   \usebibmacro{postnote}}
\DeclareFieldFormat{journaltitle}{\mkbibemph{#1},} 
\DeclareFieldFormat[inbook,thesis]{title}{\mkbibemph{#1}\addperiod} 
\DeclareFieldFormat[article]{title}{#1} 
\DeclareFieldFormat{pages}{#1}
\DeclareFieldFormat[article]{volume}{\textbf{#1}\addcolon\space}
\DefineBibliographyStrings{english}{%
  backrefpage  = {p.}, 
  backrefpages = {pp.}, 
}

\def\nbdash{\nobreakdash-\penalty0\hskip0pt\relax}

\newcommand{\bR}{\mathbb{R}}
\newcommand{\bZ}{\mathbb{Z}}
\newcommand{\bN}{\mathbb{N}}
\newcommand{\SAlg}{\mathsf{SAlg}}
\newcommand{\even}{\ensuremath{\mathrm{e}}}
\newcommand{\odd}{\ensuremath{\mathrm{o}}}
\newcommand{\reduce}[1]{\ensuremath{\tilde{#1}}}
\DeclarePairedDelimiter\parity{\lvert}{\rvert}
\DeclarePairedDelimiter\oset{\langle}{\rangle} 
\DeclareMathOperator{\bhom}{\boldsymbol{\hom}}
\DeclareMathOperator{\im}{im}
\DeclareMathOperator{\Nil}{Nil}
\DeclareMathOperator{\card}{\#}                
\newcommand{\ul}[1]{\underline{#1}}            

\DeclareMathOperator{\Hom}{Hom}
\DeclareMathOperator{\Fun}{Fun}
\DeclareMathOperator{\id}{id}
\DeclareMathOperator{\pr}{pr}
\DeclareMathOperator{\ev}{ev}
\DeclareMathOperator*{\colim}{colim}

\newcommand{\bD}{\boldsymbol{\Delta}}        
\DeclareMathOperator{\face}{\mathsf{d}}      
\DeclareMathOperator{\de}{\mathsf{s}}        
\DeclareMathOperator{\tr}{tr}                
\DeclareMathOperator{\Kan}{Kan}              
\DeclareMathOperator{\Acyc}{Acyc}            
\DeclareMathOperator{\sk}{sk}                
\DeclareMathOperator{\cosk}{cosk}            
\DeclareMathOperator{\dec}{dec}              
\DeclareMathOperator{\Dec}{Dec}              
\DeclareMathOperator{\Sp}{Sp}                
\newcommand{\Fib}[1]{\mathrm{Fib}(#1)}       

\RequirePackage{suffix}                      
\newcommand*{\Simp}[1]{{\Delta^{#1}}}        
\WithSuffix\newcommand\Simp*[2][*]{{\Delta^{#2}_{#1}}}
\newcommand*{\Horn}[2]{{\Lambda^{#1}_{#2}}}  
\newcommand*{\horn}[2]{{\lambda^{#1}_{#2}}}  

\newcommand{\Cat}[1][C]{\mathscr{#1}}
\newcommand{\initial}{\boldsymbol{0}}        
\newcommand{\terminal}{\boldsymbol{1}}       
\newcommand{\op}{{\mathrm{op}}}              
\newcommand{\blank}{\mathord{-}}             
\newcommand{\covers}{\mathcal{T}}            
\newcommand{\etale}{\textup{\'et}}           
\newcommand{\open}{\textup{open}}            
\newcommand{\subm}{\textup{subm}}            
\newcommand{\surj}{\textup{surj}}            
\DeclareMathOperator{\Psh}{\mathsf{Psh}}     

\newcommand{\Mfd}{\mathsf{Mfd}}              
\newcommand{\Topo}{\mathsf{Top}}             
\newcommand{\SMfd}{\mathsf{SMfd}}            
\newcommand{\Cats}{\mathsf{Cats}}            
\newcommand{\Sets}{\mathsf{Sets}}            
\newcommand{\SSet}{\mathsf{SSet}}            
\newcommand{\biSSet}{\mathsf{biSSet}}        
\newcommand{\Bun}{\mathsf{Bun}}              
\newcommand{\BUN}{\mathsf{BUN}}              
\newcommand{\Gpd}{\mathsf{Gpd}}              
\newcommand{\GPD}{\mathsf{GPD}}              
\newcommand{\Gen}{\mathsf{Gen}}              
\newcommand{\GEN}{\mathsf{GEN}}              
\newcommand{\ANA}{\mathsf{ANA}}              
\newcommand{\FinSet}{\mathsf{FinSet}}        
\newcommand{\SCat}{\mathsf{S}\mathscr{C}}    
\newcommand{\biSCat}{\mathsf{biS}\mathscr{C}}

\DeclareMathOperator{\pt}{pt}                
\DeclareMathOperator{\Bund}{\mathfrak{B}}    
\newcommand{\trivial}[1]{{\underline{#1}}}   
\newcommand{\unit}{{\mathsf{u}}}             
\newcommand{\source}{{\mathsf{s}}}           
\newcommand{\target}{{\mathsf{t}}}           
\newcommand{\inv}{{\mathsf{i}}}              
\newcommand{\mult}{{\mathsf{m}}}             
\newcommand{\action}{{\mu}}                  
\newcommand{\arrow}[1]{\mathrm{\Theta(#1)}}  
\newcommand{\bigon}[1]{B(#1)}                

\usepackage{filecontents}

\title{Higher Groupoid Actions, Bibundles,\\ and Differentiation}
\author{Du \surname{Li}}
\town{Hunan, China}
\refereea{Prof. Dr. Chenchang \surname{Zhu}}
\refereeb{Prof. Dr. Ralf \surname{Meyer}}
\programm{PhD School of Mathematical Sciences (SMS)}
\submitteddate{}
\submittedyear{2014}
\examinationdate{15 Juli 2014}
\publicationyear{}
\isbn{}

\begin{document}

\frontmatter
\pdfbookmark[0]{Title}{title}
\maketitle

\newpage
\thispagestyle{plain}
\begin{center}
\vspace*{20mm}
{\Huge \bf Abstract}
\phantomsection
\addcontentsline{toc}{chapter}{Abstract}

\vspace*{15mm}
\begin{minipage}[c]{125mm}
{\setlength{\parindent}{2em}
In this thesis, we employ simplicial methods to study actions, principal bundles, and bibundles of higher groupoids. Roughly, we use Kan fibrations to model actions of higher groupoids; we use pairs of a Kan fibration and a special acyclic fibration to model principal bundles of higher groupoids; we use inner Kan fibrations over the interval to model bibundles of higher groupoids. In particular, we show that our definitions given by the simplicial method agree with those given by the categorification approach to actions, principal bundles, and bibundles of 2\nbdash{}groupoids.}

{\setlength{\parindent}{2em}
In addition, we use the simplicial technique to prove a theorem on differentiation of higher Lie groupoids, which shows that the differentiation functor sends a higher Lie groupoid to a higher Lie algebroid.}
\end{minipage}
\end{center}
%



\cleardoublepage
\thispagestyle{plain}
\begin{center}
\vspace*{20mm}
{\huge \bf Acknowledgments}
\phantomsection
\addcontentsline{toc}{chapter}{Acknowledgments}

\vspace*{15mm}
\begin{minipage}[c]{125mm}
{\setlength{\parindent}{2em}
I would first like to express my sincere gratitude to my advisor Chenchang \surname{Zhu} for her encouragement, guidance, and inspiration during my graduate studies. I am also greatly indebted to Ralf \surname{Meyer} for his advice and assistance in writing this thesis.} 

{\setlength{\parindent}{2em}
My special thanks go to my M.S.\ supervisor at Peking University, \surname{Liu} Zhangju, for supporting me in pursuing a Ph.D. career in Göttingen.}

{\setlength{\parindent}{2em}
I am very thankful to the Courant Research Centre ``Higher Order Structures in Mathematics'' and Mathematisches Institut at Georg-August-Universität Göttingen, which supported my study during the last four years.}

{\setlength{\parindent}{2em}
I am glad to acknowledge the following persons for helpful discussions, suggestions, and comments: Christian \surname{Blohmann}, \surname{Liu} Bei, Weiwei \surname{Pan}, Christopher L. \surname{Rogers}, \surname{Zheng} Jiguang.}

{\setlength{\parindent}{2em}
I am indebted to the following persons for useful discussions via email, at \href{http://math.stackexchange.com}{Mathematics Stack Exchange}, or at \href{http://mathoverflow.net}{MathOverflow}: David \surname{Carchedi}, David \surname{Roberts}, Chris \surname{Schommer-Pries}, and Laura \surname{Scull} about groupoids; Zhen Lin \surname{Low} and Michał R. \surname{Przybyłek} about extensive categories.}

{\setlength{\parindent}{2em}
My appreciation also go to Suliman \surname{Albandik}, Martin \surname{Callies}, Malte \surname{Dehling}, Rohit \surname{Holkar}, George \surname{Nadareishvili} for their willingness to share their mathematical knowledge with me.}

{\setlength{\parindent}{2em}
I am grateful to Tathagata \surname{Banerjee}, \surname{Qi} Zhi, \surname{Wu} Xiaolei and \surname{Xiao} Guohui for their linguistic help and other suggestions. I thank \surname{Nguyen} Thi Thu Huong and Sutanu \surname{Roy} for sharing their \LaTeX{} templates with me.}

{\setlength{\parindent}{2em}
I wish to thank \surname{Nan} Xi for the continued encouragement. I also thank all of my friends for their help during my stay in Göttingen.}

{\setlength{\parindent}{2em}
Finally, I would like to thank my parents and sisters for their love, understanding, and support. Without them, I would never have been able to continue my education.}

{\setlength{\parindent}{2em}
This dissertation is dedicated to the memory of my dear maternal grandmother.}
\end{minipage}
\end{center}

\cleardoublepage
\phantomsection
\pdfbookmark{\contentsname}{toc}
\tableofcontents

\mainmatter
\setcounter{chapter}{-1}

\chapter{Introduction}\label{chap0}
\thispagestyle{empty}

The purpose of this thesis is twofold. First, we use the simplicial method to study actions and bibundles of higher groupoids. In particular, we compare the results for 2\nbdash{}groupoids given by the simplicial approach and those given by the usual categorification approach. Second, we use the simplicial technique to study differentiation of higher Lie groupoids.

\section{Simplicial sets and categories}
The simplicial approach is one of the most successful approaches to higher categories. The notion of quasi-categories, introduced by Boardman and Vogt~\cite{Boardman-Vogt} under the name of weak Kan complexes, is a simplicial model of \((\infty,1)\)-categories. Quasi-categories are further extensively studied by Joyal~\cite{Joyal2002,Joyal2008} and Lurie~\cite{Lurie}, among others.

A simplicial set \(X\) is said to satisfy the condition \(\Kan(m,k)\) (\(\Kan!(m,k)\)) if every horn \(\Horn{m}{k}\to X\) has a (unique) lift \(\Simp{m}\to X\); in other words, the canonical map
\begin{equation}\label{intro:eq:kan-map}
\hom(\Simp{m}, X)\to \hom(\Horn{m}{k}, X)
\end{equation}
is surjective (bijective). The horn lifting conditions play a key role in the simplicial approach to higher categories.

To illustrate, let us begin with usual categories. The nerve of a category \(\Cat\) is a simplicial set \(N\Cat\) given by
\[
[n]\mapsto \hom([n], \Cat).
\]
The nerve \(N\Cat\) satisfies \(\Kan!(m, k)\) for \(m>1\) and \(0< k<m\). Unraveling the definition, \(\Kan!(2,1)\) means that \(\Cat\) has a unique composition for composable pairs of morphisms, and \(\Kan!(3,1)\) or \(\Kan!(3,2)\) means that the composition is associative. Conversely, a simplicial set satisfying the above Kan conditions is the nerve of a category. Furthermore, groupoids correspond to simplicial sets satisfying \(\Kan!(m,k)\) for \(m>1\) and \(0\le k\le m\).

Similarly, a \((2,1)\)-category has a geometric nerve which is a simplicial set satisfying \(\Kan(m, k)\) for \(m> 1\) and \(0<k<m\) and \(\Kan!(m, k)\) for \(m>2\) and \(0<k<m\). Moreover, Duskin~\cite{Duskin02} proved that the converse is also true and that 2\nbdash{}groupoids correspond to simplicial sets satisfying \(\Kan(m, k)\) for \(m\ge 1\) and \(0\le k\le m\) and \(\Kan!(m, k)\) for \(m>2\) and \(0\le k\le m\).

Along these lines, a quasi-category is a simplicial set satisfying all inner Kan conditions. If, in addition, the unique Kan conditions are satisfied above dimension \(n\), then we get a notion of \(n\)-categories. As Joyal advocates, ``most concepts and results of category theory can be extended to quasi-categories.''

\section{2-Groupoids}

Groupoids internal to a category with a singleton Grothendieck pretopology (internal groupoids) can also be given by simplicial objects. For simplicity, let us consider Lie groupoids, which are groupoids in the category of manifolds equipped with surjective submersions as covers. The Kan conditions in this case ask the map~\eqref{intro:eq:kan-map} to be a surjective submersion or a diffeomorphism. This leads to an alternative characterisation of Lie groupoids and a definition of higher Lie groupoids~\cite{Henriques}.

There is a subtlety about the equivalence between the usual 2\nbdash{}groupoids and those given by simplicial sets. Given a simplicial set satisfying appropriate Kan conditions, to construct a 2\nbdash{}groupoid, we must choose a composition for composable pairs of 1\nbdash{}morphisms. There is no problem for sets, but we cannot do this for other categories, like manifolds. Zhu~\cite{Zhu:ngpd} observed that the composition can be canonically constructed as an HS bibundle between Lie groupoids (see below) instead of a functor. Moreover, there is a one-to-one equivalence between Lie 2\nbdash{}groupoids given by simplicial manifolds and categorified Lie groupoids (stacky Lie groupoids in~\cite{Zhu:ngpd}).

More precisely, a categorified Lie groupoid \(G\rightrightarrows \trivial{M}\) consists of a manifold of objects~\(M\) and a Lie groupoid of arrows \(G\) with HS bibundles: source and target \(\source,\target\colon G\to \trivial{M}\), unit \(\unit\colon \trivial{M}\to G\), inverse \(\inv\colon G\to G\) and multiplication \(G\times_{\source\trivial{M},\target} G\to G\); the unitality and associativity equalities are relaxed to isomorphisms of HS bibundles, and these isomorphisms as part of the structure satisfy certain coherence conditions.

\section{Actions of 2-groupoids}

Let us recall that the action of a groupoid \(X\) has three formulations:
\begin{enumerate}
  \item a set \(Y_0\) with a moment map \(J\colon Y_0\to X_0\) and an action map \(\action\colon Y_0\times_{X_0} X_1\to Y_0\) satisfying the unitality and associativity laws;
  \item a discrete fibration of groupoids \(Y\to X\), whose fiber is a trivial groupoid;
  \item a functor \(X\to \Sets\);
\end{enumerate}
We pass from formulation (i) to (ii) by taking the action groupoid, from (iii) to (ii) by the Grothendieck construction. The third formulation is the simplest one. Unfortunately, it does not apply to actions of groupoids with extra structures (like Lie groupoid actions on manifolds).

Along these lines, we should have three approaches to actions of 2\nbdash{}groupoids.

For the first approach, we follow the idea of categorification. A categorified groupoid action of \(G\rightrightarrows \trivial{M}\) on a groupoid \(E\) consists of a moment morphism \(E\to \trivial{M}\) and an action morphism, which is given by an HS bibundle \(\action\colon E\times_{\trivial{M}}G\to E\); the action satisfies the unitality and associativity laws up to isomorphisms of HS bibundles, and these isomorphisms satisfy certain coherence conditions.

The second formulation now becomes a Kan fibration between 2\nbdash{}groupoids, whose fiber is a 1\nbdash{}groupoid. The third approach does not apply to 2\nbdash{}groupoids with extra structures (like Lie 2\nbdash{}groupoids). Nevertheless, the Grothendieck construction suggests an action 2\nbdash{}groupoid construction linking the first and the second approaches. Our first main theorem establishes the equivalence between these two approaches to 2-groupoid actions.

Among groupoid actions, the principal bundles are of particular interest. A principal bundle \(Y_0\to N\) of a groupoid \(X\) is given by an \(X\)-action on \(Y_0\) such that the map \(Y_0\to N\) is invariant and that the shear map
\[
Y_0\times_{X_0} X_1\to Y_0\times_N Y_0, \quad (y, x)\mapsto (y, y\cdot x),
\]
is an isomorphism. This condition is equivalent to \(Y\to \sk_0 N\) being acyclic in the simplicial language, where \(Y\) is the nerve of the action groupoid. Both approaches can be generalised to 2\nbdash{}groupoids, and our second main theorem shows that the results are equivalent.

Our simplicial approach to actions and principal bundles is closely related to the work of Duskin~\cite{Duskin} and Glenn~\cite{Glenn:Realization}; partial results are obtained by Bakovi\'{c}~\cite{Bakovic:Bigroupoid_torsors}. Special cases of the categorification approach are studied in~\cite{Nikolaus-Waldorf:Four_gerbes, Bakovic:Bigroupoid_torsors}.

\section{Bibundles of 2-groupoids}

Bibundles (or bimodules) between groupoids \(X\) and \(Y\) also have three formulations:
\begin{enumerate}
  \item a set with a left action of \(X\) and a right action of \(Y\) subject to certain conditions;
  \item a category over the interval category \(\{0\leftarrow 1\}\) such that the two ends are \(X\) and~\(Y\);
  \item a functor \(Y^\op \times X\to \Sets\);
\end{enumerate}

The cograph construction links the first formulation and the second. The third formulation, also known as a profunctor~\cite{Benabou:dist}, is the simplest, in which the composition of bibundles can be simply written as a coend; see, for instance, \cite{Cattani-Winskel}.

The third approach may not work for internal groupoids (like Lie groupoids); the composition may not be well-behaved in general. For certain internal groupoids (like Lie groupoids), the composition of right principal bibundles (or HS bibundles) works well. There is a weak 2\nbdash{}category~\(\BUN\) whose objects are internal groupoids, 1\nbdash{}morphisms are HS bibundles, and 2\nbdash{}morphisms are isomorphisms of HS bibundles. This 2\nbdash{}category plays a key role in the study of groupoids and 2\nbdash{}groupoids; see, for instance, \cite{Moerdijk-Mrcun2003,Moerdijk-Mrcun2005,Mrcun,Schommer-Pries,Blohmann}.

Let us now consider bibundles of 2\nbdash{}groupoids. It is straightforward to generalise the first formulation above by categorification. The second formulation now becomes a simplicial object over the simplicial interval~\(\Simp{1}\) that is an inner fibration satisfying certain additional Kan conditions. Our third main theorem shows that these two approaches are equivalent.

We then construct the composition of two right principal bibundles of 2\nbdash{}groupoids. Our construction is a higher analogue of the  composition of HS bibundles of groupoids. It is closely related to 2\nbdash{}categorical coends.

We emphasise that we are working with higher groupoids in a category with a Grothendieck pretopology satisfying further appropriate conditions. These include higher Lie groupoids and higher topological groupoids as special cases.

\section{Differentiation of higher Lie groupoids}

The infinitesimal analogue of Lie groups are Lie algebras, which are sometimes called infinitesimal groups. For a Lie group the differentiation procedure gives the associated Lie algebra. Generalisations of this classical fact are well-known for Lie groupoids, crossed modules of Lie groups, and simplicial Lie groups.

Higher Lie groupoids are simplicial manifolds satisfying certain Kan conditions. Two natural questions then arise: what is a higher Lie algebroid as the infinitesimal analogue of a higher Lie groupoid and how to associate a higher Lie algebroid with a higher Lie groupoid?

\v{S}evera proposed an answer to these two questions~\cite{Severa05,Severa06}. Recall that a Lie algebra can be equivalently given by its Chevalley--Eilenberg cochain complex, which is a differential graded commutative algebra (DGCA). For a general DGCA, denote its degree operator by \(N\) (that is, the Euler vector \(Nx=\parity{x}x\)) and the differential by \(Q\).  Then the following commutation relations hold
\[
  [N, N]=0,\quad [N, Q]=Q, \quad [Q, Q]=0.
\]
As observed by Kontsevich~\cite{Kontsevich}, a DGCA structure thus can be regarded as a representation of the super Lie algebra \(\bR\ltimes D\) (or the monoid \(D^D\)), where \(D\) is the odd line.

A higher Lie algebroid, or NQ-manifold introduced in~\cite{Severa06}, is a supermanifold \(M\) equipped with an action of \(D^D\) (satisfying an additional compatibility condition which is not important for the moment). Such an action is essentially the same as a DGCA structure on the algebra of functions on~\(M\).

Let \(X\) be a Lie \(n\)\nbdash{}groupoid. \v{S}evera's differentiation of \(X\), introduced in~\cite{Severa06}, is the presheaf on the category of supermanifolds
\[
\bhom(P,X)\colon T\mapsto \hom(P\times T, X),
\]
where \(P\) is the nerve of the pair groupoid of~\(D\). This presheaf can be informally thought of as the space of all infinitesimal lines in \(X\). If this presheaf is representable, then we get a higher Lie algebroid associated with a higher Lie groupoid. The proof of the representability of this presheaf given in~\cite{Severa06} is incomplete and contains some mistakes. We present a detailed proof of this claim in the last chapter.

\section{Structure of the thesis}

The content of this thesis is organised as follows.

Chapter 1 reviews the theory of Lie groupoids. We introduce the notions of Lie groupoid actions, principal bundles, and bibundles, which we intend to generalise to higher groupoids. We also establish three equivalent 2\nbdash{}categories of Lie groupoids, of which the one given by HS bibundles will be used to defined 2\nbdash{}groupoids.

Chapter 2 collects some basic facts about simplicial sets. The lifting properties of simplicial sets will be central to our approach to higher groupoids and 2\nbdash{}groupoids. Skeleton and coskeleton functors are introduced. Collapsible extensions of simplicial sets are also studied for later use.

In Chapter 3, we define higher groupoid actions and principal bundles. We first introduce higher groupoids in a category with a Grothendieck pretopology, which are the main objects of study in this thesis. Kan fibrations and acyclic fibrations of higher groupoids are then introduced. A higher action is defined as a Kan fibration of higher groupoids (Definition~\ref{chap3:def:higher-groupoid-action}); a higher principal bundle is a Kan fibration together with a special acyclic fibration (Definition~\ref{chap3:def:higher-principal-bundles}). As an example of higher groupoid action and principal bundle, the décalage is studied (Examples~\ref{chap3:exa:decalage} and \ref{chap3:exa:decalage-principal}).

Chapter 4 is concerned with bibundles between higher groupoids. The theory of bimodules of categories is reviewed in the first section. Bibundles between two higher groupoids are defined as simplicial objects over the simplicial interval that satisfy appropriate Kan conditions; bibundles that are right or left principal and two-sided principal (Morita equivalent) are further characterised (Definition~\ref{chap4:def:bibundles=inner-over-I}). In particular, a right principal bibundle, called cograph (bundlisation), is constructed for a higher groupoid morphism. This bibundle is two-sided principal for an acyclic fibration (Proposition~\ref{chap4:prop:cograph-acyclic-is-kan}).

In Chapter 5, we study actions of 2\nbdash{}groupoids. Categorification approach to 2-groupoids and its relation with the simplicial approach are first recalled. We then define 2\nbdash{}groupoid actions on groupoids and 2\nbdash{}bundles by categorifying the notions of groupoid actions and bundles. We show that they are equivalent to those given by the simplicial approach to 2\nbdash{}groupoid actions and 2\nbdash{}bundles (Theorems~\ref{chap5:thm:2gpd-action-Kan} and~\ref{chap5:thm:princiapl-2-bundle-cat=simp}).

Chapter 6 studies bibundle of 2\nbdash{}groupoids. We first define bibundles between two 2\nbdash{}groupoids by categorifying the notion of groupoid bibundles. The equivalence between the categorification approach and the simplicial approach to 2\nbdash{}groupoid bibundles is established (Theorem~\ref{chap6:thm:bibundle-cat=simplicial}). Weak equivalences between 2\nbdash{}groupoids are then discussed. We show that they are weak acyclic fibrations in the simplicial picture (Proposition~\ref{chap6:prop:weak-equi=weak-acyc}), hence a 2-groupoid morphism is a weak equivalence if and only if its cograph gives a two-sided principal bibundle (Morita equivalence). Some examples are considered in Section~\ref{chap6:sec:example}.

Chapter 7 concentrates on the composition of bibundles between 2\nbdash{}groupoids. This is a 2\nbdash{}groupoid analogue of the composition of HS bibundles of groupoids. We also show that the bundlisation respects compositions and that the bibundle composition satisfies unitality and associativity laws up to certain equivalences (Propositions~\ref{chap6:prop:bibundle-functorial}, \ref{chap6:prop:2gpd-bibundle-comp-unit}, and~\ref{chap6:prop:2gpd-bibundle-comp-ass}).

Chapter 8 deals with differentiation of higher Lie groupoids. Our main aim is to give a complete and detailed proof of \v{S}evera's representability theorem (Theorem~\ref{chap7:thm:rep}). Some preparation on supermanifolds and NQ-manifolds are provided before proving the theorem. Additionally, three special cases of \v{S}evera's differentiation are worked out in detail. It turns out that the results coincide with the usual notion of differentiation.

\section{Categorical conventions}\label{chap0:sec:cat-convention}

Basic category and 2\nbdash{}category theory are assumed; see, for instance, \cite{MacLane,Benabou67}.

All 2\nbdash{}categorical concepts are weak ones by default. By 2\nbdash{}categories we mean weak 2\nbdash{}categories or bicategories. Indeed, all 2\nbdash{}categories in the following will be \((2,1)\)\nbdash{}categories with invertible 2\nbdash{}morphisms. Similarly, 2\nbdash{}functors mean pseudofunctors or homomorphisms, 2\nbdash{}limits mean bilimits or pseudolimits, and so on. Strict 2\nbdash{}categories are explicitly specified if necessary.

We will often need to speak of limits of some diagrams before knowing that they exist. To this end, we use the Yoneda embedding
\[
  h \colon \Cat \to [\Cat^\op, \Sets]
\]
to embed any category~\(\Cat\) into a complete category. Thus any diagram in~\(\Cat\) has a presheaf on~\(\Cat\) as a limit. It has a limit in~\(\Cat\) if and only if this presheaf is representable by some object of~\(\Cat\). Dually, we use the contravariant Yoneda embedding to embed~\(\Cat\) into the category of copresheaves if we deal with colimits.

\chapter{Lie Groupoids}\label{chap1}
\thispagestyle{empty}

In this chapter, we review the theory of Lie groupoids. After giving the definition and some first examples, we introduce some basic constructions with Lie groupoids. Three 2\nbdash{}categories of Lie groupoids are then established, one by fractions with respect to weak equivalences, one by anafunctors, and one by HS bibundles. In the end, we show that these three 2\nbdash{}categories are equivalent.

The 2\nbdash{}category given by HS bibundles will be the most important one in the following, so we devote more attention to the theory of HS bibundles.

\section{Definition and first examples}

We assume that the reader is familiar with basic category theory; see, for instance, \cite[Chapter I--V]{MacLane}. For groupoids and Lie groupoids, see~\cite{Mrcun, Moerdijk-Mrcun2003,Moerdijk-Mrcun2005}, which we follow closely.

\begin{definition}\label{chap1:def:groupoid}
  A \emph{groupoid} is a small category in which every arrow is invertible. A \emph{functor} between two groupoids is a functor between the underlying categories. A \emph{natural transformation} between two groupoid functors is a natural transformation between the underlying functors of categories.
\end{definition}

More explicitly, a groupoid \(G\) consists of the following data:
\begin{itemize}
  \item two sets: the \emph{set of object} \(G_0\) and the \emph{set of arrows} \(G_1\),
  \item five structure maps: \emph{source} and \emph{target} \(\source, \target\colon G_1\to G_0\), \emph{unit} \(\unit\colon G_1\to G_0\), \emph{inverse} \(\inv\colon G_1\to G_1\), \emph{composition} \(\circ\colon G_1\times_{\source, G_0, \target} G_1\to G_1\),
\[
\xymatrix{
G_1\times_{G_0} G_1\ar[r]^-{\circ} & G_1\ar@(dr,dl)[]^{\inv}\ar@<0.5ex>[r]^{\source} \ar@<-0.5ex>[r]_{\target}& G_0\ar@(dl, dr)[l]^{\unit}
}.
\]
\end{itemize}
These maps satisfy a series of identities:
\begin{itemize}
  \item source and target relations
  \[
  \xymatrix{
  G_0 \ar[dr]_{\unit}\ar[r]^{=}\ar[d]_{=} & G_0\\
  G_0 & G_1\ar[l]^{\target}\ar[u]_{\source}\rlap{\ ,}
  }
  \qquad
  \xymatrix{
  G_1\ar[rd]^{\inv}\ar[r]^{\source}\ar[d]_{\target} & G_0\\
  G_0 & G_1\ar[u]_{\target}\ar[l]^{\source}\rlap{\ ,}
  }
  \qquad
  \xymatrix{
  G_1\ar[d]_{\target} &G_1\times_{G_0} G_1\ar[d]^{\circ}\ar[l]_-{\pr_1}\ar[r]^-{\pr_2}& G_1 \ar[d]^{\source}\\
  G_0 &G_1\ar[l]_{\target}\ar[r]^{\source} & G_0\rlap{\ ,}
  }
  \]
  \item left and right unit identities
  \[
  \xymatrix{
  & G_1\times_{G_0} G_1\ar[d]^{\circ} & \\
  G_1\ar[ur]^-{(\id,\unit\circ\source)}\ar[r]^{=}    & G_1 & G_1\ar[ul]_-{(\unit\circ \target, \id)}\ar[l]_{=}\rlap{\ ,}
  }
  \]
  \item associativity of composition
  \[
  \xymatrix{
  G_1\times_{G_0} G_1\times_{G_0} G_1\ar[r]^-{(\id,\circ)}\ar[d]_{(\circ, \id)} &G_1 \times_{G_0} G_1\ar[d]^{\circ}\\
  G_1\ar[r]_{\circ} \times_{G_0} G_1 & G_1\rlap{\ ,}
  }
  \]
  \item inverse identities
  \[
  \xymatrix{
   & G_1\times_{G_0} G_1\ar[d]_{\circ} &\\
  G_1\ar[r]^{\unit\circ \target}\ar[ur]^{(\id,\inv)} & G_1 & G_1\ar[l]_{\unit\circ \source} \ar[ul]_{(\inv, \id)}\rlap{\ .}
  }
  \]
\end{itemize}
Here we choose to express identities by commutative diagrams. It may be inconvenient at first sight. We can easily translate a commutative diagram to a usual equation on elements. For instance, the diagram of associativity tells us exactly that
\[
(\gamma_3\circ \gamma_2) \circ \gamma_1=\gamma_3\circ (\gamma_2 \circ \gamma_1)
\]
holds for composable arrows \(\gamma_1, \gamma_2, \gamma_3 \in G_1\). The advantage of commutative diagrams is that they work in any category, even if elements are meaningless.

A functor between groupoids \(\varphi\colon G\to H \) consists of two functions \(\varphi_0\colon G_0\to H_0\) and \(\varphi_1\colon G_1\to H_1\) that respect all structure maps. We also call a functor between groupoids a \emph{morphism of groupoids}.

A natural transformation between functors \(\tau\colon \varphi \Rightarrow \psi \colon G\to H\) is a function \(\tau\colon G_0\to H_1\) such that \(\tau(y)\circ\varphi(\gamma)=\psi(\gamma)\circ\tau(x)\) for every arrow \(\gamma\colon x\to y\) in \(G\). Such a natural transformation is invertible since every arrow in a groupoid is invertible.

The composition in a groupoid is illustrated in Figure~\ref{chap1:fig:composition-of-groupoid}, where dots represent objects and arrows represent arrows in the groupoid. Two arrows in the groupoid are composable if the source and target of the corresponding arrows match. We draw arrows of a groupoid from right to left, so the composite of two arrows is in the diagrammatic order.
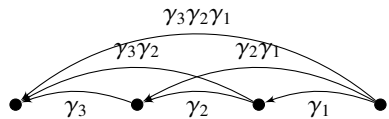
\begin{figure}[htbp]
  \centering
  \begin{tikzpicture}%
  [>=latex',every label/.style={scale=0.8},scale=0.8]
  \node[mydot,fill=black] at (6,0) (x0) {};
  \node[mydot,fill=black] at (4,0) (x1) {};
  \node[mydot,fill=black] at (2,0) (x2) {};
  \node[mydot,fill=black] at (0,0) (x3) {};
  \path[->]
     (x0) edge[out=160, in=20] node[scale=0.8,below]{$\gamma_1$} (x1)
	 edge[out=150, in=30] node[scale=0.8,above]{$\gamma_2\gamma_1$}(x2)
   edge[out=140, in=40] node[scale=0.8,above]{$\gamma_3\gamma_2\gamma_1$}(x3)
     (x1) edge[out=160, in=20] node[scale=0.8,below]{$\gamma_2$} (x2)
   edge[out=150, in=30] node[scale=0.8,above]{$\gamma_3\gamma_2$}(x3)
     (x2) edge[out=160, in=20] node[scale=0.8,below]{$\gamma_3$} (x3);
  \end{tikzpicture}
  \caption{Composition of arrows in a groupoid}\label{chap1:fig:composition-of-groupoid}
\end{figure}

Denote by \(\Mfd\) the category of manifolds, whose objects are finite-dimensional manifolds that may have different dimensions for different connected components, and whose morphisms are smooth maps.

\begin{definition}\label{chap1:def:lie-groupoid}
  A \emph{Lie groupoid} is a groupoid object internal to \(\Mfd\) such that the source and target maps \(\source, \target\) are surjective submersions. \emph{Functors between Lie groupoids} and \emph{natural transformations between Lie groupoid functors} are defined by internalization to \(\Mfd\).
\end{definition}

For internal categories, see~\cite[Chapter XII]{MacLane}. Explicitly, a Lie groupoid~\(G\) has two manifolds: \(G_0\) objects, \(G_1\) arrows, and five smooth structure maps:
source and target \(\source, \target\colon G_1\to G_0\), unit \(\unit\colon G_0\to G_1\), composition \(\circ\colon G_1\times_{\source, G_0, \target} G_1\to G_1\), and inverse \(\inv\colon G_1\to G_1\), satisfying the same identities as above. Notice that \(G_1\times_{\source, G_0, \target} G_1\) is an object in \(\Mfd\), since \(\source\) and \(\target\) are surjective submersions.

\begin{proposition}
  Lie groupoids and functors between Lie groupoids form a category, denoted by \(\Gpd\); Lie groupoids, functors between Lie groupoids, and natural transformations form a strict 2\nbdash{}category, denoted by \(\GPD\).
\end{proposition}

\begin{remark}
  Our strict 2\nbdash{}categories are sometimes called 2\nbdash{}categories, while our 2\nbdash{}categories are called bicategories or weak 2\nbdash{}categories. In the sequel, all 2\nbdash{}categories will be \((2,1)\)\nbdash{}categories in which 2\nbdash{}morphisms are invertible. For 2\nbdash{}categories see~\cite{Benabou67,Street}.
\end{remark}

\subsection{Examples}

We provide some first examples of Lie groupoids.

\begin{example}
  Let \(G\) be a Lie groupoid. Its \emph{opposite Lie groupoid}, denoted by \(G^\op\), has the same object space and arrow space, but all arrows are formally inverted. The inverse map \(\inv\colon G_1\to G_1\) induces an isomorphism \(G\to G^\op\).
\end{example}

\begin{example}
  A Lie groupoid \(G\) with \(G_0=\pt\) is a Lie group. A functor between two Lie groupoids of this form \(G, H\) is given by a homomorphism of Lie groups \(G_1\to H_1\); a natural transformation between two such functors is given by an element of \(H_1\) conjugating between the two homomorphisms.
\end{example}

\begin{example}
  Let \(f\colon M\to N\) be a surjective submersion. The associated \emph{\v{C}ech groupoid}, denoted by \(C(f)\), is the Lie groupoid with
  \begin{gather*}
  C(f)_0=M,\quad C(f)_1=M\times_N M,\\
  \source(x, y)= y,\quad \target(x, y)= x,\quad \unit(x)=(x, x),\\
  \inv(x, y)=(y, x),\quad (x, y)\circ (y, z)=(x, z).
  \end{gather*}
 When \(f\) is of the form \(M\to \pt\), we get the \emph{pair groupoid}\footnote{Pair groupoids are also called chaotic groupoids or codiscrete groupoids; trivial groupoids are also called discrete groupoids.} \(P(M)\) on~\(M\); when \(f\) is the identity map \(M\to M\), we get the \emph{trivial groupoid} on \(M\), denoted by \(\trivial{M}\). This induces an embedding of categories
  \[
   \Mfd \hookrightarrow \Gpd, \quad  M\mapsto \trivial{M}.
  \]
  For any \(f\colon M\to N\), there is a Lie groupoid functor \(c(f)\colon C(f)\to \trivial{N}\) induced by \(f\).
\end{example}

\begin{example}
  A functor from a Lie groupoid \(G\) to a trivial groupoid \(\trivial{M}\) is a map \(f\colon G_0\to M\) such that \(f\circ \source=f\circ \target\).
\end{example}

\begin{example}
  Given a right action of a Lie group \(G\) on a manifold~\(X\), its \emph{action groupoid} \(X\rtimes G \) is a Lie groupoid defined as follows:
  \begin{gather*}
  (X\rtimes G)_0=X,\quad (X\rtimes G)_1=X\times G,\\
  \source(x, g)=xg,\quad \target(x, g)=x,\quad \unit(x)=(x, e),\quad \inv(x, g)=(xg, g^{-1}),\\
  (x, g)\circ(y, h) =(x, g\circ h)\quad \text{for \(y=xg\).}
  \end{gather*}
Action groupoids for left actions are defined similarly.

If \(P\to M\) is right principal \(G\)-bundle, then \(P\to M\) induces a Lie groupoid functor \(X\rtimes G \to \trivial{M}\).
\end{example}

\section{Constructions with Lie groupoids}

In this section, we introduce some general constructions with Lie groupoids, including induced Lie groupoids, strong pullbacks, and weak pullbacks.

\subsection{Induced Lie groupoids}

\begin{definition}\label{chap1:def:induced-grouoid}
Let \(G\) be a Lie groupoid and let \(f\colon B\to G_0\) be a map. The \emph{induced groupoid} \(f^* (G)\) is a groupoid with \(f^* (G)_0=B\), \(f^* (G)_1= B\times_{G_0,\target} G_1 \times_{\source,G_0} B\), and with composition given by the composition in~\(G\).
\end{definition}

\begin{proposition}\label{chap1:prop:induced-groupoid-exists}
  The induced groupoid \(f^* (G)\) is a Lie groupoid if\/ \(\target\circ\pr_1\colon G_1\times_{\source,G_0} B\to G_0\) is a surjective submersion.
\end{proposition}
\begin{proof}
  Observe that \(f^* (G)_1= B\times_{G_0,\target} G_1 \times_{\source,G_0} B\) can be computed by two pullbacks given by the diagram
  \[
  \xymatrix{
  f^*(G)_1\ar[d]\ar[rr] & & B\ar[d]^{f}\\
  G_1\times_{\source,G_0} B\ar[d]\ar[r] & G_1\ar[d]^{\source} \ar[r]^{\target} & G_0\\
  B\ar[r]^{f}& G_0 \rlap{\ .}&
  }
  \]
  Since the composite \(G_1\times_{\source,G_0} B\xrightarrow{\pr_1} G_1\xrightarrow{\target} G_0\) is a surjective submersion, \(f^*(G)_1\) is representable\footnote{It means that the limit is a manifold; see Sections~\ref{chap0:sec:cat-convention} and \ref{chap2:sec:presheaf}.} and \(f^*(G)_1\to B\) is a surjective submersion. Hence the diagram below is a pullback square in \(\Mfd\):
  \[
  \xymatrix{
  f^*(G)_1\ar[d] \ar[r] &G_1\ar[d]^{(\source, \target)}\\
  B\times B \ar[r] & G_0\times G_0
  }
  \]
  Thus \(f^*(G)\) is a Lie groupoid and \(f\) induces a Lie groupoid functor \(f\colon f^*(G)\to G\).
\end{proof}

\subsection{Strong pullbacks}

\begin{definition}
  Given two Lie groupoid functors \(f\colon G\to K\) and \(g\colon H\to K\), the \emph{strong pullback} groupoid \(G\times_K H \) is a groupoid with \((G\times_K H)_0=G_0\times_{K_0} H_0\), \((G\times_K H)_1=G_1\times_{K_1} H_1\), and with composition defined componentwise.
\end{definition}

\begin{remark}
  If \((f_0,g_0)\) and \((f_1,g_1)\) are both transversal, then \(G\times_K H\) is a Lie groupoid. In this case, the strong pullback satisfies the usual universal property for pullbacks in the category \(\Gpd\).
\end{remark}

\begin{example}
  The induced groupoid can be written as a strong pullback. Let \(G\) and \(f\colon B\to G_0\) be as in Definition~\ref{chap1:def:induced-grouoid}. The diagram below is a strong pullback square:
  \[
  \xymatrix{
  f^*(G) \ar[d]\ar[r] & G\ar[d]\\
  P(B) \ar[r] & P(G_0)\rlap{\ .}
  }
  \]
Here \(P(B)\) and \(P(G_0)\) are pair groupoids, the Lie groupoid functor \(G\to P(G_0)\) is determined by \((\source, \target)\colon G_1\to G_0\times G_0\), and \(P(B)\to P(G_0)\) is induced by \(f\).
\end{example}

\subsection{Weak pullbacks}

For comma categories see~\cite[Section II.6]{MacLane}.
\begin{definition}
  Given two Lie groupoid functors \(f\colon G\to K\) and \(g\colon H\to K\), the \emph{weak pullback} groupoid \(G\times^w_K H \) is the comma category \((f \downarrow g)\). Explicitly, objects of \(G\times^w_K H\) are triples \((x, \alpha, y )\), where \(x\in G_0\), \(y\in H_0\) and \(\alpha\colon f(x) \to g(y)\) is an arrow in~\(K\). Arrows from \((x_1, \alpha_1, y_1)\) to \((x_2, \alpha_2, y_2)\) are pairs \((\beta, \gamma)\), where \(\beta\colon f(x_1)\to f(x_2) \) and \(\gamma\colon g(x_1)\to g(x_2) \) are arrows in~\(K\) such that \(\gamma\circ \alpha_1 =\alpha_2\circ \beta\). The composition is given componentwise.
\end{definition}

\begin{proposition}\label{chap1:prop:weak-pullback-exists-assumption}
The weak pullback \(G\times^w_K H \) is a Lie groupoid if\/ \(\target\circ\pr_2\colon  G_0\times_{K_0,\source} K_1\to K_0\) or \(\source\circ\pr_2\colon H_0\times_{K_0,\target} K_1\to K_0\) is a surjective submersion.
\end{proposition}
\begin{proof}
The space of objects of \(G\times^w_K H\) is the limit of the following diagram:
  \[
  \xymatrix{
  G_0\times_{K_0, \source} K_1 \times_{\target, K_0} H_0\ar[d] \ar[rr]&&H_0\ar[d]\\
   G_0\times_{K_0,\source} K_1 \ar[d]\ar[r]& K_1 \ar[d]^{\source}\ar[r]^{\target}&K_0\\
   G_0 \ar[r]& K_0\rlap{\ ,} &
  }\]
which is representable if \(\target\circ\pr_2\colon G_0\times_{K_0,\source} K_1\to K_0\) is a surjective submersion. Similarly, it is representable if \(\source\circ\pr_2\colon H_0\times_{K_0,\target} K_1\to K_0\) is a surjective submersion. In these cases, it is easy to see that the space of arrows,
  \[
  (G\times^w_K H)_1 =G_1\times_{\target\circ f_1,K_0,\source} K_1 \times_{\target,K_0,\source\circ f_1} H_1,
  \]
  is also a manifold, and that source and target maps are surjective submersions. This proves the claim.
\end{proof}

\begin{remark}
  The following maps
  \[
  (x, \alpha, y)\mapsto (y, \alpha^{-1}, x), \quad (\beta,\gamma)\mapsto (\gamma,\beta)
  \]
  give a canonical isomorphism between the groupoids \(G\times^w_K H\) and \(H\times^w_K G\).
\end{remark}

There are Lie groupoid functors \(G\times^w_K H\to G\) and \(G\times^w_K H\to H\) induced by projection. The composites \(G\times^w_K H \to G\to K\) and \(G\times^w_K H\to H\to K\) are not equal, but differ by a natural transformation \(\tau\) given by
\[
(G\times^w_K H)_0\to K_1,\quad (x, \alpha, y)\mapsto \alpha.
\]

\begin{remark}
  The weak pullback of Lie groupoids is a 2\nbdash{}pullback in the 2\nbdash{}category \(\GPD\) \textup(see~\cite[Section 5.3]{Moerdijk-Mrcun2003}\textup). We recall that a \emph{2\nbdash{}pullback} of \(f\colon A\to C\) and \(g\colon B\to C\) in a 2\nbdash{}category is a square
  \[
  \xymatrix{
  P\ar[d]_{q}\ar[r]^{p} & A\ar[d]^{f}\\
  B\ar[r]_{g}\ar@{{}{ }{}}[ru]|{\tau} & C
  }\]
commuting up to a 2\nbdash{}isomorphism \textup(2\nbdash{}commutative\textup), which is universal among such squares in a 2\nbdash{}categorical sense. The universal property determines the 2\nbdash{}pullback up to a unique equivalence. For 2\nbdash{}pullbacks and other 2\nbdash{}limits, see~\cite{Street,Street87}.
\end{remark}

\begin{example}\label{chap1:exa:G1=G0-times-G0-over-G}
  Let \(G\) be a Lie groupoid. The diagram below is a weak pullback square
  \[
  \xymatrix{
  \trivial{G_1} \ar[d]_{\source}\ar[r]^{\target}& \trivial{G_0}\ar[d]^{\unit}\\
  \trivial{G_0} \ar[r]_{\unit}\ar@{{}{ }{}}[ru]|{\tau}       & G \rlap{\ ,}
  }
  \]
  where the natural transformation \(\tau\) is given by \(\id\colon G_1\to G_1\).
\end{example}

\begin{example}
  Let \(M\) be a manifold. Let \(G\to \trivial{M}\) and \(H\to \trivial{M}\) be Lie groupoid functors. Then the strong pullback and the weak pullback are naturally isomorphic if one of them is a Lie groupoid.
\end{example}

\section{Weak equivalences and generalised morphisms}

After introducing the notion of weak equivalence, we will construct a 2\nbdash{}category of Lie groupoids by formally inverting weak equivalences.
\subsection{Weak equivalences}
Recall~\cite{Street} that a 1\nbdash{}morphism \(f\colon A\to B\) in a 2\nbdash{}category is an \emph{equivalence} if it has a quasi-inverse; that is, there are a morphism \(\varphi\colon B\to A\) and 2\nbdash{}isomorphisms \(\varphi\circ f\Rightarrow \id_A\) and \(f\circ \varphi \Rightarrow \id_B\).

Recall from~\cite[Section IV.4]{MacLane} that if a functor \(f\colon \Cat[A]\to \Cat[B]\) of categories is essentially surjective and fully faithful then it is an equivalence in the 2\nbdash{}category of categories.

We now consider similar notions for Lie groupoids. To distinguish various notions of equivalences, let us call equivalences in the 2\nbdash{}category \(\GPD\) \emph{strong equivalences} of Lie groupoids.

\begin{definition}\label{chap1:def:weak-equi-groupoid}
  A Lie groupoid functor \(f\colon G\to H\) is a \emph{weak equivalence} if it is
  \begin{enumerate}
    \item essentially surjective; that is, \(\target\circ \pr_2\colon G_0\times _{H_0, \source} H_1 \to H_0\) is a surjective submersion;
    \item fully faithful; that is, the diagram below is a pullback square in \(\Mfd\):
    \[
    \xymatrix{
    G_1 \ar[d]_{(\source, \target)}\ar[r]^{f_1} & H_1\ar[d]^{(\source, \target)} \\
    G_0\times G_0 \ar[r]^{f_0\times f_0}& H_0\times H_0\rlap{\ .}
    }
    \]

  \end{enumerate}
\end{definition}

\begin{remark}
We will denote a weak equivalence by \(\simeq\). By Proposition~\ref{chap1:prop:induced-groupoid-exists}, the first condition ensures that the pullback in the second condition exists. If \(f\colon G\xrightarrow{\simeq} H\) is a weak equivalence, then \(G\) is isomorphic to the induced groupoid \(f_0^* (H)\).
\end{remark}

\begin{proposition}[{\cite[Proposition 5.11]{Moerdijk-Mrcun2003}}]
  Every strong equivalence of Lie groupoids is a weak equivalence.
\end{proposition}

The following two examples show, however, that the converse is not true.

\begin{example}\label{chap1:exa:cech-groupoid-weak-equivalence}
  Let \(f\colon M\to N\) be a surjective submersion, and let \(C(f)\) be the corresponding \v{C}ech groupoid. Then \(c(f)\colon C(f)\xrightarrow{\simeq} \trivial{N}\) is a weak equivalence. Suppose that \(g\colon \trivial{N}\to C(f)\) is a quasi-inverse of \(c(f)\). Then \(c(f)\circ g\) is the identity since \(\trivial{N}\) is a trivial groupoid. Hence \(g_0\) is a section of \(f\colon M\to N \), but \(f\) need not admit a section in general.
\end{example}

\begin{example}\label{chap1:exa:principal-bundle-weak-equivalence}
  Let \(G\) be a Lie group and \(P\to M\) a right principal \(G\)-bundle. The functor \(P\rtimes G \xrightarrow{\simeq} \trivial{M}\) is a weak equivalence. That \(P\to M\) is a surjective submersion implies essential surjectivity. Since the action is principal, the diagram
  \[
    \xymatrix{
    P\times G \ar[r]\ar[d] & M\ar[d]\\
    P\times P \ar[r] &M\times M
    }
  \]
  is a pullback square, hence \(P\rtimes G \to \trivial{M}\) is fully faithful. For a similar reason as in the previous example, \(P\rtimes G \to \trivial{M}\) need not be a strong equivalence.
\end{example}

Weak equivalences of Lie groupoids enjoy the following properties:
\begin{proposition}\label{chap1:prop:weak-equivalence-properties}
Let \(G\), \(H\) and \(K\) be Lie groupoids.
\begin{enumerate}
  \item For two functors \(\varphi, \psi\colon G\to H\), if there is a natural transformation \(T\colon \varphi\Rightarrow \psi\), then \(\varphi\) is a weak equivalence if and only if \(\psi\) is.
  \item The composite of weak equivalences is a weak equivalence.
  \item For a weak equivalence \(\varphi\colon G\xrightarrow{\simeq} H\) and a functor \(\psi\colon K\to H\), their weak pullback~\(P\) exists. Moreover, \(P\to K\) is a weak equivalence and \(P_0\to K_0\) a surjective submersion \textup(see~\cite[Proposition~\textup{5.12}]{Moerdijk-Mrcun2003}\textup).
  \item For functors \(\varphi, \psi\colon G\to H\) and a weak equivalence \(\tau\colon H\xrightarrow{\simeq} K\), if there is a natural transformation \(\tau \circ \varphi \Rightarrow\tau \circ \psi\), then there is a natural transformation \(\varphi \Rightarrow \psi\) \textup(see~\cite[Section~\textup{4.1}]{Pronk}\textup).
\end{enumerate}
\end{proposition}

It is not hard to verify the following results.

\begin{lemma}\label{chap1:lem:surj-sub-2-prop}
  Surjective submersions in \(\Mfd\) satisfy the following two properties:
\begin{enumerate}
 \item Let \(L\xrightarrow{f} M\xrightarrow{g} N\) be morphisms in \(\Mfd\). If \(g\circ f\) is a surjective submersion and \(f\) is surjective, then~\(g\) is a surjective submersion;
 \item Let \(f\colon M\to N\) be a surjective submersion, and let \(g\colon L\to N\) be a morphism in \(\Mfd\). Denote the pullback \(M\times_N L\) by \(P\). If \(P\to M\) is a diffeomorphism, then so is \(L\to N\) \textup(see also Remark~\ref{chap3:rem:pullback-along-cover-iso}\textup).
\end{enumerate}
\end{lemma}

\begin{proposition}
The class of weak equivalences satisfies the 2-out-of-3 and the 2-out-of-6 properties:
\begin{enumerate}
  \item For Lie groupoid functors \( A\to B\to C\), if any two of the three functors \(A\to B\), \(B\to C\) and \(A\to C\) are weak equivalences, then so is the third.
  \item For Lie groupoid functors \( A\to B\to C\to D\), if \(A\to C\) and \(B\to D\) are weak equivalences, then every functor is a weak equivalence.
\end{enumerate}
\end{proposition}
\begin{proof}
  (i) is~\cite[Lemma 8.1]{Pronk-Scull2010}. The proof, however, is incorrect for the case that \(A\xrightarrow{\simeq} B\) and \(A\xrightarrow{\simeq} C\) being weak equivalences implies that so is \(B\to C\). To fill this gap, it suffices to show that if \(A\xrightarrow{\simeq} B\) and \(A\xrightarrow{\simeq} C\) are weak equivalences, then the right square below is a pullback diagram\footnote{This does not follow from the usual pullback lemma as suggested by~\cite{Pronk-Scull2010}.}:
  \[
  \xymatrix{
  A_1\ar[d]\ar[r] & B_1\ar[d]\ar[r] & C_1\ar[d]\\
  A_0\times A_0\ar[r] & B_0\times B_0\ar[r] & C_0\times C_0\rlap{\ .}
  }
  \]
  Without loss of generality, we may suppose that \(A_0\to B_0\) is a surjective submersion, otherwise we replace \(A\) by \(A\times^w_B B \). Consider the commutative diagram
\[
\xymatrix{
A_0\times_{C_0,\source} C_1\ar[d]\ar[r] & B_0\times_{C_0,\source} C_1\ar[d]\ar[r] & C_1\ar[d]^{\source}\ar[r]^{\target}& C_0\\
A_0 \ar[r]& B_0\ar[r] &C_0 \rlap{\ ,}
}
\]
where each square is a pullback diagram. Since \(A\to C\) is essentially surjective, so is \(B\to C\) by Lemma~\ref{chap1:lem:surj-sub-2-prop}. Denote by \(B_1'\) the pullback of the following diagram:
  \[
  \xymatrix{
           & C_1\ar[d]\\
  B_0\times B_0\ar[r] & C_0\times C_0 \rlap{\ .}
  }
  \]
  The universal property gives a natural map \(B_1 \to B_1'\). The assumption implies that the lower square and the whole square below are pullback diagrams:
  \[
  \xymatrix{
  A_1\ar[d]_{=}\ar[r] & B_1 \ar[d] \\
  A_1\ar[d]\ar[r] & B_1'\ar[d]\\
  A_0\times A_0 \ar[r]& B_0\times B_0\rlap{\ .}
  }
  \]
 Hence so is the upper square. Since \(A_0\to B_0\) is a surjective submersion, so is the map \(A_1\to B_1'\). Since the pullback of \(B\to B_1'\) along a surjective submersion is a diffeomorphism, Lemma~\ref{chap1:lem:surj-sub-2-prop} shows that \(B_1\to B_1'\) is a diffeomorphism, and we are done.

  To prove (ii), it suffices to show that \(A\to D\) is a weak equivalence; the rest then follows from the 2-out-of-3 property. Without loss of generality, we suppose that \(A_0\to C_0\) is a surjective submersion. The universal property of the pullback implies that the whole square in
  \[\xymatrix{
  A_1\ar[r]\ar[d] & B_1\ar[r]\ar[d]&C_1\ar[r]\ar[d]&D_1\ar[d]\\
  A_0\times A_0\ar[r]&B_0\times B_0\ar[r]&C_0\times C_0\ar[r]&D_0\times D_0
  }
  \]
  is a pullback diagram, so is the left most square. Since \(A_0\times A_0\to C_0\times C_0\) is a surjective submersion, a similar proof as in (i) shows that the right most square is a pullback diagram. Thus every square is a pullback diagram. It remains to show that \(A\to D\) is essential surjective. As above, \(B\to C\) being essential surjective implies that \(C\to D\) is \emph{set-theoretically} essential surjective. Since \(A\to C\) is essential surjective, \(A\to D\) is set-theoretically essential surjective. This proves the proposition \emph{set-theoretically}. Without loss of generality, we suppose that \(B_0\to C_0\) is surjective, then \(B\to D\) being essential surjective implies that \(C\to D\) is essential surjective, and we are done.
\end{proof}

\subsection{Generalised morphisms}

Examples~\ref{chap1:exa:cech-groupoid-weak-equivalence} and~\ref{chap1:exa:principal-bundle-weak-equivalence} show that we cannot invert weak equivalences in general. We can, however, formally invert them.

\begin{definition}
  A \emph{generalised morphism} between Lie groupoids \(G\) and \(H\) is a span of Lie groupoid functors with left leg being a weak equivalence:
  \[
  \xymatrix@1{
  G & K\ar[r]\ar[l]_{\simeq} & H\rlap{\ .}
  }
  \]
  An \emph{equivalence} between generalised morphisms \(G \xleftarrow{\simeq} K\to H\) and \(G \xleftarrow{\simeq} K'\to H\) is a 2\nbdash{}commutative diagram
  \begin{equation}\label{chap1:eq:GEN-2-morphisms}
  \begin{gathered}
  \xymatrix{
          & K\ar[dl]_{\simeq}\ar[dr] & \\
        G & L\ar[u]_{\simeq}\ar[d]^{\simeq} & H\\
          & K'\ar[lu]^{\simeq}\ar[ru]&
  }
  \end{gathered}
  \end{equation}
  with weak equivalences \(L\xrightarrow{\simeq} K\) and \(L\xrightarrow{\simeq} K'\).
\end{definition}

Generalised morphisms are composed by weak pullback:
\[
\xymatrix@=0.6cm{
  &   & K\times^w_{H} M\ar[ld]_{\simeq}\ar[rd] & & \\
  & K\ar[ld]_{\simeq}\ar[rd] &   & M\ar[ld]_{\simeq}\ar[rd]&   \\
G &   & H &   & K \rlap{\ .}
}
\]
The composition is not associative, but associative up to a canonical equivalence by the universal property of weak pullbacks.

\begin{definition}
The \emph{spanisation} of a Lie groupoid functor \(f\colon G\to H\) is the generalised morphism \(\mathfrak{G}(f) \colon G\xleftarrow{\id} G\xrightarrow{f} H\).
\end{definition}

\begin{proposition}[{\cite[Corollary 2.13]{Moerdijk-Mrcun2005}}]
  Lie groupoids and generalised morphism up to equivalences form a category, denoted by \(\Gen\). The spanisation \(\mathfrak{G}\) extends to a functor \(\Gpd \to \Gen\). Moreover, a generalised morphism is an isomorphism in \(\Gen\) if and only if its right leg is also a weak equivalence.
\end{proposition}

\begin{remark}
  The category \(\Gen\) is constructed by Gabriel--Zisman's calculus of fractions~\cite{Gabriel-Zisman} in the category of Lie groupoids; here the morphisms are functors up to natural transformations. This general machinery allows to formally invert a class of weak equivalences in a category. Proposition~\ref{chap1:prop:weak-equivalence-properties} verifies the conditions for a calculus of fractions. The 2-out-of-6 property implies the last statement.
\end{remark}

\begin{example}\label{chap1:exa:principal-bundle-cocycle}
  Let \(G\) be a Lie group and \(M\) a manifold. A right principal \(G\)-bundle over \(M\) can be given by local trivialization data \(\{g_{ij}\}\) with respect to an open cover \(\{U_i\}\). Let \(C(U)\) be the associated \v{C}ech groupoid. Then \(\{g_{ij}\}\) defines a Lie groupoid functor \(C(U)\to G\), and we obtain a generalised morphism
  \[
  \trivial{M} \xleftarrow{\simeq} C(U) \to G.
  \]
  Furthermore, if two trivializations have a common refinement, then the corresponding two generalised morphisms are equivalent. Under some mild assumptions, the geometric realisation of the nerve of this generalised morphism gives the classifying map \(M\to BG\), where \(BG\) is the classifying space of \(G\); see~\cite{Segal}.
\end{example}

\begin{example}\label{chap1:exa:principal-bundle-action-groupoid}
  Let \(G\) be a Lie group and \(P\to M\) a right principal \(G\)\nbdash{}bundle. There is a generalised morphism
  \[
  \trivial{M} \xleftarrow{\simeq} P\rtimes G\to G.
  \]
  Under some mild assumptions, the geometric realisation of the nerve of this generalised morphism gives the classifying map \(M\to BG\); see~\cite{May}.
\end{example}

Moreover, Pronk~\cite{Pronk} showed that there is a 2\nbdash{}category if we remember the higher equivalence relations.  Proposition~\ref{chap1:prop:weak-equivalence-properties} allows for a calculus of fractions in the 2\nbdash{}category setting.

\begin{theorem}
  The 2\nbdash{}category of Lie groupoids \(\GPD\) admits a 2\nbdash{}categorical calculus of fractions with respect to weak equivalences. We denote the resulting 2\nbdash{}category by \(\GEN\). Moreover, the spanisation \(\mathfrak{G}\) extends to a 2\nbdash{}functor \(\GPD\to \GEN\).
\end{theorem}

In the 2\nbdash{}category \(\GEN\), objects are Lie groupoids, 1-morphisms are generalised morphisms, and 2\nbdash{}morphisms are diagrams of the form~\eqref{chap1:eq:GEN-2-morphisms} \emph{up to a higher equivalence relation}.

\subsection{Anafunctors}
Using the anafunctors of Makkai~\cite{Makkai}, Roberts~\cite{Roberts} constructed a slight variant of the 2\nbdash{}category \(\GEN\). This approach uses the strong pullback and gets representatives of 2\nbdash{}morphisms instead of equivalence classes.

\begin{definition}
An \emph{anafunctor} of Lie groupoids from \(G\) to \(H\) is a generalised morphism
  \[
  \xymatrix@1{ G &K\ar@{->>}[l]_{\simeq}\ar[r] &  H}
  \]
such that \(K_0\to G_0\) is a surjective submersion. An \emph{ananatural transformation} from an anafunctor \(\xymatrix@1{ G &K\ar@{->>}[l]_{\simeq}\ar[r] &  H}\) to \(\xymatrix@1{ G &K'\ar@{->>}[l]_{\simeq}\ar[r] &  H}\) is a natural transformation between the composite functors \(K\times_G K'\to K\to  H\) and \(K\times_G K'\to  K' \to  H\).
\end{definition}

Anafunctors are composed by the strong pullback,
\[
\xymatrix@=0.6cm{
  &   & K\times_{H} M\ar@{->>}[ld]_{\simeq}\ar[rd] & & \\
  & K\ar@{->>}[ld]_{\simeq}\ar[rd] &   & M\ar@{->>}[ld]_{\simeq}\ar[rd]&   \\
G &   & H &   & K\rlap{\ ;}
}
\]
the formulas for vertical and horizontal compositions of natural transformations are involved; see~\cite{Makkai,Roberts} for details.

\begin{proposition}[\cite{Roberts}]
Lie groupoids, anafunctors, and ananatural transformations form a 2\nbdash{}category, denoted by \(\ANA\).
\end{proposition}

\begin{remark}\label{chap1:rem:ana-gen}
 A generalised morphism \(G\xleftarrow{\simeq} K\to H\) is equivalent to the anafunctor:
\[
 \xymatrix@1{ G & G\times^w_G K\ar@{->>}[l]_-{\simeq}\ar[r] &  H}.
\]
\end{remark}

\section{Groupoid actions and HS bibundles}\label{chap1:sec:action-bibundle}

Groupoid actions, principal bundles, and HS bibundles are introduced in this section. Then we construct the third 2\nbdash{}category of Lie groupoids by HS bibundles.

\subsection{Lie groupoid actions}
\begin{definition}
  Let \(G\) be a Lie groupoid, \(P\) a manifold, and \(J\colon P\to G_0\) a smooth map, called \emph{moment map}. A \emph{right action} of \(G\) on \(P\) along \(J\) is given by an action map
\[
\action\colon P\times_{G_0,\target}G_1\to P,
\]
which is also denoted by \(\action(p, g)=pg\) or \( p\cdot g\). The action map must satisfy
\[
\source(g)=J(pg),\quad p\unit(J(p))=p, \quad (pg)g'=p(gg'),
\]
for \(g, g'\in G_1\) and \(p\in P\) with \(\target(g)=J(p)\) and \(\target(g')=\source(g)\).

Given two right \(G\)\nbdash{}actions \((J_1,\action_1)\) on \(P\) and \((J_2,\action_2)\) on \(Q\), a map \(f\colon P\to Q\) is \emph{equivariant} if \(J_1=J_2\circ f\) and \(f(pg)=f(p)g\) for \((p, g)\in P\times_{G_0,\target} G_1 \).
\end{definition}

\begin{remark}\label{chap1:def:inverse-action}
  Left actions are defined analogously. Given a right \(G\)\nbdash{}action on \(P\) along~\(J\), then \(g p\coloneqq p g^{-1}\) defines a left action on $P$ along $J$, and vise versa.
\end{remark}

An action of \(G\) on \(P\) induces an equivalence relation on \(P\). The quotient space \(P/G\), which is the coequaliser of the diagram
\[
\xymatrix@1{
P\times_{G_0,\target} G_1\ar@<0.5ex>[r]^-{\action}\ar@<-0.5ex>[r]_-{\pr_1} & P,
}
\]
is usually ill-behaved even for a Lie group action.\footnote{For instance, consider the \(\bZ\)\nbdash{}action on \(\bR/\bZ\) by
\((r, n)\mapsto r+n\alpha\) for an irrational \(\alpha\).} We may, however, form a new Lie groupoid which is a better replacement of the quotient space.

\begin{definition}
  Given a right action \((J, \action)\) of a Lie groupoid \(G\) on \(P\), its \emph{action groupoid} \(P\rtimes G \) is a Lie groupoid with \((P\rtimes G)_0=P\) and \((P\rtimes G)_1= P\times_{G_0,\target} G_1\). The source map is~\(\action\), the target map is \(\pr_1\colon P\times_{G_0,\target} G_1\to P\), and the composition is given by the composition in \(G\). There is a Lie groupoid functor
  \[
  \pi\colon P\rtimes G \to G
  \]
 given by \(J\colon P\to G_0\) and \(\pr_2\colon P\times_{G_0,\target} G_1\to G_1\).
\end{definition}

\begin{remark}\label{chap1:rem:gpd-action-imaginary-picture}
  We may think of \(G_1\) as the space of arrows from black points to black points and \(P\) as the space of arrows from black points to white points. The moment map \(J\) specifies the source of an arrow in~\(P\). The action gives us a way to compose an arrow in~\(P\) and an arrow in \(G_1\) if their source and target match; see Figure~\ref{chap1:fig:groupoid-action}.
\begin{figure}[htbp]
  \centering
  \begin{tikzpicture}%
  [>=latex',every label/.style={scale=0.8},scale=0.8]
  \node[mydot,fill=black] at (6,0) (x0) {};
  \node[mydot,fill=black] at (4,0) (x1) {};
  \node[mydot,fill=black] at (2,0) (x2) {};
  \node[mydot]            at (0,0) (x3) {};
  \path[->]
     (x0) edge[out=160, in=20] node[scale=0.8,below]{$g'$} (x1)
	 edge[out=150, in=30] node[scale=0.8,above]{$gg'$}(x2)
   edge[out=140, in=40] node[scale=0.8,above]{$pgg'$}(x3)
     (x1) edge[out=160, in=20] node[scale=0.8,below]{$g$} (x2)
   edge[out=150, in=30] node[scale=0.8,above]{$pg$}(x3)
     (x2) edge[out=160, in=20] node[scale=0.8,below]{$p$} (x3);
  \end{tikzpicture}
    \caption{Groupoid action}\label{chap1:fig:groupoid-action}
\end{figure}
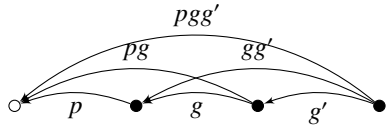
\end{remark}

\begin{example}
 Let \(G\) be a Lie groupoid. There is a natural right action on \(G_0\) with moment map \(\id\colon G_0\to G_0\) and with action map \(\source\colon G_1\cong G_0\times_{\id,G_0,\target} G_1\to G_0\). The action groupoid of this action is \(G\). This action may be regarded as a universal action. Given a right \(G\) action \((J, \action)\) on \(P\), then \(J\colon P\to G_0\) is an equivariant map.
\end{example}

\begin{example}\label{chap1:exa:two-actions-pullback}
  Let a Lie groupoid \(G\) act on manifolds \(P\) and \(Q\). Suppose that either \(P\to G_0\) or \(Q\to G_0\) is a surjective submersion. Then there is a \(G\)\nbdash{}action on \(P\times_{G_0} Q\) given by \((p, q)\cdot g= (p\cdot g, q\cdot g)\). This is the product in the category of spaces with a \(G\)\nbdash{}action.
\end{example}

\subsection{Principal bundles}
\begin{definition}\label{chap1:def:principal-bundle}
  Let \(G\) be a Lie groupoid. A right \(G\)\nbdash{}\emph{bundle} over a manifold \(M\) is a manifold \(P\) with a right \(G\)-action \((J, \action)\) and an invariant map \(\kappa\colon P\to M\); that is, \(\kappa(pg)=\kappa(p)\) for \(p\in P\) and \(g\in G_1\) with \(J(p)=\target(g)\). Left bundles are defined similarly.
  
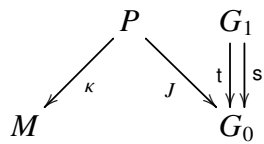
\[
  \xymatrix{
   & P\ar[dl]^{\kappa}\ar[dr]_{J} & G_1\ar@<-0.5ex>[d]_{\target}\ar@<0.5ex>[d]^{\source}\\
  M&   & G_0
  }
  \]
  We say that \(P\) is \emph{principal} or \(G\) acts on \(P\) principally if
  \begin{enumerate}
    \item the map \(\kappa\) is a surjective submersion;
    \item the \emph{shear map} below is a diffeomorphism:
    \begin{equation}\label{chap1:eq:shear-map-(p,g)->(p,pg)}
    (\pr_1, \action)\colon P\times_{G_0,\target} G_1\to P\times_M P, \quad (p, g) \mapsto (p, pg).
    \end{equation}
  \end{enumerate}
  A morphism between principal \(G\)-bundles \(\kappa\colon P\to M\) and \(\kappa'\colon P'\to M'\) is an equivariant map \(f\colon P\to P'\) that commutes with all structure maps.
\end{definition}

\begin{remark}
  Let \(G\) act on \(P\). Example~\ref{chap1:exa:G1=G0-times-G0-over-G} yields a weak pullback square
  \[
  \xymatrix{
  \trivial{P\times_{G_0,\target} G_1}\ar[r]\ar[d]  & \trivial{P}\ar[d] \\
  \trivial{P}\ar[r]& P\rtimes G \rlap{\ .}
  }
  \]
  Comparing with~\eqref{chap1:eq:shear-map-(p,g)->(p,pg)}, we may regard every action as a principal bundle \(\trivial{P}\to P\rtimes G\) with a Lie groupoid as base instead of a manifold; see also~\cite{Carchedi}.
\end{remark}

\begin{remark}
Let \(\kappa\colon P\to M\) be a principal \(G\)-bundle. The shear map~\eqref{chap1:eq:shear-map-(p,g)->(p,pg)} induces an isomorphism \(P\rtimes G\cong C(\kappa)\), where \(C(\kappa)\) is the \v{C}ech groupoid. Thus we obtain an isomorphism of quotient spaces \(P/G\cong M\) \textup(the quotient space of \(C(\kappa)\) is \(M\); see Lemma~\ref{chap3:lem:subcanonical-effective-epi}\textup).
\end{remark}

\begin{example}
  A Lie groupoid \(G\) acts on \(G_1\) on the right with \(J=\source\colon G_1\to G_0\) and with action map equal to the composition in~\(G\). Moreover, \(\target\colon G_1 \to G_0\) is a right principal \(G\)\nbdash{}bundle. Similarly, \(\source\colon G_1 \to G_0\) is a right principal \(G\)\nbdash{}bundle, and these two actions commute.
\end{example}

\begin{definition}
 Let \(\kappa\colon P\to M\) be a right \(G\)-bundle. Given a smooth map \(f\colon B\to M\), there is an \emph{induced bundle} structure on \(f^*P=B\times_M P\to B\) with action given by the action on~\(P\). If \(P\to M\) is principal, then so is \(f^* P\to B\). A right principal bundle is \emph{trivial} if it is of the form \(f^* G_1\) for some map \(f\colon B\to G_0\).
\end{definition}

The following lemma characterises trivial principal bundles. In particular, every principal bundle is locally trivial.
\begin{lemma}[{\cite[Lemma 3.19]{Lerman}}]
A right principal \(G\)-bundle \(\kappa\colon P\to M\) is trivial if and only if \(\kappa\) admits a section.
\end{lemma}

\begin{lemma}\label{chap1:lem:principal-equivariant-map-is-iso}
  An equivariant map \(P\to P'\) between principal \(G\)\nbdash{}bundles over the same manifold \(M\) is an isomorphism. If \((f_1, f_0)\) is an equivariant map between principal \(G\)\nbdash{}bundles \(P\to M\) and \(Q\to N\), then \(P\) is isomorphic to the induced bundle \(f_0^*Q\); in particular, the diagram below is a pullback square:
  \begin{equation}\label{chap1:eq:principal-bundle-equi-map}
  \begin{gathered}
  \xymatrix{
  P\ar[r]\ar[d]& M\ar[d]\\
  Q\ar[r]& N\rlap{\ .}
  }
  \end{gathered}
  \end{equation}
\end{lemma}
\begin{proof}
  The first statement appears in~\cite[p.~146]{Moerdijk-Mrcun2003}. Since the induced map \(P\to f_0^*Q\) is equivariant, the second statement follows.
\end{proof}

The following two lemmas will be used extensively.

\begin{lemma}\label{chap1:lem:invariant-cover-descent}
 Let \(\kappa\colon P\to M\) be a principal \(G\)-bundle. Let \(N\) be a manifold and \(f\colon P\to N\) an invariant surjective submersion. Then the induced map \(M\to N\) is a surjective submersion.
\end{lemma}
\begin{proof}
  Since \(M\) is the quotient \(P/G\), an invariant map \(P\to N\) induces a map \(M\to N\). Since the composite map \(P\to M\to N\) is a surjective submersion, the map \(M\to N\) is a surjective submersion by Lemma~\ref{chap1:lem:surj-sub-2-prop}.
\end{proof}

\begin{lemma}[{\cite[Lemma 5.35]{Moerdijk-Mrcun2003}}]\label{chap1:lem:action-over-principal-is-principal}
  Let \(G\) be a Lie groupoid and \(\kappa\colon P\to M\) a right principal \(G\)\nbdash{}bundle. Let \(Q\) be a manifold with a right \(G\)-action, and let \(f\colon Q\to P\) be an equivariant surjective submersion. Then the quotient \(Q/G\) is a manifold, and the induced map \(Q/G\to M \) is a surjective submersion; moreover, \(Q\to Q/G\) is a principal bundle.
\end{lemma}

\begin{corollary}\label{chap1:cor:G-bundles-P-times-Q-pullback}
  Suppose that the \(G\)-action on \(P\) in Example~\ref{chap1:exa:two-actions-pullback} is principal and that \(Q\to G_0\) is a surjective submersion. Then the induced action on \(P\times_{G_0}Q\) is principal, and we have a pullback square like~\eqref{chap1:eq:principal-bundle-equi-map}.
\end{corollary}

\subsection{HS bibundles}

\begin{definition}
  Let \(G\) and \(H\) be Lie groupoids. A \(G\)-\(H\) \emph{bibundle} is a manifold \(P\) that carries two bundle structures such that the corresponding two actions commute. There is a right \(H\)-bundle structure over \(J_l\colon P\to G_0\) with moment map \(J_r\colon P\to H_0\), and a left \(G\)-bundle structure over \(J_r\colon P\to H_0\) with moment map \(J_l\colon P\to G_0\). These two actions commute; that is, \((gp)h=g(ph)\) for \(g\in G_1\), \(p\in P\), and \(h\in H_1\) whenever either side is well-defined.
 \[
  \xymatrix{
  G_1\ar@<-0.5ex>[d]_{\target}\ar@<0.5ex>[d]^{\source}  & P\ar[dl]^{J_l}\ar[dr]_{J_r} & H_1\ar@<-0.5ex>[d]_{\target}\ar@<0.5ex>[d]^{\source}\\
  G_0&   & H_0
  }
  \]
  A morphism \(P\to P'\) of \(G\)-\(H\) bibundles is a smooth map which is equivariant with respect to both actions. If the \(H\)-bundle structure is principal, we call \(P\) a \(G\)-\(H\) \emph{Hilsum--Skandalis bibundle} \textup(HS for short\textup) or an HS bibundle from \(G\) to~\(H\).
\end{definition}

\begin{remark}
   Every morphism of right principal \(G\)-\(H\) bibundles \(P\to P'\) is an isomorphism by Lemma~\ref{chap1:lem:principal-equivariant-map-is-iso}.
\end{remark}

\begin{remark}
  There is a more general notion of bimodules between two categories, which we will review in Chapter~\ref{chap4}.
\end{remark}

\begin{remark}
By Remark~\ref{chap1:def:inverse-action}, for a \(G\)-\(H\) bibundle \(P\) there is a natural \(H\)-\(G\) bibundle structure on \(P\), denoted by \(\bar{P}\). If \(P\) is right principal, then \(\bar{P}\) is left principal.
\end{remark}

\begin{example}
  Let \(G\) be a Lie groupoid and \(M\) a manifold. A \textup(HS) bibundle between \(\trivial{M}\) and \(G\) is simply a \textup(principal\textup) \(G\)-bundle over \(M\). In particular, a \(G\)-action on a manifold \(X\) gives a bibundle between~\(*\), the one-object trivial groupoid, and~\(G\).
\end{example}

\begin{definition}
  Let \(P\) be an HS bibundle from \(G\) to~\(H\) and \(Q\) an HS bibundle from \(H\) to~\(K\). The \emph{composite} HS bibundle \(P\otimes_H Q\) is defined as follows. There is a diagonal \(H\)-action on \(P\times_{H_0} Q\) by
  \[
  (p, q) \cdot h \coloneqq (p\cdot h, h^{-1}\cdot q).
  \]
  Since the \(H\)-action on \(P\) is principal, Corollary~\ref{chap1:cor:G-bundles-P-times-Q-pullback} implies that the quotient is a manifold, denoted by \(P\otimes_H Q\). There are induced actions of \(G\) and \(K\) on \(P\otimes_H Q\). Furthermore, \(P\otimes_H Q\) is an HS bibundle from \(G\) to \(K\). For more details, see~\cite[p. 166]{Moerdijk-Mrcun2005}.
\end{definition}

\begin{remark}
The usual tensor product of two bimodules of algebras is a similar quotient \textup(colimit\textup) construction. The universal property of colimits implies that the composition of HS bibundles is associative up to a canonical isomorphism.
\end{remark}

We may think of HS bibundles as an extended notion of functors between Lie groupoids. The \(G\)-\(G\) bibundle \(G_1\) is called the \emph{unit bibundle}, denoted by \(U(G)\). The unit bibundle acts as a unit up to canonical isomorphisms for the composition of bibundles. Therefore, Lie groupoids and isomorphism classes of HS bibundles with their composition form a category, denoted by~\(\Bun\).

\begin{definition}
Let \(P\) be a two-sided principal \(G\)-\(H\) bibundle. We say that \(G\) and \(H\) are \emph{Morita equivalent} and \(P\) is a Morita equivalence between \(G\) and \(H\).
\end{definition}

\begin{proposition}[{\cite[Proposition 2.9]{Moerdijk-Mrcun2005}}]
  An HS-bibundle \(P\) from \(G\) to \(H\) is an isomorphism in the category \(\Bun\) if and only if \(P\) is a Morita equivalence.
\end{proposition}

\begin{definition}
The \emph{bundlisation} of a Lie groupoid functor \(f\colon G\to H\) is an HS bibundle from \(G\) to \(H\) given by the manifold
\[
\Bund(f)\coloneqq G_0\times_{H_0, \target} H_1
\]
with moment maps \(\pr_1\colon G_0\times_{H_0,\target} H_1 \to G_0\) and \(\source\circ\pr_2\colon G_0 \times_{H_0} H_1\to H_0\). The right \(H\)\nbdash{}action is given by the \(H\)-action on \(H_1\), and the left \(G\)-action is given by
\[
g\cdot (x, h)\coloneqq (\target(g), f(g)h)
\]
for \(g\in G_1, x\in G_0\) with \(\source(g)=x\), and \(h\in H_1\) with \(f(x)=\target(h)\).
\end{definition}

It is clear that \(\Bund(\id_G)\) is the unit bibundle. It is not hard to verify that \(\Bund\) preserves the composition up to a canonical isomorphism.
\begin{proposition}[{\cite[Proposition 2.10]{Moerdijk-Mrcun2005}}]
 The bundlisation construction extends to a functor \( \Bund\colon \Gpd \to \Bun\) which is the identity on objects. Moreover, a morphism in \(\Gpd\) is a weak equivalence if and only if its bundlisation is an isomorphism in \(\Bun\).
\end{proposition}

\begin{proposition}[{\cite[Proposition II.1.5]{Mrcun} and~\cite[Lemma 3.36]{Lerman}}]\label{chap1:prop:is-bundlisation}
  An HS bibundle \(P\) between Lie groupoids \(G\) and \(H\) is isomorphic to \(\Bund(f)\) for some functor \(f\colon G\to H\) if and only if \(P\to G_0\) admits a smooth section.
\end{proposition}

\begin{corollary}
  Every HS bibundle \(P\) from \(G\) to a trivial Lie groupoid \(\trivial{M}\) is isomorphic to \(\Bund(f)\) for some functor \(f\colon G\to \trivial{M}\).
\end{corollary}
\begin{proof}
  Since \(P\) is right principal, we have \(P\times_{G_0} P\cong P\times_M M\cong P\). Thus \(P\to G_0\) is a diffeomorphism by Lemma~\ref{chap1:lem:surj-sub-2-prop}. Then the statement follows from Proposition~\ref{chap1:prop:is-bundlisation}.
\end{proof}

\begin{proposition}\label{chap1:prop:2-morphism-of-bundles}
  Let \(f, g\colon G\to H\) be Lie groupoid functors. Morphisms of HS bibundles \(\Bund(f)\to \Bund(g)\) are in a 1-1 correspondence with natural transformations \(f\Rightarrow g\).
\end{proposition}
\begin{proof}
A morphism of HS bibundles \(\tau\colon\Bund(f)\to \Bund(g)\) must have the form \((x, \gamma)\mapsto (x, T(x)\gamma)\), where \(T(x) \in H_1\). It is easily checked that \(T\) defines a natural transformation \(f\Rightarrow g\). The argument may be reversed, and we are done.
\end{proof}

In the category \(\Bun\), unit identities and associativity hold up to canonical isomorphisms. This suggests that we should remember these isomorphisms as well, then we obtain a 2\nbdash{}category.

\begin{proposition}
  Lie groupoids, HS bibundles, and morphism of HS bibundles form a 2\nbdash{}category, denoted by \(\BUN\). The functor \(\Bund\) extends to a 2\nbdash{}functor \(\GPD\to \BUN\) \textup(see~\cite{Pronk} and \cite{Carchedi}\textup).
\end{proposition}

Since HS bibundles are morphisms in a 2\nbdash{}category, we also call them \emph{HS morphisms} to emphasise the morphism feature. The 2\nbdash{}category \(\BUN\) will play a key role in the study of Lie 2\nbdash{}groupoids.

\subsection{2-Pullbacks in the 2-category of HS bibundles}

For some simple cases, we show that 2\nbdash{}pullbacks in the 2\nbdash{}category \(\BUN\) exist.

\begin{proposition}[{\cite[Proposition 31]{Schommer-Pries}}]
  Let \(P\) be an HS bibundle \(G \to K\) and let \(Q\) be an HS bibundle \(H\to K\). Suppose that \(P\to K_0\) or \(Q\to K_0\) is a surjective submersion. Then the 2\nbdash{}pullback of \(P\) and \(Q\) in the 2\nbdash{}category \(\BUN\) exists.
\end{proposition}

An explicit 2\nbdash{}pullback square is constructed similar to the composition of HS bibundles. There is a right \(K\)-action on \(P\times_{K_0} Q\) given by
\[
 (p, q)\cdot k=(p\cdot k, q\cdot k).
\]
Corollary~\ref{chap1:cor:G-bundles-P-times-Q-pullback} implies that the quotient \((P\times_{K_0} Q)/K\) is a manifold. There are commuting left actions of \(G\) and \(H\) on \((P\times_{K_0} Q)/K\). Denote by \(L\) the action groupoid of \(G\times H\) on \((P\times_{K_0} Q)/K\). There are two Lie groupoid functors \(L\to G\) and \(L\to H\) given by projection, hence two HS bibundles \(L\to G\) and \(L\to H\). There is an canonical isomorphism between the composite HS bibundles \(L\to G\to K\) and \(L\to H\to K\), and this gives a 2\nbdash{}pullback square.

\begin{remark}
  The proof of~\cite[Lemma 30]{Schommer-Pries} has a mistake, namely, the maps \(f_0\) and \(g_0\) constructed there need not be transversal. It goes through, however, under our stronger assumption.
\end{remark}

For a weak pullback square in \(\GPD\) \textup(under the assumption in Proposition~\ref{chap1:prop:weak-pullback-exists-assumption}\textup), the 2\nbdash{}pullback square in \(\BUN\) constructed above is isomorphic to the square given by bundlisation. In particular, we have the following lemma.

\begin{corollary}\label{chap1:cor:pullback-bibundle-to-M=pullback}
 Let \(M\) be a manifold. Let \(P\colon G\to \trivial{M}\) and \(Q\colon H\to \trivial{M}\) be two HS bibundles. Assume that \(P=\Bund(\varphi)\) and \(Q=\Bund(\psi)\) for Lie groupoid functors \(\varphi\colon G\to \trivial{M}\) and \(\psi\colon H\to \trivial{M}\). Suppose that either \(\varphi_0\) or \(\psi_0\) is a surjective submersion. Then the 2\nbdash{}pullback constructed above, the weak pullback \(G\times^w_{\trivial{M}} H\), and the strong pullback \(G\times_{\trivial{M}} H \) are all isomorphic.
\end{corollary}

\section{Comparison of 2-categories of Lie groupoids}\label{chap1:sec:comparison-of-2-cat-of-lie-groupoids}

We explain that the three 2\nbdash{}categories \(\GEN, \ANA, \BUN\) are equivalent. Recall~\cite{Street} that a 2\nbdash{}functor \(f\colon \Cat[A]\to \Cat[B]\) between 2\nbdash{}categories is an \emph{equivalence} if \(f\) is
\begin{itemize}
  \item essentially surjective on objects; that is, for every object \(X\) of \(\Cat[B]\) there is an object \(Y\) of \(\Cat[A]\) such that \(X\) is equivalent to \(f(Y)\) in \(\Cat[B]\);
  \item locally equivalent; that is, the functors \(\Cat[A](X, Y)\to \Cat[B](f(X), f(Y))\) are equivalences of categories for all objects \(X\) and \(Y\) of \(\Cat[A]\).
\end{itemize}

We first compare the 2\nbdash{}categories \(\GEN\) and \(\ANA\). Although the size of 1-morphisms of \(\ANA\) is smaller than that of \(\GEN\), Remark~\ref{chap1:rem:ana-gen} shows that every 1-morphism of \(\GEN\) is equivalent to a 1-morphism that is in \(\ANA\). The following theorem is a direct consequence of the universal property of the calculus of fractions~\cite[Proposition 24]{Pronk}.

\begin{theorem}[{\cite[Theorem 7.2]{Roberts}}]
  The 2\nbdash{}functor \(\Gpd\to \ANA\) extends to an equivalence of 2\nbdash{}categories \(\GEN\to\ANA\).
\end{theorem}

We now compare HS bibundles and generalised morphisms of Lie groupoids. Given an HS bibundle  \(P\colon G\to H\), we form a \emph{biaction groupoid} \(G\ltimes P\rtimes H\) with
\[
  (G\ltimes P\rtimes H)_0=P,\quad (G\ltimes P\rtimes H)_1=G_1\times_{\source,G_0} P\times_{H_0,\target} H_1,
\]
and with composition given by the compositions in \(G\) and \(H\). There is a span
\[
   G\leftarrow G\ltimes P\rtimes H\to H
\]
given by the projections \(G_1\times_{G_0} P\times_{H_0} H_1\to G_1\) and \(G_1\times_{G_0} P\times_{H_0} H_1\to H_1\). Since the \(H\)\nbdash{}action is principal over \(G_0\), the functor \(G\ltimes P\rtimes H\to G\) is a weak equivalence. Thus we get a generalised morphism from \(G\) to~\(H\).

\begin{proposition}[{\cite[Section 2.6]{Moerdijk-Mrcun2005}}]
  The biaction groupoid construction induces a functor \(\Bun \to \Gen\). The bundlisation extends to a functor \(\Gen\to\Bun\). These two functors are mutually inverse functors.
\end{proposition}

\begin{remark}
 It follows that the two generalised morphisms in Examples~\ref{chap1:exa:principal-bundle-cocycle} and~\ref{chap1:exa:principal-bundle-action-groupoid} are equivalent.
\end{remark}

\begin{theorem}\label{chap1:thm:GPD=BUN}
  The 2\nbdash{}functor \(\Bund\colon \GPD\to \BUN\) extends to an equivalence of 2\nbdash{}categories \(\GEN\to \BUN\).
\end{theorem}

\begin{remark}
   This well-known theorem follows from the universal property of the 2\nbdash{}categorical calculus of fractions~\cite[Proposition 24]{Pronk} and Proposition~\ref{chap1:prop:2-morphism-of-bundles}. See also~\cite{Carchedi} for a direct proof of the equivalence of \(\GEN\) and \(\BUN\). We notice that the 2\nbdash{}category of \emph{differential stacks} is also equivalent to these three 2\nbdash{}categories. For a detailed treatment of differential stacks, see~\cite{Behrend-Xu, Carchedi}. See~\cite{Carchedi} for a review on the equivalence of these 2\nbdash{}categories.
\end{remark}

\chapter{Simplicial Sets}
\thispagestyle{empty}

This chapter introduces the basics of simplicial sets. The lifting properties of simplicial sets will be central to our approach to higher groupoids and 2-groupoids.

\section{Preliminaries on presheaves}\label{chap2:sec:presheaf}

We collect some basic facts on presheaves that will be used later. All the material is standard; see, for instance, \cite{MacLane}. Let~\(\Cat\) be a locally small category throughout this section.

\begin{definition}
   The \emph{category of presheaves} on \(\Cat\) is the functor category \([\Cat^\op, \Sets]\).
\end{definition}

More concretely, a presheaf \(f\) is a functor \( f\colon \Cat^\op \to \Sets\) from the opposite category of \(\Cat\) to \(\Sets\), the category of sets. A morphism of presheaves is a natural transformation between the corresponding functors. In particular, a morphism of presheaves \(\tau\colon f\to g\) is a \emph{subpresheaf} if for each object \(C\in \Cat\) the map \(\tau\colon f(C)\to g(C) \) is an injection.

\begin{example}\label{chap2:poset-opensubsets-presheaf}
  Given a partially ordered set (poset) \((S, \prec)\), we view it as a category with objects given by elements in \(S\), and with a unique arrow from \(i\) to \(j\) if \(j\prec i\).\footnote{Our convention is the same as in~\cite{Zhu:ngpd}, but differs from the perhaps more common one.} Let \((\Cat[U],\subset)\) be the poset of open subsets of a topological space \(X\). A presheaf on \(\Cat[U]^\op\) as defined above is a presheaf on \(X\) in the usual sense.
\end{example}

\begin{definition}
The presheaf \emph{represented} by an object \(C\in\Cat\) is the presheaf defined by
  \[
    T \mapsto \hom(T, C), \quad  T\in \Cat,
  \]
denoted by \(\hom(\blank, C)\) or \(h_C\). A presheaf isomorphic to one of this form is called \emph{representable presheaf}. Given a morphism \(C\to D\) in \(\Cat\), there is a morphism of presheaves \(h_C\to h_D\). This defines the \emph{Yoneda functor}
\[
 h\colon \Cat \to [\Cat^\op, \Sets].
\]
\end{definition}

\begin{lemma}[Yoneda Lemma]\label{chap2:lem:Yoneda}
  The Yoneda functor is full and faithful. Moreover, given \(C\in \Cat\) and a presheaf~\(f\), there is a canonical isomorphism
  \[
    [\Cat^\op, \Sets](h_C, f)=f(X),
  \]
  which is natural in \(C\) and \(f\).
\end{lemma}

\begin{remark}
  The category of presheaves on \(\Cat\) has all limits and colimits. They are both computed pointwise. The Yoneda functor preserves and reflects all limits, but does not preserve colimits in general.
\end{remark}

\begin{remark}
  We may use the Yoneda embedding
  \[
  h \colon \Cat \to [\Cat^\op, \Sets],
  \]
  to embed any category~\(\Cat\) into a complete category. Thus any diagram in~\(\Cat\) has a presheaf on~\(\Cat\) as a limit. It has a limit in~\(\Cat\) if and only if this presheaf is representable by some object of~\(\Cat\).
\end{remark}

\begin{remark}
There is a dual version of copresheaves and corepresentable copresheaves. Any diagram in~\(\Cat\) has a copresheaf as a colimit. It has a colimit in~\(\Cat\) if and only if this copresheaf is corepresentable.
\end{remark}

\begin{definition}
  Let \(f\) be a presheaf on \(\Cat\). For an object \(T\in \Cat\), elements of \(f(T)\) are called \emph{\(T\)-elements} of~\(f\). In particular, \(T\)-elements of \(h_C\) are also called \(T\)-elements of~\(C\).
\end{definition}

By the Yoneda Lemma, a \(T\)-element of \(f\) is the same as a presheaf morphism \(h_T\to f\), and a \(T\)-element of \(C\) is the same as a morphism \(T\to C\). A morphism between presheaves \(f\to g\) is the same as a family of maps from \(T\)-elements of \(f\) to \(T\)-elements of \(g\) that is functorial in \(T\). This point of view will be important for us: when there is no notion of elements of an object in the category \(\Cat\), we can work with \(T\)-elements.

\begin{definition}
The \emph{category of elements} of a presheaf \(f\) on \(\Cat\) is the category
\begin{itemize}
 \item whose objects are \(T\)-elements \(h_T\to f\), for all \(T\in \Cat\), and
 \item morphisms from \(h_T\to f\) to \(h_S\to f\) are commutative triangles of presheaves
\[
  \xymatrix{
  h_T\ar[rd]\ar[rr] & & h_S\ar[ld]\\
      &f\rlap{\ .}&
 }
\]
\end{itemize}
\end{definition}

The following co-Yoneda Lemma or density theorem~\cite[Exercise 3, p.~62]{MacLane} implies that every presheaf is a colimit of representables.

\begin{lemma}\label{chap2:lem:co-Yoneda}
  Let \(f\) be a presheaf on \(\Cat\). There is a canonical isomorphism
\[
  f \cong\colim_{h_T\to f} h_T,
\]
where the colimit is taken over the category of elements.
\end{lemma}

\section{Simplicial sets}

In this section, we recall the notion of a simplicial set. Our main references are~\cite{Goerss-Jardine,Hovey}.

\begin{definition}
The \emph{simplex category} \(\bD\) is the category with finite, non-empty, totally ordered sets as objects
\[
[n]=\{0<1<\dots<n\} ,\quad n \ge 0,
\]
and weakly order-preserving maps as morphisms.
\end{definition}

Regarding the poset \([n]\) as a category as in Example~\ref{chap2:poset-opensubsets-presheaf}, then the category \(\bD\) is a full subcategory of \(\Cats\), the category of small categories.

\begin{definition}
 The category of \emph{simplicial sets} \(\SSet\) is the presheaf category \([\bD^\op,\Sets]\). Subpresheaves in this case are called \emph{simplicial subsets}.
\end{definition}

More generally, \emph{simplicial objects} in a category~\(\Cat\) are functors \(\bD^\op\to \Cat\).

Since a map between two sets can be written uniquely as the composite of a surjection followed by an injection, morphisms in \(\bD\) decompose uniquely as composites of two classes of simple maps:
\begin{enumerate}
  \item coface map \(\delta_n^i\colon [n-1] \to [n]\) for \(0\le i\le n\), the unique injection whose image omits~\(i\),
  \item codegeneracy map \(\sigma_n^i\colon [n+1]\to [n]\) for \(0\le i\le n\), the unique surjection that identifies \(i\) and \(i+1\),
\end{enumerate}
These maps satisfy some obvious relations called cosimplicial identities. We get an alternative description of simplicial sets:

\begin{proposition}
A simplicial set \(X\) consists of a collection of sets \(X_n\) for \(n \ge 0\) with \emph{face} maps \(\face^n_i\colon X_n\to X_{n-1}\) and \emph{degeneracy} maps \(\de^n_i\colon X_n\to X_{n+1}\) satisfying the following \emph{simplicial identities}:
\begin{equation*}
\begin{aligned}
    \face^{n-1}_i \face^{n}_j &= \makebox[0pt][l]{\(\face^{n-1}_{j-1} \face^n_i\)}
    \hphantom{\begin{cases} \de^{n-2}_{j-1} \face^{n-1}_i\\0\end{cases}}\kern-\nulldelimiterspace
    \text{if } i<j,  \\
    \de^{n}_i \de^{n-1}_j &= \makebox[0pt][l]{\(\de^{n}_{j+1} \de^{n-1}_i\)}
    \hphantom{\begin{cases} \de^{n-2}_{j-1} \face^{n-1}_i\\0\end{cases}}\kern-\nulldelimiterspace
    \text{if } i\le j,\\
  \face^n_i \de^{n-1}_j &=
    \begin{cases}
      \de^{n-2}_{j-1} \face^{n-1}_i &\text{if } i<j,\\
      \id                         &\text{if } i=j\text{ or } j+1,\\
      \de^{n-2}_j \face^{n-1}_{i-1} &\text{if } i>j+1.
    \end{cases} &
\end{aligned}
\end{equation*}
Accordingly, a simplicial map \(f\colon X\to Y\) consists of a collection of maps \(f_n\colon X_n\to Y_n\) for \(n\ge 0\) that commute with all face and degeneracy maps.
\end{proposition}

If there is no danger of confusion, we will write \(\face^n_i, \de^n_i\) as \(\face_i, \de_i\), respectively.

\begin{example}
  First examples of simplicial sets include the standard \emph{simplicial} \emph{\(n\)\nbdash{}simplex}~\(\Simp{n}\), its \emph{boundary} \(\partial \Simp{n}\), and its \emph{\(k\)-th horn} \(\Horn{n}{k}\) for each \(0\le k\le n\) described as follows. Let~\(\Simp{n}\) be the simplicial set given by the representable functor \(\hom(-, [n])\colon \bD^\op\to \Sets\),
  \[
  [m]\mapsto \hom([m], [n]).
  \]
  The simplicial boundary \(\partial \Simp{n}\) is the simplicial subset of \(\Simp{n}\) defined by
  \[
  [m]\mapsto \{f\in  \hom([m], [n])\mid f([m])\neq [n]\},
  \]
 The \(k\)-th horn \(\Horn{n}{k}\) is the simplicial subset of \(\Simp{n}\) defined by
  \[
  [m]\mapsto \{f\in  \hom([m], [n])\mid [n]-\{k\}\not\subset f([m])\}.
  \]
  We call the natural morphism \(\Horn{n}{k}\to \Simp{n} \) for \(n\ge 1\) and \(0\le k\le n\) \emph{horn inclusion}, and \(\partial \Simp{n}\to \Simp{n}\) for \(n \ge 0\) \emph{boundary inclusion}. Notice that \(\partial\Simp{0}=\emptyset\), the initial simplicial set. The above formula also gives \(\Horn{0}{0}=\emptyset\), but we do not use this definition and left \(\Horn{0}{0}\) undefined. The reason will be explained in Remark~\ref{chap4:rem:horn-0-0}.
\end{example}

As a presheaf category, the category of simplicial sets has all limits and colimits, and these are computed pointwise. For a simplicial set \(X\), the Yoneda Lemma shows that there is a canonical isomorphism
\[
\hom(\Simp{n}, X)\cong X_n.
\]
Thus we call \(X_n\) the set of \emph{\(n\)-simplices}. An \(n\)-simplices \(x_n\in X_n\) is \emph{degenerate} if there are \(x_{n-1}\in X_{n-1}\) and some \(i\) such that \(x_n=\de^{n-1}_i x_{n-1}\).

The category of elements of a simplicial set \(X\) is also called the \emph{category of simplices} in~\(X\). The co-Yoneda Lemma implies that every simplicial set can be expressed by a colimit of standard simplices,
\[
X\cong \colim_{\Simp{n}\to X} \Simp{n},
\]
where the colimit is taken over the category of simplices in \(X\).

\begin{remark}\label{chap2:rem:smaller-category-of-simplicies}
In practice, for a specific simplicial set \(X\), it is often sufficient to consider a much smaller category (see~\cite[Lemma 3.1.4 and Errata]{Hovey}). For instance, the simplicial boundary \(\partial\Delta\) is the union of \(n\) copies of \(\Simp{n-1}\) as the boundary of \(\Simp{n}\), and the \(k\)-th horn \(\Horn{n}{k}\) is the union of \(n-1\) copies of \(\Simp{n-1}\), leaving out the one opposite to the vertex \(k\).
\end{remark}

\begin{remark}
  Covariant functors \([\bD, \Sets]\) and \([\bD, \Cat]\) are called \emph{cosimplicial sets} and \emph{cosimplicial objects} in \(\Cat\), respectively. For instance, the assignment
  \[
  [n]\mapsto \Delta^n
  \]
  is a cosimplicial object in \(\SSet\). Let \(X\) be a cosimplicial object in \(\Cat\). Then for each object \(C\in \Cat\), the sets \(\hom(X_i, C)\) form a simplicial set.
\end{remark}

\begin{example}\label{chap2:exa:nerve}
  The inclusion \(\bD\to \Cats\) is a cosimplicial object in \(\Cats\). The \emph{nerve} of a small category \(\Cat\) is the simplicial set \(N\Cat\) given by
  \[
  N_n\Cat=\Fun([n], \Cat).
  \]
  This defines a functor \(N\colon \Cats\to \SSet\).

  More concretely, \(N_n\Cat\) is the set of composable strings of arrows with length \(n\)
   \[
   C_0 \longleftarrow C_1 \longleftarrow\cdots\longleftarrow C_n.
   \]
   The face map \(\face^n_i\) composes two arrows for \(0<i<n \) or forget the first (last) arrow for \(i=0\) (\(i=n\)); degeneracy maps insert identity arrows.
\end{example}

\section{Lifting properties}

Let \(\Cat\) be a locally small category throughout this section.

\begin{definition}
Let \(i\colon A\to B\) and \(p\colon X\to Y \) be morphisms in \(\Cat\). We say that \(p\) satisfies the \emph{right lifting property} (RLP) with respect to \(i\) and that \(i\) satisfies the \emph{left lifting property} (LLP) with respect to \(p\) if, for every commutative square of the form
\begin{equation}\label{chap2:eq:commutative-square}
\begin{gathered}
  \xymatrix{
  A\ar[d]_{i}\ar[r]   &  X\ar[d]^{p}\\
  B\ar[r]\ar@{.>}[ru]   &  Y \rlap{\ ,}
  }
\end{gathered}
\end{equation}
there is a \emph{lift} \(B\to X\) such that the two triangles commute. If the lift is unique, then we speak of \emph{unique} right and left lifting properties.
\end{definition}

\begin{lemma}\label{chap2:lem:lift-composition-pullback}
  Let \(i\colon A\to B\) be a morphism in \(\Cat\).
  \begin{enumerate}
   \item If two morphisms \(p\colon X\to Y\) and \(q\colon Y\to Z\) have the RLP with respect to~\(i\), then so does~\(q\circ p\);
   \item If a morphism \(p\colon X\to Y\) has the RLP with respect to \(i\), then so does the pullback of~\(p\) along any morphism \(f\colon Z\to Y\).
  \end{enumerate}
\end{lemma}

\begin{lemma}\label{chap2:lem:lift-composition-pushout}
  Let \(p\colon X\to Y\) be a morphism in \(\Cat\).
  \begin{enumerate}
    \item If two morphisms \(i\colon A\to B\) and \(j\colon B\to C\) have the LLP with respect to~\(p\), then so does~\(j\circ i\);
    \item If a morphism \(i\colon A\to B\) has the LLP with respect to \(p\), then so does the pushout of \(i\) along any morphism \(f\colon A\to C\).
  \end{enumerate}
\end{lemma}

For a straightforward proof of these two lemmas, see, for instance, \cite[Lemma 4.1]{Goerss-Jardine}. We give an alternative proof below. We work with the spaces of \emph{all} maps instead of picking up a single map as in~\cite{Goerss-Jardine}.

\subsection{Arrow category}
Denote by \(I\) the category~\([1]\), which is also called the \emph{interval category}.

\begin{definition}
The \emph{arrow category} of~\(\Cat\) is the functor category \([I, \Cat]\).
\end{definition}

Explicitly, objects of \([I, \Cat]\) are arrows in \(\Cat\), and morphisms of \([I, \Cat]\) are commutative squares in \(\Cat\). For arrows \(i\colon A\to B\) and \(p\colon X\to Y\) in \(\Cat\), let
\[
\hom(A\xrightarrow{i} B, X\xrightarrow{p} Y)\coloneqq [I, \Cat](A\xrightarrow{i} B, X\xrightarrow{p} Y)
=\hom(A, X)\times_{\hom(A, Y)} \hom(B, Y)
\]
be the set of commutative squares~\eqref{chap2:eq:commutative-square} with fixed \((i,p)\). Every commutative square~\eqref{chap2:eq:commutative-square} with fixed \((i,p)\) has a (unique) lift if and only if
\[
  \hom(B, X)\to \hom(A\xrightarrow{i} B, X\xrightarrow{p} Y)
\]
is a surjection (bijection).

\subsection{Technical lemmas}\label{chap2:ssec:technical-lemmas}

Let us collect some simple but useful lemmas.

\begin{lemma}\label{chap2:lem:lift-composition-diagram}
For arrows \(A\to B\) and \(X\to Y\to Z\) in \(\Cat\), there is a pullback square
  \begin{equation}\label{chap2:eq:lift-composition-AB-XYZ}
  \begin{gathered}
  \xymatrix{
  \hom(A\to B, X\to Y)\ar[d]\ar[r]& \hom(B, Y)\ar[d]\\
  \hom(A\to B, X\to Z)\ar[r]      & \hom(A\to B, Y\to Z)\rlap{\ .}
  }
  \end{gathered}
  \end{equation}
Dually, for arrows \(A\to B\to C\) and \(X\to Y\) in \(\Cat\) there is a pullback square
  \[
  \xymatrix{
  \hom(B\to C, X\to Y)\ar[d]\ar[r]&\hom(B, X)\ar[d]\\
  \hom(A\to C, X\to Y)\ar[r]&\hom(A\to B, X\to Y)\rlap{\ .}
  }
  \]
\end{lemma}
\begin{proof}
By duality, it suffices to show the first statement. Consider the commutative diagram
  \[
  \xymatrix{
  \hom(A\to B, X\to Y)\ar[r]\ar[d] & \hom(B, Y)\ar[d]&\\
  \hom(A\to B, X\to Z)\ar[r]\ar[d] & \hom(A\to B, Y\to Z)\ar[r]\ar[d]& \hom(B, Z)\ar[d]\\
  \hom(A, X)\ar[r] & \hom(A, Y) \ar[r]& \hom(A, Z) \rlap{\ .}
  }
  \]
As the bottom right square and the bottom whole rectangle are pullback diagrams, so is the bottom left square. Since the left whole rectangle is also a pullback, so is the top left square, as desired.
\end{proof}

\begin{lemma}\label{chap2:lem:lift-pullback-pushout-diagram}
 Given a pullback square in \(\Cat\)
  \[
  \xymatrix{
  P\ar[r]\ar[d] & X\ar[d]\\
  Z\ar[r] & Y\rlap{\ ,}
  }
  \]
the diagram below is a pullback square for every morphism \(A\to B\) in~\(\Cat\):
  \begin{equation}\label{chap2:eq:lift-pullback-AB-XYZ}
  \begin{gathered}
  \xymatrix{
  \hom(B, P)\ar[r]\ar[d] & \hom(B, X)\ar[d] \\
  \hom(A\to B, P\to Z)\ar[r] & \hom(A\to B, X\to Y)\rlap{\ .}
  }
  \end{gathered}
  \end{equation}
Dually, given a pushout square in \(\Cat\),
  \[
  \xymatrix{
  A\ar[r]\ar[d] & C\ar[d]\\
  B\ar[r] & Q\rlap{\ ,}
  }
  \]
the diagram below is a pullback square for every morphism \(X\to Y\) in~\(\Cat\):
  \[
  \xymatrix{
  \hom(Q, X)\ar[r]\ar[d] & \hom(B, X)\ar[d] \\
  \hom(C\to Q, X\to Y)\ar[r] & \hom(A\to B, X\to Y)\rlap{\ .}
  }
  \]
\end{lemma}
\begin{proof}
By duality, it suffices to show the first statement. Consider the commutative diagram
\[
 \xymatrix{
 \hom(A\to B, P\to Z)\ar[r]\ar[d] & \hom(A, P)\ar[r]\ar[d] & \hom(A, X)\ar[d]\\
 \hom(B, Z)\ar[r] & \hom(A, Z)\ar[r] &\hom(A, Y)\rlap{\ ,}
}
\]
where the left square is a pullback diagram. Since \(P=X\times_{Y}Z\), the right square is a pullback diagram, and so is the whole rectangle. Consider now the commutative diagram
\[
 \xymatrix{
 \hom(B, P)\ar[r]\ar[d] & \hom(B, X)\ar[d] &\\
 \hom(A\to B, P\to Z)\ar[r]\ar[d] & \hom(A\to B, X\to Y)\ar[r]\ar[d] & \hom(A, X)\ar[d]\\
 \hom(B, Z)\ar[r] & \hom(B, Y)\ar[r] &\hom(A, Y) \rlap{\ .}
}
\]
The bottom right square and the bottom whole rectangle are pullback diagrams, hence so is the bottom left square. Since the left whole rectangle is a pullback diagram, so is the top left square, and we are done.
\end{proof}

Lemmas~\ref{chap2:lem:lift-composition-diagram} and~\ref{chap2:lem:lift-pullback-pushout-diagram} combined give the following result:

\begin{corollary}\label{chap2:cor:lift-composition-pushout-combined}
Suppose that the square below is a pushout diagram in \(\Cat\):
  \[
  \xymatrix{
  D\ar[r]\ar[d] & E\ar[d] & \\
  A\ar[r] & B\ar[r] & C\rlap{\ .}
  }
  \]
 For every morphism \(X\to Y\) in~\(\Cat\), the induced square below is a pullback diagram:
  \[
  \xymatrix{
  \hom(B\to C, X\to Y)\ar[r]\ar[d] & \hom(E, X)\ar[d]\\
  \hom(A\to C, X\to Y)\ar[r]       & \hom(D\to E, X\to Y)\rlap{\ .}
  }
  \]

\end{corollary}

\subsection{Proof of Lemma~\ref{chap2:lem:lift-composition-pullback}}
We are now ready to prove Lemma~\ref{chap2:lem:lift-composition-pullback}. There is also a dual argument for Lemma~\ref{chap2:lem:lift-composition-pushout}.

\begin{proof}
(i) Consider the pullback square~\eqref{chap2:eq:lift-composition-AB-XYZ} in Lemma~\ref{chap2:lem:lift-composition-diagram}. By the assumption that
\[
  \hom(B, Y)\to \hom(A\to B, Y\to Z)
\]
is a surjection, we deduce that
\[
  \hom(A\to B, X\to Y)\to \hom(A\to B, X\to Z)
\]
is a surjection. Since \(\hom(B, X)\to \hom(A\to B, X\to Y)\) is also a surjection by assumption, the composite of two surjections
\[
\hom(B, X)\to \hom(A\to B, X\to Z)
\]
is a surjection, as desired.

(ii) Consider the pullback square~\eqref{chap2:eq:lift-pullback-AB-XYZ} in Lemma~\ref{chap2:lem:lift-pullback-pushout-diagram}. By the assumption that
\[
\hom(B, X)\to \hom(A\to B, X\to Y)
\]
is a surjection, we deduce that
\[
\hom(B, P)\to \hom(A\to B, P\to Z)
\]
is a surjection, as claimed.
\end{proof}

The advantage of our approach is that we need not work with elements of the hom-sets of \(\Cat\), so it may work in the more general situation that we will consider in Section~\ref{chap3:sec:Hom}.

\section{Lifting properties of simplicial sets}
In this section, we study the right lifting property of simplicial sets with respect to certain inclusions. A simplicial horn \(\Horn{n}{k}\) is called an \emph{inner horn} if \(0<k<n\), a \emph{left horn} if \(0\le k<n\), a \emph{right horn} if \(0<k\le n\), an \emph{outer horn} if \(k=0\) or \(k=n\).

\begin{definition}
A morphism of simplicial sets \(p\colon X\to Y\) satisfies the (\emph{unique}) \emph{Kan condition} \(\Kan(n,k)\) (\(\Kan!(n,k)\)) if \(p\) satisfies the (unique) RLP with respect to the horn inclusion \(\Horn{n}{k} \to \Simp{n}\). Furthermore, \(p\) is
\begin{itemize}
  \item a \emph{Kan fibration}, if \(p\) satisfies the Kan conditions for all \(n\ge 1\) and \(0\le k\le n\);
  \item an \emph{inner Kan fibration}, if \(p\) satisfies the inner Kan conditions for all \(n\ge 2\) and \(0<k<n\).
\end{itemize}
In particular, \(X\) is a \emph{Kan complex} if the canonical map \(X\to \Simp{0}\) is a Kan fibration, and \(X\) is an \emph{inner Kan complex} if \(X\to \Simp{0}\) is an inner Kan fibration.
\end{definition}

\begin{definition}
  Let \(p\colon X\to Y\) be a morphism of simplicial sets. We say that \(p\) satisfies the (\emph{unique}) \emph{acyclicity condition} \(\Acyc(n)\) (\(\Acyc!(n)\)) if \(p\) has the (unique) RLP with respect to the boundary inclusion \(\partial\Simp{n}\to \Simp{n}\). We call \(p\) an \emph{acyclic fibration} if \(p\) satisfies all conditions \(\Acyc(n)\) for \(n\ge 0\).
\end{definition}

\begin{remark}
 Inner Kan complexes are introduced by Boardman and Vogt~\cite{Boardman-Vogt} under the name of weak Kan complexes, which are also called quasi-categories~\cite{Joyal2002} and \(\infty\)-categories~\cite{Lurie}. Inner Kan fibrations are introduced by Joyal~\cite{Joyal2002} under the name of quasi-fibrations, which are also called weak Kan fibrations~\cite{Goerss-Jardine}. Acyclic fibrations are also called trivial fibrations~\cite{Goerss-Jardine}.
\end{remark}

Adding the \(k\)-th face to the horn \(\Horn{n}{k}\) gives the boundary \(\partial\Simp{n}\). Thus an acyclic fibration is also a Kan fibration. Lemma~\ref{chap2:lem:lift-composition-pullback} implies the following:

\begin{corollary}
Kan \textup(inner Kan, or acyclic\textup) fibrations are closed under composition and stable under pullback.
\end{corollary}

\begin{proposition}\label{chap2:prop:category-is-inner-kan}
  The nerve of a category described in Example~\ref{chap2:exa:nerve} is an inner Kan complex; the nerve of a groupoid is a Kan complex.
\end{proposition}
\begin{proof}
We verify the inner Kan conditions for a category; the outer Kan conditions for a groupoid can be checked similarly. First,  \(\Kan!(2,1)\) is clear, since a horn \(\Horn{2}{1}\) in \(N\Cat\) is the same as two composable arrows in \(\Cat\). Given a horn \(f\colon \Horn{3}{1}\to N\Cat\), considering the 2-simplex with vertices \(\{0,1,2\}\) in \(N\Cat\), we have \(f([0, 2])=f([0, 1])\circ f([1, 2])\). Similarly, \(f([1,3]) =f([1,2])\circ f([2,3])\) and \(f([0,3]) =f([0,1])\circ f([1,3])\), thus \(f([0, 3])=f([0,2])\circ f([2,3])\), which implies that we can uniquely lift \(f\) to a 3-simplex in \(N\Cat\). This proves \(\Kan!(3, 1)\). Similarly, we can check \(\Kan!(3, 2)\).

When \(n>3\), the horn \(\Horn{n}{k}\) contains all 2-simplices of \(\Simp{n}\), hence
\[
f([ik])=f([ij])f([jk]) \quad\text{for \(0\le i<j<k\le n\)},
\]
thus we can uniquely lift every morphism \(\Horn{n}{k}\to N\Cat\) to a morphism \(\Simp{n}\to N\Cat\).
\end{proof}

\begin{remark}\label{chap2:rem:cat-Kan!(n,0)-n>=3}
The proof shows that \(\Kan!(2,1)\) encodes the composition, and \(\Kan!(3,1)\) or \(\Kan!(3,2)\) encodes the associativity of the composition. It is clear that \(\Kan!(n, k)\) holds for all \(n>3\) and \(0\le k\le n\). Every lift in the previous proposition is in fact unique for a category and for a groupoid except for \(n=1\).  It is easy to show that the converse is also true:
\end{remark}

\begin{proposition}\label{chap2:prop:inner-kan-complex-is-category}
If an inner Kan complex \(X\) satisfies \(\Kan!(n, k)\) for all \(n> 1\) and \(0<k<n\), then \(X\) is the nerve of a category. Moreover, if a Kan complex \(X\) satisfies \(\Kan!(n, k)\) for all \(0\le k\le n\) and \(n>1\), then \(X\) is the nerve of a groupoid.
\end{proposition}

\begin{proposition}\label{chap2:prop:nerve-of-functor-is-inner-kan}
The nerve of a functor of categories \(Nf\colon N\Cat\to N\Cat[D]\) is an inner Kan fibration.
\end{proposition}
\begin{proof}
To verify \(\Kan!(2, 1)\), let us consider a lifting problem
\[
\xymatrix{
\Horn{2}{1}\ar[d]\ar[r]^{\varphi} & N\Cat\ar[d]\\
\Simp{2}\ar[r]^{\psi}\ar@{.>}[ru]^{\tilde{\varphi}} & N\Cat[D] \rlap{\ .}
}
\]
Let \(\tilde{\varphi}\colon \Simp{2}\to N\Cat\) be the unique lift of \(\varphi\). Then we have
\[
f( \tilde{\varphi}[0,2])=f(\varphi([0,1])\circ \varphi([1,2]))=\psi([0,1])\circ\psi([1,2])=\psi([0,2]),
\]
which shows that \(\tilde{\varphi}\) is a lift for \((\varphi,\psi)\). Conversely, every lift for \((\varphi,\psi)\) is given by such \(\tilde{\varphi}\). Therefore, \(\Kan!(2,1)\) holds. Higher Kan conditions follow from \(\Kan!(2, 1)\) and the fact that both \(N\Cat\) and \(N\Cat[D]\) satisfy unique Kan conditions.
\end{proof}

Lemmas~\ref{chap2:prop:category-is-inner-kan} and \ref{chap2:prop:inner-kan-complex-is-category} imply that a category can be equivalently given by its nerve, which is an inner Kan complex satisfying appropriate unique Kan conditions. Along similar lines, we define higher categories and higher groupoids as simplicial sets satisfying appropriate Kan conditions.

\begin{definition}
 An \emph{\(n\)-category} is an inner Kan complex satisfying \(\Kan!(m,k)\) for all \(m> n\) and \(0<k<m\). An \emph{\(n\)-groupoid} is a Kan complex satisfying \(\Kan!(m,k)\) for all \(m> n\) and \(0\le k\le m\).
\end{definition}

When \(n=\infty\), we speak of \emph{\(\infty\)-categories} and \emph{\(\infty\)-groupoids}, which are inner Kan complexes and Kan complexes, respectively. More precisely, our notion of an \(\infty\)-category is a model for \((\infty, 1)\)-categories~\cite{Bergner}, in which each \(k\)-morphism is invertible if \(k>1\). The simplicial approach is one of the most successful approaches to higher category theory. It has been extensively studied by Joyal~\cite{Joyal2002,Joyal2008}, Lurie~\cite{Lurie}, and many others.

\section{Skeleton and coskeleton functors}\label{chap2:sec:sekeleton-and-coskeleton}

We saw that all higher dimensional simplices of the nerve of a category are obtained in an automatic way from the lower dimensions. Skeleton and coskeleton functors provide a general framework to study simplicial sets whose higher dimensions are determined by their lower dimensions.

Let \(\bD_{\le n}\) be the full subcategory of \(\bD\) with objects \([0], \dots, [n]\). The inclusion \(\bD_{\le n}\hookrightarrow \bD\) induces a \emph{truncation functor} remembering the low dimensions of a simplicial set:
\[
  \tr_n \colon \SSet=[\bD^\op, \Sets] \to [\bD_{\le n}^\op, \Sets].
\]

\begin{definition}
 The truncation functor \(\tr_n\) admits a left adjoint, called \emph{\(n\)-skeleton},
\[
\sk'_n\colon [\bD_{\le n}^\op, \Sets]\to \SSet,
\]
and a right adjoint, called \emph{\(n\)-coskeleton},
\[
 \cosk'_n\colon [\bD_{\le n}^\op, \Sets]\to \SSet.
\]
\end{definition}

Let \(X\in [\bD_{\le n}^\op, \Sets]\). The skeleton \(\sk'_n X\) is the smallest simplicial set generated by~\(X\): lower dimensional simplices of \(\sk'_n X\) are the same as that of \(X\) and all higher dimensional simplices of \(\sk'_n X\) above \(n\) are degenerate. The coskeleton \(\cosk'_n X\) is characterised by
\[
(\cosk'_n X)_i=\hom(\tr_n \Simp{i}, X), \quad \forall\, i.
\]
We denote the composite functors by
\begin{align*}
\sk_n&\coloneqq\sk'_n\circ \tr_n\colon \SSet\to \SSet,\\
\cosk_n&\coloneqq\cosk'_n\circ \tr_n\colon \SSet\to \SSet.
\end{align*}
It follows that \(\sk_n\) is left adjoint to \(\cosk_n\). By a slight abuse of language, we also call \(\sk_n\) the \(n\)-skeleton and \(\cosk_n\) the \(n\)-coskeleton. If there is no danger of confusion, we will denote \(\sk'_n\) and \(\cosk'_n\) by \(\sk_n\) and \(\cosk_n\), respectively.

\begin{definition}
  A simplicial set \(X\) is \emph{\(n\)-coskeletal} if the natural morphism \( X\to \cosk_n X\) is an isomorphism.
\end{definition}

\begin{proposition}
  The following statements about a simplicial set \(X\) are equivalent:
  \begin{enumerate}
    \item \(X\) is \(n\)-coskeletal.
    \item \(X_k\to \hom(\tr_n \Simp{k}, \tr_n X)\) is a bijection for all \(k> n\).
    \item \(X\to \Simp{0}\) satisfies \(\Acyc!(k)\) for all \(k>n\). 
  \end{enumerate}
\end{proposition}
\begin{proof}
   The first two statements are equivalent by definition. Suppose that \(X\) is \(n\)-coskeletal. Then for \(k>n\)
  \[
  \hom(\partial \Simp{k}, X)=\hom(\tr_n\partial \Simp{k}, \tr_n X)=\hom(\tr_n \Simp{k}, \tr_n X)=X_n,
  \]
  proving the third statement. Conversely, suppose that the third statement is true. Adding all \((n+1)\)-simplices of \(\Simp{k}\) to \(\sk_{n}\Simp{k}\) gives \(\sk_{n+1}\Simp{k}\). It follows that
  \[
  \hom(\sk_n \Simp{k}, X)=\hom(\sk_{n+1}\Simp{k}, X)
  \]
  for \(k>n\), hence
  \[
  \hom(\sk_n \Simp{k}, X)=\dots=\hom(\sk_{k-1}\Simp{k}, X)=X_k,
  \]
  and this shows that \(X\) is \(n\)-coskeletal. This proves the claim.
\end{proof}

An \(n\)-coskeletal simplicial set is clearly \(m\)-coskeletal for \(m>n\). The following proposition shows that unique Kan conditions and being coskeletal are closely related.

\begin{proposition}\label{chap2:prop:coskeletal-kan}
  Let \(X\) be a simplicial set.
\begin{enumerate}
  \item If \(X\) is \(n\)-coskeletal, then \(X\) satisfies \(\Kan!(m, k)\) for all \(m\ge n+2\) and \(0\le k\le m\).
  \item If \(X\) satisfies \(\Kan!(m, k)\) for all \(0\le k\le m\) and \(m\ge n\), then \(X\) is \(n\)-coskeletal.
\end{enumerate}
\end{proposition}
\begin{proof}
  (i) Since \(X\) is \(n\)-coskeletal and  \(\sk_n \Horn{m}{k}=\sk_n \Simp{m}\) for \(m\ge n+2\), we have
  \[
  \hom(\Horn{m}{k}, X)=\hom(\Simp{m}, X),
  \]
  which proves the first statement.

  (ii) To show \(\Acyc!(m)\) for \(m>n\), consider the commutative diagram
  \[
  \xymatrix{
  \hom(\partial \Simp{m}, X)\ar[r]^{f} & \hom(\Horn{m}{k}, X)\\
  \hom(\Simp{m}, X)\ar[ur]_{h}\ar[u]^{b} \rlap{\ .}&
  }
  \]
The map \(\hom(\Simp{m}, X)\xrightarrow{h} \hom(\Horn{m}{k}, X)\) is a bijection by assumption. Let \(\tilde{f}=b\circ h^{-1}\), we shall prove that \(\tilde{f}\) is the inverse of \(f\). Since \(f\circ \tilde{f}=f\circ b\circ h^{-1}=\id\), it remains to show that \(\tilde{f}\circ f=\id\). Given a morphism \(\varphi\colon \partial \Simp{m}\to  X\), we observe that the \(k\)-th face of the morphism \(\tilde{f}\circ f(\varphi)\colon \partial \Simp{m}\to  X\) is determined by first lifting the horn \(\Horn{m}{k}\to X\) to a simplex \(\Simp{m}\to X\) and then restricting to the \(k\)-th face. We need only show the \(k\)-th faces are the same. By the assumption \(\Kan!(m-1, j)\) for all \(j\), the \(k\)-th face is determined by a lower horn, hence by its boundary. Since \(\varphi\) and \(\tilde{f}\circ f(\varphi)\)  agree on the boundary of the \(k\)-th face, they must be the same, and this proves \(\Acyc!(m)\).
\end{proof}

The following examples for coskeletal simplicial sets are easy to verify.
\begin{example}
  The nerve of a category is 2-coskeletal. The nerve of a \v{C}ech groupoid is 1-coskeletal. The nerve of a pair groupoid is 0-coskeletal.
\end{example}

\section{Collapsible extensions}

Lemma~\ref{chap2:lem:lift-composition-pushout} implies that the collection of morphisms satisfying the LLP with respect to a given morphism is stable under composition and pushouts. In this section, we consider several classes of morphisms of simplicial sets that are stable under composition and pushouts. They will be of important use in our study of Kan fibrations, acyclic fibrations, and inner Kan fibrations of higher groupoids.

The following definition generalises that of collapsible simplicial sets in~\cite{Henriques}.

\begin{definition}\label{chap2:def:collapsible}
  Let~\(S\hookrightarrow T\) be an inclusion of finite simplicial sets (that is, each dimension is finite). We call \(S\hookrightarrow T\) a \emph{collapsible extension}, if there is a filtration
  \[
  S= S_0 \subset S_1 \subset \dotsb \subset S_l=T
  \]
  such that each \(S_i\) is obtained from~\(S_{i-1}\) by filling a horn; that is, there are
  \(n\ge 1\), \(0\le k\le n\), and a morphism~\(\Horn{n}{k}\to  S_{i-1}\) such that \(S_i=S_{i-1}\cup_{\Horn{n}{k}} \Simp{n}\). If there is a collapsible extension \(\Simp{0}\hookrightarrow T\), then we simply call \(T\) \emph{collapsible}.
\end{definition}

Of course, if \(S\hookrightarrow T\) and \(T\hookrightarrow U\) are collapsible extensions, then so is \(S\hookrightarrow U\). By definition, the collection of all collapsible extensions is the closure of horn inclusions \(\horn{n}{k}\colon\Horn{n}{k}\to \Simp{n}\) for \(n\ge 1\) and \(0\le k\le n\) under finitely many compositions and pushouts.

If we replace horn inclusions by inner/right/left horn inclusions or by boundary inclusions, we obtain \emph{inner}/\emph{right}/\emph{left collapsible extensions}, or \emph{boundary extensions}, respectively. These classes of maps are special cases of  (inner/left/right) anodyne extensions, or injections, which are defined as the closure of the corresponding classes of maps under transfinite compositions, pushouts, and retractions; see~\cite{Goerss-Jardine,Joyal2008,Lurie} for more information. Some general statements for (inner/left/right) anodyne extensions often remain valid for our special cases. In future applications, we allow only finite compositions and pullbacks, that is, we require a finite filtration as above. Lemma~\ref{chap2:lem:lift-composition-pushout} shows that a (inner/right/left) Kan fibration satisfies the RLP with respect to all (inner/right/left) collapsible extensions.

\subsection{Join of simplicial sets}

We provide some examples of collapsible extensions using the join operation.

\begin{definition}\label{chap2:def:join-simplicial-sets}
  The \emph{join} of two simplicial sets \(X\) and \(Y\) is the simplicial set \(X\star Y\) given by (formulas for face and degeneracy maps are omitted here)
  \[
  (X\star Y)_k=X_k \sqcup \Big(\bigsqcup_{i+j=k-1} X_i\times Y_j\Big) \sqcup Y_k.
  \]
\end{definition}

The join operation is determined by the following two properties~\cite{Joyal2002}:
  \begin{enumerate}
   \item The join of standard simplices is given by ordinal sum, \(\Simp{n}\star \Simp{m}=\Simp{n+m+1}\);
   \item The join is a bifunctor preserving colimits in both variables.
  \end{enumerate}
The join is better defined for augmented simplicial sets, see Section~\ref{chap4:sec:aug-simplicial-bisimplicial}.

\begin{example}
  For a simplicial set \(T\), the join \(T\star \Simp{0}\) is the cone over \(T\), and \(\Simp{0}\star T \) is the cone under \(T\). The left picture in Figure~\ref{chap2:fig:2cones-simp2*simp0-simp2*simp0} represents \(\Simp{2}\star\Simp{0}\), and the right picture represents \(\Simp{0}\star\Simp{2}\).
\begin{figure}[htb]
  \centering
  \begin{tikzpicture}
  [>=latex', mydot/.style={draw,circle,fill=black, inner sep=1.5pt}, %
    every to/.style={<-},every label/.style={scale=0.8},scale=0.8]
   \draw [fill=gray]
       (0,0) -- (0.5,-0.5) -- (2,0) -- cycle;
  \node[mydot, label=left:{$0$}]  (a) at (0,0) {};
  \node[mydot, label=below:{$1$}] (b) at (0.5,-0.5) {};
  \node[mydot, label=right:{$2$}] (c) at (2,0) {};
  \node[mydot, label=above:{$3$}] (d) at (0.8,1.5) {};
  \foreach \to/\from in {a/b,a/d,b/c,b/d,c/d}
  \draw [<-] (\to)--(\from);
  \draw [<-,dashed] (a)--(c);
  \begin{scope}[xshift=4cm]
      \draw [fill=gray]
      (0.5,-0.5) -- (2,0) -- (0.8,1.5) -- cycle;
  \node[mydot, label=left:{$0$}]  (a) at (0,0) {};
  \node[mydot, label=below:{$1$}] (b) at (0.5,-0.5) {};
  \node[mydot, label=right:{$2$}] (c) at (2,0) {};
  \node[mydot, label=above:{$3$}] (d) at (0.8,1.5) {};
  \foreach \to/\from in {a/b,a/d,b/c,b/d,c/d}
  \draw [<-] (\to)--(\from);
  \draw [<-,dashed] (a)--(c);
  \end{scope}
  \end{tikzpicture}
  \caption{Cones \(\Simp{2}\star\Simp{0}\) and \(\Simp{0}\star\Simp{2}\)}\label{chap2:fig:2cones-simp2*simp0-simp2*simp0}
\end{figure}
\end{example}

The following identities are shown in~\cite[Lemma 3.3]{Joyal2002}:
\begin{equation}\label{chap2:eq:join}
\begin{aligned}
   (\partial\Simp{m}\star\Simp{n})\cup (\Simp{m}\star\partial\Simp{n})&=\partial\Simp{m+1+n},&\text{for \(m\ge 0, n\ge 0\)},\\
   (\Horn{m}{k}\star\Simp{n}) \cup (\Simp{m}\star \partial \Simp{n})&=\Horn{m+1+n}{k},&\text{for \(m>0, n\ge 0\)},\\
   (\partial\Simp{m}\star \Simp{n})\cup (\Simp{m}\star\Horn{n}{k})&=\Horn{m+1+n}{m+1+k},&\text{for \(m\ge 0, n>0\)}.
\end{aligned}
\end{equation}

\begin{lemma}\label{chap2:lem:join-composition-pushout}
\begin{enumerate}
  \item For morphisms of simplicial sets \(A\to B\to C\) and \(X\to Y\), the following diagram is a pushout square:
\[
  \xymatrix{
  A\star Y \cup_{A\star X} B\star X \ar[r]\ar[d] & A\star Y \cup_{A\star X} C\star X\ar[d] \\
  B\star Y \ar[r] & B\star Y \cup_{B\star X} C\star X\rlap{\ .}
 }
\]
  \item Given a morphism of simplicial sets \(X\to Y\) and a pushout square of simplicial sets
\[
  \xymatrix{
  A\ar[r]\ar[d] & C\ar[d] \\
  B\ar[r] & Q \rlap{\ ,}
  }
\]
the following diagram is a pushout square:
\[
  \xymatrix{
  A\star Y \cup_{A\star X} B\star X \ar[r]\ar[d] & C\star Y \cup_{C\star X} Q\star X \ar[d] \\
  B\star Y \ar[r] & Q\star Y\rlap{\ .}
  }
\]
\end{enumerate}
\end{lemma}
\begin{proof}
 The proof is parallel to those of Lemmas~\ref{chap2:lem:lift-composition-diagram} and \ref{chap2:lem:lift-pullback-pushout-diagram}.
 Notice that the proofs there only use that \(\hom(\blank,\blank)\) is a continuous bifunctor.
\end{proof}

\begin{proposition}\label{chap2:prop:join-pushout-collapsible}
 Let \(f\colon A\to B\) and \(g\colon X\to Y\) be inclusions of simplicial sets. Suppose that \(f\) is a right collapsible extension and \(g\) is a boundary extension, or that \(f\) is a boundary extension and \(g\) is a left collapsible extension. Then the induced inclusion
\begin{equation}\label{chap2:eq:g-star-g-inclusion}
 A\star Y\cup_{A\star X} B\star X\to B\star Y
\end{equation}
is an inner collapsible extension. If \(f\) is a left collapsible extension, and \(g\) is a boundary extension, then~\eqref{chap2:eq:g-star-g-inclusion} is a left collapsible extension \textup(modified from \cite[Theorem~\textup{3.7}]{Joyal2008}\textup).
\end{proposition}
\begin{proof}
Applying Lemma~\ref{chap2:lem:join-composition-pushout}, we are reduced to the case when \(f\) and \(g\) are horn inclusions or boundary inclusions. This is established by~\eqref{chap2:eq:join}.
\end{proof}

\begin{lemma}
  The inclusion of a face \(\Simp{k} \to \Simp{n}\) is a collapsible extension for \(0\le k\le n\).
\end{lemma}
\begin{proof}
  Suppose that the statement holds for \(1,\dots, n-1\), and let us consider \(n\). It suffices to show that the \(n\)-th face inclusion \(\face^n\colon \Simp{n-1}\to \Simp{n}\) is a collapsible extension. By assumption \(\Simp{0}\to \Simp{n-1}\) is collapsible. Proposition~\ref{chap2:prop:join-pushout-collapsible} shows that the morphism
  \[
  \Simp{n-1}\cup_{\Simp{0}\{0\}} (\Simp{0}\{0\}\star \Simp{0}\{n\}) \to \Simp{n-1}\star \Simp{0}\{n\}=\Simp{n}
  \]
  is a collapsible extension. Composing with the collapsible extension
  \[
  \Simp{n-1}\to \Simp{n-1}\cup_{\Simp{0}\{0\}} (\Simp{0}\{0\}\star \Simp{0}\{n\}),
  \]
  we deduce that the composite \(\face^n\colon \Simp{n-1}\to \Simp{n}\) is a collapsible extension.
\end{proof}

\begin{definition}\label{chap2:def:spine}
The \emph{spine} \(\Sp(n)\) of the standard simplicial set \(\Simp{n}\) is the union of all edges \(\Simp{1}\{i, i+1\}\) for \(0\le i\le n-1\).
\end{definition}

\begin{lemma}\label{chap2:lem:spine-horn-inner}
 Let \(\Horn{n}{k}\) be an inner horn. Then the inclusion \(\Sp(n)\to \Horn{n}{k}\) is an inner collapsible extension \textup(compare with \cite[Proposition \textup{2.13}]{Joyal2008}\textup).
\end{lemma}
\begin{proof}
The statement is clear for \(n=2\). When \(n \ge 3\), we prove the stronger statement that the following inclusion is an inner collapsible extension:
\[
\Simp{n-1}\cup_{\Simp{0}\{n-1\}} \Simp{1}\{n-1,n\}\to \Horn{n}{k}.
\]
Then the lemma follows by applying the result successively.

To show the stronger statement, we show that
\[
f\colon S\coloneqq\Simp{n-1} \cup_{ \Simp{n-2}\{1,\cdots,n-1\} } \Simp{n-1}\{1,\cdots, n \} \to \Horn{n}{k}
\]
is an inner collapsible extension, then the stronger statement follows by induction. All the missing simplices in \(S\) are those containing \(\Simp{1}\{0, n\}\), thus filling a horn of \(S\) corresponds to filling a horn in \(\Simp{n-2}\{1,\dots, n-1\}\). Since \(\Horn{n-2}{k-1}\) is collapsible, \(\Horn{n}{k}\) can be obtained from~\(S\) by filling horns. Since every horn filled contains \(\Simp{1}\{0, n\}\), it is an inner horn. This shows that \(f\) is an inner collapsible extension.
\end{proof}

\chapter{Higher Groupoid Actions}
\thispagestyle{empty}

We have seen that a groupoid can be equivalently given by its nerve, which is a Kan complex satisfying appropriate unique Kan conditions. In this chapter, we establish a similar result for Lie groupoids, which leads to a notion of higher groupoids in a category with a singleton Grothendieck pretopology. We then study higher groupoid Kan fibrations, which can be viewed as actions of higher groupoids. Next, we define principal bundles of higher groupoids by pairs of a Kan fibration and an acyclic fibration.

\section{Singleton Grothendieck pretopology}

In Definition~\ref{chap1:def:lie-groupoid}, we require the source and target maps to be surjective submersions. This ensures that the space of composable arrows is a manifold. The collection of surjective submersions in \(\Mfd\) is a singleton Grothendieck pretopology:

\begin{definition}
  Let \(\Cat\) be a locally small category. A \emph{singleton Grothendieck pretopology} (henceforth pretopology) on \(\Cat\) is a collection of morphisms \(\covers\), called \emph{covers}, such that
  \begin{enumerate}
    \item isomorphisms are covers;
    \item composites of covers are covers;
    \item if \( Y\to X\) is a cover, then, for any morphism \(Z\to X\), the pullback \(Y\times_X Z\) is representable and \(Y\times_X Z\to Z\) is a cover.
  \end{enumerate}
\end{definition}

In the sequel, we assume further that the category \(\Cat\) has a terminal object \(\terminal\) and that, for every object \(X \in \Cat\), the map \(X\to \terminal\) is a cover.

\begin{example}\label{chap3:exa:singleton-pretopoloty}
The notion of pretopology is a natural generalisation of the usual notion of covering, as shown by the following examples:
\begin{enumerate}
  \item surjections in \(\Sets\), denoted by \((\Sets, \covers_\surj)\);
  \item open surjections in \(\Topo\), denoted by \((\Topo, \covers_\open)\);
  \item surjective local homeomorphisms in \(\Topo\), denoted by \((\Topo,\covers_\etale)\);
  \item surjective submersions in \(\Mfd\), denoted by \((\Mfd, \covers_\subm)\);
  \item surjective local diffeomorphism in \(\Mfd\), denoted by \((\Mfd, \covers_\etale)\).
\end{enumerate}
In particular, given an open cover \(\{Y_i\}_{i\in I}\) of \(X\) in \(\Topo\), let \(Y=\bigsqcup_{i\in I} Y_i\). Then the natural map \(Y\to X\) is a cover in \((\Topo, \covers_\open)\) and \((\Topo, \covers_\etale)\).
\end{example}

Let \((\Cat, \covers)\) be a category with a pretopology throughout this chapter.

\begin{definition}\label{chap3:def:sheaf}
  A presheaf \( f\colon \Cat^{\op} \to \Sets\) is a \emph{sheaf} with respect to \(\covers\) if
  \[
  \xymatrix@1{
  f(X)\ar[r] & f(Y)\ar@<-0.5ex>[r]\ar@<0.5ex>[r] & f(Y\times_X Y)
  }
  \]
  is an equaliser diagram for every cover \(Y\to X\). Morphisms of sheaves are morphisms of the underlying presheaves.
\end{definition}

The sheaf condition generalises the fact that, if \(\{Y_i\}\) is an open cover of a topological space \(X\), then a function on \(X\) is the same as a matching family of functions on \(\{Y_i\}\).

\begin{remark}\label{chap3:rem:sheaf}
The category of sheaves on \(\Cat\) with respect to \(\covers\) is a full subcategory of the category of presheaves on \(\Cat\). A presheaf limit of sheaves is a sheaf, and a finite presheaf colimit of sheaves is a sheaf.
\end{remark}

\begin{definition}
  A pretopology \(\covers\) on \(\Cat\) is \emph{subcanonical} if every representable presheaf is a sheaf with respect to \(\covers\).
\end{definition}

\begin{lemma}\label{chap3:lem:subcanonical-effective-epi}
  A pretopology \(\covers\) on \(\Cat\) is subcanonical if and only if every cover is an effective epimorphism.\footnote{A morphism \(f\colon Y\to X\) in \(\Cat\) is an effective epimorphism if~\eqref{chap3:eq:coequaliser-diagram} is a coequaliser diagram. An effective epimorphism is an epimorphism~\cite{Borceux}.}
\end{lemma}
\begin{proof}
If \(\covers\) is subcanonical, then for every object \(Z\) in \(\Cat \), the diagram
  \[
  \xymatrix@1{
  \hom(X, Z) \ar[r] & \hom(Y, Z)\ar@<-0.5ex>[r]\ar@<0.5ex>[r] & \hom(Y\times_X Y, Z)
  }
  \]
is an equaliser diagram for every cover \(Y\to X\). It follows from the definition of limits that
\begin{equation}\label{chap3:eq:coequaliser-diagram}
\xymatrix@1{X & Y \ar[l] & Y\times_X Y \ar@<-0.5ex>[l]\ar@<0.5ex>[l] }
\end{equation}
is a coequaliser diagram for every cover \(Y\to X\). Thus \(Y\to X\) is an effective epimorphism. It is easy to check the converse, and we are done.
\end{proof}

\begin{example}
  All five examples in Example~\ref{chap3:exa:singleton-pretopoloty} are subcanonical. An example of a non-subcanonical pretopology is the pretopology on \(\Sets\) whose covers are given by all morphisms; notice that not every cover is an effective epimorphism.
\end{example}

\subsection{Internal groupoids}\label{chap3:subs:internal-groupoids}
A category \(\Cat\) with a pretopology has enough pullbacks to define groupoids in it.

\begin{definition}\label{chap3:def:groupoid-pretopology}
 A \emph{groupoid object} \(G\) in \((\Cat, \covers)\) consists of two objects \(G_0\) and \(G_1\) with five structure maps: source and target \(\source, \target\colon G_0\to G_1\) in \(\covers\), unit \(\unit\), inverse \(\inv\), and composition \(\circ\colon G_1\times_{\source, G_0, \target} G_1 \to G_1\) satisfying identities as in Definition~\ref{chap1:def:groupoid}.
\end{definition}

\begin{remark}
Lie groupoids are groupoids in \((\Mfd,\covers_\subm)\). Almost all results in Chapter~\ref{chap1} can be translated to groupoids in \((\Cat,\covers)\) if \((\Cat, \covers)\) satisfies some additional axioms; this is discussed thoroughly by Meyer and Zhu~\cite{Meyer-Zhu}. These axioms are fulfilled by each case in Example~\ref{chap3:exa:singleton-pretopoloty}. In particular, there is a 2\nbdash{}category \(\BUN\) given by HS bundles, and Theorem~\ref{chap1:thm:GPD=BUN} is also valid; see also~\cite{Carchedi} for some special cases.
\end{remark}

Since the source and target maps are covers, the space of composable pairs of arrows \(G_1\times_{\source, G_0, \target} G_1\) is representable. Likewise, the space of composable chains of arrows of length~\(n\),
\[
G_1\times_{\source, G_0, \target} G_1\times_{\source, G_0, \target} \dots \times_{\source, G_0, \target} G_1,
\]
is representable. Thus the nerve of a groupoid in \((\Cat, \covers)\) is a simplicial object in \(\Cat\).

\section{The Hom functor}\label{chap3:sec:Hom}

Let \(S\) be a finite simplicial set and \(X\) a simplicial object in a locally small category~\(\Cat\). We will describe a presheaf on \(\Cat\) as ``the space of diagrams of shape \(S\) in \(X\)'', denoted by \(\Hom(S, X)\). This is not the naive hom; when \(\Cat\) is \(\Sets\), the two notions coincide.

First, take \(S\) a finite set and \(X\) an object in \(\Cat\). Since \(S\) and \(X\) are objects in two different categories, \(\hom(S, X)\) does not make sense. We define \(\Hom(S, X)\) to be the presheaf
\[
Z\mapsto \hom(Z, X)^S, \quad Z\in \Cat
\]
where \(\hom(Z, X)^S\) is the \(S\)-power\footnote{Let \(\Cat\) be a category, and \(S\) be a finite set. Recall from~\cite[Chapter III]{MacLane} that the \(S\)-power of an element \(Z\in \Cat\) is the product of \(S\) copies of \(Z\), denoted by \(Z^S\), and that the \(S\)-copower of \(Z\) is the coproduct of \(S\) copies of \(Z\), denoted by \(S\otimes Z\).} of the set \(\hom(Z, X)\). If \(\Cat\) has \(S\)-copowers \(\otimes\), then we may write \(\Hom(S, X)=\hom( S\otimes \blank , X)\). If \(\Cat\) is \(\Sets\), then \(\Hom(S, X)\) is the presheaf represented by \(\hom(S, X)\).

Given two simplicial sets \(R\) and \(S\), we observe that the set of simplicial morphisms \(\hom(R, S)\) fits into an equaliser diagram (more precisely, an end~\cite[Section IX.5]{MacLane}),
\[
\xymatrix@1{
\hom(R, S)\ar[r] &\prod_{j}\hom(R_j, S_j)\ar@<-0.5ex>[r]\ar@<0.5ex>[r] &\prod_{[j]\to [i]} \hom(R_i, S_j)
},
\]
where \([j]\to [i]\) runs over all morphisms in \(\bD\). This allows us to make the following definition.

\begin{definition}
  Let \(S\) be a finite simplicial set, \(X\) a simplicial object in \(\Cat\). We define \(\Hom(S, X)\) to be the presheaf on \(\Cat\) given by the equaliser diagram
  \[
  \xymatrix@1{
  \Hom(S, X)\ar[r] &\prod_{j}\Hom(S_j, X_j)\ar@<-0.5ex>[r]\ar@<0.5ex>[r] &\prod_{[j]\to [i]} \Hom(S_i, X_j)
  }.
  \]
Let \(R\to S\) be a morphism of finite simplicial sets, and let \(X\to Y\) be a morphism of simplicial objects in \(\Cat\). The presheaf \(\Hom(R\to S, X\to Y)\) is the pullback
\[
\Hom(R\to S, X\to Y)\coloneqq \Hom(R, X)\times_{\Hom(R, Y)} \Hom(S, Y).
\]
\end{definition}

\begin{remark}
As before, if \(\Cat\) has copowers, then \(\Hom(S, X)\) has a simpler form
\[
  Z\mapsto \hom(S\otimes Z, X), \quad Z\in \Cat,
\]
where \(S\otimes Z\) is the simplicial object induced by levelwise copower. Similarly, \(\Hom(R\to S, X\to Y)\) is given by
\[
  Z\mapsto \hom( R\otimes Z\to S\otimes Z, X\to Y), \quad Z\in \Cat.
\]
\end{remark}

\begin{example}
If \(S=\Simp{n}\), then \(\Hom(S, X)=\hom(-, X_n)\), the presheaf represented by \(X_n\).
\end{example}

Recall that the \(\hom\) functor turns colimits in the first variable into limits and preserves limits in the second variable. Since limits commute with equalisers, the \(\Hom\) functor turns colimits in the first variable into limits and preserves limits in the second variable. It follows that all lemmas we established in Section~\ref{chap2:ssec:technical-lemmas} remain valid for \(\Hom\).

Recall that a simplicial set \(S\) is the colimit of standard simplices over the category of simplices of \(S\). Thus we can compute \(\Hom(S, X)\) as a limit of presheaves \(\hom(-, X_n)\) over the same category. Remark~\ref{chap2:rem:smaller-category-of-simplicies} indicates that the limit defining \(\hom(S,X)\) often simplifies considerably by choosing a smaller subcategory.

\section{Higher groupoids in a category with pretopology}

We are now ready to give an alternative definition of Lie groupoids in terms of simplicial manifolds. Indeed, we will define higher groupoids in a category with a pretopology.

Let \(G\) be a groupoid in \((\Cat,\covers)\) as in Definition~\ref{chap3:def:groupoid-pretopology}. Its nerve \(NG\) is a simplicial objects in \(\Cat\). Let us consider Kan conditions for \(NG\). First, the map
\[
   G_1=\Hom(\Simp{1}, G)\to \Hom(\Horn{1}{0}, G)=G_0
\]
is \(\target\colon G_1\to G_0\), so it is a cover. Similarly, \(\Hom(\Simp{1}, G)\to \Hom(\Horn{1}{1}, G)\) is also a cover. Second, the map
\[
  \Hom(\Simp{2}, G)\to \Hom(\Horn{2}{1}, G)
\]
is an isomorphism in \(\Cat\). Similar to the case of \(\Sets\), we can show that
\[
\Hom(\Simp{m}, G)\to \Hom(\Horn{m}{k}, G)
\]
is an isomorphism for any \(m\ge 2\) and \(0\le k\le m\). This suggests that we shall replace surjections in Kan conditions by covers in \(\covers\).

\begin{definition}
  A simplicial object \(X\) in \((\Cat, \covers)\) satisfies the \emph{Kan condition} \(\Kan(m, k)\) if the natural morphism
  \[
 \Hom(\Simp{m}, X)\to \Hom(\Horn{m}{k}, X)
  \]
  is a cover. It satisfies the \emph{unique Kan condition} \(\Kan!(m,k)\) if this morphism is an isomorphism in~\(\Cat\). In both cases, the representability of the presheaf \(\Hom(\Horn{m}{k},X)\) is assumed.
\end{definition}

\begin{definition}\label{chap3:def:n-groupoid-pretopology}
  An \(n\)\nbdash{}\emph{groupoid object} (for \(0\le n \le \infty\)) in \((\Cat, \covers)\) is a simplicial object \(X\) in \(\Cat\) that satisfies \(\Kan(m,j)\) for all \(m\ge1\)  and \(0\le k\le m\) and \(\Kan!(m,k)\) for all \(m>n\) and \(0\le k\le m\).
\end{definition}

\begin{remark}\label{chap3:rem:duskin-hypergroupoid}
This definition is due to Henriques~\cite{Henriques}. Duskin~\cite{Duskin,Glenn:Realization} defined a notion of \(n\)-hypergroupoid, which is a simplicial set \(X\) satisfying only the conditions \(\Kan!(m, k)\) for \(m>n\) and \(0\le k\le m\).
\end{remark}

\begin{remark}\label{chap3:rem:Kan_one}
  The conditions \(\Kan(1,j)\) are automatic for simplicial \emph{sets} because the degeneracy map provides a section,  but they are not automatic in general. These conditions tell us that face maps from \(X_1\) to \(X_0\) are covers.
\end{remark}

\begin{remark}\label{chap3:rem:rep}
  The conditions \(\Kan(l,k)\) for \(1\le l<m\) imply that \(\Hom(\Horn{m}{j},X)\) is representable.  This is a special case of the proof of Lemma~\ref{chap3:lem:covers_in_groupoid} below.
\end{remark}

\begin{example}
  \label{chap3:exa:zero-groupoid}
  A \(0\)\nbdash{}groupoid is a constant simplicial object \(\sk_0 X_0\) given by an object~\(X_0\) of~\(\Cat\). Simplicial sets of this type are usually called discrete; this would, however, cause confusion for simplicial topological spaces. To avoid this, we call them 0-groupoids or trivial groupoids.
\end{example}

\begin{example} \label{chap3:exa:one-groupoid}
  A \(1\)\nbdash{}groupoid is the nerve of a groupoid object in~\((\Cat,\covers)\) in the sense of Definition~\ref{chap3:def:groupoid-pretopology}; the object space is~\(X_0\), the arrow space is~\(X_1\), the target and source maps are \(\face_0, \face_1\colon X_1\to X_0\), respectively, and the unit is \(\de_0\colon X_0\to X_1\).  The composition is the composite map
  \[
  \Hom(\Horn{2}{1},X) \cong  \Hom(\Simp{2},X) = X_2 \xrightarrow{\face_1} X_1,
  \]
  where the condition \(\Kan!(2,1)\) is used. The condition \(\Kan!(3,1)\) or \(\Kan!(3,2)\) implies that the composition is associative. The conditions \(\Kan!(2,0)\) and \(\Kan!(2,2)\) give left and right inverses, respectively.
\end{example}

The Kan conditions imply that various maps of the form \(\Hom(S,X) \to \Hom(T,X)\) for simplicial sets \(T\subseteq S\) are covers. The lemma below and Lemma \ref{chap3:lem:representable} later will be used to solve most cases of representability problems. They are similar to~\cite[Lemma 2.4]{Henriques}.

\begin{lemma}\label{chap3:lem:covers_in_groupoid}
  Let~\(T\) be a collapsible extension of a simplicial set~\(S\) and let~\(X\) be a higher groupoid in \((\Cat,\covers)\). Assume that \(\Hom(S,X)\) is representable. Then \(\Hom(T,X)\) is representable and the map \(\Hom(T,X) \to \Hom(S,X)\) is a cover.
\end{lemma}
\begin{proof}
  Let \(S=S_0\subset S_1\subset \dotsb \subset S_l=T\) be a filtration as in Definition~\ref{chap2:def:collapsible}.  Since composites of covers are covers, we may assume without loss of generality that \(l=1\), so that \(T=S\cup_{\Horn{n}{j}} \Simp{n}\) for  some \(n,j\).  Since~\(X\) is a higher groupoid, \(\Hom(\Simp{n},X) \to \Hom(\Horn{n}{j},X)\) is a cover and \(\Hom(\Horn{n}{j},X)\) is representable. Then
  \[
  \Hom(T,X) = \Hom(S,X) \times_{\Hom(\Horn{n}{j},X)} \Hom(\Simp{n},X)
  \]
  is representable, and the map \(\Hom(T,X)\to\Hom(S,X)\) is a cover because it is a pullback of a cover.
\end{proof}

For a category, both \(\Kan!(3,1)\) and \(\Kan!(3,2)\) express the associativity of the composition. As a generalisation, we have

\begin{lemma}\label{chap3:lem:kan(n+1,j)-implies-all}
 Let \(X\) be a simplicial object in \((\Cat, \covers)\) satisfying \(\Kan!(n, k)\) for \(0<k<n\). If\/ \(\Kan!(n+1, j_0)\) holds for some \(j_0\) with \(0<j_0<n+1\), then \(\Kan!(n+1, j)\) hold for all \(0<j<n+1\). The statement is still true if we replace the conditions \(0<k<n\) and \(0<j<n+1\), by \(0<k\le n\) and \(0<j\le n+1\), or by \(0\le k< n\) and \(0\le j< n+1\).
\end{lemma}
\begin{proof}
   Fix \(0<i<j<n+1\). Let \(S\) be the simplicial set generated by all \(n\)-simplices of \(\Simp{n+1}\) containing \(\{i,j\}\). Attaching to \(S\) the \(n\)-simplex without \(\{j\}\) gives \(\Horn{n+1}{i}\); that is
   \[
   \Horn{n+1}{i}=S\cup_{\Horn{n}{i} \{0,\dots,\hat{j},\dots, n+1\}} \Simp{n}\{0,\dots,\hat{j},\dots, n+1\}.
   \]
   Similarly,
   \[
   \Horn{n+1}{j}=S\cup_{\Horn{n}{j-1} \{0,\dots,\hat{i},\dots, n+1\}} \Simp{n}\{0,\dots,\hat{i},\dots, n+1\}.
   \]
  Applying \(\Kan!(n, i)\) and \(\Kan!(n, j-1)\), we deduce that
  \[
  \Hom(\Horn{n+1}{i},X)\cong \Hom(S, X)\cong \Hom(\Horn{n+1}{j}, X).
  \]
  Therefore, if \(\Kan!(n+1, i)\) holds, that is,
  \[
  \Hom(\Simp{n+1}, X)\cong \Hom(\Horn{n+1}{i},X),
  \]
  then \(\Kan!(n+1, j)\) also holds. The rest follows immediately.
\end{proof}

\begin{remark}
  The notion of being coskeletal is easy to adapt to simplicial objects. An argument parallel to the proof of   Proposition~\ref{chap2:prop:coskeletal-kan} shows that an \(n\)-groupoid object \(X\) in \((\Cat,\covers)\) is \((n+1)\)-coskeletal, hence \(X\) is determined by \(X_{\le n+2}\). Equivalently, \(X\) is fully determined by \(X_{\le n}\) together with \((n+1)\)-multiplications, which come from \(\Kan!(n+1, k)\) for \(0\le k\le n+1\). Moreover, these \((n+1)\)-multiplications satisfy, among other things, some compatibility conditions that come from \(\Kan!(n+2, k)\) for \(0\le k \le n+2\). Lemma~\ref{chap3:lem:kan(n+1,j)-implies-all} implies that it suffices to verify \(\Kan!(n+2, k)\) for a single \(0\le k\le n+2\).
\end{remark}

\section{Higher groupoid Kan fibrations}

Our definition of a groupoid action is based on the notion of Kan fibrations between simplicial objects in a category with pretopology.

\begin{definition}\label{chap3:def:Kan_arrow}
  Let \(f\colon X\to Y\) be a morphism of simplicial objects in~\((\Cat,\covers)\). We say that~\(f\) satisfies the \emph{Kan condition} \(\Kan(n,j)\) if the presheaf \(\Hom(\Horn{n}{j}\to\Simp{n}, X\to Y)\) is representable and the canonical map
  \[
    X_n = \Hom(\Simp{n},X)\to \Hom(\Horn{n}{j}\to\Simp{n}, X\to Y)
  \]
  is a cover. If this map is an isomorphism, then we say that \(f\) satisfies the \emph{unique Kan condition} \(\Kan!(n, j)\).
  We call~\(f\) a \emph{Kan fibration} if it satisfies \(\Kan(m,j)\) for all \(m\ge 1\), \(0\le j\le m\).
\end{definition}

\begin{remark}
  Since \(\Cat\) has a terminal object \(\terminal\), the natural map \(X \to \terminal\) is a Kan fibration if and only if \(X\) is an \(\infty\)-groupoid in \((\Cat, \covers)\).
\end{remark}

\begin{definition}\label{chap3:def:kan-fibration-n-groupoid}
Let \(f\colon X\to Y\) be a morphism between \(n\)-groupoids in \((\Cat, \covers)\). Suppose that \(f\) is a Kan fibration. We call \(f\) an \(n\)-\emph{groupoid Kan fibration} if, in addition, the unique Kan conditions \(\Kan!(n, k)\) for \(0\le k \le n\) are satisfied.
\end{definition}

\begin{remark}
The definition of \(n\)-groupoid Kan fibration has a stronger condition than that of Definition~\ref{chap3:def:Kan_arrow}. It is easy to see that, if \(X\) and \(Y\) are \(n\)\nbdash{}groupoids, then \(\Kan!(m, k)\) for \(m>n\) hold automatically.
\end{remark}

\begin{remark}
We may allow in Definition~\ref{chap3:def:kan-fibration-n-groupoid} that \(X\) is only a simplicial object. Proposition~\ref{chap3:prop:comp} below implies that \(X\) will be an \(n\)\nbdash{}groupoid in \((\Cat, \covers)\). Composing \(f\colon X \to Y\) with \(Y\to \terminal\), we see that \(X\) satisfies the desired Kan conditions.
\end{remark}

The following lemma is a relative version of Lemma~\ref{chap3:lem:kan(n+1,j)-implies-all}.

\begin{lemma}\label{chap3:lem:morphism-kan(n+1,j)-implies-all}
  Let \(f\colon X\to Y\) be a morphism of simplicial objects in \((\Cat,\covers)\) satisfying \(\Kan!(n, k)\) for \(0< k < n\). If \(f\) satisfies \(\Kan!(n+1, j_0)\) for some \(j_0\) with \(0< j_0< n+1\), then it satisfies \(\Kan!(n+1, j)\) for all \(0< j< n+1\). The statement remains true if we replace the conditions \(0<k<n\) and \(0<j<n+1\) by \(0<k\le n\) and \(0<j\le n+1\), or by \(0\le k< n\) and \(0\le j< n+1\).
\end{lemma}
\begin{proof}
  The proof is similar to that of Lemma~\ref{chap3:lem:kan(n+1,j)-implies-all}. Fix \(0<i<j<n+1\). Let \(S\) be the simplicial set generated by all \(n\)-simplices of \(\Simp{n+1}\) containing \(\{i,j\}\). We have
  \begin{gather*}
   \Horn{n+1}{i}=S\cup_{\Horn{n}{i} \{0,\dots,\hat{j},\dots, n+1\}} \Simp{n}\{0,\dots,\hat{j},\dots, n+1\},\\
   \Horn{n+1}{j}=S\cup_{\Horn{n}{j-1} \{0,\dots,\hat{i},\dots, n+1\}} \Simp{n}\{0,\dots,\hat{i},\dots, n+1\}.
  \end{gather*}
 Corollary~\ref{chap2:cor:lift-composition-pushout-combined} implies a pullback square
 \[
 \xymatrix{
 \Hom(\Horn{n+1}{i}\to \Simp{n+1},f)\ar[r]\ar[d] &\Hom(\Simp{n}\{\hat{j}\},f)\ar[d]\\
 \Hom(S\to\Simp{n+1}, f)\ar[r] & \Hom(\Horn{n}{i}\{\hat{j}\}\to \Simp{n}\{\hat{j}\},f)\rlap{\ ,}
 }
 \]
 where \(\Simp{n}\{\hat{j}\}\) and \(\Horn{n}{i}\{\hat{j}\}\) are short for \(\Simp{n}\{0,\dots,\hat{j},\dots, n+1\}\) and \(\Horn{n}{i} \{0,\dots,\hat{j},\dots, n+1\}\), respectively. Applying \(\Kan!(n, i)\) for \(f\) gives
 \[
 \Hom(\Horn{n+1}{i}\to \Simp{n+1},f)\cong \Hom(S\to\Simp{n+1}, f).
 \]
 Similarly, \(\Kan!(n, j-1)\) for \(f\) implies that
 \[
 \Hom(\Horn{n+1}{j}\to \Simp{n+1}, f)\cong\Hom(S\to\Simp{n+1}, f).
 \]
 Therefore, if \(f\) satisfies \(\Kan!(n+1, i)\), that is,
 \[
 \Hom(\Simp{n+1}, X)\cong \Hom(\Horn{n+1}{i}\to \Simp{n+1},f),
 \]
 then \(f\) satisfies \(\Kan!(n+1, j)\). The rest follows immediately.
\end{proof}

We now show that the assumption of representability in Definition~\ref{chap3:def:Kan_arrow} is not necessary to define a higher groupoid Kan fibration; the conditions \(\Kan(l,i)\) for \(l<m\) ensure that the presheaves \(\Hom(\Horn{m}{j}\to\Simp{m},X\to Y)\) are representable.

\begin{lemma} \label{chap3:lem:representable}
  Let \(k\ge 0\) and let \(\Simp{0}\to S\subset\Simp{k}\) be collapsible extensions of simplicial sets. Let \(f\colon X\to Y\) be a morphism of simplicial objects in~\(\Cat\). Assume that~\(f\) satisfies \(\Kan(m,j)\) for \(m<k\) and \(0\le j\le m\) and that \(\Hom(\Simp{0}\to\Simp{k}, f)=X_0\times_{Y_0}Y_k\) is representable. Then the presheaf \(\Hom(S\to\Simp{k}, f)\) is representable. Furthermore, if \(\Simp{0}\subset S \subset T\subset \Simp{k}\) are collapsible extensions, then \(\Hom(T\to\Simp{k}, f)\to \Hom(S\to\Simp{k}, f)\) is a cover.
\end{lemma}
\begin{proof}
  The difference of this lemma to \cite[Lemma 2.4]{Henriques} is that we do not require \(X_0=Y_0=\terminal\). The induction step is the same as in~\cite{Henriques}. The base case of the induction holds by assumption. For the induction step, suppose that \(S_i=S_{i-1}\cup_{\Horn{n_i}{j_i}}\Simp{n_i}\), and that the statement is true for \(S_{i-1}\).
  Corollary~\ref{chap2:cor:lift-composition-pushout-combined} gives a pullback square
 \[
  \xymatrix{
  \Hom(S_i\to \Simp{n}, X\to Y)\ar[r]\ar[d]&\Hom(\Simp{n_i}, X)\ar[d]\\
  \Hom(S_{i-1}\to \Simp{n}, X\to Y) \ar[r]         &\Hom(\Horn{n_i}{j_i}\to \Simp{n_i}, X\to Y)\rlap{\ .}
  }
  \]
 By assumption, \(\Hom(S_{i-1}\to \Simp{n}, X\to Y)\) is representable and the map
  \[
  \Hom(\Simp{n_i}, X)\to \Hom(\Horn{n_i}{j_i}\to \Simp{n_i}, X\to Y)
  \]
  is a cover. Thus the presheaf \(\Hom(S_i\to \Simp{n}, X\to Y)\) is representable, and
  \[
  \Hom(S_i\to \Simp{n}, X\to Y)\to \Hom(S_{i-1}\to \Simp{n}, X\to Y)
  \]
  is a cover. This proves the first statement. The second statement follows because composites of covers are covers.
\end{proof}

\begin{remark}
  If \(Y\) is a higher groupoid in \((\Cat, \covers)\), then \(X_0\times_{Y_0}Y_k\) is representable, since \(Y_k\to Y_0\) is a cover by Lemma~\ref{chap3:lem:covers_in_groupoid}.
\end{remark}

The argument above still works if we replace \(S\to T\to \Simp{n}\) by arbitrary collapsible extensions \(A\to B\to C\).

\begin{corollary}
Let \(f\colon X\to Y\) be a Kan fibration between simplicial objects in \((\Cat, \covers)\). Let \( A\to B\to C\) be collapsible extensions of finite simplicial sets. Assume that \(\Hom(A\to C, X\to Y)\) is representable.
Then \(\Hom(B\to C, X\to Y)\) is representable, and
\[
 \Hom(B\to C, X\to Y)\to \Hom(A\to C, X\to Y)
\]
is a cover \textup(compare with~\cite[Lemma~\textup{2.10}]{Henriques}\textup).
\end{corollary}

\subsection{Composition of Kan fibrations}

Composites of Kan fibrations are again Kan fibrations for simplicial sets. We show a similar result for Kan fibrations of simplicial objects in \((\Cat,\covers)\).

\begin{proposition} \label{chap3:prop:comp}
  Let \(f\colon X\to Y\) and \(g\colon Y\to Z\) be morphisms of simplicial objects in~\(\Cat\). If\/ \(\Hom(\Horn{n}{j}\to \Simp{n}, X\to Z)\) is representable and \(f\) and~\(g\) satisfy \(\Kan(n,j)\), then so does \(g\circ f\). If
  \(f\) and~\(g\) satisfy \(\Kan!(n,j)\), then so does \(g\circ f\).
\end{proposition}

\begin{proof}
Lemma~\ref{chap2:lem:lift-composition-diagram} gives a pullback square:
\[
\xymatrix{
\Hom(\Horn{n}{j}\to \Simp{n}, X\to Y) \ar [r] \ar[d] &\Hom(\Horn{n}{j}\to \Simp{n}, X\to Z)\ar[d] \\
    \Hom(\Simp{n}, Y) \ar[r] & \Hom(\Horn{n}{j} \to \Simp{n}, Y\to Z) \rlap{\ .}
}
\]
The condition \(\Kan(n, j)\) for \(g\) implies that \(\Hom(\Simp{n}, Y) \to \Hom(\Horn{n}{j} \to \Simp{n}, Y\to Z)\) is a cover. Therefore, the map \(\Hom(\Horn{n}{j}\to \Simp{n}, X\to Y) \to \Hom(\Horn{n}{j}\to \Simp{n}, X\to Z) \) is a cover. The condition \(\Kan(n, j)\) for \(f\) gives a cover \(\Hom(\Simp{n}, X) \to \Hom(\Horn{n}{j}\to \Simp{n}, X\to Y)\). Thus the map \(\Hom(\Simp{n}, X) \to \Hom (\Horn{n}{j}\to \Simp{n}, X\to Z)\), being the composite of two covers, is a cover. This proves the first claim. The second claim follows similarly.
\end{proof}

\begin{corollary}
  Composites of Kan fibrations of simplicial objects in~\((\Cat, \covers)\) are again Kan fibrations.
\end{corollary}

\subsection{The fibre of a Kan fibration}

Given a higher groupoid Kan fibration, its fibre is a well-defined higher groupoid.

\begin{definition}
  \label{chap3:def:fibre_Kan_fibration}
  Let \(f\colon X\to Y\) be a higher groupoid Kan fibration in \((\Cat,\covers)\). The \emph{fibre} of~\(f\) is the simplicial object~\(\Fib{f}\) with \(\Fib{f}_k=\Hom(\Simp{k}\to \Simp{0}, X\to Y)\) and with face and degeneracy
  maps induced by morphisms in~\(\bD\).
\end{definition}

\begin{remark}\label{chap3:rem:Kan-fibre-adjunction}
The functor
\[
 S\mapsto \Hom(S\to \Simp{0}, X\xrightarrow{f} Y), \quad S\in \SSet,
\]
sends colimits to limits. We claim that there is a natural isomorphism
\[
\Hom(S\to \Simp{0}, X\xrightarrow{f} Y)\cong \Hom(S, \Fib{f})
\]
for every \(S\in \SSet\). First, this isomorphism holds for \(S=\Simp{n}\) by definition. The claim then follows by taking limits on both sides.
\end{remark}

\begin{lemma} \label{chap3:lem:fibre_exists}
  The presheaf\/ \(\Hom(\Simp{k}\to \Simp{0}, X\to Y)\) is representable in~\(\Cat\), and\/~\(\Fib{f}\) is a simplicial object
  in~\(\Cat\). Moreover, \(\Fib{f}\) is a higher groupoid in \((\Cat, \covers)\).

  If \(f\) is an \(n\)\nbdash{}groupoid Kan fibration, then \(\Fib{f}\) is an \((n-1)\)-groupoid in \((\Cat, \covers)\).
\end{lemma}
\begin{proof}
First, \(\Hom(\Simp{0}\xrightarrow{0} \Simp{k}, X\to Y) \) is representable. Since \(\Simp{0}\xrightarrow{0} \Simp{k}\) is a collapsible extension, Lemma~\ref{chap3:lem:representable} implies that the map
\[
X_k\to \Hom(\Simp{0}\xrightarrow{0} \Simp{k}, X\to Y)
\]
is a cover in~\(\Cat\). Lemma~\ref{chap2:lem:lift-composition-diagram} gives a pullback square
\[
\xymatrix{
\Hom(\Simp{k}\to \Simp{0}, X\to Y)\ar[r]\ar[d] & \Hom(\Simp{k}, X)\ar[d]\\
\Hom(\Simp{0}\to \Simp{0}, X\to Y)\ar[r] & \Hom(\Simp{0}\xrightarrow{0} \Simp{k}, X\to Y)\rlap{\ ,}
}\]
thus \(\Hom(\Simp{k}\to \Simp{0}, X\to Y)\) is representable and \(\Fib{f}\) is a simplicial object in \(\Cat\).

To prove Kan conditions for \(\Fib{f}\), we consider the following pullback square implied by Lemma~\ref{chap2:lem:lift-composition-diagram}:
\[\xymatrix{
\Hom(\Simp{k}\to \Simp{0}, X\to Y) \ar[r] \ar[d] & \Hom(\Simp{k}, X) \ar[d]\\
\Hom(\Horn{k}{i}\to \Simp{0}, X\to Y) \ar[r]     & \Hom(\Horn{k}{i}\to \Simp{k} , X\to Y )\rlap{\ .}
}\]
Remark~\ref{chap3:rem:Kan-fibre-adjunction} shows that \(\Kan(k,i)\) or \(\Kan!(k,i)\) of \(f\) implies the same for \(\Fib{f}\).
\end{proof}

\subsection{Pullbacks of Kan fibrations}

A pullback of a Kan fibration of simplicial sets is still a Kan fibration. We now prove a similar result for pullbacks of higher groupoid Kan fibrations in \((\Cat, \covers)\).

It is clear that \(\Fib{f}\) fits into a pullback diagram
  \[
  \xymatrix{
  \Fib{f}\ar[r]\ar[d] & \sk_0 Y_0\ar[d]\\
  X\ar[r]             & Y \rlap{\ ,}
  }
  \]
where \(\sk_0 Y_0\) is the constant simplicial object with value \(Y_0\).

\begin{proposition}\label{chap3:prop:pullback-Kan-fibration}
Let \(P\) be the pullback \textup(as a simplicial object in presheaves on \(\Cat\)\textup) of the following diagram
\[
 \xymatrix{
 P\ar[r]\ar[d] & X\ar[d]\\
 Z\ar[r] & Y\rlap{\ ,}
 }
\]
where \(X, Y, Z\) are higher groupoids in \((\Cat, \covers)\). Assume that \(X\to Y\) is a Kan fibration and that \(P_0\) is representable. Then \(P\) is a higher groupoid object in \((\Cat, \covers)\) and \(P\to Z\) is a Kan fibration.
\end{proposition}
\begin{proof}
For any morphism of simplicial sets \(A\to B\), Lemma~\ref{chap2:lem:lift-pullback-pushout-diagram} yields a pullback square
\begin{equation}\label{chap3:eq:diagram-pullback}
\begin{gathered}
  \xymatrix{
  \Hom(B, P)\ar[r]\ar[d] & \Hom(B, X)\ar[d] \\
  \Hom(A\to B, P\to Z)\ar[r] & \Hom(A\to B, X\to Y)\rlap{\ .}
  }
\end{gathered}
\end{equation}
Let \(A\to B\) be a collapsible extension \(\Simp{0}\to B\). Since \(\Hom(B, Z)\to Z_0\) is a cover by Lemma~\ref{chap3:lem:covers_in_groupoid} and \(P_0\) is representable, we deduce that \(\Hom(\Simp{0}\to B, P\to Z)=P_0\times_{Z_0} \Hom(B, Z)\) is representable. Lemma~\ref{chap3:lem:representable} implies that \(\Hom(B, X)\to \Hom(\Simp{0}\to B, X\to Y)\) is a cover, hence \(\Hom(B, P)\) is representable. In particular, \(P\) is a simplicial object in~\(\Cat\).

Next let \(\Simp{0}\to A\to B\) be collapsible extensions. We consider the pullback diagram
\[
 \xymatrix{
 \Hom(A\to B, P\to Z)\ar[r]\ar[d] & \Hom(A, P)\ar[d]\\
 \Hom(B, Z)\ar[r] & \Hom(A, Z) \rlap{\ .}
 }
\]
Lemma~\ref{chap3:lem:covers_in_groupoid} implies that \(\Hom(B, Z)\to \Hom(A, Z)\) is a cover, hence \(\Hom(A\to B, P\to Z)\) is representable. Considering the pullback square~\eqref{chap3:eq:diagram-pullback}, we deduce that \(\Hom(B, P)\to \Hom(A\to B, P\to Z)\) is a cover because \(X\to Y\) is a Kan fibration. This proves that \(P\to Z\) is a Kan fibration.
\end{proof}

For examples of Kan fibrations we refer to the next section after we clarify what Kan fibrations mean in the setting of higher groupoid actions.

\section{Higher groupoid actions}\label{chap3:sec:2gpd-action}

Let us first recall Lie groupoid actions. Let \((J, \action)\) be a \(G\)-action on a manifold~\(E\). The action groupoid~\(A\) is a Lie groupoid with \(A_0=E\), \(A_1= E\times_{G_0,\target} G_1\), \(\source(x, g)= xg\), and \(\target(x, g)=x\). The natural projection \(\pi\colon A \to G\), given by \(\pi_1 =\pr_G\), and \(\pi_0=J\colon E\to G_0\), is a Lie groupoid functor.

\begin{lemma}\label{chap3:lem:1-action-kan}
  The projection \(\pi\colon A\to G\) is a Lie 1-groupoid Kan fibration.\footnote{Such Lie 1-groupoid Kan fibrations are also called discrete fibrations in the literature.}
\end{lemma}

\begin{proof}
  The condition \(\Kan!(1,0)\) for~\(\pi\) says that \(A_1\to A_0 \times_{\pi_0,G_0,\target} G_1\) is an isomorphism. This follows
  from the construction of~\(A_1\). The condition \(\Kan!(1,1)\) follows by applying the inverse map. Higher Kan conditions hold automatically, and we are done.
\end{proof}

The converse is also true: given a Lie 1-groupoid Kan fibration \(\pi\colon A\to G\), there is a natural action of \(G\) on \(A_0\) with action groupoid~\(A\). Thus, the Kan conditions combine all the data for a groupoid action. This motivates the following definition.

\begin{definition}\label{chap3:def:higher-groupoid-action}
  An \emph{action} of an \(n\)-groupoid object \(G\) in \((\Cat, \covers)\) is an \(n\)-groupoid Kan
  fibration \(A \xrightarrow{\pi} G\).
\end{definition}

We shall justify our definition for 2\nbdash{}groupoids in Chapter~\ref{chap5}. Theorem~\ref{chap5:thm:2gpd-action-Kan} shows that such a 2\nbdash{}groupoid Kan fibration is equivalent to a categorified action of a 2-groupoid.

\begin{remark}\label{chap3:rem:compare}
  In the category of sets, similar subjects have been studied: Duskin~\cite{Duskin} defined a notion of \(n\)-torsors for \(K(\Pi, n)\), which is a simplicial model for the Eilenberg--MacLane space. Glenn~\cite{Glenn:Realization} studied similar higher actions for Duskin's \(n\)-hypergroupoids (as in Remark~\ref{chap3:rem:duskin-hypergroupoid}). The action only asks Kan conditions for \(m>n\) as well. Later Baković~\cite{Bakovic:Bigroupoid_torsors} applies such actions to 2\nbdash{}groupoids, which are called bigroupoids in~\cite{Bakovic:Bigroupoid_torsors}, (thus lower Kan conditions are required as well), but the actions still have no lower Kan conditions. The author shows that a categorified action gives an action of a 2\nbdash{}groupoid in the sense of Glenn. But the opposite  direction is missing. Indeed, the lower Kan conditions must be required for this direction, as we can see in the
  reconstruction procedure in Section~\ref{chap5:sec:simp-cat}.
\end{remark}

\begin{remark}
  The fibre of an \(n\)\nbdash{}groupoid Kan fibration \(\pi\colon A\to G\) is an \((n-1)\)-groupoid by Proposition~\ref{chap3:lem:fibre_exists}. Intuitively, we may thought of \(G\) acts on the fibre as from the right (left) if we use \(\Kan(1,0)\) (\(\Kan(1,1)\)) to compose a 0-simplex in \(\Fib{\pi}\) and a 1-simplex in~\(G\).
\end{remark}

\begin{example}\label{chap3:exa:higher_groupoid_acting_on_objects}
  For a higher groupoid~\(G\) in \((\Cat,\covers)\), the identity map \(G\xrightarrow{\id} G\) is a Kan fibration.  The fibre~\(\Fib{\id}=\sk_0 G_0\) is a 0-groupoid.  This Kan fibration describes the obvious action of~\(G\) on its own object space~\(G_0\). This action may be regarded as the universal action of \(G\).
\end{example}

\begin{example}\label{chap3:exa:decalage}
  Given an \(n\)-groupoid \(G\) in \((\Cat,\covers)\), the \emph{d\'ecalage} of Illusie~\cite{Illusie} \(\dec(G)\) is the simplicial object obtained by shifting the simplicial degree up by one. More precisely, the simplicial object \(\dec(G)\) is characterised by the adjunction
  \[
  \Hom(S\star \Delta^0, G)=\Hom(S, \dec(G)), \quad\text{for any \(S\in \SSet\) and \(S\neq\emptyset\)},
  \]
  where \(\star\) is the join operation (Definition~\ref{chap2:def:join-simplicial-sets}). There is a natural map \(\pi\colon\dec(G)\to G\) induced by the inclusions \(S\hookrightarrow S\star \Simp{0}\). Replacing \(S\star\Simp{0}\) by \(\Simp{0}\star S\) gives a symmetric construction \(\pi'\colon\dec'(G)\to G\). We have
  \begin{align*}
  \Hom (\Horn{k}{i}\to\Simp{k}, \dec(G) \to G)&\cong\Hom (\Horn{k}{i}\star\Simp{0},G)\times_{\Hom(\Horn{k}{i},G)}\Hom(\Simp{k},  G))\\
  &=\Hom(\Horn{k}{i}\star\Simp{0}\cup_{\Horn{k}{i}}\Simp{k},  G)\\
  &=\Hom(\Horn{k+1}{i},G),
  \end{align*}
  where the last identify follows from Equation~\eqref{chap2:eq:join}. The commutative diagram
  \[
  \xymatrix{
    \Hom (\Simp{k}, \dec(G))\ar[d]\ar[r]^{\cong}& \Hom(\Simp{k+1},G)\ar[d]\\
    \Hom (\Horn{k}{i}\to\Simp{k}, \dec(G) \to G) \ar[r]^-{\cong}& \Hom(\Horn{k+1}{i},G)
  }
  \]
  shows that \(\Kan(k+1, i)\) for~\(G\) implies \(\Kan(k,i)\) for \(\pi\colon \dec(G) \to G\).  Thus, \(\pi\) is a Kan fibration between \(n\)\nbdash{}groupoids and \(\dec(G)\) is also an
  \(n\)-groupoid. We have \(\Fib{\pi}_0= G_1\), the space of arrows in~\(G\). We interpret \(\Fib{\pi}\) and \(\Fib{\pi'}\)  as higher groupoids of arrows in~\(G\).  The Kan fibrations \(\pi\) and \(\pi'\) correspond to the canonical actions of~\(G\) on their higher groupoids of arrows by right and left translations. Later in Example~\ref{chap3:exa:decalage-principal}, we will see that these higher actions are also principal, as one may expect.
\end{example}

\begin{example}
Let \(K_1\), \(G_1\), and \(H_1\) be groups, and let \(K\), \(G\), and \(H\) be the nerves of the corresponding one-object groupoids. Given a group extension (not necessarily abelian),
\[
1\to K_1\to G_1 \to H_1\to 1,
\]
the induced map of nerves \(G\to H\) is a 2\nbdash{}groupoid Kan fibration. The condition \(\Kan(1, j)\) holds because \(G_1\to H_1\) is surjective. For \(m \ge 2\), the condition \(\Kan!(m, j)\) for \(G\to H\) follows from the same conditions for \(H\) and \(G\) separately because
\begin{align*}
\hom(\Horn{m}{j} \to \Simp{m} , G \to H)&= \hom(\Horn{m}{j}, G)\times_{\hom(\Horn{m}{j}, H) } \hom(\Simp{m}, H),\\
                                        &\cong \hom(\Simp{m}, G).
\end{align*}
We know that there is no group action of \(H_1\) on \(K_1\), in general; however, a group extension gives an action of the 2-groupoid \(H\) on the 1-groupoid \(K\) with action 2-groupoid~\(G\).
\end{example}

\section{Higher principal bundles}
Let \(G\) be a Lie groupoid. After the introduction of acyclic fibrations, we will see that a \(G\)-bundle \(P\to M\) is principal if and only if the nerve of the projection \(P\rtimes G\to  \trivial{M}\) is an acyclic fibration. Then we define higher principal bundles by acyclic fibrations.

\subsection{Acyclic fibrations}
\begin{definition}\label{chap3:def:acyclic-fibration}
   A morphism \(f\colon X\to Y\) of simplicial objects in \(\Cat\) is an \emph {acyclic fibration} if it satisfies \(\Acyc(m)\) for all \(m\ge 0\), where \(\Acyc(m)\) means that \(\Hom(\partial\Simp{m}\to \Simp{m},X\to Y)\) is representable and the canonical map
  \begin{equation}\label{chap3:eq:acyclic-fibration}
    \Hom(\Simp{m}, X) \to \Hom(\partial\Simp{m}\to \Simp{m},X\to Y)
  \end{equation}
  is a cover.
\end{definition}

\begin{remark}
  As shown in~\cite[Lemma 2.4]{Zhu:ngpd}, if the conditions \(\Acyc(l)\) for \(0\le l<m\) are satisfied, then the presheaf \(\Hom(\partial\Simp{m}\to \Simp{m},X\to Y)\) is representable. Moreover, an acyclic fibration is a Kan fibration because the horn inclusions \(\Horn{m}{j} \to \Simp{m}\) are boundary extensions.
\end{remark}

\begin{lemma}[{\cite[Lemma~2.5]{Zhu:ngpd}}]\label{chap3:lem:acyclic-finite}
Let \(f\colon X\to Y\) be a morphism of \(n\)\nbdash{}groupoids in \((\Cat, \covers)\). If \(f\) satisfies \(\Acyc(m)\) for \(0\le m<n\) and \(\Acyc!(n)\), where \(\Acyc!(n)\) means that~\eqref{chap3:eq:acyclic-fibration} is an isomorphism, then \(\Acyc!(m)\) for \(m\ge n+1\) automatically holds, thus \(f\) is an acyclic fibration \textup(or hypercover in \cite{Zhu:ngpd}\textup).
\end{lemma}

Lemma~\ref{chap3:lem:acyclic-def} below shows the converse. Thus, unlike for Kan fibrations, an \(n\)-groupoid morphism \(f\colon X\to Y\) is an acyclic fibration of \(n\)\nbdash{}groupoids if and only if it is an acyclic fibration of simplicial objects.

\begin{lemma}\label{chap3:lem:acyclic-def}
  Let \(f\colon X\to Y\) be a morphism of \(n\)-groupoids in \((\Cat,\covers)\). If \(f\) satisfies \(\Acyc(k)\) for all \(k \ge n\), then it satisfies \(\Acyc!(k)\) for \(k \ge n\).
\end{lemma}

\begin{proof}
  When \(k \ge n+2\), since \(n\)-groupoids are \((n+1)\)-coskeletal, we have
  \[
  \Hom(\partial\Simp{k} \to \Simp{k}, X \to Y)=\Hom(\Simp{k} \to \Simp{k}, X \to Y)=\Hom(\Simp{k}, X).
  \]
  Thus \(\Acyc!(k)\) for \(f\) holds automatically.

  For \(k =n+1\), there is a pullback diagram
  \[
  \xymatrix{
    \Hom(\partial\Simp{k+1} \to \Simp{k+1}, X\to Y)\ar[r]^{u'}\ar[d]^{v'} & \Hom(\Horn{k+1}{0} \to \Simp{k+1}, X\to Y)\ar[d]^{v}\\
    \Hom(\Simp{k} \to \Simp{k}, X\to Y)\ar[r]^{u} &\Hom(\partial\Simp{k} \to \Simp{k}, X\to Y) \rlap{\ ,}
  }
  \]
  where \(v\) is induced by \(\face^0\colon \Simp{k}\to \Simp{k+1}\). The condition \(\Kan!(k+1,0)\) for \(X\) and \(Y\) implies that
  \[
  \Hom(\Horn{k+1}{0} \to \Simp{k+1}, X\to Y)=\Hom(\Simp{k+1} \to \Simp{k+1}, X\to Y)\cong\Hom(\Simp{k+1}, X).
  \]
  Since \(X\) is an \(n\)\nbdash{}groupoid and \(\face^0\colon \Simp{k}\to \Simp{k+1}\) is a collapsible extension, the induced map \(\Hom(\Simp{k+1}, X )\to \Hom(\Simp{k}, X)\) is a cover. Since \(\Hom(\Simp{k}, X )\xrightarrow{u} \Hom(\partial\Simp{k} \to \Simp{k}, X\to Y)\) is a cover by \(\Acyc(k)\), the composite
  \(\Hom(\Simp{k+1}, X)\xrightarrow{v} \Hom(\partial\Simp{k}\to \Simp{k}, X \to Y)\) is a cover. By the condition \(\Acyc!(k+1)\) that we just proved, \(u'\) is an isomorphism. Lemma~\ref{chap3:lemma:iso} below shows that \(\Acyc!(k)\) for \(k=n+1\) holds. A similar argument proves \(\Acyc!(n)\), and we are done.
\end{proof}

\begin{lemma}\label{chap3:lemma:iso}
  Let \(X\xrightarrow{u} Z\) and \(Y\xrightarrow{v} Z\) be covers in \((\Cat, \covers)\).  If the natural map \(X \times_Z Y  \xrightarrow{u'} Y\) is an isomorphism, then \(X\xrightarrow{u} Z\) is an isomorphism.
\end{lemma}

\begin{proof}
Recall from Lemma~\ref{chap3:lem:subcanonical-effective-epi} that in a subcanonical pretopology \((\Cat, \covers)\), covers \(u\) and~\(v\) are effective epimorphisms.

We now prove that~\(u\) is a monomorphism.  Let \(f, g \colon  A\to X\) be a pair of morphisms such that \(u\circ f=u  \circ g\).  By the universal property, there are morphisms \(f' , g'\colon A\times_Z Y \to X\times_Z Y\) such that \(u'\circ f'=u' \circ g'\) and the two left squares, involving \(f\) and \(f'\) or \(g\) and \(g'\), in the diagram below are pullback diagrams:
  \[
  \xymatrix{
    A\times_Z Y\ar@{.>}@<.5ex>[r]^{f'} \ar@{.>}@<-.5ex>[r]_{g'}\ar[d]^{v''}   &  X\times_Z Y\ar[d]^{v'} \ar[r]^-{u'}   &Y\ar[d]^{v} \\
    A\ar@<.5ex>[r]^{f} \ar@<-.5ex>[r]_{g}           &      X\ar[r]^{u}         &Z\rlap{\ .}
  }
  \]
As \(u'\) is an isomorphism, \(u'\circ f'=u' \circ g'\) implies \(f'=g'\). Therefore, \(f \circ v''=v' \circ f'=v' \circ g'=g \circ v''\).  Since \(v''\) is a cover, hence an effective epimorphism, we obtain \(f=g\), and this proves that \(u\) is a monomorphism.

We claim that if \(u\colon X \to Z\) is a monomorphism and an effective epimorphism, then it is an isomorphism. Consider the coequaliser diagram
  \[
  \xymatrix@1{
    X\times_Z  X \ar@<.5ex>[r]^-{p_1} \ar@<-.5ex>[r]_-{p_2} &  X\ar[r]^{u} &  Z
  }.
  \]
Since~\(u\) is a monomorphism, \(u \circ p_1= u \circ p_2\) implies \(p_1=p_2\).  Therefore, \(u\) is an isomorphism. This completes the proof.
\end{proof}

\begin{remark}\label{chap3:rem:pullback-along-cover-iso}
  The assumption that \(u\) is a cover in the above lemma is redundant. In this case, we can still prove that \(u\) is a momomorphism. Since \(u\circ v'=v\circ u'\) is an effective epimorphism, \(u\circ v'\) is an extremal epimorphism, and so is \(u\). Now \(u\) is a momomorphism and an extremal epimorphism, thus it is an isomorphism. For the definition of extremal epimorphisms and the reason of this remark see~\cite[Section 4.3]{Borceux}.
\end{remark}

\begin{proposition}\label{chap3:prop:acyclic-stable-comp-pullbak}
 Acyclic fibrations of \(n\)-groupoids in \((\Cat, \covers)\) are closed under composition and stable under pullback along any morphism.
\end{proposition}
\begin{proof}
 This is stated in Lemmas 2.6, 2.7, and 2.8 in~\cite{Zhu:ngpd}.
\end{proof}

\begin{corollary}
  Let \(f\colon X\to Z\) and \(g\colon Y\to Z\) be \(n\)-groupoid Kan fibrations in \((\Cat,\covers)\). Suppose that \(h\colon X\to Y\) is an acyclic fibration such that \(f=g\circ h\). Then the induced morphism \(\Fib{f}\to \Fib{g}\) is an acyclic fibration of \((n-1)\)\nbdash{}groupoids.
\end{corollary}
\begin{proof}
  By definition, each square below is a pullback diagram
  \[
  \xymatrix{
  \Fib{f}\ar[r]\ar[d]&\Fib{g}\ar[r]\ar[d]&\sk_0 Z_0\ar[d]\\
  X\ar[r]^{h} & Y\ar[r]^{g}& Z\rlap{\ .}
  }
  \]
  Applying Proposition~\ref{chap3:prop:acyclic-stable-comp-pullbak} to the left square proves the claim.
\end{proof}

\subsection{Higher principal bundles}\label{chap3:sec:hihger-principal-bundles}

We first characterise acyclic fibrations of Lie groupoids.
\begin{lemma}
  Let \(f\colon G\to H\) be a Lie groupoid functor. Then the nerve of \(f\) is acyclic if and only if \(f\) is a weak equivalence such that \(f_0\) is a cover.
\end{lemma}
\begin{proof}
  The condition \(\Acyc(0)\) for \(Nf\) says that \(f_0\) is a cover. The condition \(\Acyc!(1)\) for \(Nf\) says that the map
  \[
  G_1\to \Hom(\partial\Simp{1}\to \Simp{1}, NG\to NH)=(G_0\times G_0)\times_{H_0\times H_0} H_1
  \]
  is an isomorphism. This proves the statement by Lemma~\ref{chap3:lem:acyclic-finite}.
\end{proof}

\begin{lemma}
  Let \(\pi\colon A\to G\) be a Lie groupoid Kan fibration with fiber~\(P\). Let \(N\) be a manifold with a Lie groupoid functor \(\kappa\colon A\to \trivial{N}\). Then \(\kappa_0\colon P\to N\) is a \(G\)-bundle.
\end{lemma}
\begin{proof}
  Since \(A\to \trivial{N}\) is a Lie groupoid functor, the map \(P\to N\) is \(G\) invariant.
\end{proof}

These two lemmas combined characterise principal bundles of Lie groupoids.
\begin{proposition}
  Let \(G\) be Lie groupoid. Then a principal \(G\)-bundle over a manifold \(N\) is the same as a Lie groupoid Kan fibration \(\pi\colon A\to G\) with an acyclic fibration \(\kappa\colon A\to \trivial{N}\).
\end{proposition}
\begin{proof}
  Let \(P\to N\) be a \(G\)-bundle with action groupoid \(A\). The condition \(\Acyc(0)\) for \(NA\to \sk_0 N\) says that the map \(A_0=P\to N\) is a cover. The condition \(\Acyc!(1)\) for \(NA\to\sk_0 N\) says that the map
  \[
  A_1\to A_0\times_N A_0,\quad (p, g)\mapsto (p, pg),
  \]
  is an isomorphism, and we are done.
\end{proof}

Consequently, a Lie groupoid principal bundle can be given by a Lie groupoid Kan fibration and an acyclic fibration. This motivates the following definition.
\begin{definition}\label{chap3:def:higher-principal-bundles}
  Let \(\pi\colon A\to G\) be an \(n\)-groupoid Kan fibration in \((\Cat, \covers)\) giving an action of \(G\) on \(\Fib{\pi}\). Let \(N\) be an object in \(\Cat\) with a morphism \(\kappa\colon A\to\sk_0 N\). We call \(\Fib{\pi}\) a \emph{\(G\)\nbdash{}bundle} over~\(N\). If, in addition, \(\kappa\colon A\to \sk_0 N\) is acyclic, then we call \(\Fib{\pi}\) a \emph{principal} \(G\)-\emph{bundle} over~\(N\).
\end{definition}

We will justify our definition for 2\nbdash{}groupoids in Section~\ref{chap5:sec:principal-2-bundles} by showing that a categorified principal bundle is equivalent to a principal 2\nbdash{}bundle as defined above.

\begin{remark}
 In the literature, principal bundles are also known as torsors. Duskin~\cite{Duskin} defined an \(n\)\nbdash{}torsor for \(K(\Pi, n)\) to be a Kan fibration \(A\to K(\Pi, n)\) such that \(A\) is acyclic and \((n-1)\)-coskeletal. Let \(G\) be a Duskin \(n\)-hypergroupoid. Glenn~\cite{Glenn:Realization} defined an \(n\)\nbdash{}torsor of \(G\) over \(N\) by simplicial maps \(\sk_0 N\leftarrow A\to G\) such that \(A\) gives an action as in Remark~\ref{chap3:rem:compare} and \(A\to \sk_0 N\) is acyclic and \((n-1)\)-coskeletal. Later, Baković~\cite{Bakovic:Bigroupoid_torsors} shown that a set-theoretical 2\nbdash{}torsor by categorification gives a 2\nbdash{}torsor in Glenn's sense. The converse, however, is also missing.
\end{remark}

\begin{remark}
  For a 1-groupoid acyclic fibration \(f\colon A\to \sk_0 N\) in \((\Cat, \covers)\), \(N\) is the quotient \(A_0/A_1\), that is, the coequaliser of two parallel morphisms \(\face_1,\face_2\colon A_1\to A_0\). For certain categories like \((\Mfd,\covers_\surj)\), it is possible to speak of principal bundles of groupoids without mentioning the quotient. For \((\Mfd,\covers_\surj)\), it is well-known that a free and proper action yields a principal bundle.
\end{remark}

\begin{remark}
  Nikolaus--Schreiber--Stevenson~\cite{NSS:general,NSS:presentation} studied principal \(\infty\)-bundles of \(\infty\)-groups in an \(\infty\)-topos. In the discrete case, their \(\infty\)-groups are simplicial groups, which are not the same as ours but a special type; their principal bundles are weakly principal Kan simplicial bundles and the base spaces are generally Kan simplicial sets other than sets (see~\cite[Section 4.1]{NSS:presentation}).
\end{remark}

\begin{example}\label{chap3:exa:decalage-principal}
  Recall the décalage \(\dec(G)\) for an \(n\)\nbdash{}groupoid \(G\) constructed in
  Example~\ref{chap3:exa:decalage}. The inclusion \(\Simp{0}\to S\star\Simp{0}\) induces a morphism of simplicial objects \(\kappa\colon \dec(G)\to \sk_0 G_0\). We have
  \[
  \Hom (\partial\Simp{k}\to\Simp{k}, \dec(G) \to\sk_0 G_0)\cong\Hom (\partial\Simp{k}\star\Simp{0}, G)=\Hom(\Horn{k+1}{k+1},G),
  \]
  where the last identity follows from Equation~\eqref{chap2:eq:join}. This gives a commutative diagram
  \[
  \xymatrix{
    \Hom (\Simp{k}, \dec(G))\ar[d]\ar[r]^{\cong}& \Hom(\Simp{k+1},G)\ar[d]\\
    \Hom (\partial\Simp{k}\to\Simp{k}, \dec(G) \to\sk_0 G_0) \ar[r]^-{\cong}& \Hom(\Horn{k+1}{k+1},G)\rlap{\ .}
  }
  \]
  Hence \(\Kan(k+1,k+1)\) for~\(G\) implies \(\Acyc(k)\) for \(\kappa\colon \dec(G) \to \sk_0 G_0\). Thus~\(\kappa\) is an acyclic fibration. These two statements combined imply that \(\dec(G)\) is the action \(n\)-groupoid of a principal \(G\)-bundle over \(G_0\).
\end{example}

\chapter{Higher Groupoid Bibundles}\label{chap4}
\thispagestyle{empty}

In this chapter, our goal is to define bibundles between higher groupoids.

We first review the theory of bimodules between categories. A bimodule can be equivalently given by its cogragh, which is a category with a natural functor to the interval category \(I\) such that the preimages of the two ends are the given two categories.

This machinery works also for Lie groupoids. Furthermore, HS bibundles correspond to simplicial objects in \(\Mfd\) over \(\Simp{1}\) satisfying appropriate left Kan conditions.

Next, we collect some facts about augmented simplicial objects and bisimplicial objects in an extensive category. We then define colored Kan conditions. Bibundles between higher groupoids are defined by simplicial objects over \(\Simp{1}\) that satisfy appropriate Kan conditions. For example, a morphism of higher groupoids gives such a bibundle.

\section{Bimodules between categories and cographs}

We start by reviewing the classical notion of bimodules between two categories. This section is based on \cite{Benabou:dist}; see also~\cite{Borceux,Street} and~\cite[Section 2.3.1]{Lurie}. Let \(\Cat[A]\) and \(\Cat[B]\) be small categories throughout this section.

\subsection{Bimodules and profunctors}

The notions of actions of categories on sets, equivariant maps, and action categories are defined as in Section~\ref{chap1:sec:action-bibundle}. Notice that the inverse is not used in the definitions. Similarly, \emph{bimodules between categories} are defined like bibundles between Lie groupoids.

\begin{remark}\label{chap4:rem:action=functor}
  A left action of \(\Cat[A]\) on \(P\) is equivalent to a functor \(\psi\colon \Cat[A]\to \Sets\) such that \(P=\bigsqcup_{A\in \Cat[A]_0} \psi(A)\). Given a functor \(\psi\), then \(f\cdot p=\psi(f)(p)\) for \(f\colon A\to A'\) and \(p\in \psi(A)\) defines an action. The category of elements of \(\psi\) is exactly the associated action category.
\end{remark}

\begin{definition}\label{chap4:def:cat-profunctor}
  A \emph{profunctor} from \(\Cat[A]\) to \(\Cat[B]\) is a functor \(\psi\colon \Cat[B]^{\op}\times \Cat[A]\to \Sets\).
\end{definition}

\begin{remark}\label{chap4:rem:bimodule=profunctor}
We construct a bimodule from a profunctor \(\psi\). Let \(P=\bigsqcup_{A\in \Cat[A]_0, B\in \Cat[B]_0}\psi(B, A)\) with natural maps \(P\to \Cat[A]_0\) and \(P\to \Cat[B]_0\). For \(f\colon A\to A'\) and \(p\in \psi(B, A)\), the action is given by \(f\cdot p=\psi(\id_B, f)p\). The right action is defined similarly. Conversely, given an \(\Cat[A]\)-\(\Cat[B]\) bimodule \(P\), we can construct a profunctor. We deduce that profunctors and bimodules are equivalent notions. With this equivalence understood, bimodules are also known as distributors~\cite{Benabou:dist} and correspondences~\cite{Lurie}.
\end{remark}

\begin{remark}
 The composition of profunctors is best written as a coend. It then follows from some basic properties of coends that there is a 2-category with categories as objects and profunctors as 1-morphisms. See~\cite{Cattani-Winskel} for more details. It is straightforward to translate these results into bimodules.
\end{remark}

\begin{figure}[htbp]
  \centering
  \begin{tikzpicture}%
  [>=latex', mydot/.style={draw,circle, inner sep=1.5pt},every label/.style={scale=0.8}]
  \node[mydot,fill=black]    at (6, 0) (h0) {};
  \node[mydot,fill=black]    at (4,0)  (h1) {};
  \node[mydot]  at (2,0)  (h2) {};
  \node[mydot]  at (0,0)  (h3) {};
  \path[->]
     (h0) edge[out=160, in=20] node[scale=0.8,below]{$b$} (h1)
	 edge[out=150, in=30] node[scale=0.8,above]{$pb$}(h2)
   edge[out=140, in=40] node[scale=0.8,above]{$apb$}(h3)
     (h1) edge[out=160, in=20] node[scale=0.8,below]{$p$} (h2)
   edge[out=150, in=30] node[scale=0.8,above]{$ap$}(h3)
     (h2) edge[out=160, in=20] node[scale=0.8,below]{$a$} (h3);
  \end{tikzpicture}
  \caption{A bimodule between categories}\label{chap4:fig:category-bimodule}
\end{figure}
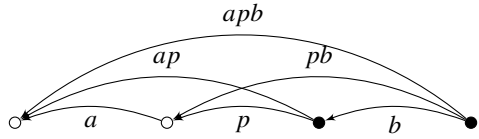
We depict a bimodule between categories as in Figure~\ref{chap4:fig:category-bimodule}, where white points and black points represent objects of \(\Cat[A]\) and \(\Cat[B]\), respectively. Arrows with ends of the same color represent arrows in the corresponding categories, and arrows from black points to white points represent elements of \(P\).

\subsection{Cographs of bimodules}

For a bimodule between categories, all the arrows in Figure~\ref{chap4:fig:category-bimodule} form a new category:

\begin{definition}
  Let \(P\) be a bimodule between \(\Cat[A]\) and \(\Cat[B]\). Its \emph{cograph}~\(\Gamma\), also called collage~\cite{Street81} or linking category~\cite{Blohmann}, is a category with
  \begin{gather*}
  \Gamma_0=\Cat[A]_0\sqcup \Cat[B]_0,         \quad \Gamma_1=\Cat[A]_1\sqcup P\sqcup \Cat[B]_1, \\
  \source=\source_{\Cat[A]}\sqcup J_r \sqcup \source_{\Cat[B]},\quad \target=\target_{\Cat[A]}\sqcup J_l \sqcup \target_{\Cat[B]},
  \end{gather*}
  and with composition induced by the compositions in \(\Cat[A]\) and \(\Cat[B]\) and the two actions.
\end{definition}

\begin{example}
  Let \(P=\Cat[A]_0\times \Cat[B]_0\) and let \(P\to \Cat[A]_0\) and \(P\to \Cat[B]_0\) be the projections. Then~\(P\) is a bimodule with trivial actions. Its cograph is a category obtained from \(\Cat[A]\sqcup \Cat[B]\) by adding a unique arrow from \(B\) to \(A\) for each pair \((A, B)\in \Cat[A]_0\times \Cat[B]_0\). This category is also called the  \emph{join of two categories}, denoted by \(\Cat[A]\star \Cat[B]\). In particular, \([m]\star [n]=[m+n+1]\) is the ordinal sum. The nerve preserves joins
  ~\cite[Corollary 3.3]{Joyal2008}:
  \[
   N(\Cat[A]\star\Cat[B])=N\Cat[A]\star N\Cat[B],
  \]
  where \(\star\) on the right hand side is the join of simplicial sets (Definition~\ref{chap2:def:join-simplicial-sets}). It is easy to see that \(P\) is the terminal object among \(\Cat[A]\)-\(\Cat[B]\) bimodules.
\end{example}

\begin{example}
  As for Lie groupoids, given a functor \(f\colon \Cat[A]\to \Cat[B]\), there is a bimodule structure on \(P=\Cat[A]_0\times_{\Cat[B]_0, \target} \Cat[B]_1\). Biequivariant maps between two bimodules of this type are in a bijection with natural transformations between the corresponding two functors. Therefore, we may regard bimodules as generalised functors.
\end{example}

\subsection{Bimodules and categories over the interval}

Denote by \(I\) the interval category \([1]=\{0 \leftarrow 1\}\). The cograph \(\Gamma\) admits a natural functor to~\(I\) that sends \(\Cat[A]\) to \(0\), \(\Cat[B]\) to \(1\), and \(P\) to the unique non-identity arrow. Conversely, let \(\Gamma\) be a category with a functor \(\Gamma\to I\) such that \(\Cat[A]\) is the preimage of \(0\), \(\Cat[B]\) is the preimage of~\(1\). Then the preimage of the unique non-trivial arrow \(P\) is a bimodule between \(\Cat[A]\) and~\(\Cat[B]\). These two constructions are mutually inverse, so we get:

\begin{proposition}\label{chap4:prop:cat-bimodule=cats-over-I}
  The category of \(\Cat[A]\)-\(\Cat[B]\) bimodules is isomorphic to the subcategory of \(\Cats/I\) where the two ends are \(\Cat[A]\) and \(\Cat[B]\).
\end{proposition}

The nerve of the cograph decomposes as follows
\[
N_n\Gamma=\bigsqcup_{i+j+1=n} N_{i,j} \Gamma, \quad i, j\ge -1,
\]
where \(N_{i,j}\) has \(i\) copies of \(\Cat[A]_1\), \(1\) copy of \(P\), and \(j\) copies of \(\Cat[B]_1\) for \(i,j\ge 0\),
\begin{equation}\label{chap4:eq:Nij-Gamma-cograph}
N_{i,j} \Gamma=\underbrace{\Cat[A]_1\times_{\Cat[A]_0} \dots \times_{\Cat[A]_0} \Cat[A]_1}_{\text{\(i\) copies }}\times_{\Cat[A]_0} P\times_{\Cat[B]_0} \underbrace{\Cat[B]_1\times_{\Cat[B]_0} \cdots \times_{\Cat[B]_0} \Cat[B]_1}_{\text{\(j\) copies}},
\end{equation}
while \(N_{i,-1}\Gamma=N_i \Cat[A]\) and \(N_{-1,j}\Gamma=N_j\Cat[B]\). The natural map \(N\Gamma\to NI\) sends \(N_{i,j}\Gamma\) to the unique \((i+j+1)\)-simplex of \(\Simp{1}\) with \(i+1\) copies of \(0\) and \(j+1\) copies of \(1\).

\section{Bibundles of Lie groupoids via simplicial manifolds}

Let \(P\) be a bibundle between Lie groupoids \(G\) and \(H\) with cograph~\(\Gamma\). Since the source and target maps of Lie groupoids are covers in the surjective submersion pretopology, \(N_{i,j}\Gamma \) in~\eqref{chap4:eq:Nij-Gamma-cograph} is representable in~\(\Mfd\). Therefore, \(N \Gamma\) is a simplicial manifold and \(\Gamma\) is a category object in \(\Mfd\). Proposition~\ref{chap4:prop:cat-bimodule=cats-over-I} now reads:

\begin{proposition}\label{chap4:prop:groupoid-bibunble=cats-over-I}
  The category of bibundles between Lie groupoids \(G\) and \(H\) is isomorphic to the category of categories in \(\Mfd\) over \(I\) with the two ends given by \(G\) and~\(H\), where we equip \(I\) with discrete manifold structures on object space and arrow space.
\end{proposition}

Proposition~\ref{chap2:prop:nerve-of-functor-is-inner-kan} implies that \(N\Gamma\to NI=\Simp{1}\) is an inner Kan fibration. We show that the right principality (Definition~\ref{chap1:def:principal-bundle}) can be encoded by appropriate left Kan conditions.

\begin{lemma}\label{chap4:lem:groupoid-Gamma-I-kan10}
  Let \(P\) be a \(G\)-\(H\) bibundle with cograph \(\Gamma\). The morphism \(N\Gamma \to \Simp{1}\) satisfies \(\Kan(1, 0)\) if and only if \(J_l\colon P\to G_0\) is a surjective submersion.
\end{lemma}
\begin{proof}
The commutative diagram
\[
\xymatrix{
\Hom(\Simp{1}, N\Gamma)\ar[r]^{=}\ar[d]& G_1\sqcup P\sqcup H_1\ar[d]^{\target\sqcup J_l\sqcup \target}\\
\Hom(\Horn{1}{0}\to\Simp{1}, N\Gamma\to \Simp{1}) \ar[r]^-{=}& G_0\sqcup G_0\sqcup H_0
}
\]
proves the claim.
\end{proof}

\begin{lemma}\label{chap4:lem:groupoid-Gamma-I-kan20}
  Suppose that \(P\) is as in Lemma~\emph{\ref{chap4:lem:groupoid-Gamma-I-kan10}} and that \(J_l\colon P\to G_0\) is a surjective submersion. The morphism \(N\Gamma \to \Simp{1}\) satisfies \(\Kan!(2, 0)\) if and only if the map
    \[
    P\times_{H_0,\target} H_1\to P\times_{G_0} P ,\quad (p, h)\mapsto (p, p\cdot h),
    \]
  is a diffeomorphism.
\end{lemma}
\begin{proof}
We calculate
\begin{equation}\label{chap4:eq:hom-simp2-Gamma}
\Hom(\Simp{2}, N\Gamma)\cong G_1\times_{G_0} G_1\sqcup G_1\times_{G_0} P\sqcup P\times_{H_0} H_1\sqcup H_1\times_{H_0} H_1,
\end{equation}
and
\begin{equation}\label{chap4:eq:hom-horn20-Gamma}
\Hom(\Horn{2}{0}\to\Simp{2}, N\Gamma\to \Simp{1})\cong
 G_1\times_{G_0} G_1\sqcup G_1\times_{G_0} P\sqcup P\times_{G_0} P \sqcup H_1\times_{H_0} H_1.
\end{equation}
The natural map from~\eqref{chap4:eq:hom-simp2-Gamma} to~\eqref{chap4:eq:hom-horn20-Gamma} sends the \(i\)-th component to the \(i\)-th component for \(i\in\{1, 2, 3, 4\}\) by an appropriate map \((a, b)\mapsto (a, ab)\). Since \(G\) and \(H\) are Lie groupoids, it remains to consider the map of the third components
\[
P\times_{H_0} H_1 \to P\times_{G_0} P ,\quad (p, h)\mapsto (p, ph),
\]
and this proves the lemma.
\end{proof}

Two lemmas above combined give an alternative formulation of HS bibundles:

\begin{proposition}
  Let \(P\) be a \(G\)-\(H\) bibundle with cograph \(\Gamma\). Then \(P\) is right principal if and only if \(N\Gamma \to \Simp{1}\) satisfies \(\Kan(1,0)\) and \(\Kan!(n,0)\) for \(n\ge 2\).
\end{proposition}
\begin{proof}
  It remains to show that right principality implies \(\Kan!(n, 0)\) for \(n\ge 3\). We have \(\Kan!(2, 1)\),  \(\Kan!(2,0)\) and \(\Kan!(3,i)\) for \(0<i<3\). Lemma~\ref{chap3:lem:morphism-kan(n+1,j)-implies-all} yields \(\Kan!(3,0)\). An induction shows \(\Kan!(n, 0)\) for \(n\ge 4\).
\end{proof}

To summarise, we obtain an enhanced version of Proposition~\ref{chap4:prop:groupoid-bibunble=cats-over-I}.

\begin{proposition}\label{chap4:prop:groupoid-HSbibunble=left-kan-cats-over-I}
  Let \(G\) and \(H\) be Lie groupoids. The category of \(G\)-\(H\) HS bibundles is isomorphic to the category of categories \(\Gamma\) in \(\Mfd\) over \(I\) with the two ends given by \(G\) and~\(H\) such that \(N\Gamma\to \Simp{1}\) satisfies \(\Kan(1, 0)\) and \(\Kan!(n, 0)\) for \(n\ge 2\).
\end{proposition}

\subsection{Action groupoids in the cogragh}

Let \(P\) be a \(G\)-\(H\) bibundle with cograph \(\Gamma\). There are three action groupoids: for the \(G\)\nbdash{}action on \(P\), for the \(H\)-action on \(P\), and for the biaction (Section~\ref{chap1:sec:comparison-of-2-cat-of-lie-groupoids}). We shall describe how these action groupoids sit inside the cograph. Recall that the nerve of \(\Gamma\) decomposes as \(N_n\Gamma=\bigsqcup_{i+j+1=n}N_{i,j} \Gamma\) with \(N_{i, -1}\Gamma=N_i G\) and \(N_{-1, i}\Gamma=N_i H\).

\begin{proposition}
  The action groupoid \(P\rtimes H\) is given by the collection of \(N_{0, i}\Gamma\) for \(i\ge 0\), and the functor \(P\rtimes H\to H\) is given by \(N_{0, i}\Gamma\to N_{-1, i}\Gamma\) for \(i\ge 0\); Similarly, the functor \(G\ltimes P\to H\) is given by \(N_{i, 0}\Gamma\to N_{i, -1}\Gamma\) for \(i\ge 0\).
\end{proposition}
\begin{proof}
By symmetry, it suffices to prove the first statement. By definition,
\[
N_{0, i}\Gamma=P\times_{H_0} H_1\times_{H_0} \dots \times_{H_0} H_1,\quad i \ge 0.
\]
The sets of \(N_{0, i}\Gamma\) for \(i\ge 0\) with face maps
\[
\face_j=\face_{j+1}^{N\Gamma}\colon N_{0, i}\Gamma\to N_{0,i-1}\Gamma,
\]
and degeneracy maps
\[
\de_j=\de_{j+1}^{N\Gamma}\colon N_{0, i}\Gamma\to N_{0, i+1}\Gamma
\]
form a simplicial manifold, which is easily identified with the nerve of the action groupoid \(P\rtimes H\). Moreover, the collection of maps \(\face_0^{N\Gamma}\colon N_{0, i}\Gamma\to N_{-1, i}\Gamma\) gives the nerve of the projection \(P\rtimes H\to H\), and this completes the proof.
\end{proof}

\begin{proposition}\label{chap4:prop:biaction-nerve=diagonal}
  The biaction groupoid \(G\ltimes P\rtimes H\) is given by the collection of \(N_{i,i}\) for \(i\ge 0\).
\end{proposition}
\begin{proof}
  The nerve of the biaction groupoid is given by
  \[
  N_i (G\ltimes P\rtimes H)
  =G_1\times_{G_0} \dots \times_{G_0} G_1 \times_{G_0} P\times_{H_0} H_1\times_{H_0} \dots \times_{H_0} H_1
  =N_{i,i}\Gamma.
  \]
  The face map \(\face_j\colon N_{i,i}\Gamma\to N_{i-1,i-1} \Gamma\) is given by \(\face^{N\Gamma}_{i+j}\circ \face^{N\Gamma}_{i-j}\). The degeneracy map \(\de_j\colon N_{i,i}\to N_{i+1, i+1}\) is given by \(\de^{N\Gamma}_{i+j+2}\circ \de^{N\Gamma}_{i-j}\).
\end{proof}

We will see that \(N_{i,j}\Gamma\) is an augmented bisimplicial object. This allows to simplify the description of these action groupoids (see Example~\ref{chap4:ex:Lie-groupoid-bibundle-aug-bisimp}).

\section{Augmented simplicial sets and bisimplicial sets}\label{chap4:sec:aug-simplicial-bisimplicial}

In this section, we collect some basic facts about augmented simplicial sets and (augmented) bisimplicial sets. This section is based on~\cite[Chapter 3]{Joyal2008}. We will see that augmented simplicial sets over the augmented simplicial interval and augmented bisimplicial sets are the same.

\subsection{Augmented simplicial sets}

Let \(\bD_{+}\) be the category of all finite ordered sets with order-preserving maps as morphisms; that is, \(\bD_{+}=\bD\cup [-1]\), where \([-1]=\emptyset\).

\begin{definition}\label{chap4:def:augmented-simplicial}
  The category of \emph{augmented simplicial sets} \(\SSet_+\) is the functor category \([\bD_{+}^{\op}, \Sets]\).
\end{definition}

The inclusion \(\iota\colon \bD\hookrightarrow \bD_{+}\) induces a forgetful functor
\[
\iota^*\colon \SSet_+\to \SSet
\]
sending an augmented simplicial set \(X_{i\ge -1}\) to the simplicial set \(X_{i\ge 0}\), called the \emph{reduced simplicial set} of~\(X\). Moreover, an augmented simplicial set \(X_{i\ge -1}\) is the same as a simplicial set \(X_{i\ge 0}\) with an \emph{augmentation} map
\[
X_{i\ge 0}\xrightarrow{d^{-1}} \sk_0 X_{-1},
\]
where \(\sk_0 X_{-1}\) is the constant simplicial set with value \(X_{-1}\).

The functor \(\iota^*\) admits a right adjoint,\footnote{In general, such adjoint functors can be computed by Kan extensions (equivalently, ends and coends), see~\cite[Chapter X]{MacLane}.} called the \emph{trivial augmentation} functor
\[
\iota_*\colon \SSet\to \SSet_+,\quad X\mapsto X_*,
\]
which sends a simplicial set \(X\) to the \emph{trivially augmented simplicial set} \(X_*=\iota_*(X)\) with \(\iota_*(X)_i=X_i\) for \(i\ge 0\) and \(X_{-1}=*\). The functor \(\iota_*\) embeds \(\SSet\) as a full subcategory into \(\SSet_+\). Denote by \(\Simp*[+]{n}\) the augmented simplicial set represented by \([n]\), then \(\Simp*[+]{n}=\Simp*{n}\). We will often identify a simplicial set with the corresponding trivially augmented simplicial set.

The functor \(\iota^*\) admits a left adjoint \(\iota_!\colon \SSet\to \SSet_+\), which sends a simplicial set \(X\) to \(\iota_!(X)\) with \(\iota_!(X)_i=X_i\) for \(i\ge 0\) and \(X_{-1}\) being the coequaliser of \(\face_1,\face_0\colon X_1\to X_0\).

\begin{definition}
A simplicial set \(X\) is \emph{connected} if \(\iota_!(X)=\iota_*(X)\).
\end{definition}

For a connected simplicial set \(X\) and an augmented simplicial set \(Y\), we have
\[
\hom(\iota^* Y, X)=\hom(Y, X_*), \quad \hom(X, \iota^* Y)=\hom(X_*, Y).
\]

It is easy to see that \(\Simp{n}\) for \(n\ge 0\), \(\Horn{n}{k}\) for \(n\ge 1\) and \(0\le k\le n\), and \(\partial\Simp{n}\) for \(n\ge 2\) are connected.

\begin{lemma}\label{chap4:lem:aug-acyclic}
  Let \(X\) be an augmented simplicial set. There is a commutative diagram
  \[
  \xymatrix{
  \hom(\Simp{n}, \iota^* X)\ar[d]\ar[r]^{=} & \hom(\iota_*\Simp{n}, X)\ar[d]\\
  \hom(\partial\Simp{n}\to\Simp{n},\iota^*X\to \sk_0 X_{-1})\ar[r]^-{=}&\hom(\iota_*\partial\Simp{n}, X)\rlap{\ .}
  }
  \]
  Thus we call \(X\) acyclic if \(\iota^*X\to \sk_0 X_{-1}\) is acyclic.
\end{lemma}
\begin{proof}
  Since \(\Simp{n}\) is connected, \(\hom(\Simp{n}, \iota^* X)=\hom(\iota_*\Simp{n}, X)\). We observe also that
  \[
  \hom(\iota_*\partial\Simp{n}, X)=\hom(\partial\Simp{n}\to \Simp{0},\iota^* X\to \sk_0 X_{-1})
                           =\hom(\partial\Simp{n}\to \Simp{n},\iota^* X\to \sk_0 X_{-1}),
  \]
  and this proves the statement.
\end{proof}

\subsection{Augmented bisimplicial sets}
\begin{definition}\label{chap4:def:bisimplicial}
  The category of \emph{augmented bisimplicial sets} \(\biSSet_{+}\) is the functor category \([\bD^\op_+\times \bD^\op_+, \Sets]\).
\end{definition}

\begin{definition}\label{chap4:def:bisimplicial-row-n}
  The functor \(r_n\colon\bD_+\to \bD_+\times \bD_+\) given by \([i]\mapsto ([n],[i])\) induces the \emph{row} functor
  \[
  R_n\coloneqq r_n^*\colon \biSSet_+ \to \SSet_+,\quad
  X\mapsto \big\{[i]\mapsto X_{n,i}\big\},
  \]
  We call \(R_n'=\iota^*\circ R_n\) the \emph{reduced row} functor. We also have \emph{column} functors \(C_n\) and \emph{reduced column} functors \(C_n'\).
\end{definition}

\begin{definition}\label{chap4:def:box-product}
  The product bifunctor \(\Sets\times \Sets\to \Sets\) induces a functor
  \begin{gather*}
  \boxtimes\colon \SSet_+\times\SSet_+\to \biSSet_+,\\
  (X, Y)\mapsto \big\{([i],[j]) \mapsto A_i\times B_j\big\},
  \end{gather*}
  called the \emph{box product}.
\end{definition}

\begin{example}
  Denote by \(\Simp*[+]{i,j}\) the augmented bisimplicial set represented by \(([i], [j])\). Then \(\Simp*[+]{i,j}=\Simp*{i}\boxtimes \Simp*{j}\).
\end{example}

The following result is an easy coend calculation.
\begin{proposition}\label{chap4:prop:box-row-adjoint}
The box product with \(\Simp*{n}\) on the left \textup(right\textup)
\[
\Simp*{n}\boxtimes \blank\colon \SSet_+\to \biSSet_+
\]
is left adjoint to the row functor \(R_n\) \textup(column functor \(C_n\)\textup).
\end{proposition}

The projection to the first factor \(\pr_1\colon \bD_+\times \bD_+\to \bD_+\) induces a functor
\begin{gather*}
\pr_1^*=\blank\boxtimes \Simp*{0}\colon\SSet_+\to \biSSet_+,\\
 X\mapsto \big \{([i,j])\mapsto X_i \big\},
\end{gather*}
which sends an augmented simplicial set \(X\) to an augmented bisimplicial set with every column given by \(X\). By Proposition~\ref{chap4:prop:box-row-adjoint}, \(\pr_1^*\) is left adjoint to the column functor \(C_0\). It is not hard to see that \(\pr_1^*\) is right adjoint to the column functor \(C_{-1}\).

\begin{definition}\label{chap4:def:diagonal-bisimp-simp}
  The diagonal inclusion \(\delta\colon \bD_+\to \bD_+\times \bD_+\) induces the \emph{diagonal} functor
  \[
  \delta^*\colon \biSSet_+\to \SSet_+,\quad X\mapsto \big\{ [i]\mapsto X_{i,i}\big\},
  \]
  sending an augmented bisimplicial to its diagonal.
\end{definition}

\begin{definition}
The ordinal sum,
\[
\sigma\colon \bD_{+} \times \bD_{+}\to \bD_{+}, \quad ([i], [j])\mapsto [i+j+1],
\]
extends cocontinuously to the \emph{join of augmented simplicial sets}
\begin{gather*}
\star \colon \SSet_+\times \SSet_+\to \SSet_+,\\
(X, Y)\mapsto \Big\{[n]\mapsto \bigsqcup_{i+j+1=n} X_i\times Y_j\Big\}.
\end{gather*}
\end{definition}

\begin{remark}
  The unit of the join operation is \(\Simp*[+]{-1}\), the augmented simplicial set represented by \([-1]=\emptyset\). The join of augmented simplicial sets restricts to the join of simplicial sets (Definition~\ref{chap2:def:join-simplicial-sets}); that is, for simplicial sets \(X\) and \(Y\), we have \(\iota_* X\star\iota_* Y=\iota_*(X\star Y)\).
\end{remark}

\begin{remark}\label{chap4:rem:horn-0-0}
  Recall that \(\Horn{0}{0}\) is left undefined. It is meaningful to set \(\Horn{0}{0}=\emptyset\), the constant presheaf on \(\bD_+\) with value \(\emptyset\), as an augmented simplicial set. Then~\eqref{chap2:eq:join} is also valid for~\(\Horn{0}{0}\). Notice that \(\Horn{0}{0}\subsetneq \iota_*\partial\Simp{0}=\Simp*[+]{-1}\).
\end{remark}

\begin{definition}
The ordinal sum on \(\bD_+\) induces the \emph{total décalage} of Illusie~\cite{Illusie}
\begin{gather*}
\Dec\coloneqq \sigma^*\colon \SSet_+\to \biSSet_+,\\
X\mapsto \big\{([i],[j])\mapsto X_{i+j+1}\big\}.
\end{gather*}
The functor \(\Dec\) admits a left adjoint,
\begin{gather*}
T\coloneqq\sigma_!\colon \biSSet_+\to \SSet_+,\\
 X\mapsto \Big\{[n]\mapsto \bigsqcup_{i+j+1=n} X_{i,j}\Big\},
\end{gather*}
called the \emph{total augmented simplicial set} functor.\footnote{The term ``total augmented simplicial set'' sometimes refers to the right adjoint of \(\Dec\), which we will not use in this thesis.}
\end{definition}

It is not hard to prove
\begin{equation}\label{chap4:eq:Tboxtims=star}
T\circ \boxtimes=\star.
\end{equation}
Therefore, the reduced column \(0\) and reduced row \(0\) of the total décalage give the décalages defined in Example~\ref{chap3:exa:decalage}; in formulas,
\[
C'_0 \circ \Dec =\dec,\quad R'_0 \circ \Dec=\dec'.
\]

The total augmented simplicial set functor \(T\) sends the terminal augmented bisimplicial set \(\Simp*[+]{0,0}\) to \(\Simp*{1}\), the \emph{augmented simplicial interval}. Thus \(T\) descends to a functor
\[
T_/\colon \biSSet_+ \to \SSet_+/\Simp*{1},
\]
where \(\SSet_+/\Simp*{1}\) denotes the over category.

\begin{proposition}\label{chap4:prop:augbisimp=aug-simp-over-I}
The functor \( T_/\colon \biSSet_+ \to \SSet_+/\Simp*{1}\) is an isomorphism of categories.
\end{proposition}
\begin{proof}
  We observe that \(\bD_+\times\bD_+\) is isomorphic to the category of elements of \(\Simp*{1}\). Then the claim follows from~\cite[Exercise III.8(a)]{MacLane-Moerdijk}.
\end{proof}

\begin{example}\label{chap4:ex:Lie-groupoid-bibundle-aug-bisimp}
  Let \(\Gamma\) be the cograph of a bibundle between Lie groupoids \(G\) and \(H\). Since there is a natural functor \(\Gamma\to I\), the nerve \(N\Gamma\) gives a natural augmented bisimplicial manifolds \(N_{i,j}\Gamma\) by setting \(N_{-1,-1}\Gamma=\pt\). Its reduced row -1 is the nerve of~\(H\), and its reduced row 0 is the nerve of the \(H\)-action groupoid. The nerve of the biaction groupoid in Proposition~\ref{chap4:prop:biaction-nerve=diagonal} is the reduced diagonal.
\end{example}

\section{Augmented bisimplicial objects}
Replace \(\Sets\) in Section~\ref{chap4:sec:aug-simplicial-bisimplicial} by a category \(\Cat\). Almost all results carry over  easily to this context, except for Proposition~\ref{chap4:prop:augbisimp=aug-simp-over-I}. The aim of this section is to establish a similar result when \(\Cat\) is an extensive category. We then define colored Kan conditions.

\subsection{Extensive categories}
We collect some basic facts about extensive categories. All the material is taken from \cite{Adameka-Milius-Velebil,Carboni-Lack-Walters}. Roughly, an extensive category is a category with finite coproducts that interact well with pullbacks. Denote the initial object in a category by \(\initial\).

\begin{definition}
An \emph{extensive category} is a category \(\Cat\) with finite coproducts such that for each pair \(A, B\in \Cat\) the coproduct induces an isomorphism of categories
\[
\Cat/A \times \Cat/B\to \Cat/(A\sqcup B).
\]
\end{definition}

\begin{proposition}
  A category with finite coproducts is extensive if and only if it has pullbacks along injections of coproducts and every commutative diagram
\[
\xymatrix{
A_1\ar[d]_{f_1}\ar[r] & A\ar[d]^{f} & A_2\ar[d]^{f_2}\ar[l]\\
B_1\ar[r]             & A_2\sqcup B_2  &  B_2\ar[l]
}
\]
consists of a pair of pullback squares if and only if the top row is a coproduct diagram with \(f=f_2\sqcup f_2\).
\end{proposition}

\begin{definition}
A coproduct \(A\sqcup B\) in a category \(\Cat\) is \emph{disjoint} if coproduct injections are monomorphisms and their pullback is an initial object. Equivalently, this means that the following three diagrams are pullback squares:
\[
\xymatrix{
A \ar[r]\ar[d]& A\ar[d]\\
A\ar[r] & A\sqcup B
}\quad
\xymatrix{
B \ar[r]\ar[d]& B\ar[d]\\
B\ar[r] & A\sqcup B
}\quad
\xymatrix{
\emptyset \ar[r]\ar[d]& B\ar[d]\\
A\ar[r] & A\sqcup B\rlap{\ .}
}
\]
\end{definition}

\begin{proposition}
  A category with finite coproducts is extensive if and only if coproducts are disjoint, pullbacks of finite coproduct injections along arbitrary morphisms are representable, and the pullbacks of coproduct diagrams are again coproduct diagrams.
\end{proposition}

We will use the following facts about extensive categories.

\begin{proposition}
In an extensive category \(\Cat\) the following statements hold:
\begin{enumerate}
  \item Coproducts of pullbacks squares are pullback squares;
  \item Initial objects are \emph{strict}; that is, every morphism \(f\colon A\to \initial\) is an isomorphism;
  \item If the first two diagrams below are pullback squares, then so is the third:
\[
  \xymatrix{
  A_1 \ar[r]\ar[d]& B_1\ar[d]\\
  C\ar[r] & D
  }\qquad
  \xymatrix{
  A_2 \ar[r]\ar[d]& B_2\ar[d]\\
  C\ar[r] & D
  }\qquad
  \xymatrix{
  A_1\sqcup A_2 \ar[r]\ar[d]& B_1\sqcup B_2\ar[d]\\
  C\ar[r] & D\rlap{\ .}
  }
\]
\end{enumerate}
\end{proposition}

\subsection{Augmented bisimplicial objects}

Let \(\Cat\) be an extensive category with a terminal object \(\terminal\).

Let \(S\) be a finite set and \(C\in \Cat\). Denote the coproduct of \(S\) copies of \(C\) by \(S\otimes C\) (copower). Let \(j\colon \FinSet \to \Cat\) be the functor induced by \(S\mapsto S\otimes\terminal\). There is an isomorphism of categories \(\Cat^S\cong \Cat/j(S)\).

Let \(f\colon R\to S\) be a map of finite sets. Denote by \(F\) the associated cograph, which is a category whose object set is \(R \sqcup S\) and whose non-trivial arrows are given by \(r\to s\) if \(f(r)=s\) for \(r\in R, s\in S\). The projection \(\varphi\colon F\to I\) induces a functor \(\varphi^*\colon [I, \Cat]\to [F, \Cat]\). This functor admits a left adjoint \(\varphi_!\colon [F,\Cat]\to [I, \Cat]\), which sends the terminal object in \([F,\Cat]\) to \(j(f)\).

\begin{lemma}\label{chap4:lem:iso:I-cat-j(f)}
The functor \(\varphi_!\) induces an isomorphism of categories
 \[
 [F, \Cat]\cong [I, \Cat]/j(f).
 \]
\end{lemma}
\begin{proof}
  An object in \([I, \Cat]/j(F)\) is a commutative square in \(\Cat\)
  \[
  \xymatrix{
  C_R\ar[r]^{\psi}\ar[d] & C_S\ar[d]\\
  j(R)\ar[r]^{j(f)} & j(S) \rlap{\ ,}
  }
  \]
  where \(C_R=\bigsqcup_{r\in R}C_r\) and \(C_S=\bigsqcup_{s\in S} C_s\). Since \(\Cat\) is extensive, the map \(\psi\colon C_R\to C_S\) decomposes as
  \[
  \bigsqcup_{s\in S} \Big(\bigsqcup_{f(r)=s} C_r\xrightarrow{\psi_{r,s}} C_s\Big),
  \]
  which is the image of an object in \([F, \Cat]\) under the functor \(\varphi_!\). Following these lines, we can verify that the two categories have the same arrows.
\end{proof}

Let us denote by \(\SCat, \SCat_+\) and \(\biSCat_+\) the categories of simplicial objects, augmented simplicial objects, and augmented bisimplicial objects in \(\Cat\), respectively.

\begin{proposition}\label{chap4:prop:abisimpC=asimpC-I}
  Let \(T\) be the functor defined by
  \[
  \biSCat_+\to \SCat_+,\quad X\mapsto \Big\{[n]\mapsto \bigsqcup_{i+j+1=n} X_{i,j}\Big\}.
  \]
 It induces an isomorphism of categories
  \[
  \biSCat_+=\SCat_+/j(\Simp*{1}).
  \]
  In particular, \(\SCat/j(\Simp{1})\) corresponds to the category of augmented bisimplicial objects~\(Y\) with \(Y_{-1,-1}=\terminal\).
\end{proposition}
\begin{proof}
Let \(C\in \SCat_+\) with a morphism \(g\colon C\to j(\Simp{1})\). Then every object \(C_n\) can be decomposed as \(\bigsqcup_{i+j+1=n} C_{i,j}\), where \(n, i, j\ge -1\). Applying Lemma~\ref{chap4:lem:iso:I-cat-j(f)}, we can decompose face and degeneracy maps. It is routine to verify that the decomposed objects with the decomposed face and degeneracy maps form an augmented bisimplicial object in \(\Cat\).
\end{proof}

Similar to Section~\ref{chap4:sec:aug-simplicial-bisimplicial}, we can define row functors and column functors \(\biSCat_+\to \SCat_+\), and a box product \(\boxtimes\colon \SSet_+\times \SCat_+\to \biSCat_+\). The box product, row functors, and column functors satisfy adjointness relations similar to those in Proposition~\ref{chap4:prop:box-row-adjoint}.

\subsection{Colored Kan conditions}

Let \(\pi\colon X\to j(\Simp{1})\) be an simplicial object in \(\Cat\) over \(j(\Simp{1})\). It may be regarded as a \emph{colored simplicial object}: vertices of \(X\) are colored by two colors \(0\) or \(1\), and vertices of an \(n\)-simplex for \(n\ge 0\) in \(X\) are of the same color or colored by \((0,\dots,0,1,\dots,1)\).

\begin{definition}
  Let \(R, S \in \SSet_+\) and \(Y\in \biSCat_+\). Let \(\Hom_{\Simp*{1}}(R\star S, TY)\) be the presheaf on \(\Cat\) given by \(\Hom(R\boxtimes S, Y)\); By Proposition~\ref{chap4:prop:abisimpC=asimpC-I}, its value at \(C\in \Cat\) is the set of commutative triangles
  \[
  \xymatrix{
  (R\star S) \otimes C\ar@{..>}[r]\ar[d] & TY\ar[ld]\\
  j(\Simp*{1}) \rlap{\ ,}
  }
  \]
  where the left vertical arrow is induced by \(R\star S\to \Simp*{1}\) and \(C\to\terminal\).
\end{definition}

Fix \(i,j\ge 0\) and let \(m= i+j-1\). Let \(\Simp{m}\xrightarrow{\chi_{i,j}} \Simp{1}\) be the canonical map \(\Simp{i-1}\star\Simp{j-1}\to \Simp{1}\).

\begin{definition}
Let \(\Hom(\Horn{m}{k}[i,j], X)\coloneqq \Hom_{/\Simp{1}}(\Horn{m}{k}, X)\) be the presheaf on \(\Cat\) whose value at \(C\in \Cat\) is the set of commutative squares of the following type
\[
  \xymatrix{
  \Horn{m}{k}\otimes C\ar[d]\ar@{..>}[r] & X\ar[d]\\
  \Simp{m}\otimes C\ar[r] & j(\Simp{1})\rlap{\ ,}
}
\]
where the bottom arrow is induced by \(\Simp{m}\xrightarrow{\chi_{i,j}}\Simp{1}\) and \(C\to \terminal\). Similarly, we define \(\Hom(\Simp{m}[i,j], X)\coloneqq \Hom_{j(\Simp{1})}(\Simp{m}, X)\) to be the presheaf given by sending \(C\) to the set of arrows in \(\hom(\Simp{m}\otimes C, X)\) such that the lower right triangle in the above square commutes.
\end{definition}

\begin{definition}\label{chap4:def:colored-Kan-condition}
  A colored simplicial object \(\pi\colon X\to j(\Simp{1})\) satisfies the \emph{colored Kan condition} \(\Kan(m, k)[i, j]\) if \(\Hom(\Simp{m}[i,j], X)\) and \(\Hom(\Horn{m}{k}[i,j], X)\) are presentable and the map
  \[
  \Hom(\Simp{m}[i,j], X)\to \Hom(\Horn{m}{k}[i,j], X)
  \]
  is a cover. If this map is an isomorphism, then we say that \(X\to j(\Simp{1})\) satisfies the \emph{colored unique Kan condition}  \(\Kan!(m, k)[i, j]\).
\end{definition}

Let \((\Cat, \covers)\) be an extensive category with a singleton Grothendieck pretopology. We assume from now on that the following additional assumption holds. This is fulfilled by each instance in Example~\ref{chap3:exa:singleton-pretopoloty}.

\begin{assumption}\label{chap4:asmp:extensive-pretopology}
Let \(f_i\colon Y_i\to X_i\), \(i=1, 2\) be two morphisms in \(\Cat\). The coproduct morphism  \(f_1\sqcup f_2\) is a cover in \(\covers\) if and only if \(f_1\) and \(f_2\) are so.
\end{assumption}

\begin{proposition}
  Let \(X\to j(\Simp{1})\) be as above. Then it satisfies \(\Kan(m, k)\) for \(m\ge 1\) and \(0\le k\le m\) if and only if it satisfies \(\Kan(m, k)[i, j]\) for all \(i, j\ge 0\) with \(i+j=m+1\).
\end{proposition}
\begin{proof}
Suppose that
\[
\Hom(\Horn{m}{k}\to \Simp{m}, X\to j(\Simp{1}))=\Hom(\Horn{m}{k}, X)\times_{\Hom(\Horn{m}{j}, j(\Simp{1}))} \Hom(\Simp{m}, j(\Simp{1}))
\]
is representable. Observe that
\[
\Hom(\Simp{m}, j(\Simp{1}))=j(\Simp{1}_m)=\bigsqcup_{\chi_{i,j}\in \Simp{1}_m} j(\chi_{i,j}).
\]
Since \(\Cat\) is extensive, we get isomorphisms in \(\Cat\)~\footnote{Warning: these isomorphisms are not as presheaves, since coproducts are not compatible with the Yoneda functor.}
\begin{align*}
&\Hom(\Horn{m}{k}, X)\times_{\Hom(\Horn{m}{j}, j(\Simp{1}))} \Hom(\Simp{m},j(\Simp{1}))\\
\cong&\bigsqcup_{\chi_{i,j}\in \Simp{1}_m}  \Hom(\Horn{m}{k}, X)\times_{\Hom(\Horn{m}{k}, j(\Simp{1}))} j(\chi_{i,j})\\
\cong &\bigsqcup_{i+j=m+1} \Hom(\Horn{m}{k}[i,j], X).
\end{align*}
Similarly, we have
\[
\Hom(\Simp{m}, X)\cong \bigsqcup_{i+j=m+1}\Hom(\Simp{m}[i,j], X).
\]
Then the claim follows from Assumption~\ref{chap4:asmp:extensive-pretopology}.
\end{proof}

\section{Cographs of higher groupoid morphisms}\label{chap4:sec:cograph-morphism-higher-gpd}

In this section, we define a cograph of a higher groupoid morphism. We establish a property that models bibundles between higher groupoids.

There is a natural transformation
\[
\pr_1\Rightarrow\sigma \colon \bD_+\times\bD_+\to \bD_+,\quad ([i], [j])\mapsto \{[i]\to [i+j+1]\},
\]
where \([i]\to [i+j+1]\) is given by the inclusion \([i]\to [i]\star[j]=[i+j+1]\).

\begin{definition}
Let \(f\colon X\to Y\) be a morphism of trivially augmented simplicial objects in \(\Cat\). Its \emph{cograph} \(\Gamma\) is the augmented bisimplicial object in \(\Psh(\Cat)\) given by the pullback square
\[
  \xymatrix{
  \Gamma\ar[r]\ar[d] & \Dec Y\ar[d]\\
  \pr_1^* X \ar[r]      &\pr_1^* Y \rlap{\ ,}
  }
\]
where the vertical right arrow is induced by the natural transformation \(\pr_1\Rightarrow\sigma\). Explicitly, \(\Gamma_{i,j}=X_i\times_{Y_i} Y_{i+j+1}\) for \(i, j\ge -1\); in particular, \(R_{-1}\Gamma=Y\) and \(C_{-1}\Gamma=X\).
\end{definition}

If \(X\) and \(Y\) are higher groupoids in \((\Cat, \covers)\), then \(\Gamma\) is an augmented bisimplicial object in \(\Cat\). We are interested in Kan conditions for the induced morphism \(T\Gamma\to j(\Simp{1})\).

\begin{proposition}\label{chap4:prop:hihger-groupoid-cogragh-Kan-conditions}
  Let \(\Gamma\) be the cograph of a higher groupoid morphism \(X\to Y\) in \((\Cat, \covers)\). The morphism \(T\Gamma\to j(\Simp{1})\) satisfies all colored Kan conditions \(\Kan(m, k)[i,j]\) except \(\Kan(m, m)[m,1]\) for \(m\ge 1\).
\end{proposition}

We first establish a useful lemma.

\begin{lemma}
Let \(A\to B\) and \(C\to D\) be inclusions of augmented simplicial sets. Suppose that \(C\) and \(D\) are trivially augmented. Then the diagram below is a pullback square:
\[
\xymatrix{
\Hom_{/\Simp{1}}(A\star D\cup_{A\star C} B\star C, T\Gamma)\ar[r]\ar[d]&\Hom(A\star D\cup_{A\star C} B\star C, Y)\ar[d] \\
\Hom(B, X)\ar[r] &\Hom(B, Y)\rlap{\ .}
}
\]
\end{lemma}
\begin{proof}
First, since \(C_{-1}\) is the left adjoint of \(\pr_1^*\), we have
  \[
  \Hom(B\boxtimes D, \pr_1^* X)=\Hom(C_{-1}(B\boxtimes D), X)=\Hom(B, X),
  \]
where \(C_{-1}(B\boxtimes D)=B\) because \(D\) is trivially augmented. We deduce that
\[
\Hom(A\boxtimes D\cup_{A\boxtimes C} B\boxtimes C, \pr_1^* X)=\Hom(B, X).
\]
Similarly, \(\Hom(A\boxtimes D\cup_{A\boxtimes C} B\boxtimes C, \pr_1^* Y)=\Hom(B, Y)\).

Second, since \(\Hom(B\boxtimes D, \Dec Y)=\Hom(B\star D, Y)\), we have
\[
  \Hom(A\boxtimes D\cup_{A\boxtimes C} B\boxtimes C, \Dec Y)=\Hom(A\star D\cup_{A\star C} B\star C, Y).
\]

Third, since \(\Hom(B\boxtimes D, \Gamma)=\Hom_{/\Simp{1}}(B\star D, T\Gamma)\), we have
\[
  \Hom(A\boxtimes D\cup_{A\boxtimes C} B\boxtimes C, \Gamma)=\Hom_{/\Simp{1}}(A\star D\cup_{A\star C} B\star C, T\Gamma).
\]

All results above combined prove the claim.
\end{proof}

\begin{proof}[Proof of Proposition~\ref{chap4:prop:hihger-groupoid-cogragh-Kan-conditions}]
First, let \(A\) be a trivially augmented simplicial set. Propositions~\ref{chap4:prop:augbisimp=aug-simp-over-I} and~\ref{chap4:prop:box-row-adjoint} imply that
\[
\Hom_{/\Simp{1}}(A\star\Simp*[+]{-1}, T\Gamma)=\Hom(A\boxtimes \Simp*[+]{-1}, \Gamma)=\Hom(A, C_{-1}\Gamma).
\]
Since $T\Gamma$ and $C_{-1}\Gamma$ are both trivially augmented, condition \(\Kan(l,k)[l+1,0]\) for \(T\Gamma\to \Simp{1}\) is  implied by \(\Kan(l, k)\) for \(X\). Similarly, \(\Kan(l, k)[0, l+1]\) for \(T\Gamma\to \Simp{1}\) follows from \(\Kan(l, k)\) for~\(Y\).

Replacing \(A\to B\) in the previous lemma by \(\Horn{l}{k}\to \Simp{l}\) for \(l\ge 0\), and \(C\to D\) by \(\partial \Simp{m}\to \Simp{m}\) for \(m\ge 0\), we get a commutative diagram (recall Remark~\ref{chap4:rem:horn-0-0} and~\eqref{chap2:eq:join})
\[
\xymatrix{
\Hom_{/\Simp{1}}(\Simp{l+m+1}, T\Gamma)\ar[r]\ar[d] & \Hom(\Simp{l+m+1}, Y)\ar[d]\\
\Hom_{/\Simp{1}}(\Horn{l+m+1}{k}, T\Gamma)\ar[r]\ar[d] & \Hom(\Horn{l+m+1}{k}, Y)\ar[d]\\
\Hom(\Simp{l}, X)\ar[r] & \Hom(\Simp{l}, Y)\rlap{\ ,}
}
\]
where each square is a pullback diagram. Since \(\Simp{l}\to \Horn{l+m+1}{k}\to \Simp{l+m+1}\) are collapsible extensions, the right vertical arrows are covers. Thus
\[
\Hom_{/\Simp{1}}(\Simp{l+m+1}, T\Gamma)\to \Hom_{/\Simp{1}}(\Horn{l+m+1}{k}, T\Gamma)
\]
is a cover between two representable objects. This proves \(\Kan(l+m+1,k)[l+1, m+1]\).

Similarly, let \(A\to B\) in the previous lemma be \(\partial\Simp{l}\to \Simp{l}\) for \(l\ge 0\) and \(C\to D\) be \(\Horn{m}{k}\to\Simp{m}\) for \(m\ge 1\). We deduce that
\[
\Hom_{/\Simp{1}}(\Simp{l+m+1}, T\Gamma)\to \Hom_{/\Simp{1}}(\Horn{l+m+1}{l+1+k}, T\Gamma)
\]
is a cover. This proves \(\Kan(l+m+1, l+1+k)[l+1,m+1]\).

The only missing case is \(m=k=0\), that is, \(\Kan(l+1, l+1)[l+1, 1]\) for \(l\ge 0\).
\end{proof}

\begin{proposition}\label{chap4:prop:cograph-acyclic-is-kan}
  Let \(\Gamma\) be the cograph of an acyclic fibration of higher groupoids \(X\to Y\) in \((\Cat, \covers)\). Then \(T\Gamma\to j(\Simp{1})\) is a Kan fibration.
\end{proposition}
\begin{proof}
It remains to consider \(\Kan(l+1, l+1)[l+1, 1]\) for \(l\ge 0\). By assumption, the morphism
\[
\Hom(\Simp{l}, X)\to  \Hom(\partial\Simp{l}\to \Simp{l}, X\to Y)
\]
is a cover. Consider the commutative diagram
\[
\xymatrix{
\Hom_{/\Simp{1}}(\Simp{l}\star\Simp{0}, T\Gamma)\ar[r]\ar[d]  & P\ar[r]\ar[d] & \Hom(\Simp{l+1}, Y)\ar[d]\\
\Hom(\Simp{l}, X)\ar[r]& \Hom(\partial\Simp{l}\to \Simp{l}, X\to Y)\ar[r]\ar[d] & \Hom(\Simp{l}, Y)\ar[d]\\
& \Hom(\partial\Simp{l}, X)\ar[r]& \Hom(\partial\Simp{l}, Y) \rlap{\ ,}
}
\]
in which every square is a pullback diagram, where \(P\) is the pullback of the right whole square. We deduce that
\(\Hom_{/\Simp{1}}(\Simp{l}\star\Simp{0}, T\Gamma)\to P\) is a cover. Consider the commutative diagram
\[
\xymatrix{
P\ar[r]\ar[d] & \Hom(\Simp{l+1}, Y)\ar[d]\\
\Hom_{/\Simp{1}}(\partial\Simp{l}\star\Simp{0}, T\Gamma)\ar[r]\ar[d] & \Hom(\Horn{l+1}{l+1}, Y)\ar[d]\\
\Hom(\partial\Simp{l}, X)\ar[r]& \Hom(\partial\Simp{l}, Y) \rlap{\ ,}
}
\]
in which each square is a pullback diagram. It follows that \(P\to \Hom_{/\Simp{1}}(\partial\Simp{l}\star\Simp{0}, T\Gamma)\) is a cover. The claim then follows because
\[
\Hom_{/\Simp{1}}(\Simp{l}\star\Simp{0}, T\Gamma)\to \Hom_{/\Simp{1}}(\partial\Simp{l}\star\Simp{0}, T\Gamma)
\]
is the composite of two covers.
\end{proof}

\section{Bibundles of higher groupoids}

Inspired by the previous sections, we propose the following definition:

\begin{definition}\label{chap4:def:bibundles=inner-over-I}
  Let \(X\) and \(Y\) be \(n\)-groupoids in \((\Cat, \covers)\). A \emph{bibundle} between \(X\) and \(Y\) is an augmented bisimplicial object \(\Gamma\) with \(C_{-1}\Gamma=X_*\) and \(R_{-1}\Gamma=Y_*\), such that \(T\Gamma\to j(\Simp{1})\) satisfies the inner Kan conditions \(\Kan(m, k)\) for \(m\ge 2\) and \(0<k<m\), and the outer Kan conditions \(\Kan(m, 0)[i, j]\) for \(i\ge 2\), and \(\Kan(m, m)[i, j]\) for \(j\ge 2\); moreover, when \(m> n\) the unique version of all prescribed Kan conditions hold.

  If, in addition, \(\Kan(m, 0)\) are satisfied for \(m\ge 1\), and when \(m> n\) the unique version of all desired Kan conditions hold, then we call such a bibundle \emph{right principal}. Two-sided principal bibundles are also called Morita bibundles.
\end{definition}

\begin{remark}\label{chap4:rem:low-inner-kan-imply-repr}
  If the lower inner Kan conditions \(\Kan(l, k)\) hold for \(0<k<l\) and \(l=2,\dots, m-1\), then \(\Hom(\Horn{m}{k}\to \Simp{m}, T\Gamma\to j(\Simp{1}))\) for \(0<k<m\) are representable. This follows because \(\Sp(m)\to \Horn{m}{k}\) is an inner collapsible extension by Proposition~\ref{chap2:lem:spine-horn-inner} and \(\Hom(\Sp(m)\to \Simp{m}, T\Gamma\to j(\Simp{1}))\) is representable.
\end{remark}

\begin{remark}\label{chap4:rem:low-outer-kan-imply-repr}
If the lower colored outer Kan conditions \(\Kan(l,0)[s,t]\) hold for \(s \ge 2\) and \(l<m\), then \(\Hom(\Horn{m}{0}[i,j],T\Gamma)\) is representable for \(i\ge 2\). This follows because \(\Horn{m}{0}\) can be built from \(\Simp{1}\{0,1\}\cup \Simp{m-1}\{0,2,\cdots,m\}\) by filling 0-th horns containing \(\{0, 1\}\), and \(\Hom(\Simp{1}\{0,1\}\cup \Simp{m-1}\{0,2,\cdots,m\} [i,j], T\Gamma)\) is representable by assumption.
\end{remark}

\begin{remark}
  Joyal~\cite{Joyal2002} studied first inner Kan fibrations under the name of mid fibrations. Lurie~\cite{Lurie} used inner Kan fibrations over \(\Simp{1}\) to model correspondences between \(\infty\)-categories. Our definition of bibundle is slightly modified by adding some special outer Kan conditions since we deal with higher groupoids. Following \cite{Joyal2002}, we show in Appendix~\ref{app:sec:colored-out-kan} that, under some additional assumption on the category and pretopology, the inner Kan conditions imply the outer Kan conditions \(\Kan(m, 0)[i, j]\) for \(i\ge 2\), and \(\Kan(m, m)[i, j]\) for \(j\ge 2\). For instance, this is true for \((\Sets,\covers_\surj)\) and \((\Mfd,\covers_\subm)\).
\end{remark}

\begin{remark}\label{chap4:rem:opposite-bibundle}
  Let \(\Gamma\to j(\Simp{1})\) be a bibundle between \(n\)\nbdash{}groupoids \(X\) and \(Y\). Regarding \(\Gamma\) as a colored simplicial object and formally inverting all arrows from \(1\) to \(0\), we obtain a bibundle between \(Y\) and \(X\). We denoted this bibundle by \(\overline{\Gamma}\) and call it the \emph{opposite bibundle} associated with \(\Gamma\). Notice that \(\overline{\Gamma}\) is not the opposite simplicial object, since we do not invert arrows with target and source of the same color.
\end{remark}

\begin{corollary}
  Let \(f\colon X\to Y\) be a morphism of \(n\)-groupoid objects in \((\Cat, \covers)\). Then its cograph is a bibundle between \(X\) and \(Y\). Moveover, if \(f\) is an acyclic fibration of \(n\)-groupoids, then the cograph is right principal.
\end{corollary}
\begin{proof}
  This follows from Propositions~\ref{chap4:prop:hihger-groupoid-cogragh-Kan-conditions} and~\ref{chap4:prop:cograph-acyclic-is-kan}.
  By considering the dimensions, it is routine to verify that all desired unique Kan conditions hold for \(m> n\).
\end{proof}

Given a bibundle between \(X\) and \(Y\), we should be able to extract actions of \(X\) and \(Y\), respectively. If \(\Gamma\) is right principal, then the \(Y\)-action should be principal.

\begin{proposition}\label{chap4:prop:TX-I-lifting-properties}
Let \(\Gamma\) be a bibundle between \(n\)-groupoids \(X\) and \(Y\) in \((\Cat, \covers)\). Then \(R'_l \Gamma\to R'_{-1}\Gamma=Y\) and \(C'_l \Gamma \to C'_{-1}\Gamma=X\) are \(n\)-groupoid Kan fibrations for each \(l\ge 0\).
\end{proposition}
\begin{proof}
It suffices to consider \(R'_l \Gamma\to R'_{-1}\Gamma\). Let \(A\to B\) be an inclusion of trivially augmented connected simplicial sets. Since \(B\) is connected, we have
\[
\Hom(\iota^* B, R'_l \Gamma)\cong\Hom(B, R_l \Gamma).
\]
Proposition~\ref{chap4:prop:box-row-adjoint} and~\eqref{chap4:eq:Tboxtims=star} imply that
\[
\Hom(B, R_l \Gamma)\cong\Hom(\Simp{l} \boxtimes B, \Gamma)\cong\Hom_{/\Simp{1}}(\Simp{l}\star B, T\Gamma),
\]
natural in each variable. Therefore, by naturality we get
\[
\Hom(A\to B, R_l \Gamma\to R_{-1} \Gamma)\cong \Hom_{/\Simp{1}}(\Simp{l}\star A\cup_A B, T\Gamma),
\]
and a commutative diagram
\[
\xymatrix{
\Hom(B, R_l \Gamma)\ar[r]^{\cong}\ar[d] & \Hom_{/\Simp{1}}(\Simp{l}\star B, T\Gamma)\ar[d]\\
\Hom(A\to B, R_l \Gamma\to R_{-1} \Gamma)\ar[r]^{\cong} & \Hom_{/\Simp{1}}(\Simp{l}\star A\cup_A B, T\Gamma)\rlap{\ .}
}
\]

Replace \(A\to B\) above by \(\Horn{m}{k}\to \Simp{m}\). Proposition~\ref{chap2:prop:join-pushout-collapsible} shows that
\[
\Simp{l}\star \Horn{m}{k}\cup_\Horn{m}{k}  \Simp{m}\to \Simp{l}\star\Simp{m}
\]
is an inner collapsible extension for \(0\le k<m\). Thus \(R'_l \Gamma\to R'_{-1}\Gamma=Y\) satisfies \(\Kan(m, k)\) for \(0\le k<m\).

It remains to consider \(k=m\). We know that \(\Simp{l}\star \Horn{m}{m-1}\cup_\Horn{m}{m-1}  \Simp{m}\to \Simp{l}\star\Simp{m}\) is an inner collapsible extension. Switching the role of the vertices \(m\) and \(m-1\) gives a filtration
\[
\Simp{l}\star \Horn{m}{m}\cup_\Horn{m}{m}  \Simp{m}\hookrightarrow \cdots\hookrightarrow \Simp{l}\star\Simp{m}
\]
such that each step is attaching either an inner horn or a right outer horn containing \(m-1\) and \(m\). By assumption, \(\Kan(s,s)[i,j]\) holds for \(j\ge 2\), and we deduce that \(\Kan(m, m)\) holds for \(R'_l \Gamma\to R'_{-1}\Gamma=Y\).

Finally, it is routine to verify all desired unique Kan conditions for \(m>n\).
\end{proof}

The following proposition justifies the definition of right principality.
\begin{proposition}
  Let \(\Gamma\) be a right principal bibundle between \(n\)-groupoids \(X\) and \(Y\) in \((\Cat, \covers)\). Then
  \(R'_l \Gamma\to\sk_0 \Gamma_{l, -1}\) is acyclic for \(l\ge 0\).
\end{proposition}
\begin{proof}
  By Lemma~\ref{chap4:lem:aug-acyclic}, it suffices to show that \(R_l\Gamma\) is acyclic. The proof of Proposition~\ref{chap4:prop:TX-I-lifting-properties} gives a commutative diagram
  \[
  \xymatrix{
  \Hom(\iota_*\Simp{m}, R_l\Gamma)\ar[d]\ar[r]^-{\cong}& \Hom_{/\Simp{1}}(\Simp{l}\star\Simp{m}, T\Gamma)\ar[d]\\
  \Hom(\iota_*\partial\Simp{m}, R_l\Gamma)\ar[r]^-{\cong}&\Hom_{/\Simp{1}}(\Simp{l}\star\partial\Simp{m},T\Gamma)\rlap{\ .}
  }
  \]
  Proposition~\ref{chap2:prop:join-pushout-collapsible} implies that \(\Simp{0}\{0\}\star\Simp{m}\cup \Simp{l}\star\partial\Simp{m}\to\Simp{l}\star\Simp{m}\) is a left collapsible extension. It is clear that \(\Simp{l}\star\partial\Simp{m}\to \Simp{0}\{0\}\star\Simp{m}\cup \Simp{l}\star\partial\Simp{m}\) is a left collapsible extension. Thus \(\Simp{l}\star\partial\Simp{m}\to \Simp{l}\star\Simp{m}\) is a left collapsible extension. We deduce that
  \[
  \Hom_{/\Simp{1}}(\Simp{l}\star\Simp{m}, T\Gamma)\to \Hom_{/\Simp{1}}(\Simp{l}\star\partial\Simp{m},T\Gamma)
  \]
  is a cover, and this proves the statement.
\end{proof}

\chapter{Actions of 2-Groupoids}\label{chap5}
\thispagestyle{empty}

In this chapter, we recall another approach to 2\nbdash{}groupoids, namely, in terms of categorification. We then define actions of  2\nbdash{}groupoids on groupoids and 2\nbdash{}bundles in this framework. The main part of this chapter is devoted to establishing the equivalence between the simplicial approach and the categorification approach to 2\nbdash{}groupoid actions and 2\nbdash{}bundles.

\section{2-groupoids}\label{chap5:sec:2-groupoids}

We recall 2\nbdash{}groupoids defined by categorification in this section. We then explain the relation between categorified groupoids and 2\nbdash{}groupoids given by simplicial objects.

\subsection{Categorified groupoids}
Recall from Section~\ref{chap3:subs:internal-groupoids} that for a nice enough \((\Cat,\covers)\) there is a 2\nbdash{}category \(\BUN\) that has groupoids in \((\Cat,\covers)\) as objects, HS bibundles as 1\nbdash{}morphisms, and isomorphisms of HS bibundles as 2\nbdash{}morphisms. In Chapters~\ref{chap5}, \ref{chap6}, and \ref{chap6:composition}, we let \((\Cat,\covers)\) be such a nice category with a pretopology. The readers are free to take \((\Cat, \covers)\) to be \((\Mfd,\covers_\subm)\).

\begin{definition}\label{chap5:def:cat-groupoid}
A \emph{categorified groupoid} \(G \rightrightarrows \trivial{M}\) in \((\Cat, \covers)\) is a groupoid
object in the 2\nbdash{}category \(\BUN\) such that the space of objects is an object in~\(\Cat\).

More precisely, such a 2\nbdash{}groupoid consists of the following data:
\begin{itemize}
  \item a space of objects \(M\), which is an object in \(\Cat\), and a groupoid of arrows \(G= (\bar\source,\bar\target\colon G_1\rightrightarrows G_0)\), which is a groupoid in \((\Cat, \covers)\);
  \item the source and target morphisms \(\source, \target\colon G\to \trivial{M}\), which are given by covers \(\source_0,\target_0\colon G_0\to M\) with \(\source_0\circ\bar\source= \source_0\circ\bar\target\) and \(\target_0\circ\bar\source= \target_0\circ\bar\target\);
  \item the multiplication\footnote{The multiplication and the inversion are also called horizontal multiplication and horizontal inversion, respectively. Not to be confused with the vertical multiplication and vertical inversion of the groupoid \(G\).} \(\mult\), which is given by an HS bibundle~\(E_\mult\) from \(G\times_{\source,M,\target} G\) to~\(G\).
  \item the unit morphism \(\unit \colon \trivial{M}\to G\), which arises from a groupoid functor; the inversion morphism \(\inv\colon G\to G\), which is given by an HS bibundle \(E_\inv\).
\end{itemize}
These morphisms satisfy the same target and source relations as for a groupoid, while the identities of unit, associativity, and inverse are replaced by isomorphisms of HS bibundles:
\begin{itemize}
 \item The multiplication is associative up to a 2\nbdash{}morphism, called \emph{associator}: there is a 2\nbdash{}morphism between HS morphisms \(G\times_{\source,\trivial{M},\target} G\times_{\source,\trivial{M},\target} G\to G\),
\begin{equation}\label{chap5:eq:2-groupoid-associator}
  \alpha\colon \mult\circ (\mult\times \id)\Rightarrow \mult\circ(\id\times\mult).
\end{equation}
 \item The unit identities hold up to 2\nbdash{}morphisms: there are 2\nbdash{}morphisms between HS morphisms \(G\to G\)
\begin{equation}\label{chap5:eq:unitors}
\begin{aligned}
  \rho\colon \mult\circ (\id, \unit\circ \source)&\Rightarrow  \id,\\
  \lambda\colon \mult\circ (\unit\circ \target, \id)&\Rightarrow \id,
\end{aligned}
\end{equation}
called right and left \emph{unitors}.
 \item The inverse identities hold up to 2\nbdash{}morphisms: there are 2\nbdash{}morphisms between HS morphisms \(G\to G\)
\begin{equation}\label{chap5:eq:universor}
\begin{aligned}
 \eta \colon \mult\circ(\id, \inv) &\Rightarrow \unit \circ \target, \\
 \varepsilon \colon \mult\circ (\inv, \id) &\Rightarrow \unit \circ \source,
\end{aligned}
\end{equation}
called right and left \emph{inversors}.
\end{itemize}
The associator satisfies the usual \emph{pentagon condition}, and the unitors satisfy the \emph{triangle condition}, which are both expressed by 2\nbdash{}commutative diagrams (see \cite[Definition 3.4]{Zhu:ngpd}).
\end{definition}

\begin{remark}
  We can also use other 2\nbdash{}categories of groupoids to define 2\nbdash{}groupoids. For instance, if we use the 2\nbdash{}category of differential stacks, the resulting categorified groupoids are also called stacky groupoids. Stacky groupoids were introduced by Tseng and Zhu~\cite{Tseng-Zhu} as the global objects integrating Lie algebroids; later in~\cite{Zhu:ngpd}, the definition was modified by adding some higher coherence conditions.
\end{remark}

\begin{remark}
  A categorified group is a categorified groupoid such that the object space is a single point. Categorified groups in \((\Sets,\covers_\surj)\) were studied in~\cite{Ulbrich,Laplaza,Baez-Lauda}. Stacky groups were studied by Blohmann~\cite{Blohmann}.
\end{remark}

\begin{remark}
  The groupoid \(G\times_{\source,\trivial{M},\target}G\) is the 1-categorical strong pullback of groupoids in \((\Cat, \covers)\), which is also a 2\nbdash{}pullback in the 2\nbdash{}category \(\BUN\) by Corollary~\ref{chap1:cor:pullback-bibundle-to-M=pullback}. The 2\nbdash{}categorical universal property guarantees that all morphisms and 2\nbdash{}morphisms in the formulas above exist and that the coherence conditions for 2-morphisms are well-defined. The groupoids \((G\times_{\source,\trivial{M},\target} G)\times_{\source,\trivial{M},\target} G\) and  \(G\times_{\source,\trivial{M},\target} (G\times_{\source,\trivial{M},\target} G)\) are identified by the canonical isomorphism. Another convention was made in~\cite{Zhu:ngpd}: the coherence 2\nbdash{}morphisms from the 2\nbdash{}category \(\BUN\) itself are omitted to simplify diagrams. We will follow this convention.
\end{remark}

\begin{remark}
  The unit morphism \(\unit\colon \trivial{M}\to G\) is assumed to arise from the bundlisation of a groupoid functor \(\trivial{M}\to G\). It is shown that, up to equivalence, we can always replace a general HS bibundle by one arising from some groupoid functor \(\trivial{M}\to G\). For details see~\cite[Section 3.1]{Zhu:ngpd}.
\end{remark}

\begin{remark}
  If we ignore the inversion and replace \(\BUN\) by the strict 2\nbdash{}category \(\Cats\), then we get the definition of a 2\nbdash{}category. As for 2\nbdash{}categories, the coherence conditions in Definition~\ref{chap5:def:cat-groupoid} ensure that ``all diagrams'' (without involving the inversion) commute.
\end{remark}

\begin{remark}
  The higher coherence conditions for inversors are two \emph{zig-zag conditions} similar to those for 2-groups given in~\cite{Baez-Lauda,Ulbrich, Laplaza}. It is shown there that we can always achieve them by an appropriate modification; the argument can be carried over to 2-groupoids. Moveover, the condition on the inversion can be replaced by the requirement that \(E_\mult\times_{J_r, G_0,\unit} \trivial{M}\) is a Morita equivalence from \(G\) to \(G^\op\); see~\cite[Section 3.2]{Zhu:ngpd}. This is, in turn, equivalent to the following two shear morphisms
  \begin{gather*}
  G\times_{\source,\trivial{M},\target}G\to G\times_{\source, \trivial{M},\source} G,\quad  (g_1, g_2)\mapsto (g_1\cdot g_2, g_2)\\
  G\times_{\source,\trivial{M},\target}G \to G\times_{\target, \trivial{M},\target} G,\quad (g_1, g_2)\mapsto (g_1, g_1\cdot g_2)
  \end{gather*}
  being Morita equivalences. In fact, the first morphism is \(f_1\) in~\cite[Lemma~3.12]{Zhu:ngpd}. With our assumption, the rest of the proof of Lemma~3.12 carries on, hence Lemma~3.13 in~\cite{Zhu:ngpd} also follows.
\end{remark}

\subsection{A one-to-one correspondence}

Zhu~\cite{Zhu:ngpd} proved that Definition~\ref{chap5:def:cat-groupoid} agrees with Definition~\ref{chap3:def:n-groupoid-pretopology} for~\(n=2\).

\begin{theorem}\label{chap5:thm:2-groupoid-cat=simp}
Categorified groupoids in \((\Cat, \covers)\) are in a one-to-one correspondence with \(2\)\nbdash{}groupoid objects in \((\Cat, \covers)\).
\end{theorem}

We briefly recall the proof given in~\cite{Zhu:ngpd}.

\subsubsection{From categorification to the simplicial picture}

Given a categorified groupoid \(G \rightrightarrows \trivial{M}\), we may construct a simplicial object~\(X\) with
\[
X_0=M, \quad X_1=G_0, \quad X_2=E_\mult, \quad X_3=\Hom(\Horn{3}{0}, X_{\le 2}).
\]
The degeneracy map \(X_0\to X_1\) is given by the unit, the degeneracy maps \(X_1\to X_2\) are given by the unitors. The face maps \(X_1\to X_0\) are given by~\(\source\) and~\(\target\), the face maps \(X_2\to X_1\) are given by the two moment maps of the multiplication HS bibundle. One can use the existence of an associator to show that each horn in \(X_{\le 2}\) determines the missing face uniquely, hence\(X_3=\Hom(\Horn{3}{0}, X_{\le 2})\) is well-defined.

Let \(X=\cosk_3(X_{\le 3})\). Explicitly, for \(m>3\) we have
\[
X_m=\{ f\in \Hom(\sk_2(\Delta_m), X_{\le 2})\mid f\circ (d_0\times d_1 \times
d_2 \times d_3) (\sk_3(\Delta_m)) \subset X_3 \}.
\]
We can show that \(\{X_m\}_{m\ge 0}\) forms a simplicial object and satisfies the appropriate Kan conditions. So we obtain a 2\nbdash{}groupoid object in \((\Cat, \covers)\).

\subsubsection{From the simplicial picture to categorification}

Conversely, for a 2\nbdash{}groupoid object in \((\Cat, \covers)\) given by a simplicial object \(X\), we construct a categorified groupoid.

Let \(\arrow{X}\) be the fibre of the décalage \(\dec'(X)\to X\) defined in Example~\ref{chap3:exa:decalage}. We know that \(\arrow{X}\) is a 1-groupoid. This groupoid is described as a \emph{groupoid of bigons} in~\cite{Zhu:ngpd}. The space of bigons \(\bigon{X_2}=\face_0^{-1}(\de_0 X_0)\subset X_2\) is a groupoid over \(X_1\). Applying Kan conditions, one can show that another symmetric choice gives the same result.

The face maps \(\face_0, \face_1\colon X_1\to X_0\) give two morphisms \(\source,\target\colon \arrow{X}\to \trivial{X_0}\). We then show that \(\arrow{X}\) acts on \(X_2\) from three sides. This gives the multiplication HS morphism \(\arrow{X}\times_{\source,\trivial{X_0},\target} \arrow{X}\to \arrow{X}\). The space of inverse bigons \(\face_1^{-1}(\de_0 X_0)\subset X_2\) gives the inversion bibundle~\(E_\inv\). Then \(\arrow{X} \rightrightarrows \trivial{X_0}\) is a categorified groupoid. All the data to build a categorified groupoid are encoded by simplicial identities and Kan conditions.

\subsection{Geometric nerve of a 2-category}\label{chap5:ssec:geometric-nerve}

The relation between a categorified groupoid and its corresponding simplicial object is analogous to that between a 2\nbdash{}category and its geometric nerve; this becomes precise when every HS bibundle in Definition~\ref{chap5:def:2gpd-action} is a groupoid functor.

The (unitary) \emph{geometric nerve} \(N\Cat[C]\) of a 2\nbdash{}category \(\Cat[C]\) is a simplicial set that completely encodes all the 2-categorical structure. The 0-simplices of \(N\Cat[C]\) are objects in \(\Cat[C]\), 1-simplices are 1-morphisms, 2-simplices are oriented 2-commutative triangles, and 3-simplices are oriented 2-commutative tetrahedra. The following theorem of Duskin~\cite{Duskin02} characterises the geometric nerve of a 2-category.

\begin{theorem}\label{chap5:thm:nerve=2-category}
The geometric nerve of a \((2,1)\)\nbdash{}category \textup(2\nbdash{}groupoid\textup) is an inner Kan complex \textup(Kan complex\textup) satisfying unique horn filling conditions above dimension \(2\); conversely, every inner Kan complex \textup(Kan complex\textup) satisfying unique horn filling conditions above dimension \(2\) arises from the geometric nerve of a \((2,1)\)-category \textup(2\nbdash{}groupoid\textup).
\end{theorem}

The reconstruction procedure of a 2-category from a simplicial set is non-canonical: the composition of 1-simplices is only defined up to unique 2-simplices. Theorem~\ref{chap5:thm:2-groupoid-cat=simp} implies that we obtain a well-defined composition if we use HS bibundles as morphisms between groupoids.

\section{Actions of Lie 2-groupoids via categorification}

We categorify the notion of action in this section. Our definition is modeled on the notion of action of a monoidal category on a category; see, for instance, \cite{Benabou67}.

\begin{definition}
  \label{chap5:def:2gpd-action}
  A (right) \emph{categorified action} of a categorified groupoid \(G \rightrightarrows \trivial{M}\) on a groupoid \(E\) consists of the following data:
  \begin{enumerate}
  \item\label{chap5:item:action_1}
  a \emph{moment morphism}, which is an HS morphism~\(J\) from~\(E\) to~\(\trivial{M}\) or, equivalently, a map \(J_0\colon E_0\to \trivial{M}\) in~\(\Cat\) with \(J_0\circ \source_E = J_0\circ \target_E\);
  \item\label{chap5:item:action_2}
  an \emph{action morphism}, which is an HS morphism~\(\action\) from~\(E\times_{J,\trivial{M},\target}G\) to~\(E\),
  such that \(J\circ \action = \source \circ  \pr_2\colon E \times_{J, \trivial{M},\target} G \to  \trivial{M}\);
  \item\label{chap5:item:action_3} a 2\nbdash{}morphism between HS morphisms \(E\times_{J,\trivial{M},\target} G\times_{\source,\trivial{M},\target}G\to E \),
    \begin{equation}
      \label{chap5:eq:associator}
      \alpha\colon \action\circ(\action\times \id) \Rightarrow \action\circ(\id\times \mult),
    \end{equation}
    called \emph{associator}, where~\(\mult\) is the multiplication in the 2\nbdash{}groupoid \(G\rightrightarrows \trivial{M}\);
  \item\label{chap5:item:action_4}
  a 2\nbdash{}morphism \(\rho\colon \action\circ (\id,\unit\circ J) \Rightarrow \id_E\), called \emph{unitor}.
  \end{enumerate}
  The associator satisfies the \emph{pentagon condition}, described as follows. Let the 2\nbdash{}morphisms on the each face of the cube be \(\alpha_i\) (all the \(\alpha_i\)'s are generated by \(\alpha\) or by \(\alpha_G\), the associator of \(G\rightrightarrows \trivial{M}\), except that \(\alpha_4\) is \(\id\)) arranged in the following way: front face~\(\alpha_1\),  back~\(\alpha_5\); up~\(\alpha_4\), down~\(\alpha_2\); left~\(\alpha_6\), right~\(\alpha_3\); the following cube
  \begin{equation} \label{chap5:eq:pentagon}
  \begin{gathered}
  \xymatrix@=5pt{
     & & E\times_\trivial{M}G\times_\trivial{M}G \ar[dr]^-{\action\times \id} \ar[ddd]^{\rotatebox{-90}{\(\scriptstyle \id\times \mult\)}}& \\
  E \times_\trivial{M}G\times_\trivial{M}G\times_{\trivial{M}}G \ar[urr]^{\id\times \id\times \mult} \ar[dr]_{\action\times \id\times \id} \ar[ddd]_{\rotatebox{-90}{\(\scriptstyle \id\times \mult\times \id\)}} & & & E \times_{\trivial{M}}G \ar[ddd]^{\action} \\
     & E\times_\trivial{M}G\times_\trivial{M}G\ar[urr]^(0.4){\id\times \mult} \ar[ddd]^{\rotatebox{-90}{\(\scriptstyle \action\times \id\)}} & & & \\
     & & E \times_\trivial{M}G \ar[dr]^{\action} & \\
  E\times_\trivial{M}G\times_\trivial{M}G \ar[urr]^(0.4){\id\times \mult} \ar[dr]_{\action\times \id} & & & E \\
     & E\times_\trivial{M}G \ar[urr]^{\action} & &
  }
   \end{gathered}
  \end{equation}
   is 2\nbdash{}commutative, that is,
  \[
 (\alpha_6\times \id )\circ(\id\times \alpha_2)\circ(\alpha_1\times \id)=
 (\id\times \alpha_5) \circ (\alpha_4 \times \id) \circ (\id \times \alpha_3).
  \]

  The unitor satisfies the \emph{triangle condition}: the diagram
  \begin{equation}\label{chap5:eq:x1g}
    \begin{gathered}
    \xymatrix{
       && E\times_\trivial{M} G \ar[rrd]^{\action}&&\\
     E\times_\trivial{M} G\times_\trivial{M} G \ar[rru]^{\action\times \id}\ar[rrd]_{\id\times \mult}&&
     E\times_\trivial{M} G \ar[ll]^-{(\id, \unit\circ J)\times \id)}\ar[u]^{\id}\ar[d]_{\id}  &&E \\
       && E\times_\trivial{M} G\ar[rru]_{\action}&&
     }
    \end{gathered}
  \end{equation}
  is 2\nbdash{}commutative, where the 2\nbdash{}morphism on the whole square is induced by \(\alpha\), on the upper triangle by \(\rho\), on the lower triangle by \(\lambda_G\), the left unitor 2\nbdash{}morphism of~\(G\rightrightarrows \trivial{M}\), and on the right square by the identity.
\end{definition}

\begin{remark}\label{chap5:rem:the-other-axiom}
  It is natural to ask another diagram similar to~\eqref{chap5:eq:x1g} to 2\nbdash{}commute:
  \begin{equation}\label{chap5:eq:xg1}
    \begin{gathered}
    \xymatrix{
       && E\times_\trivial{M} G \ar[rrd]^{\action}&&\\
     E\times_\trivial{M} G\times_\trivial{M} G \ar[rru]^{\action\times \id}\ar[rrd]_{\id\times \mult}&&
     E\times_\trivial{M} G \ar[ll]^-{\id\times (\id, \unit\circ \source)}\ar[rr]^{\action}\ar[d]_{\id}  &&E \\
       && E\times_\trivial{M} G\ar[rru]_{\action}&&
     }
    \end{gathered}
  \end{equation}
  where the 2\nbdash{}morphism on the whole square is induced by \(\alpha\), on the upper square by \(\rho\), on the lower left triangle by \(\rho_G\), the right unitor of~\(G\rightrightarrows \trivial{M}\), and on the lower right triangle by the identity.  This is assumed in~\cite[Definition 6.1]{Breen:Bitorseurs} and~\cite[Definition 13.1]{Bakovic:Bigroupoid_torsors}. This may, however, be deduced from the other axioms with an argument similar to~\cite{Kelly:MacLane_conditions}.
\end{remark}

The next two sections are devoted to proving the following main result of this chapter.

\begin{theorem}\label{chap5:thm:2gpd-action-Kan}
  Given a 2\nbdash{}groupoid object \(X\) in \((\Cat, \covers)\), there is a one-to-one correspondence between categorified actions of the 2\nbdash{}groupoid \(X\) and 2\nbdash{}groupoid Kan fibrations \(A\to X\). The 2\nbdash{}groupoid \(A\) is called the \emph{action 2\nbdash{}groupoid} of the corresponding action.
\end{theorem}

\begin{remark}
  A categorified action of a 2\nbdash{}groupoid \(G\rightrightarrows\trivial{M}\) can be viewed as a categorified bibundle between the trivial 2-groupoid and \(G\rightrightarrows\trivial{M}\), which will be studied in the next chapter. For a categorified bibundle we will construct an augmented bisimplicial object, which contains our action 2\nbdash{}groupoid \(A\) given below as a row.
\end{remark}

\begin{remark}
  For a 2-category \(\Cat[C]\), we can define actions on categories as in Definition~\ref{chap5:def:2gpd-action}. Similar to Remark~\ref{chap4:rem:action=functor}, such an action can be equivalently given by a 2-functor \(\Cat[C]\to \Cats\); see~\cite{Garner-Shulman}. The Grothendieck construction~\cite{Street} produces a 2\nbdash{}category \(\Cat[A]\) with a 2-functor into~\(\Cat[C]\). Our construction of action 2\nbdash{}groupoids can be regarded as a geometrical Grothendieck construction; a similar construction appeared in~\cite{Bakovic:Bigroupoid_torsors}. Theorem~\ref{chap5:thm:2gpd-action-Kan} can be viewed as a higher analogue of the well-known equivalence between fibrations over \(\Cat[C]\) and 2\nbdash{}functors \(\Cat[C]\to \Cats\) for a category \(\Cat[C]\). For an \(\infty\)-categorical generalisation along these lines, see~\cite[Chapter 2]{Lurie}.
\end{remark}

\section{From the simplicial picture to categorification} \label{chap5:sec:simp-cat}

Let \(\pi\colon A\to X\) be a Kan fibration between 2\nbdash{}groupoids in \((\Cat, \covers)\) that satisfies \(\Kan(1, 0)\), \(\Kan(1,1)\) and \(\Kan!(m, j)\) for \(m\ge 2\) and \(0\le j \le m\). Let~\(\arrow{X}\) denote the groupoid of bigons \(\bigon{X_2} \rightrightarrows X_1\). Let~\(E=\Fib{\pi}\), which is a 1-groupoid object in \((\Cat,\covers)\) by Lemma~\ref{chap3:lem:fibre_exists}. The map \(\pi_0\colon E_0=A_0\to X_0\) gives a groupoid functor \(J\colon E\to \trivial{X_0}\). We shall show that the 2\nbdash{}groupoid \(\arrow{X}\rightrightarrows \trivial{X_0}\) acts on~\(E\) along \(J\) as in Definition~\ref{chap5:def:2gpd-action}.

\subsection{The action morphism}

\begin{lemma}\label{chap5:lem:a1-bibundle}
  The space \(A_1\) is an HS bibundle \(E\times_{J,\trivial{X_0},\target} \arrow{X} \to E\), denoted by \(\action\).
\end{lemma}

\begin{proof}
  There is a left action of the groupoid \(E\) on \(A_1\) along the map \(\face_1\colon A_1\to E_0=A_0\) given by the composed map
  \[
  E_1\times_{\source_E, A_0, \face_1} A_1 \to \Hom(\Horn{2}{1}\to\Simp{2},A\to X)
   \cong A_2 \xrightarrow{\face_1} A_1,
  \]
  where the map \(E_1\times_{\source_E,A_0,\face_1} A_1 \to \Hom(\Horn{2}{1}\to\Simp{2},A\to X)\) is given by \((e,a)\mapsto (e,a,\de_1(\pi_1(a)))\). The Kan conditions for \(\pi\colon A\to X\) imply that this defines an action: \(\Kan!(2,1)\) implies the unit identity and \(\Kan!(3,1)\) implies the associativity. Similarly, there are a left action of \(\arrow{X}\) on \(A_1\) along the map \(\pi\colon A_1\to \arrow{X}_0=X_1\) and a right action of \(E\) on \(A_1\) along the map \(\face_0\colon A_1 \to A_0\).
  The Kan conditions imply that these three actions mutually commute.  Since \(\face_0\circ\pi_1=\pi_0\circ\face_0\colon A_1\to X_0\), the groupoid \(E\times_{\trivial{X_0}} \arrow{X}\) acts on~\(A_1\). Thus, \(A_1\) is an \(E\times_{\trivial{X_0}} \arrow{X}\)-\(E\) bibundle, illustrated by Figure~\ref{chap5:fig:bibundle_A1}, where white dots are imaginary as in Remark~\ref{chap1:rem:gpd-action-imaginary-picture}.
  \begin{figure}[htbp]
    \centering
    \begin{tikzpicture}
      [>=latex', mydot/.style={draw,circle,inner sep=1.5pt}, every label/.style={scale=0.6},scale=0.75]
      \node[scale=0.6] at (0,0) (a) {\(A_1\)};
      \node[mydot,label=210:\(0\)]              at (-0.866,-0.5)  (a0) {};
      \node[mydot,fill=black,label=90:\(1\)]    at (0,1)          (a1) {};
      \node[mydot,fill=black,label=-30:\(2\)]   at (0.866,-0.5)   (a2) {};
      \path[<-]
      (a0) edge (a1)
           edge (a2)
      (a1) edge (a2);

      \begin{scope}[xshift=-3cm,<-]
        \node[mydot,label=210:\(0\)]              at (-0.866,-0.5)  (b0) {};
        \node[mydot,fill=black,label=90:\(1\)]    at (0,1)          (b1) {};
        \node[mydot,fill=black,label=-30:\(2\)]   at (0.866,-0.5)   (b2) {};
        \path
        (b0) edge[bend right=15] (b1)
             edge[bend left=15] node[scale=0.6,above left]{\(E\)}  (b1)
        (b1) edge[bend right=15] (b2)
        edge[bend left=15] node[scale=0.6,above right]{\(\Theta(X)\)}(b2);
      \end{scope}
      \begin{scope}[xshift=+3cm, <-]
        \node[mydot,label=210:\(0\)]              at (-0.866,-0.5)  (c0) {};
        \node[mydot,fill=black,label=-30:\(2\)]   at (0.866,-0.5)   (c2) {};
        \path
        (c0) edge[bend right=15] node[scale=0.6,below]{\(E\)} (c2)
             edge[bend left=15] (c2);
      \end{scope}
    \end{tikzpicture}
    \caption{The bibundle structure on~\(A_1\)}
    \label{chap5:fig:bibundle_A1}
  \end{figure}
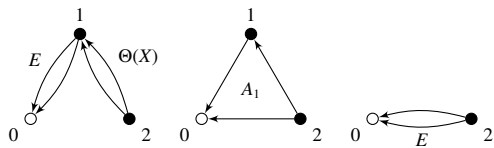

  Moreover, \(\Kan(1,0)\) for \(\pi\) implies that the left moment map \(A_1 \to E_0 \times_{X_0} \arrow{X}_0\) is a cover in \((\Cat,\covers)\), and \(\Kan!(2,0)\) for \(\pi\) implies that the right action of \(E\) is principal. Therefore, \(A_1\) gives an HS morphism \(\action\colon E\times_{J,\trivial{X_0},\target} \arrow{X} \to E\).
\end{proof}

Since \(\face_1\pi_1=\pi_0\face_1\), Lemma~\ref{chap5:lem:bundles-triangle-to-M} below implies that \(J\circ \action=\source\circ\pr_2\).

\begin{lemma}\label{chap5:lem:bundles-triangle-to-M}
Let \(P\) be an HS bibundle between groupoids \(K\) and \(H\). Suppose that \(M\) is an object in \(\Cat\). Let \(\Bund(\varphi)\) and \(\Bund(\psi)\) be the bundlisation of groupoid functors \(\varphi\colon K\to\trivial{M}\) and \(\psi\colon H\to \trivial{M}\). The composite HS bibundle \(P\otimes\Bund(\psi)\) is isomorphic to \(\Bund(\varphi)\) if and only if the composites \(P\to K_0\to M\) and \(P\to H_0\to M\) are equal.
\end{lemma}
\begin{proof}
  Consider the diagram of groupoids
  \[
  \xymatrix{
  K\ltimes P\rtimes H\ar[r]^-{\pr_H}\ar[d]_{\pr_K}& H\ar[d]^{\psi}\\
  K\ar[r]^{\varphi}& \trivial{M}\rlap{\ .}
  }
  \]
  Since the categories \(\Bun\) and \(\Gen\) are isomorphic and \(\trivial{M}\) is a trivial groupoid, it suffices to show that \(\psi\circ \pr_H\circ \pr_K^{-1}=\varphi\) in \(\Gen\) if and only if the composites \(P\to K_0\to M\) and \(P\to H_0\to M\) are equal. We know that \(\psi\circ \pr_H\circ \pr_K^{-1}=\varphi\) in \(\Gen\) if and only if \(\psi\circ \pr_H=\varphi\circ \pr_K\) in \(\Gpd\), and the rest is obvious.
\end{proof}

To show that \(\action\colon E\times_{J,\trivial{X_0},\target} \arrow{X} \to E\) gives a categorified action of \(X\) on \(E\) along \(J\), we must construct the data for~\ref{chap5:item:action_3} and~\ref{chap5:item:action_4} in Definition~\ref{chap5:def:2gpd-action}.

\subsection{The associator}

We now construct the associator for the action morphism \(\action\).

The HS morphism \(\action\circ (\action \times \id)\colon E\times_{\trivial{X_0}} \arrow{X} \times_{\trivial{X_0}} \arrow{X}\to E\) is given by the HS bibundle
  \[
  ( (A_1 \times_{X_0} \arrow{X}_1)\times_{E_0\times_{X_0} \arrow{X}_0} A_1 )/( E \times_{\trivial{X_0}} \arrow{X}).
  \]
Since \(\arrow{X}\) acts on \(\arrow{X}_1\) by multiplication, we have
  \[
  ( (A_1 \times_{X_0} \arrow{X}_1)\times_{E_0\times_{X_0} \arrow{X}_0} A_1
  )/( E \times_{\trivial{X_0}} \arrow{X} )\cong ( A_1 \times_{E_0} A_1 )/ E.
  \]
Similarly, the HS morphism \(\action\circ(\id \times \mult)\) is given by the HS bibundle
  \[
  ( A_1 \times_{\arrow{X}_0} X_2 ) / \arrow{X},
  \]
where \(\mult\) is the multiplication of the 2\nbdash{}groupoid \(\arrow{X} \rightrightarrows \trivial{X_0}\).

In particular, the \(E\)-action on \(A_1 \times_{E_0} A_1\) and the \(\arrow{X}\) action on \(A_1 \times_{\arrow{X}_0} X_2 \) are principal. We now construct a one-to-one correspondence, illustrated by Figure~\ref{chap5:fig:associator},
  \begin{gather*}
  \alpha\colon (A_1 \times_{E_0} A_1 )/ E \to ( A_1 \times_{\arrow{X}_0}  X_2 )/ \arrow{X}, \\
   [(a,b)] \mapsto [(c,d)]\quad\text{if } (d,b,c,a) \in  A_2.
  \end{gather*}
  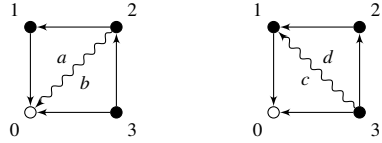
\begin{figure}[htbp]
    \centering
    \begin{tikzpicture}
      [>=latex',
      scale=0.8,
      every label/.style={scale=0.6}
      ]
      \node[mydot,label=225:\(0\)] at (225:1cm)(a0){};
      \foreach \i in {1,2,3}{
        \node[mydot,fill=black,label=225-\i*90:\(\i\)]  at (225-\i*90:1cm) (a\i){};
      }
      \path[<-]
      (a0)  edge        (a1)
      edge        (a3)
      edge[snake back]
      node[scale=0.6,above left]  {\(a\)}
      node[scale=0.6,below right] {\(b\)} (a2)
      (a1)  edge        (a2)
      (a2)  edge        (a3);
      \begin{scope}[xshift=+4cm]
        \node[mydot,label=225:\(0\)] at (225:1cm)(b0){};
        \foreach \i in {1,2,3}{
          \node[mydot,fill=black,label=225-\i*90:\(\i\)]  at (225-\i*90:1cm) (b\i){};
        }
        \path[<-]
        (b0)  edge        (b1)
        edge        (b3)
        (b1)  edge        (b2)
        edge[snake back]
        node[scale=0.6,above right] {\(d\)}
        node[scale=0.6,below left]  {\(c\)} (b3)
        (b2)  edge        (b3);
      \end{scope}
    \end{tikzpicture}
    \caption{The construction of the associator}\label{chap5:fig:associator}
  \end{figure}

First, we show that this map is well-defined. Since~\(\pi\) satisfies \(\Kan!(3,1)\), our construction is independent of the choice of \((a,b)\) in the equivalence class \([(a, b)]\). Similarly, it is independent of the choice of \((c, d)\) because~\(\pi\) satisfies \(\Kan!(3,2)\).

Furthermore, the map \(\alpha\) is equivariant with respect to the right action of \(E\) by using \(\Kan!(3,2)\) for \(\pi\), and it is equivariant with respect to the left action of \(E\) and both actions of \(\arrow{X}\) by the conditions \(\Kan!(3,1)\), \(\Kan!(3,2)\), and \(\Kan!(3,3)\) for \(\pi\), respectively.

Finally, since \(A_1\to \Hom(\Horn{1}{0}, A\to X)\) is a cover, the map
\[
  \Hom(\Horn{2}{1}\to \Simp{2}, A\to X)\cong\Hom(\Horn{2}{0}\to \Simp{2}, A\to X) \to A_1\times_{X_1} X_2
\]
is a cover. This map is invariant under the principal action of \(E\) and equivariant under the principal action of \( \arrow{X}\). Therefore, by Lemma~\ref{chap1:lem:invariant-cover-descent} and Lemma~\ref{chap1:lem:action-over-principal-is-principal}, it descends to a cover, which is \(\alpha\) by definition.

In conclusion, \(\alpha\) defines an isomorphism of HS bibundles. The pentagon condition for~\(\alpha\) follows by a similar argument as in \cite[Section 4.1]{Zhu:ngpd}, we will treat it in detail later in Chapter~\ref{chap6} for bibundles between 2\nbdash{}groupoids.

\subsection{The unitor}

We now construct the unitor for the action and then verify the triangle condition.

The unit morphism \(\unit\colon \trivial{X_0}\to \arrow{X}\) is given by \(\de_0\colon X_0\to \arrow{X}_0\).
Thus the morphism \(\action\circ(\id, \unit\circ J)\) is given by the bibundle \(\pi_1^{-1}\circ \de_0(X_0) \subset A_1\), which is isomorphic to \(E_1\) by construction. We take this isomorphism to be the unitor \(\rho\colon \action\circ(\id, \unit\circ J) \Rightarrow \id_E\).

The HS morphism corresponding to \((a\cdot 1)\cdot  g\) is given by the HS bibundle \((E_1 \times_{\face_0,A_0,\face_2} A_1 )/E\), while the one corresponding to \(a\cdot(1\cdot g)\) is given by the HS bibundle \(( A_1\times_{\pi_1,\arrow{X}_0,\face_1}  \arrow{X}_1 )/\arrow{X}\). They are both isomorphic to \(A_1\) by the isomorphisms \(\rho\) and \(\lambda_X\). Moreover, applying \(\Kan!(2, 1 )\) to a triple \((\de_0\pi_1\eta, \eta, \de_0\face_0\eta)\), where \(\eta\in A_1\), proves that these isomorphisms are compatible with the associator~\(\alpha\). Thus the triangle condition~\eqref{chap5:eq:x1g} holds.

This completes the proof of Theorem~\ref{chap5:thm:2gpd-action-Kan} in one direction.

\section{Finite data of Kan fibrations}

We recall from~\cite[Section 2.3]{Zhu:ngpd} that a 2\nbdash{}groupoid object \(X\) in \((\Cat, \covers)\) is equivalently given by \(X_{\le 4}\), or by \(X_{\le 2}\) with extra data (the so called \(3\)-multiplications and coherence conditions). Given a 2\nbdash{}groupoid object \(X\) in \((\Cat, \covers)\), consider 2\nbdash{}groupoid Kan fibrations \(\pi\colon A\to X\) in \((\Cat, \covers)\). We show that such a Kan fibration is fully determined by \(\pi_{\le 3}\colon A_{\le 3}\to X_{\le 3}\), or by \(\pi_{\le 1}\colon A_{\le 1}\to X_{\le 1}\) with extra data.

Let \(\horn{m}{j}\) denote the horn inclusion \(\Horn{m}{j}\to \Simp{m}\) for \(0\le j\le m\).

\begin{proposition}\label{chap5:prop:kan-fibration-finite-data}
Let \(X\) be a 2\nbdash{}groupoid in \((\Cat, \covers)\). A 2\nbdash{}groupoid Kan fibration \(A\to X\) in \((\Cat, \covers)\) can also be described by \(\pi_{\le 1}\colon A_{\le 1}\to X_{\le 1}\) and the following data:
\begin{itemize}
  \item \(A_{\le 1}\) satisfies appropriate simplicial identities, and \(\pi_{\le 1}\) commutes with face and degeneracy maps.  The conditions \(\Kan(1,0)\) and \(\Kan(1,1)\) hold, that is,
  \[
  A_1\to A_0\times_{\pi_0,X_0,\face^1_0} X_1 \quad\text{and}\quad A_1\to A_0\times_{\pi_0,X_0,\face^1_0} X_1
  \]
  are covers.
  \item morphisms \textup(\(3\)\nbdash{}multiplications\textup):
  \[
  m_j\colon \Hom(\horn{2}{j}, \pi)\to A_1,\quad j=0,1,2.
  \]
  For instance, given compatible \(\eta_{01}, \eta_{02} \in A_1\) and \(\gamma\in X_2\) that form an element in \(\Hom(\horn{2}{0}, \pi)\), \(m_0\) determines a unique \(\eta_{12}\in A_1\). The \(3\)-multiplications satisfy the following coherence conditions:
  \begin{enumerate}
    \item \label{chap5:item:coherence-3-mult-iso}
     the induced morphisms \(\Hom(\horn{2}{i},\pi)\to \Hom(\horn{2}{j}, \pi)\) are all isomorphisms for \(i, j=0,1,2\).
    \item \label{chap5:item:coherence-3-mult-de}
    compatibility of \(3\)-multiplications, face, and degeneracy maps: given \(\eta\in A_1\), then
    \begin{gather*}
    m_1(\de_0\pi_1\eta, \eta, \de_0\face_0\eta)=\eta,\\
    (\text{or equivalently, }
    m_0(\de_0\pi_1\eta, \eta, \de_0\face_0\eta)=\eta,\quad m_2(\de_0\pi_1\eta, \eta, \eta)=\de_0\face_0\eta);\\
    m_1(\de_1\pi_1\eta, \de_0\face_1\eta, \eta)=\eta,\\
    (\text{or equivalently, }
    m_2(\de_1\pi_1\eta, \de_0\face_1\eta, \eta)=\eta, \quad m_1(\de_1\pi_1\eta, \eta, \eta)=\de_0\face_1\eta).
    \end{gather*}
    \item \label{chap5:item:coherence-3-mult-ass}
     associativity for 3-multiplications: given \(\eta_{01}, \eta_{12}, \eta_{23}\in A_1\) and \(\gamma\in \Hom(\Simp{3}, X)\) that are compatible in the obvious way, the following two methods to determine~\(\eta_{03}\) give the same result. We can either obtain \(\eta_{02}\) first:
   \[
   \eta_{02}=m_1(\gamma_{012}, \eta_{12},\eta_{01}),\quad \eta_{03}=m_1(\gamma_{023}, \eta_{23}, \eta_{02}),
   \]
   or we can obtain \(\eta_{13}\) first:
   \[
    \eta_{13}=m_1(\gamma_{123}, \eta_{23}, \eta_{12}),\quad  \eta_{03}=m_1(\gamma_{013}, \eta_{13}, \eta_{01}).
   \]
   Equivalently, given the same data, the two methods to determine \(\eta_{02}\) or \(\eta_{13}\) give the same results.
  \end{enumerate}
\end{itemize}
\end{proposition}

\begin{remark}
  We may think of the 3-multiplication \(m_1\) as a composition and \(m_0, m_2\) as inverses. Then the coherence condition~\ref{chap5:item:coherence-3-mult-iso} says that the inverses are true inverses with respect to the composition; condition~\ref{chap5:item:coherence-3-mult-de} is the unit condition; and condition~\ref{chap5:item:coherence-3-mult-ass} is the associativity condition. We will see in the proof below that 3-multiplications and the coherence conditions all are encoded in \(A_{\le 3}\to X_{\le 3}\).
\end{remark}

\subsection{Reconstruction of higher dimensions}
To prove Proposition~\ref{chap5:prop:kan-fibration-finite-data}, we shall reconstruct the simplicial object \(A\) in \((\Cat, \covers)\) with a morphism \(\pi\colon A\to X\) satisfying \(\Kan(1, 0), \Kan(1, 1)\) and \(\Kan!(m, j)\) for \(0\le j \le m\) and \(m\ge 2\). Then \(A\) is automatically a 2\nbdash{}groupoid in \((\Cat, \covers)\) and \(\pi\colon A\to X\) is a 2\nbdash{}groupoid Kan fibration.

\subsubsection{Dimension 2}
Let
\[
A_2=\Hom(\horn{2}{j}, \pi),\qquad \forall\, j=0, 1, 2,
\]
which is well-defined by assumption. The conditions \(\Kan!(2,j)\) for \(j=0,1,2\) for~\(\pi\) hold automatically. The face maps \(A_2\to A_1\) and the projection \(\pi_2\colon A_2 \to X_2\) are natural to define. The degeneracy maps \(\de_0, \de_1\colon A_1 \to A_2\) are given by
\begin{align*}
\de_0\colon \eta&\mapsto (\de_0\pi_1\eta, \eta, \de_0\face_0\eta) \in \Hom(\horn{2}{1},\pi),\\
 \de_1\colon \eta&\mapsto (\de_1\pi_1\eta, \de_0\face_1\eta, \eta)\in \Hom(\horn{2}{1}, \pi).
\end{align*}
The coherence condition~\ref{chap5:item:coherence-3-mult-de} implies that \(\de_0, \de_1\) are well-defined and that the simplicial identities \(\face_0\circ \de_0=\id= \face_1 \circ \de_1\) hold. The other simplicial identities and commutativity with~\(\pi\) are left to the reader.

\subsubsection{Dimension 3}
Let
\begin{equation} \label{chap5:eq:an}
  A_m\coloneqq \Hom(\sk_2\Simp{m}\to\Simp{m},\pi),\qquad \forall\, m \ge 3.
\end{equation}
Since the above expression depends only on \(\pi_j\) for \(j=0, 1, 2\), the space~\(A_m\) is already well-defined.

\begin{lemma}
  The space~\(A_3\) in~\eqref{chap5:eq:an} is representable, and
  \[
  A_3=\Hom(\horn{3}{j},\pi),\qquad \forall\, 0\le j\le 3.
  \]
\end{lemma}

\begin{proof}
  Lemma~\ref{chap3:lem:covers_in_groupoid} shows that \(\Hom(\horn{3}{j},\pi)\) is representable for \(0\le j \le 3\). It suffices to show that each element of
  \[
  \Hom(\sk_2\Simp{3}\to \Simp{3}, \pi)=\Hom(\partial\Simp{3}\to \Simp{3}, \pi)
  \]
  can be extended uniquely from an element of \(\Hom(\horn{3}{j},\pi)\) for \(0\le j \le 3\).
  For instance, given an element \(\eta\in \Hom(\horn{3}{0},\pi)\), we have \(\eta_{01}, \eta_{12}\) and \(\eta_{23}\) and \(\gamma\in \Hom(\Simp{3}, X)\) that are compatible. The coherence condition~\ref{chap5:item:coherence-3-mult-ass} implies that the two ways to determine \(\eta_{03}\) give the same results. That is, we obtain a unique element of \(\Hom(\sk_2\Simp{3}\to \Simp{3}, \pi)\). We can prove the other cases by using different formulations of the coherence condition, or by Lemma~\ref{chap3:lem:morphism-kan(n+1,j)-implies-all}.
\end{proof}

The face maps, degeneracy maps, and the projection \(\pi_3\colon A_3\to X_3\) are natural to define. The Kan conditions follow from the lemma above; moreover, \(\Acyc!(3)\) hold by definition.

\subsubsection{Higher dimensions}
We have shown that \(\Kan!(3,j)\) and \(\Acyc!(3)\) hold. Suppose that \(A_{k}\) for \(3 \le k \le m\) given by~\eqref{chap5:eq:an} are representable. Suppose that \(\Kan!(k,j)\) holds for \(0\le j\le k\) and \(2 \le k \le m\), and that \(\Acyc!(k)\) holds for \(3\le k\le m\). We define \(A_{m+1}\coloneqq \Hom(\sk_2\Simp{m+1}\to\Simp{m+1}, \pi)\), and show that \(\Kan!(m+1,j)\) and \(\Acyc!(m+1)\) hold.

Since \(\sk_2\Simp{m+1}=\sk_2\Horn{m+1}{j}\), we have
\[
\Hom(\sk_2\Simp{m+1}\to\Simp{m+1},
\pi)=\Hom(\sk_2\Horn{m+1}{j}\to\Simp{m+1}, \pi).
\]
By \(\Acyc!(k)\) for \(3\le k\le m\), the same argument as in the proof of~\cite[Lemma~2.5]{Zhu:ngpd} tells us that the natural maps
\[
\Hom(\partial\Simp{m+1}\to \Simp{m+1}, \pi)\to \Hom(\horn{m+1}{j},\pi)
\to \Hom(\sk_2\Horn{m+1}{j}\to\Simp{m+1}, \pi)
\]
are both isomorphisms. Since we have the exact condition \(\Acyc!(k)\) instead of \(\Acyc(k)\), the conclusion is strengthened to being isomorphisms instead of being covers as in the original lemma. Hence
\begin{multline}
  \label{chap5:eq:pld}
  \Hom(\partial\Simp{m+1}\to \Simp{m+1},\pi)
  = \Hom(\horn{m+1}{j},\pi)\\
  = \Hom(\sk_2\Simp{m+1}\to\Simp{m+1},\pi)
  = A_{m+1}.
\end{multline}
The presheaf \(\Hom(\horn{m+1}{j},\pi)\) is representable by Lemma~\ref{chap3:lem:representable}, so \(A_{m+1}\) is representable.
There are natural face maps, degeneracy maps, and a natural projection \(\pi_{m+1}\colon A_{m+1}\to X_{m+1}\). The conditions \(\Kan(m+1,j)\) for \(0\le j\le m+1\) and \(\Acyc!(m+1)\) hold by~\eqref{chap5:eq:pld}. This complete the proof of Proposition~\ref{chap5:prop:kan-fibration-finite-data}.

\section{From categorification to the simplicial picture}\label{chap5:sec:cat-simp}

Given a categorified action of a 2\nbdash{}groupoid \(G\rightrightarrows \trivial{M}\) on a groupoid \(E\), we now construct the finite data for the expected Kan fibration \(A\to X\).

\subsection{Dimensions 0 and 1}

First, we let
\begin{alignat*}{2}
  A_0&\coloneqq E_0,&\qquad
  A_1&\coloneqq E_\action,\\
  \face_0&\coloneqq J_r\colon E_\action\to E_0,&\qquad
  \face_1&\coloneqq \pr_1 \circ J_l\colon E_\action\to E_0,\\
  \pi_0&\coloneqq J_0\colon E_0\to X_0,&\qquad
  \pi_1&\coloneqq \pr_2\circ J_l\colon E_\action \to X_1,
\end{alignat*}
where~\(E_\action\) is the bibundle giving the action morphism \(E\times_{J,\trivial{X_0},\target} G\to E\).  The 2\nbdash{}morphism
\(\rho\colon \action\circ  (\id, \unit\circ J) \Rightarrow \id_E\) induces an isomorphism of bibundles
\[
\rho\colon  J_l^{-1}(E_0\times_{X_0} X_0 )\cong E_1.  
\]
We take \(\de_0\) to be the composite
\[
\de_0=\rho^{-1} \circ \unit_E\colon A_0=E_0\to E_1 \to E_\action=A_1.
\]
Then \(\face_0 \circ \de_0= J_r\circ \rho^{-1} \circ \unit_E=\target_E \circ \unit_E=\id\).  Similarly, \(\face_1 \circ \de_0=\id\).  Therefore, the simplicial identities involving \(A_i\), \(i=0,1\), hold.

Furthermore, \(\pi_i\) for \(i=0,1\) commute with face and degeneracy maps. In fact, we have
\begin{gather*}
  \de_0\circ\pi_0=\pi_1 \circ \de_0\colon A_0 \to X_1,\\
  \face_0\circ\pi_1=\pi_0 \circ \face_0\colon A_1 \to X_0
\end{gather*}
by construction. Since \(J\circ \action = \source \circ  \pr_2\), Lemma~\ref{chap5:lem:bundles-triangle-to-M} implies that
\[
  \face_1\circ\pi_1=\pi_0 \circ \face_1\colon A_1 \to X_0.
\]

Next, we verify the conditions \(\Kan(1,0)\) and \(\Kan(1,1)\) for \(\pi\colon A \to X\).  The map
\[
A_1 \to \Hom(\Horn{1}{0}\to \Simp{1},\pi) = E_0\times_{J_0,X_0,\target} X_1
\]
is a cover because it is the left moment map \(E_\action \to E_0\times_{J_0,X_0,\target} X_1\).  This proves \(\Kan(1,0)\) for~\(\pi\). The condition \(\Kan(1, 1)\) follows from the following lemma.

\begin{lemma}\label{chap5:lem:left-E-is-principal}
  The left \(E\)-action on \(E_\action\) gives a left principal \(E\)-bundle \((J_r, \pr_2\circ J_l)\colon E_\action \to  E_0 \times_{J_0,X_0,\source_0} G_0\).
\end{lemma}

\begin{proof}
  The proof is similar to that of \cite[Lemma 3.12]{Zhu:ngpd}. Since~\(X\) has a horizontal inversion, the morphism
  \[
  E\times_{\trivial{M},\target} G\to E\times_{\trivial{M},\source} G, \quad (x, g)\mapsto (x\cdot g, g),
  \]
  is a Morita equivalence. Therefore, the left \(E\times_{\trivial{M},\target} G\)-bundle
  \[
  E_\action\times_{\pr_2\circ J_l,G_0,\target} G_1 \to E_0 \times_{J_0,X_0,\source_0} G_0
  \]
  is principal. Since the left \(G\)-action is given by multiplication, we have
  \[
  (E_\action\times_{\pr_2\circ J_l,G_0,\target} G_1)/G=E_\action
  \]
  as a left \(E\)-bundle, and this proves the lemma.
\end{proof}

\subsection{Dimension 2}

The presheaf \(\Hom(\horn{2}{j}, \pi)\) is already defined because it only involves \(\pi_0\) and~\(\pi_1\); it is representable by Lemma~\ref{chap3:lem:representable}.

Our next goal is to prove the following lemma.
\begin{lemma}\label{chap5:lem:kan2}
  There are natural isomorphisms
  \[
  \Hom(\horn{2}{0}, \pi) \cong \Hom(\horn{2}{1}, \pi)\cong
  \Hom(\horn{2}{2}, \pi).
  \]
\end{lemma}

We first establish two auxiliary lemmas.

\begin{lemma}\label{chap5:lem:iso-of-bibundles}
  There is a natural isomorphism of HS bibundles
  \[
  ( E_\action \times_{E_0} E_\action ) /E \cong ( E_\action \times_{G_0} E_\mult)/G,
  \]
  where~\(E_\mult\) is the HS bibundle giving the horizontal multiplication \(\mult\) of the 2\nbdash{}groupoid~\(X\).
\end{lemma}
\begin{proof}
  The 2\nbdash{}morphism~\eqref{chap5:eq:associator} implies an isomorphism of HS bibundles
  \begin{multline*}
    ((E_\action\times_{X_0} G_1)\times_{E_0\times_{X_0}
      G_0} E_\action) \bigm/ ( E \times_{\trivial{X_0}} G)\cong ((E_1\times_{X_0} E_\mult) \times_{E_0\times_{X_0}
      G_0} E_\action) \bigm/ ( E \times_{\trivial{X_0}} G).
  \end{multline*}
  Since \(G\) acts by multiplication, we have
  \[
  ((E_\action\times_{X_0} G_1)\times_{E_0\times_{X_0} G_0} E_\action) /( E\times_{\trivial{X_0}}G)
  \cong (E_\action \times_{E_0} E_\action) / E .
  \]
  Similarly, we get
  \[
  ((E_1\times_{X_0} E_\mult)\times_{E_0\times_{X_0} G_0} E_\action) /( E \times_{\trivial{X_0}}G)
  \cong (E_\action \times_{G_0} E_\mult) / G ,
  \]
  proving the lemma.
\end{proof}

In particular, the \(E\)-action on \(E_\action \times_{E_0} E_\action\) and the \(G\)-action on \(E_\action \times_{G_0} E_\mult\) are principal. Lemma~\ref{chap1:lem:principal-equivariant-map-is-iso} implies the following lemma.
\begin{lemma}\label{chap5:lem:pull-back}
The following two diagrams are pullback squares in~\(\Cat\):
\[
  \xymatrix{
    E_\action\times_{E_0} E_\action \ar[d]\ar[r] & (E_\action\times_{E_0} E_\action) / E \ar[d]\\
    E_\action                       \ar[r]       & E_0\times_{G_0}G_1 \rlap{\ ,}
  }\quad
  \xymatrix{
  E_\action\times_{E_0} E_\mult\ar[d]\ar[r] & (E_\action \times_{G_0} E_\mult)/ G \ar[d] \\
  E_\mult  \ar[r] & G_0\times_{X_0}G_0\rlap{\ .}
  }
\]
\end{lemma}

\begin{proof}[Proof of Lemma~\ref{chap5:lem:kan2}]
  Lemma~\ref{chap5:lem:pull-back} implies natural isomorphisms
  \[
  \Hom(\horn{2}{0}, \pi)\cong
  (E_\action \times_{G_0} E_\mult)/ G\times_{E_0\times_{X_0}G_0\times_{X_0}G_0 }( E_\action
  \times_{X_1} E_\mult)
  \]
  and
  \[
  \Hom(\horn{2}{1}, \pi)\cong
  (E_\action \times_{E_0} E_\action) / E\times_{E_0\times_{X_0}G_0\times_{X_0}G_0 } (E_\action
  \times_{X_1} E_\mult).
  \]
  Applying Lemma~\ref{chap5:lem:iso-of-bibundles}, we get a natural isomorphism
  \[
  \Hom(\horn{2}{0}, \pi)\cong \Hom(\horn{2}{1}, \pi),
  \]
  and \(\Hom(\horn{2}{2}, \pi)\cong \Hom(\horn{2}{1}, \pi)\) follows similarly.
\end{proof}

It follows from Lemma~\ref{chap5:lem:kan2} that there are 3-multiplications
\[
  m_j\colon \Hom(\horn{2}{j}, \pi)\to A_1,\quad j=0,1,2,
\]
and the coherence condition~\ref{chap5:item:coherence-3-mult-iso} follows immediately.

\begin{lemma}\label{chap5:lem:triangle-condition-imply-3mult-face-de}
The triangle conditions~\eqref{chap5:eq:x1g} and~\eqref{chap5:eq:xg1} imply the coherence conditions~\ref{chap5:item:coherence-3-mult-de}.
\end{lemma}
\begin{proof}
The proof is similar to argument in~\cite[the coherence conditions, p. 4127]{Zhu:ngpd}. The triangle condition~\eqref{chap5:eq:xg1} gives a commutative triangle
\[
\xymatrix{
(E_\action \times_{E_0} E_1)/E\ar[rr]^-{\cong}\ar[rd]_{\cong} &&
(E_\action \times_{G_0} G_1)/G\ar[ld]^{\cong}\\
  & E_\action\rlap{\ ,} &
}
\]
in which every morphism is an isomorphism. It follows that for \(\eta \in A_1=E_\action\) we have
\[
m_1(\face_{1}\pi_1\eta, \de_{0}\face_{1}, \eta)=\eta.
\]
Other cases can be proved similarly.
\end{proof}

The pentagon condition~\eqref{chap5:eq:pentagon} implies the coherence condition~\ref{chap5:item:coherence-3-mult-ass}. Details of this argument can be found in the later Section~\ref{chap6:sec:cat-2-simp}.

We have shown that the spaces \(\{A_i\}_{i\ge 0}\) with the face and degeneracy maps constructed above give a simplicial object in \(\Cat\), denoted by \(A\).  Moreover, the maps \(\pi_i\colon A_i\to X_i\) give a simplicial morphism \(\pi\), which is a Kan fibration.  This completes the proof of Theorem~\ref{chap5:thm:2gpd-action-Kan}.

\section{Principal 2-bundles}\label{chap5:sec:principal-2-bundles}

In this section, we define principal 2\nbdash{}bundles by categorification and show that they agree with the principal 2\nbdash{}bundles defined by the simplicial method in Section~\ref{chap3:sec:hihger-principal-bundles}.

\begin{definition}\label{chap5:def:cat-2-bundle}
  Let \(G\rightrightarrows \trivial{M}\) be a categorified groupoid in \((\Cat,\covers)\), and let~\(N\) be an object in \(\Cat\). A \emph{categorified bundle of \(G\rightrightarrows \trivial{M}\)} over \(N\) consists of
  \begin{itemize}
    \item a groupoid \(E\) with an HS morphism \(\kappa\colon E\to \trivial{N}\) arising from a map \(\kappa_0\colon E_0\to N\),
    \item a right categorified action of \(G\rightrightarrows \trivial{M}\) as in Definition~\ref{chap5:def:2gpd-action} \((\action, J)\) on \(E\),
  \end{itemize}
  such that the action respects \(\kappa\); that is, we have
  \begin{equation}\label{chap5:eq:cat-bundle-respect}
  \kappa\circ \action=\kappa\circ \pr_1\colon E\times_{J,\trivial{M},\target} G\to \trivial{N}.
  \end{equation}
\end{definition}

The following two lemmas verify that this definition of 2\nbdash{}bundles agrees with the one given by the simplicial method.

\begin{lemma}\label{chap5:lem:cat-2-bun-to-simp}
  Let \(\kappa\colon E\to \trivial{N}\) be a categorified \(G\rightrightarrows \trivial{M}\) bundle. Denote the action 2\nbdash{}groupoid by \(A\). Then \(\kappa_0\) gives a canonical morphism \(A\to \sk_0 N\).
\end{lemma}
\begin{proof}
 We need to show that \(\kappa_0\face_1=\kappa_0\face_0\colon A_1\to N\). We know that \(A_1=E_\action\) and that \(\face_1\) is given by \(J_r\), \(\face_0\) is given by \(\pr_1\circ J_l\). Now \eqref{chap5:eq:cat-bundle-respect} and Lemma~\ref{chap5:lem:bundles-triangle-to-M} give \(\kappa_0\face_1=\kappa_0\face_0\).
\end{proof}

\begin{lemma}
  Suppose that we have a categorified action of a categorified groupoid \(G\rightrightarrows \trivial{M}\) on a groupoid \(E\) with action 2\nbdash{}groupoid \(A\). Let \(N\) be an object in \((\Cat,\covers)\). Given a morphism \( A\to \sk_0 N\), then \(E\) is a categorified bundle over~\(N\).
\end{lemma}
\begin{proof}
  The morphism \(A\to \sk_0 N\) is given by a map \(\kappa_0\colon E_0=A_0\to N\) such that \(\kappa_0\face_1=\kappa_0\face_0\colon A_1\to N\). It follows that \(\kappa_0\face_1|_{E_1}=\kappa_0\face_0|_{E_1}\), hence \(\kappa_0\) gives a morphism \(\kappa\colon E\to \trivial{N}\). Since \(\kappa_0\face_1=\kappa_0\face_0\), Lemma~\ref{chap5:lem:bundles-triangle-to-M} proves~\eqref{chap5:eq:cat-bundle-respect}.
\end{proof}

\begin{definition}\label{chap5:def:2-bundle-principal}
  Let \(\kappa\colon E\to \trivial{N}\) be a categorified \(G\rightrightarrows \trivial{M}\) bundle. We say that \(E\) is \emph{principal} if \(\kappa_0\) is a cover and the shear morphism
  \[
  (\pr_1,\action)\colon E\times_{J, \trivial{M},\target} X\to E\times_{N} E,\quad  (e ,g)\mapsto (e, eg),
  \]
  is a Morita equivalence.
\end{definition}

\begin{remark}
  For strict 2\nbdash{}groups or crossed modules of groups, there are many equivalent methods to define such principal 2\nbdash{}bundles, which generalise various equivalent methods of defining usual principal (1)-bundles; see~\cite{Wockel:Principal,
  Behrend-Xu, Laurent-Gengoux-Stienon-Xu:Non-abelian_gerbes, Bartels}. We refer to~\cite{Nikolaus-Waldorf:Four_gerbes} for such a review of various versions.
\end{remark}

The following theorem shows that categorified principal bundles agree with 2\nbdash{}bundles given by the simplicial method in Definition~\ref{chap3:def:higher-principal-bundles}.

\begin{theorem}\label{chap5:thm:princiapl-2-bundle-cat=simp}
  Let \(\kappa\colon E\to \trivial{N}\) be a categorified \(G\rightrightarrows \trivial{M}\) bundle with action 2\nbdash{}groupoid~\(A\). It is principal if and only if the induced morphism \(A\to \sk_0 N\) given by Lemma~\ref{chap5:lem:cat-2-bun-to-simp} is acyclic.
\end{theorem}

First, condition \(\Acyc(0)\) means that \(\kappa_0\colon E_0\to N\) is a cover. So let us assume that \(\kappa_0\colon E_0\to N\) is a cover. It remains to show that, if \((\pr_1,\action)\) is a Morita equivalence, then \(\Acyc(1)\) and \(\Acyc!(2)\) hold, and vise versa. We divide the proof into three lemmas.

\begin{lemma}\label{chap5:lem:shear-morphism-ismorita=left-principal}
  The morphism \((\pr_1,\action)\colon E\times_{J, \trivial{M},\target} G\to E\times_{N} E\) is a Morita equivalence if and only if \((\pr_1\circ J_l, J_r)\colon E_\action \to E_0\times_{N} E_0\) is a left principal \(G\)-bundle.
\end{lemma}
\begin{proof}
  The morphism \((\pr_1,\action)\) is given by the HS bibundle
  \[
  E_1\times_{\target,E_0,\pr_1\circ J_l} E_\action.
  \]
  It is a Morita equivalence if and only \(E_1\times_{\target,E_0,\pr_1\circ J_l} E_\action\to E_0\times_{N} E_0\) is a left \(E\times_{\trivial{M},\target} G\) principal bundle. Since the left \(E\)-action is given by multiplication, this \(E\)-action is principal and
  \[
  (E_1\times_{\target,E_0,\pr_1\circ J_l} E_\action)/E\cong E_\action
  \]
  as a left \(G\) bibundle. The rest follows immediately.
\end{proof}

\begin{lemma}
 If \((\pr_1,\action)\colon E\times_{J, \trivial{M},\target} G\to E\times_{N} E\) is a Morita equivalence, then \(\Acyc(1)\) and \(\Acyc!(2)\) for \(A\to \sk_0 N\) hold.
\end{lemma}
\begin{proof}
Lemma~\ref{chap5:lem:shear-morphism-ismorita=left-principal} implies \(\Acyc(1)\). It remains to prove \(\Acyc!(2)\). Since \(E_\action \to E_0\times_{N} E_0\) is a left principal \(G\)-bundle, Lemma~\ref{chap1:lem:action-over-principal-is-principal} shows that the diagram
\[
\xymatrix{
  E_\action\times_{E_0} E_\mult\ar[d]\ar[r] & (E_\action\times_{E_0} E_\mult)/G\ar[d]\\
  E_\action \ar[r]& E_0\times_{N} E_0
}
\]
is a pullback square. Combining this with Lemmas~\ref{chap5:lem:iso-of-bibundles} and~\ref{chap5:lem:pull-back}, we get
\[
  \Hom(\partial\Simp{2}\to\Simp{2}, A\to \sk_0 N)\cong \Hom(\horn{2}{0},\pi)
\]
by an argument similar to the proof of Lemma~\ref{chap5:lem:kan2}. This completes the proof.
\end{proof}

\begin{lemma}
 If \(A\to \sk_0 N\) is acyclic, then \((\pr_1,\action)\colon E\times_{J, \trivial{M},\target} G\to E\times_{N} E\) is a Morita equivalence.
\end{lemma}
\begin{proof}
By Lemma~\ref{chap5:lem:shear-morphism-ismorita=left-principal}, it suffices to show that \(E_\action \to E_0\times_{N} E_0\) is a left principal \(G\)-bundle. The condition \(\Acyc(1)\) implies that \(E_\action \to E_0\times_{N} E_0\) is a cover. The diagram
\[
  \xymatrix{
  E_\action\times_{G_0} G_1\ar[r] \ar[d] & E_\action \times_{E_0\times_{N} E_0} E_\action\ar[d] \\
  \Hom(\horn{2}{0},\pi)\ar[r]^-{\cong}& \Hom(\partial\Simp{2}\to\Simp{2}, A\to \sk_0 N)
}
\]
is a pullback square, where the vertical arrows are induced by adding a degenerate 1-simplex in \(A_1\). It follows that \(E_\action \to E_0\times_{N} E_0\) is a left principal \(G\)-bundle.
\end{proof}

\chapter{Bibundles of 2-Groupoids}\label{chap6}
\thispagestyle{empty}

In this chapter, we first define categorified bibundles between two categorified groupoids. Roughly, such a bibundle is a groupoid equipped with two categorified actions of 2\nbdash{}groupoids, and these two actions commute up to a higher isomorphism satisfying certain coherence conditions. Categorified principal bibundles between 2\nbdash{}groupoids are also defined in this fashion.

Then we establish the equivalence of the two approaches to bibundles between 2\nbdash{}groupoids by categorification and by the simplicial method. This extends the equivalence of these two approaches to 2\nbdash{}groupoids and actions of 2\nbdash{}groupoids.

We then study weak equivalence of 2-groupoids. We show that their cographs are two-sided principal bibundles. At the end of this chapter, some examples of bibundles between groupoids are given.

The following notation will be used throughout the sequel: let \(G\rightrightarrows \trivial{M}\) and \(H\rightrightarrows \trivial{N}\) be categorified groupoids in \((\Cat,\covers)\), and let \(X\) and \(Y\) be the corresponding simplicial objects in~\(\Cat\).

\section{Categorified bibundles between 2-groupoids}\label{chap6:sec:bibdl}
In this section, we categorify the notion of bibundles. Recall the notion of categorified bundles and principal bundles from Section~\ref{chap5:sec:principal-2-bundles}. Our definition is modeled after the notion of biaction of a monoidal category on a category, see~\cite{Benabou67}.

\begin{definition}\label{chap6:def:cat-bibundle}
A \emph{categorified bibundle} between \(G\rightrightarrows \trivial{M}\) and \(H\rightrightarrows \trivial{N}\) is a groupoid \(E\) in \((\Cat,\covers)\) equipped with the following data:
\[
  \xymatrix{
  G\ar@<-0.5ex>[d]_{\target}\ar@<0.5ex>[d]^{\source}\ar@{}[r]^\curvearrowright_q &E\ar[dl]^{J_l}\ar[dr]_{J_r}\ar@{}[r]^\curvearrowleft_p&H\ar@<-0.5ex>[d]_{\target}\ar@<0.5ex>[d]^{\source}\\
  \trivial{M} &  & \trivial{N}
  }
\]
\begin{itemize}
  \item a categorified left \(G\rightrightarrows \trivial{M}\) bundle \(J_r\colon E\to \trivial{N}\), which has moment morphism \(J_l\colon E\to \trivial{M}\), action morphism \(q\colon G\times_{\source,\trivial{M},J_l}E\to E\) given by an HS bibundle \(E_q\), associator \(\alpha_l\), and unitor \(\rho\);
  \item a categorified right \(H\rightrightarrows \trivial{N}\) bundle \(J_l\colon E\to \trivial{M}\), which has moment morphism \(J_r\colon E\to \trivial{N}\), action morphism \(p\colon E\times_{J_r,\trivial{N},\target}H\to E\) given by an HS bibundle \(E_p\), associator \(\alpha_r\), and unitor \(\lambda\);
\end{itemize}
These two actions commute up to a 2\nbdash{}morphism \(\alpha_{lr}\), called \emph{associator},
   \begin{equation}\label{chap6:eq:associator}
     \alpha_{lr} \colon q\circ(p\times \id) \Rightarrow p\circ(\id\times q),
   \end{equation}
   between HS morphisms \(G\times_\trivial{M} E \times_\trivial{N} H\to E\). The 2-morphisms \(\alpha_{lr}\), \(\alpha_{l}\), and \(\alpha_{r}\) satisfy two extra pentagon conditions. One of them identifies two 2-morphisms between two HS morphisms \(G\times_\trivial{M} G \times_\trivial{M} E \times_\trivial{N} H\to E\), that is, the following diagram is 2-commutative (and the other is given by symmetry):
 \begin{equation} \label{chap6:eq:pentagon}
  \begin{gathered}
  \xymatrix@=5pt{
  &&G\times_\trivial{M}G\times_\trivial{M}E \ar[dr]^-{\mult\times \id}\ar[ddd]^{\rotatebox{-90}{\(\scriptstyle \id\times q\)}}& \\
  G\times_\trivial{M} G\times_\trivial{M} E\times_\trivial{M} H\ar[urr]^{\id\times \id\times p} \ar[dr]_{\mult\times \id\times \id} \ar[ddd]_{\rotatebox{-90}{\(\scriptstyle \id\times q\times \id\)}} & & & G\times_{\trivial{M}}E \ar[ddd]^{q} \\
     & G\times_\trivial{M}E\times_\trivial{N}H \ar[urr]^(0.4){\id\times p} \ar[ddd]^{\rotatebox{-90}{\(\scriptstyle q\times \id\)}} & & & \\
     & & G\times_\trivial{M}E \ar[dr]^{q} & \\
  G\times_\trivial{M}E\times_\trivial{N} H \ar[urr]^(0.4){\id\times p} \ar[dr]_{q\times \id} & & & E\rlap{\ .} \\
     & E\times_\trivial{N}H \ar[urr]^{p} & &
  }
   \end{gathered}
  \end{equation}
\end{definition}

\begin{remark}
  As in Remark~\ref{chap5:rem:the-other-axiom}, the pentagon conditions and triangle conditions (Equation~\ref{chap5:eq:x1g} and a symmetric one) imply two other triangle conditions; one of them is the following 2\nbdash{}commutative diagram (and the other is given by symmetry):
  \begin{equation}\label{chap6:eq:1eh}
    \begin{gathered}
    \xymatrix{
                       &&E\times_{\trivial{N}} H\ar[rrd]^{p} &&\\
    G\times_{\trivial{M}} E\times_{\trivial{N}} H\ar[rru]^{q\times\id}\ar[rrd]_{\id\times p}&&
    E\times_{\trivial{N}}H\ar[ll]_-{(\unit\circ J_l, \id)\times\id}\ar[rr]^{p}\ar[u]^{=} &&E\rlap{\ ,}\\
                       &&G\times_{\trivial{M}} E\ar[rru]_{q} &&
    }
    \end{gathered}
  \end{equation}
where the 2\nbdash{}morphism on the whole square is induced by \(\alpha_{lr}\), the upper left triangle is induced by \(\lambda\), the upper right triangle is \(\id\), and the lower square is induced by \(\lambda\).
\end{remark}

\begin{remark}
   If it is clear from the context, we denote all associators by \(\alpha\), regardless whether they come from 2\nbdash{}groupoids, or actions, or \(\alpha_{lr}\). If there is no danger of confusion, we denote actions and horizontal multiplications of 2\nbdash{}groupoids by \(\mult\), or by \(\cdot\).
\end{remark}

\begin{remark}\label{chap6:rem:bibundle-2-cat}
  Replacing the 2-category \(\BUN\) by \(\GPD\), we get another sort of categorified bibundles. Recall that the cograph construction for a groupoid bibundle gives a category over~\(I\), the interval category. A similar construction for a bibundle of 2\nbdash{}groupoids gives a 2-category with a 2\nbdash{}functor to \(I\). This point of view appeared in~\cite{Benabou67} for biactions of monoidal categories.
\end{remark}

\begin{remark}
  For 2-categories \(\Cat[A]\) and \(\Cat[B]\), we can similarly define bimodules to be categories with two actions commuting up to coherent 2-morphisms. As in Remark~\ref{chap4:rem:bimodule=profunctor}, 2-category bimodules are closely related to 2\nbdash{}profunctors \(\Cat[B]^{\op}\times\Cat[A]\to \Cats\); see~\cite{Garner-Shulman}. This point of view will give us a hint how to compose 2-groupoid bibundles.
\end{remark}

Right principal bibundles are also categorified in a similar fashion.

\begin{definition}\label{chap6:def:cat-prinbibd}
Let \(E\) be a categorified bibundle between \(G \rightrightarrows \trivial{M}\) and \(H\rightrightarrows \trivial{N}\). We say that \(E\) is \emph{right principal} if the right \(H\rightrightarrows \trivial{N}\) bundle is principal. Left principal bibundles are defined similarly. Two-sided principal bibundles are also called \emph{Morita bibundles}.
\end{definition}

By Definition~\ref{chap5:def:2-bundle-principal}, being right principal means that the left moment morphism \(J_l \colon E\to \trivial{M}\) is given by a cover \(E_0\to M\) and that the shear morphism
\begin{equation}\label{chap6:eq:r-principal-shear-is-morita}
  (\pr_1, p)\colon E\times_{J_r,\trivial{N},\target}H \to E\times_{\trivial{M}}E, \quad
  (e, h)\mapsto (e, e\cdot h),
\end{equation}
is a Morita equivalence.

The following theorem establishes that two approaches to bibundles of 2\nbdash{}groupoids are equivalent. The proof given in the next two sections will largely follow the same lines as in the previous chapter.

\begin{theorem}\label{chap6:thm:bibundle-cat=simplicial}
 Categorified bibundles between \(G\rightrightarrows \trivial{M}\) and \(H\rightrightarrows \trivial{N}\) are in a one-to-one correspondence with bibundles between \(X\) and \(Y\) given by simplicial objects. Moreover, right principal \textup(Morita\textup) categorified bibundles correspond to right principal \textup(Morita\textup) bibundles given by simplicial objects.
\end{theorem}

Recall from Definition~\ref{chap4:def:bibundles=inner-over-I} that a bibundle between \(X\) and \(Y\) is a simplicial object \(\pi\colon \Gamma\to \Simp{1}\) with the two ends given by \(X\) and~\(Y\) such that the morphism \(\pi\) satisfies \(\Kan(m, k)\) for \(1<k<m\), \(\Kan(m,0)[i,j]\) for \(i\ge 2\), and \(\Kan(m, m)[i,j]\) for \(j\ge 2\). Such a bibundle is right principal if, in addition, \(\pi\) satisfies left \(\Kan(m, 0)\) for \(m\ge 1\). Two-sided principal bibundles satisfy \(\Kan(m,k)\) for \(m\ge 1\) and \(0\le k\le m\). Finally, when \(m>n\) the unique version of all these Kan conditions hold.

\begin{remark}
Theorem~\ref{chap6:thm:bibundle-cat=simplicial} is not surprising in view of Theorems~\ref{chap5:thm:2-groupoid-cat=simp} and~\ref{chap5:thm:nerve=2-category}. Consider the categorified bibundle as in Remark~\ref{chap6:rem:bibundle-2-cat}. The associated simplicial object given below is the geometric nerve of the 2\nbdash{}category corresponding to this bibundle.
\end{remark}

\section{From categorification to simplicial picture}\label{chap6:sec:cat-2-simp}

In this section, we show one direction of Theorem~\ref{chap6:thm:bibundle-cat=simplicial}. Given a categorified bibundle~\(E\) between 2\nbdash{}groupoids \(G\rightrightarrows \trivial{M}\) and \(H\rightrightarrows \trivial{N}\) as in Definition~\ref{chap6:def:cat-bibundle}, we construct a simplicial object \(\Gamma\) in \(\Cat\) with a morphism \(\Gamma\to \Simp{1}\) satisfying the desired Kan conditions. If \(E\) is right principal, we show that the corresponding \(\Gamma\to \Simp{1}\) satisfies the desired left Kan conditions.

Recall from Section~\ref{chap4:sec:aug-simplicial-bisimplicial} that a simplicial object \(\Gamma\) over \(\Simp{1}\) can be viewed as a simplicial object colored by 0 (white) and 1 (black). Denote by \(\Gamma_{i,j}\) the subspace of \(\Gamma_{i+j+1}\) with \(i+1\) white and \(j+1\) black vertexes, and let \(\Gamma_{-1,-1}\) be the terminal object. Then \(\Gamma_{i,j}\) is the corresponding augmented bisimplicial object.

Since \(X\) is the simplicial object corresponding to the categorified groupoid \(G\rightrightarrows \trivial{M}\), \(X_0=M\), \(X_1=G_0\), and \(X_2=E_\mult^G\), the HS bibundle giving the multiplication.

\subsection{The dimensions 0, 1, 2}\label{chap6:ssec:levels012}
We construct \(\Gamma\) as an augmented bisimplicial object, so the morphism \(\Gamma\to \Simp{1}\) is built into the construction by Proposition~\ref{chap4:prop:abisimpC=asimpC-I}.

\subsubsection{The construction of the spaces}

The desired simplicial object \(\Gamma\) in dimensions \(m=0, 1, 2\) is displayed as follows:
\[
\xymatrix{
  E_\mult^G\ar@3{->}[dr] & &  E_q\ar@2{->}[dr]^{\face_0}_{\face_1}\ar[dl]^{\face_2} &&  E_p\ar[dr]_{\face_0}\ar@2{->}[dl]_{\face_2}^{\face_1} && E_\mult^H\ar@3{->}[dl] \\
   &  G_0\ar@2{->}[dr]^{\face_0}_{\face_1}&&  E_0\ar[dr]_{J_r}\ar[dl]^{J_l}   &&  H_0\ar@2{->}[dl]^{\face_0}_{\face_1}    \\
   & &  M &&   N\rlap{\ .}
  }
\]
First, let \(\Gamma_{i,-1}=X_i\) and \(\Gamma_{-1,i}=Y_i\) for \(i=0,1,2\). Let
\[
  \Gamma_{0,0}=E_0,\quad \Gamma_{1,0}=E_q,\quad \Gamma_{0,1}= E_p,
\]
where \(E_q\) and \(E_p\) are the HS bibundles giving the action morphisms \(q\) and \(p\), respectively.

The face maps on \(\Gamma_{i,-1}\) and \(\Gamma_{-1, i}\) for \(i=0,1,2\) are those of \(X\) and \(Y\), respectively.
The face maps with domain \(E_0\) are defined by
\[
  \face_0=J_r|_{E_0}\colon E_0\to N,\qquad \face_1=J_l|_{E_0}\colon E_0\to M.
\]
The face maps with domain \(E_p\) are defined by
\begin{gather*}
  \face_0=\pr_2\circ J_l^p\colon E_p\to H_0\\
  \face_1=J_r^p,\quad \face_2=\pr_1\circ J_l^p \colon E_p\to E_0,
\end{gather*}
where \(J_r^p\) and \(J_l^p\) are the moment maps of the bibundle \(E_p\). Similarly, the face maps with domain \(E_q\) are defined by
\begin{gather*}
  \face_2=\pr_1\circ J_r^q\colon E_q\to G_0,\\
  \face_1=J_l^q,\quad \face_0=\pr_2\circ J_r^q \colon E_q\to E_0.
\end{gather*}

The degeneracy maps on \(\Gamma_{i,-1}\) and \(\Gamma_{-1, i}\) for \(i=0,1,2\) are those of \(X\) and \(Y\), respectively. The degeneracy map \(\de_0\colon E_0\to E_q\) is induced by the unitor \(\lambda\). In fact, the 2-morphism \(\lambda\colon q|_{\trivial{M}\times_{\trivial{M}}E } \Rightarrow\id_{E}\) gives an isomorphism of HS bibundles
\[
  \lambda \colon E_q\supset (J_l^q)^{-1}(\trivial{M}\times_{\trivial{M}}E)\to E_1.
\]
Let \(\de_0=\lambda^{-1}\circ \unit_{E}\). Similarly, we define \(\de_1\colon E_0\to E_p\) using the unitor~\(\rho\).

\subsubsection{Simplicial identities}

We need to verify that \(\Gamma_{\le 2}\) constructed above satisfies the simplicial identities.

Consider the simplicial identities with domain \(E_p\); those with domain \(E_q\) can be shown similarly. The identity \(\face_0\circ \face_2=\face_1\circ \face_0\) holds by construction. The identity
\(\face_1\circ \face_2=\face_1\circ \face_1\) follows from Lemma~\ref{chap5:lem:bundles-triangle-to-M} and
\[
J_r\circ p=J_r\circ \pr_1 \colon E\times_{\trivial{N},\target} H\to \trivial{M}.
\]
Similarly, \(J_l\circ p=\source\circ \pr_2\) implies \(\face_0\circ \face_1=\face_0\circ \face_0\).

Consider the simplicial identities with domain \(E_0\). The identity \(\face_2\circ \de_0=\de_0\circ \face_1\) holds by construction. By definition, we have
\begin{gather*}
  \face_0\circ \de_0=\face_0\circ (J_l^q)^{-1}\circ \unit_E=\source\circ \unit_E=\id, \\
  \face_1\circ \de_0=J_r^q\circ (J_l^q)^{-1}\circ \unit_E=\target\circ \unit_E=\id,
\end{gather*}
where we used that HS bibundle isomorphisms respect moment maps. The other simplicial identities hold by symmetry.

It is clear that there is a natural map from \(\Gamma_{\le 2}\) to \(\tr_2\Simp{1}\) (see Section~\ref{chap2:sec:sekeleton-and-coskeleton}).

\subsubsection{Kan conditions}

Although \(\Gamma_m\) is yet not defined for \(m>2\), all presheaves \(\Hom(\Horn{m}{k}[i, j],\Gamma)\) are well-defined for \(0\le k\le m\le 3\) and \(i, j\ge 0\) with \(i+j=m+1\). Thus it is meaningful to verify some Kan conditions.

Recall the colored Kan conditions from Chapter~\ref{chap4}. There are four cases for \(\Kan(2,1)\), two of which are given by \(\Kan(2,1)\) for \(X\) and \(Y\). The other two cases are \(\Kan(2,1)[2,1]\) and \(\Kan(2,1)[1,2]\), which say that two maps below are covers:
\[
  E_q\to G_0\times_{M}E_0,\quad E_p\to E_0\times_{N} H_0.
\]
This holds because they are the left moment maps of the HS bibundles \(E_q\) and \(E_p\); see Figure~\ref{chap6:fig:Kan(2,1)-21-12}.
\begin{figure}[htbp]
\centering
\begin{tikzpicture}
  [scale=0.7,
  >=latex',
  mydot/.style={draw,circle,inner sep=1.5pt},
  every label/.style={scale=0.6}
  ]
  \node[mydot,label=120+90:$0$] (P0) at (120+90:1cm){};
  \node[mydot,label=90:$1$]     (P1) at (90:1cm){};
  \node[mydot,fill=black,label=90-120:$2$] (P2) at (90-120:1cm){};
  \draw[<-]
    (P0) edge (P1)
    (P1) edge (P2)
    (P0) edge[dashed] (P2);
  \begin{scope}[xshift=4cm]
    \node[mydot,label=120+90:$0$] (Q0) at (120+90:1cm){};
    \node[mydot,fill=black,label=90:$1$]    (Q1) at (90:1cm){};
    \node[mydot,fill=black,label=90-120:$2$] (Q2) at (90-120:1cm){};
    \draw[<-]
     (Q0) edge (Q1)
     (Q1) edge (Q2)
     (Q0) edge[dashed] (Q2);
  \end{scope}
\end{tikzpicture}
\caption{\(\Kan(2,1)[2,1]\) and \(\Kan(2,1)[1,2]\)}\label{chap6:fig:Kan(2,1)-21-12}
\end{figure}
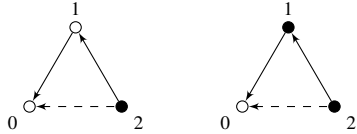

Additionally, if the bibundle \(E\) is right principal, then \(\Kan(1, 0)\) holds. It suffices to consider \(\Kan(1,0)[1,1]\), which says that \(E_0\to M\) is a cover. This follows from Definition~\ref{chap6:def:cat-prinbibd}.

\begin{lemma}
  The left \(E\)-action on \(E_q\) gives a principal \(E\)-bundle
  \[
  (\pr_1\circ J^q_l, J^q_r)\colon E_q\to G_0\times_{\target,N} E_0;
  \]
  Similarly, the left \(E\)-action on \(E_p\) gives a principal \(E\)-bundle
  \[
  (J^p_r,\pr_2\circ J^p_l)\colon E_p\to E_0\times_{M,\source} H_0;
  \]
  If the bibundle \(E\) is right principal, then the left \(H\)-action on \(E_p\) gives a principal \(H\)-bundle
  \[
  (\pr_1\circ J^p_l, J^p_r)\colon E_p\to E_0\times_{M} E_0.
  \]
\end{lemma}
\begin{proof}
This is the same as Lemmas~\ref{chap5:lem:left-E-is-principal} and~\ref{chap5:lem:shear-morphism-ismorita=left-principal}.
\end{proof}

Consequently, the colored outer Kan conditions \(\Kan(2,0)[2,1]\) and \(\Kan(2,2)[1,2]\) hold. If \(E\) is right principal, then \(\Kan(2,0)[1,2]\) holds; see Figure~\ref{chap6:fig:special-Kan(2,0)[2,1]-etc}.
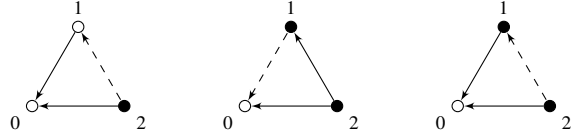
\begin{figure}[htbp]
\centering
\begin{tikzpicture}
  [scale=0.7,
  >=latex',
  mydot/.style={draw,circle,inner sep=1.5pt},
  every label/.style={scale=0.6}
  ]
  \node[mydot,label=90+120:$0$]            (P0) at (90+120:1cm){};
  \node[mydot,label=90:$1$]                (P1) at (90:1cm){};
  \node[mydot,fill=black,label=90-120:$2$] (P2) at (90-120:1cm){};
  \draw[<-]
     (P0) edge (P1)
          edge (P2)
     (P1) edge[dashed] (P2);
  \begin{scope}[xshift=4cm]
    \node[mydot,label=90+120:$0$]            (Q0) at (90+120:1cm){};
    \node[mydot,fill=black,label=90:$1$]     (Q1) at (90:1cm){};
    \node[mydot,fill=black,label=90-120:$2$] (Q2) at (90-120:1cm){};
    \draw[<-]
    (Q0) edge[dashed] (Q1)
         edge (Q2)
    (Q1) edge (Q2);
  \end{scope}
  \begin{scope}[xshift=8cm]
    \node[mydot,label=90+120:$0$]            (Q0) at (90+120:1cm){};
    \node[mydot,fill=black,label=90:$1$]     (Q1) at (90:1cm){};
    \node[mydot,fill=black,label=90-120:$2$] (Q2) at (90-120:1cm){};
    \draw[<-]
    (Q0) edge (Q1)
         edge (Q2)
    (Q1) edge[dashed] (Q2);
  \end{scope}
\end{tikzpicture}
\caption{\(\Kan(2,0)[2,1], \Kan(2,2)[1,2]\), and \(\Kan(2,0)[1,2]\)}\label{chap6:fig:special-Kan(2,0)[2,1]-etc}
\end{figure}

\subsection{Dimension 3}\label{chap6:ssec:level3}

First, let \(\Gamma_{-1,3}=X_3\) and \(\Gamma_{3,-1}=Y_3\). For the remaining cases, we use the following lemma.

\begin{lemma}\label{chap6:lem:level3:isos}
There are natural isomorphisms of representable presheaves
\begin{gather}
  \Hom(\Horn{3}{0}[3,1],\Gamma)\cong\Hom(\Horn{3}{1}[3,1],\Gamma)\cong \Hom(\Horn{3}{2}[3,1],\Gamma)\label{chap6:eq:Gamma-20},\\
  \Hom(\Horn{3}{1}[1,3],\Gamma)\cong\Hom(\Horn{3}{2}[1,3],\Gamma)\cong \Hom(\Horn{3}{3}[1,3],\Gamma)\label{chap6:eq:Gamma-02},\\
  \Hom(\Horn{3}{0}[2,2],\Gamma)\cong\Hom(\Horn{3}{1}[2,2],\Gamma)
  \cong\Hom(\Horn{3}{2}[2,2],\Gamma)\cong \Hom(\Horn{3}{3}[2,2],\Gamma).\label{chap6:eq:Gamma-11}
\end{gather}
If the bibundle \(E\) is right principal, then
\begin{equation}\label{chap6:eq:Gamma-02-principal}
  \Hom(\Horn{3}{0}[1,3],\Gamma)\cong \Hom(\Horn{3}{1}[1,3],\Gamma)
  \cong\Hom(\Horn{3}{2}[1,3],\Gamma)\cong \Hom(\Horn{3}{3}[1,3],\Gamma).
\end{equation}
\end{lemma}

Let \(\Gamma_{0,2}, \Gamma_{1,1}\), and \(\Gamma_{2,0}\) be given by \eqref{chap6:eq:Gamma-02}, \eqref{chap6:eq:Gamma-11}, and~\eqref{chap6:eq:Gamma-20}, respectively. We shall verify all the required conditions for a simplicial object up to dimension~\(3\).

First, Remarks~\ref{chap4:rem:low-inner-kan-imply-repr} and~\ref{chap4:rem:low-outer-kan-imply-repr} imply that all the presheaves in Lemma~\ref{chap6:lem:level3:isos} are representable. Since~\eqref{chap6:eq:Gamma-02} and \eqref{chap6:eq:Gamma-20} are essentially the same as Lemma~\ref{chap5:lem:kan2}, we only prove \eqref{chap6:eq:Gamma-11} and~\eqref{chap6:eq:Gamma-02-principal}.

\begin{lemma}\label{chap6:lem:HS-bibundle-iso-alpha}
There is a natural isomorphism of HS bibundles from \( G\times_{\trivial{M}} E\times_{\trivial{N}}H\) to \(E\),
\[
  (E_q\times_{E_0} E_p)/E\overset{\alpha} \cong (E_p\times_{E_0}E_q)/E,
\]
where the omitted maps can be read from Figure~\ref{chap6:fig:associator-implies-iso}, in which snake lines stand for quotients.
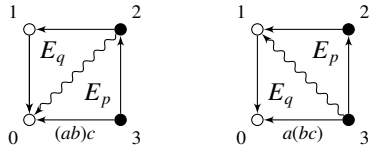
\begin{figure}[htbp]
  \centering
\begin{tikzpicture}
    [scale=1.2,
    <-,
    >=latex',
    mydot/.style={draw,circle,inner sep=1.5pt},
    every label/.style={scale=0.6}]
  \node[mydot,label=-135:$0$]           (P0) at (0,0){};
  \node[mydot,label=135:$1$]            (P1) at (0,1){};
  \node[mydot,fill=black,label=45:$2$]  (P2) at (1,1){};
  \node[mydot,fill=black,label=-45:$3$] (P3) at (1,0){};
  \draw
  (P0) edge (P1)
  (P1) edge (P2)
  (P2) edge node(ML){} (P3)
  (P0) edge (P1)
       edge node[below,scale=0.6]{$(ab)c$}(P3)
  (P2) edge[snake arrow]
  node[scale=0.8,above left]{\(E_q\)}
  node[scale=0.8,below right]{\(E_p\)} (P0);
  \begin{scope}[xshift=2.5cm]
    \node[mydot,label=-135:$0$] (Q0) at (0,0){};
    \node[mydot,label=135:$1$]  (Q1) at (0,1){};
    \node[mydot,fill=black,label=45:$2$]   (Q2) at (1,1){};
    \node[mydot,fill=black,label=-45:$3$]  (Q3) at (1,0){};
  \draw
  (Q0) edge node(MR){}(Q1)
  (Q1) edge (Q2)
  (Q2) edge (Q3)
  (Q0) edge node[below,scale=0.6]{$a(bc)$}(Q3)
  (Q3) edge[snake arrow]
  node[scale=0.8,above right]{\(E_p\)}
  node[scale=0.8,below left]{\(E_q\)} (Q1);
  \end{scope}
\end{tikzpicture}
\caption{Associator implies an isomorphism of bibundles}\label{chap6:fig:associator-implies-iso}
\end{figure}
\end{lemma}
\begin{proof}
The associator 2-morphism
\[
  \alpha\colon p\circ(q\times\id)\Rightarrow q\circ(\id\times p)\colon  G\times_{\trivial{M}} E\times_{\trivial{N}}H\to E
\]
implies an isomorphism of HS bibundles
\[
( (E_q\times_{N} H_1)\times_{E_0\times_{N} G_0} E_p)/(E\times_{\trivial{N}}H)\xrightarrow{\alpha}
((G_1\times_{M} E_p)\times_{G_0\times_{M} E_0} E_q)/(G\times_{\trivial{M}}E).
\]
Since \(H\) acts by multiplication, we have
\[
  ( (E_q\times_{N} H_1)\times_{E_0\times_{N} G_0} E_p)/(E\times_{\trivial{N}}H)\cong
  (E_q\times_{E_0} E_p)/E.
\]
Similarly,
\[
  ((G_1\times_{M} E_p)\times_{G_0\times_{M} E_0} E_q)/(G\times_{\trivial{M}}E)\cong
  (E_p\times_{E_0}E_q)/E,
\]
and this proves the lemma.
\end{proof}

The proof shows that the \(E\)-actions on \(E_q\times_{E_0} E_p\) and \(E_p\times_{E_0}E_q\) are principal. The following lemma is a consequence of Lemma~\ref{chap1:lem:principal-equivariant-map-is-iso}.
\begin{lemma}\label{chap6:lem:pullback1}
The diagram
\[
  \xymatrix{
  E_q\times_{E_0} E_p\ar[d]\ar[r] & (E_q\times_{E_0} E_p)/E\ar[d] \\
  E_q \ar[r]&                   G_0\times_{M}E_0
  }
\]
is a pullback square in \(\Cat\). We illustrate this pullback square by Figure~\ref{chap6:fig:pullback1}, where the space given by the left picture is the pullback of the two spaces given by the middle and the right pictures over an obvious space.
\begin{figure}[htbp]
  \centering
  \begin{tikzpicture}
  [<-,
    >=latex',
    mydot/.style={draw,circle,inner sep=1.5pt},
    every label/.style={scale=0.6},
    scale=0.8
  ]
  \begin{scope}
  \node[mydot,label=-135:$0$]           (P0) at (0,0){};
  \node[mydot,label=135:$1$]            (P1) at (0,1){};
  \node[mydot,fill=black,label=45:$2$]  (P2) at (1,1){};
  \node[mydot,fill=black,label=-45:$3$] (P3) at (1,0){};
  \draw
  (P0) edge (P1)
  (P1) edge (P2)
  (P2) edge (P3)
  (P0) edge (P1)
       edge (P3)
  (P2) edge[->] (P0);
  \end{scope}
  \begin{scope}[xshift=2.5cm]
  \node[mydot,label=-135:$0$]           (Q0) at (0,0){};
  \node[mydot,label=135:$1$]            (Q1) at (0,1){};
  \node[mydot,fill=black,label=45:$2$]  (Q2) at (1,1){};
  \node[mydot,fill=black,label=-45:$3$] (Q3) at (1,0){};
  \draw
  (Q0) edge (Q1)
  (Q1) edge (Q2)
  (Q2) edge (Q3)
  (Q0) edge (Q1)
       edge (Q3)
  (Q2) edge[snake arrow] (Q0);
  \end{scope}
  \begin{scope}[xshift=5cm]
  \node[mydot,label=-135:$0$]           (R0) at (0,0){};
  \node[mydot,label=135:$1$]            (R1) at (0,1){};
  \node[mydot,fill=black,label=45:$2$]  (R2) at (1,1){};
  \draw
  (R0) edge (R1)
  (R1) edge (R2)
  (R0) edge (R1)
  (R2) edge[->] (R0);
  \end{scope}
  \end{tikzpicture}
  \caption{An illustration of a pullback diagram}\label{chap6:fig:pullback1}
\end{figure}
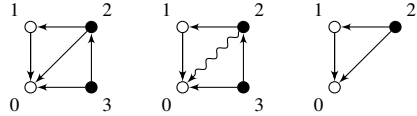
\end{lemma}

\begin{remark}\label{chap6:rem:pullback3}
  This lemma follows because \(E_q\times_{E_0} E_p\to E_q\) is an equivariant map of principal \(E\) bundles. Thus, there are three other pullback diagrams, illustrated by Figure~\ref{chap6:fig:pullback3}, where every row gives a pullback diagram.
  \begin{figure}[htbp]
  \centering
  \begin{tikzpicture}
  [<-,
    >=latex',
    mydot/.style={draw,circle,inner sep=1.5pt},
    every label/.style={scale=0.6},
    scale=0.8
  ]
  \begin{scope}
  \node[mydot,label=-135:$0$]           (P0) at (0,0){};
  \node[mydot,label=135:$1$]            (P1) at (0,1){};
  \node[mydot,fill=black,label=45:$2$]  (P2) at (1,1){};
  \node[mydot,fill=black,label=-45:$3$] (P3) at (1,0){};
  \draw
  (P0) edge (P1)
  (P1) edge (P2)
  (P2) edge (P3)
  (P0) edge (P1)
       edge (P3)
  (P2) edge[->] (P0);
  \end{scope}
  \begin{scope}[xshift=2.5cm]
  \node[mydot,label=-135:$0$]           (Q0) at (0,0){};
  \node[mydot,label=135:$1$]            (Q1) at (0,1){};
  \node[mydot,fill=black,label=45:$2$]  (Q2) at (1,1){};
  \node[mydot,fill=black,label=-45:$3$] (Q3) at (1,0){};
  \draw
  (Q0) edge (Q1)
  (Q1) edge (Q2)
  (Q2) edge (Q3)
  (Q0) edge (Q1)
       edge (Q3)
  (Q2) edge[snake arrow] (Q0);
  \end{scope}
  \begin{scope}[xshift=5cm]
  \node[mydot,label=-135:$0$]           (R0) at (0,0){};
  \node[mydot,fill=black,label=45:$2$]  (R2) at (1,1){};
  \node[mydot,fill=black,label=-45:$3$] (R3) at (1,0){};
  \draw[->]
  (R3) edge (R0)
       edge (R2)
  (R2) edge (R0);
  \end{scope}
  \begin{scope}[yshift=-2cm]
  \node[mydot,label=-135:$0$]           (P0) at (0,0){};
  \node[mydot,label=135:$1$]            (P1) at (0,1){};
  \node[mydot,fill=black,label=45:$2$]  (P2) at (1,1){};
  \node[mydot,fill=black,label=-45:$3$] (P3) at (1,0){};
  \draw
  (P0) edge (P1)
  (P1) edge (P2)
       edge (P3)
  (P2) edge (P3)
  (P0) edge (P1)
       edge (P3);
  \end{scope}
  \begin{scope}[xshift=2.5cm,yshift=-2cm]
  \node[mydot,label=-135:$0$]           (Q0) at (0,0){};
  \node[mydot,label=135:$1$]            (Q1) at (0,1){};
  \node[mydot,fill=black,label=45:$2$]  (Q2) at (1,1){};
  \node[mydot,fill=black,label=-45:$3$] (Q3) at (1,0){};
  \draw
  (Q0) edge (Q1)
  (Q1) edge (Q2)
  (Q2) edge (Q3)
  (Q0) edge (Q1)
       edge (Q3)
  (Q3) edge[snake arrow] (Q1);
  \end{scope}
  \begin{scope}[xshift=5cm,yshift=-2cm]
  \node[mydot,label=-135:$0$]           (R0) at (0,0){};
  \node[mydot,label=135:$1$]            (R1) at (0,1){};
  \node[mydot,fill=black,label=-45:$3$] (R3) at (1,0){};
  \draw
  (R0) edge (R1)
  (R1) edge (R3)
  (R0) edge (R1)
  (R3) edge[->] (R0);
  \end{scope}
  \begin{scope}[yshift=-4cm]
  \node[mydot,label=-135:$0$]           (P0) at (0,0){};
  \node[mydot,label=135:$1$]            (P1) at (0,1){};
  \node[mydot,fill=black,label=45:$2$]  (P2) at (1,1){};
  \node[mydot,fill=black,label=-45:$3$] (P3) at (1,0){};
  \draw
  (P0) edge (P1)
  (P1) edge (P2)
       edge (P3)
  (P2) edge (P3)
  (P0) edge (P1)
       edge (P3);
  \end{scope}
  \begin{scope}[xshift=2.5cm,yshift=-4cm]
  \node[mydot,label=-135:$0$]           (Q0) at (0,0){};
  \node[mydot,label=135:$1$]            (Q1) at (0,1){};
  \node[mydot,fill=black,label=45:$2$]  (Q2) at (1,1){};
  \node[mydot,fill=black,label=-45:$3$] (Q3) at (1,0){};
  \draw
  (Q0) edge (Q1)
  (Q1) edge (Q2)
  (Q2) edge (Q3)
  (Q0) edge (Q1)
       edge (Q3)
  (Q3) edge[snake arrow] (Q1);
  \end{scope}
  \begin{scope}[xshift=5cm,yshift=-4cm]
  \node[mydot,label=135:$1$]            (R1) at (0,1){};
  \node[mydot,fill=black,label=45:$2$]  (R2) at (1,1){};
  \node[mydot,fill=black,label=-45:$3$] (R3) at (1,0){};
  \draw
  (R1) edge (R2)
       edge (R3)
  (R2) edge (R3);
  \end{scope}
  \end{tikzpicture}
  \caption{An illustration of three pullback diagrams}\label{chap6:fig:pullback3}
\end{figure}
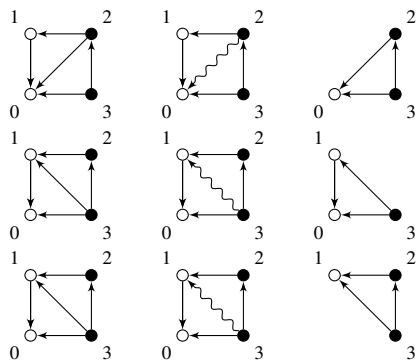
\end{remark}

\begin{proof}[Proof of Lemma~\ref{chap6:lem:level3:isos}]
Lemma~\ref{chap6:lem:pullback1} and Remark~\ref{chap6:rem:pullback3} imply that \(\Hom(\Horn{3}{0}[2,2],\Gamma)\) is the limit of the three spaces (over obvious spaces) in the first row of Figure~\ref{chap6:fig:hom-horn-30-and-31-2-2-to-Gamma}.
\begin{figure}[htbp]
  \centering
  \begin{tikzpicture}
  [<-,
    >=latex',
    mydot/.style={draw,circle,inner sep=1.5pt},
    every label/.style={scale=0.6},
    scale=0.8
  ]
  \begin{scope}
  \node[mydot,label=-135:$0$]           (P0) at (0,0){};
  \node[mydot,label=135:$1$]            (P1) at (0,1){};
  \node[mydot,fill=black,label=-45:$3$] (P3) at (1,0){};
  \draw
  (P0) edge (P1)
       edge (P3)
  (P1) edge (P3);
  \end{scope}
  \begin{scope}[xshift=2.5cm]
  \node[mydot,label=-135:$0$]           (Q0) at (0,0){};
  \node[mydot,label=135:$1$]            (Q1) at (0,1){};
  \node[mydot,fill=black,label=45:$2$]  (Q2) at (1,1){};
  \node[mydot,fill=black,label=-45:$3$] (Q3) at (1,0){};
  \draw
  (Q0) edge (Q1)
  (Q1) edge (Q2)
  (Q2) edge (Q3)
  (Q0) edge (Q1)
       edge (Q3)
  (Q2) edge[snake arrow] (Q0);
  \end{scope}
  \begin{scope}[xshift=5cm]
  \node[mydot,label=-135:$0$]           (R0) at (0,0){};
  \node[mydot,label=135:$1$]            (R1) at (0,1){};
  \node[mydot,fill=black,label=45:$2$]  (R2) at (1,1){};
  \draw[->]
  (R1) edge (R0)
  (R2) edge (R0)
       edge (R1);
  \end{scope}
  \begin{scope}[yshift=-2cm]
  \node[mydot,label=-135:$0$]           (P0) at (0,0){};
  \node[mydot,label=135:$1$]            (P1) at (0,1){};
  \node[mydot,fill=black,label=-45:$3$] (P3) at (1,0){};
  \draw
  (P0) edge (P1)
       edge (P3)
  (P1) edge (P3);
  \end{scope}
  \begin{scope}[xshift=2.5cm,yshift=-2cm]
  \node[mydot,label=-135:$0$]           (Q0) at (0,0){};
  \node[mydot,label=135:$1$]            (Q1) at (0,1){};
  \node[mydot,fill=black,label=45:$2$]  (Q2) at (1,1){};
  \node[mydot,fill=black,label=-45:$3$] (Q3) at (1,0){};
  \draw
  (Q0) edge (Q1)
  (Q1) edge (Q2)
  (Q2) edge (Q3)
  (Q0) edge (Q1)
       edge (Q3)
  (Q3) edge[snake arrow] (Q1);
  \end{scope}
  \begin{scope}[xshift=5cm,yshift=-2cm]
  \node[mydot,label=-135:$0$]           (R0) at (0,0){};
  \node[mydot,label=135:$1$]            (R1) at (0,1){};
  \node[mydot,fill=black,label=45:$2$]  (R2) at (1,1){};
  \draw[->]
  (R1) edge (R0)
  (R2) edge (R0)
       edge (R1);
  \end{scope}
  \end{tikzpicture}
  \caption{\(\Hom(\Horn{3}{0}[2,2],\Gamma)\) and \(\Hom(\Horn{3}{1}[2,2],\Gamma)\)}\label{chap6:fig:hom-horn-30-and-31-2-2-to-Gamma}
\end{figure}

Similarly, the second row of Figure~\ref{chap6:fig:hom-horn-30-and-31-2-2-to-Gamma} gives the space \(\Hom(\Horn{3}{1}[2,2],\Gamma)\). Lemma~\ref{chap6:lem:HS-bibundle-iso-alpha} gives \(\Hom(\Horn{3}{0}[2,2],\Gamma)\cong\Hom(\Horn{3}{1}[2,2],\Gamma)\). The remaining isomorphisms in~\eqref{chap6:eq:Gamma-11} follow similarly.

If the bibundle \(E\) is right principal, then we have similar pullback diagrams as in Lemma~\ref{chap6:fig:pullback1}, and~\eqref{chap6:eq:Gamma-02-principal} follows similarly.
\end{proof}

Lemma~\ref{chap6:lem:level3:isos} implies that \(\Hom(\Horn{3}{1},\Gamma)\cong \Hom(\Horn{3}{2},\Gamma)\). This isomorphism gives two  3\nbdash{}multiplication maps
\begin{equation}\label{chap6:eq:3-multiplication}
m_i\colon \Hom(\Horn{3}{i},\Gamma)\to\Hom(\Simp{2},\Gamma),\quad \text{for } i=1, 2,
\end{equation}
which determine the remaining face of a horn of shape \(\Horn{3}{1}\) or \(\Horn{3}{2}\) in \(\Gamma\). In fact, considering the coloring carefully, we have more 3\nbdash{}multiplication maps; but these two maps are sufficient for our purposes.

\begin{lemma}\label{chap6:lem:3mult}
The \(3\)-multiplications are compatible with face and degeneracy maps in the following sense: for \(\eta\in \Gamma_2\), we have
\begin{gather*}
  \eta=m_1(\eta,\source_0\circ \face_1(\eta),\de_0\circ \face_2(\eta)), \\
  (\text{which is equivalent to } \de_0\circ\face_1 (\eta)=m_2(\eta,\eta, \source_0\circ \face_2(\eta)));\\
  \eta=m_1(\de_0\circ \face_0(\eta),\eta,\de_1\circ \face_2(\eta)), \\
  (\text{which is equivalent to } \eta=m_2(\de_0\circ \face_0(\eta),\eta,\de_1\circ \face_2(\eta)));\\
  \de_1\circ \face_1(\eta)=m_1(\de_1\circ \face_0(\eta),\eta,\eta),\\
  (\text{which is equivalent to } \eta=m_2(\de_1\circ \face_0(\eta),\de_1\circ \face_1(\eta),\eta)).
\end{gather*}
\end{lemma}
\begin{proof}
This follows from the triangle conditions. The proof is similar to \cite[the coherence conditions, p. 4127]{Zhu:ngpd}; see also Lemma~\ref{chap5:lem:triangle-condition-imply-3mult-face-de}.
\end{proof}

Now, we set
\[
  \Gamma_3=\Hom(\Horn{3}{1},\Gamma)\cong \Hom(\Horn{3}{2},\Gamma);
\]
that is, for \(i+j=2\) and \(i, j\ge -1\), we have
\[
 \Gamma_{i, j}=\Hom(\Horn{3}{1}[i+1, j+1],\Gamma)\cong \Hom(\Horn{3}{2}[i+1, j+1],\Gamma).
\]
The face maps \(\face_i\colon \Gamma_3\to \Gamma_2\) for \(0\le i\le 3\) are natural to define. Using Lemma~\ref{chap6:lem:3mult}, we can define degeneracy maps \(\de_i\colon \Gamma_2\to \Gamma_3\) for \(0\le i\le 2\) by
\begin{gather*}
  \Gamma_2\to \Hom(\Horn{3}{1},\Gamma),\\
  \de_0\colon \eta\mapsto(\eta,\source_0\circ \face_1(\eta),\de_0\circ \face_2(\eta)), \\
  \de_1\colon \eta\mapsto (\de_0\circ \face_0(\eta),\eta,\de_1\circ \face_2(\eta)), \\
  \de_2\colon \eta\mapsto (\de_1\circ \face_0(\eta),\eta,\eta).
\end{gather*}
The simplicial identities
\[
  \face^2_i \face^3_j=\face^2_{j-1}\face^3_i, \quad i<j
\]
hold by definition. All the simplicial identities involving degeneracy maps \(\de^2_i\) follow from the construction and  Lemma~\ref{chap6:lem:3mult}.

All desired colored Kan conditions for \(\Gamma\) hold by the definition and Lemma~\ref{chap6:lem:level3:isos}. This completes dimension \(3\) of \(\Gamma\).

There is a natural map \(\Gamma_{\le 3}\to \tr_3 \Simp{1}\). We know from the lemma below that \(\Kan!(3, k)\) for \(\Gamma\to \Simp{1}\) is equivalent to the same condition for \(\Gamma\). This property also holds for higher dimensions, so we can safely omit the morphism \(\Gamma\to \Simp{1}\) in higher dimensions.

\begin{lemma}
Let \(\Gamma\) be a simplicial object in \(\Cat\) with a morphism \(\Gamma\to\Simp{1}\). For \((m,k)=(2,1)\) or \(m>2\), the conditions \(\Kan(m, k)\) for\/ \(\Gamma\to \Simp{1}\) and for \(\Gamma\) are equivalent.
\end{lemma}
\begin{proof}
If \((m,i)=(2,1)\) or \(m>2\), then \(\Hom(\Simp{m},\Simp{1})=\Hom(\Horn{m}{i},\Simp{1})\). Consider the pullback diagram
\[
  \xymatrix{
  \Hom(\Horn{m}{i}\to\Simp{1},\Gamma\to\Simp{1})\ar[r]\ar[d] & \Hom(\Horn{m}{i}, \Gamma)\ar[d]\\
  \Hom(\Simp{m},\Simp{1})\ar[r]^=&\Hom(\Horn{m}{i}, \Simp{1}) \rlap{\ .}
  }
\]
The bottom row is an isomorphism, thus so is the top row. This proves the claim.
\end{proof}

\subsection{Dimension 4}\label{chap6:ssec:level4}

Since the desired simplicial object \(\Gamma\) should be 3-coskeletal, we let \(\Gamma=\cosk_3 \Gamma_{\le 3}\);
that is, for \(m\ge 4\), we have
\begin{align*}
  \Gamma_m &=\Hom(\sk_3\Simp{m},\Gamma), \\
      &=\{f\in\Hom(\sk_2\Simp{m},\Gamma)|f\circ(\face_0\times \face_1\times \face_2\times \face_3)(\sk_3\Simp{m}) \subset \Gamma_3\},
\end{align*}
We shall prove that this completes the simplicial object \(\Gamma\).

Let us consider dimension \(m=4\). An element \(f\in\Hom(\sk_3\Simp{m},\Gamma)\) is given by a collection of triangles in \(\Gamma_2\) such that they are compatible in the following sense: four faces of every 3\nbdash{}simplex form an element in~\(\Gamma_3\). Using \(m_1\) or \(m_2\) in~\eqref{chap6:eq:3-multiplication}, we can uniquely determine the 1st or 2nd face of a 3\nbdash{}simplex. Therefore, we may remove some triangles of \(f\) such that \(f\) can be recovered from the remaining triangles. If there are two different ways of recovering the same face, then the results should be equal.

\begin{lemma}\label{chap6:lem:f024}
Given triangles \(f_{012}\), \(f_{013}\), \(f_{014}\), \(f_{123}\), \(f_{124}\), and \(f_{234} \in \Gamma_2\), then the following two ways to determine \(f_{024}\in \Gamma_2\) using \(m_1\) and \(m_2\) yield the same results:
\begin{equation}\label{chap6:eq:f024-1step}
    f_{024}=m_1(f_{124},f_{014},f_{012}),
\end{equation}
and
\begin{equation}\label{chap6:eq:f024-4step}
\begin{aligned}
    f_{134}&=m_1(f_{234},f_{124},f_{123}),\\
    f_{034}&=m_1(f_{134},f_{014},f_{013}),\\
    f_{023}&=m_1(f_{123},f_{013},f_{012}),\\
    f_{024}&=m_2(f_{234},f_{034},f_{023}).
\end{aligned}
\end{equation}
\end{lemma}
\begin{proof}
Translate the pentagon conditions into commutative diagrams of HS bibundles. The pentagon condition~\eqref{chap6:eq:pentagon} is illustrated by Figure~\ref{chap6:fig:pentagons}, where each pentagon represents an HS bibundle and every arrow between two pentagons is induced by a suitable instance of~\(\alpha\); there are five other cases depending on the coloring.
\begin{figure}[htbp]
\centering
\begin{tikzpicture}
  [>=latex',
  every label/.style={scale=0.5},
  mylb/.style={below,scale=0.6},
  scale=0.8]
  \foreach \j in {0,...,4}{
    \foreach \i in {0,...,2}{
      \node[mydot,label=-\i*72-126:$\i$] at ($(\j*72-54:3cm)+(-\i*72-126:1cm)$) (P\j\i){};
    }
    \foreach \i in {3,4}{
      \node[mydot,fill,label=-\i*72-126:$\i$] at ($(\j*72-54:3cm)+(-\i*72-126:1cm)$) (P\j\i){};
    }
    \draw[<-]
    (P\j0) edge node(M\j01){} (P\j1)
    (P\j1) edge node(M\j12){} (P\j2)
    (P\j2) edge node(M\j23){} (P\j3)
    (P\j3) edge node(M\j34){} (P\j4);
   }
   \draw[<-]
    (P00) edge node[mylb](M040){$a((bc)d)$} (P04)
    (P10) edge node[mylb](M140){$a(b(cd))$} (P14)
    (P20) edge node[mylb](M240){$(ab)(cd)$} (P24)
    (P30) edge node[mylb](M340){$((ab)c)d$} (P34)
    (P40) edge node[mylb](M440){$(a(bc))d$} (P44);
\begin{scope}
  \path[every edge/.append style=snake back]
   (P01) edge (P04)
         edge (P03)
   (P11) edge (P14)
   (P12) edge (P14)
   (P20) edge (P22)
   (P22) edge (P24)
   (P30) edge (P32)
         edge (P33)
   (P41) edge (P43)
   (P40) edge (P43);
\end{scope}
\begin{scope}[>=stealth,thick,bend right=15]
   \draw[->] (M023) to (M140);
   \draw[<-] (M112) to (M234);
   \draw[<-] (M201) to (M323);
   \draw[->] (M340) to (M412);
   \draw[->] (M434) to (M001);
\end{scope}
\end{tikzpicture}
\caption{Five Pentagons}\label{chap6:fig:pentagons}
\end{figure}
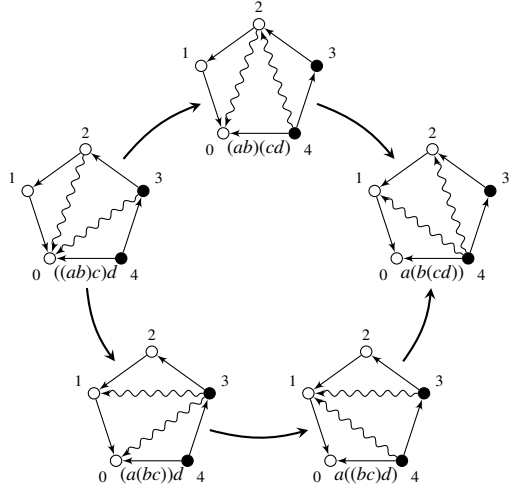

Since the triangles \(f_{012}\), \(f_{013}\), \(f_{014}\), \(f_{123}\),\(f_{124}\), and \(f_{234}\) are given, we can get all the other triangles using \(m_1\) and \(m_2\). Chasing Figure~\ref{chap6:fig:pentagons} from the rightmost pentagon labeled by \(a(b(cd))\) to the topmost pentagon labeled by \((ab)(cd)\), we have two ways to determine~\(f_{024}\): with one step as in~\eqref{chap6:eq:f024-1step} or with four steps as in~\eqref{chap6:eq:f024-4step}. They are exactly the same, as we claimed.
\end{proof}

\begin{lemma}\label{chap6:lem:level4}
There are natural isomorphisms of representable presheaves
\[
\Hom(\sk_3\Simp{4},\Gamma)\cong\Hom(\Horn{4}{i},\Gamma),\quad 0<i<4.
\]
\end{lemma}
\begin{proof}
Let \(S\) be the union of the following six 2-simplices
\begin{gather*}
 \Simp{2}\{0,1,2\},\quad \Simp{2}\{0,1,3\},\quad \Simp{2}\{0,1,4\},\\
 \Simp{2}\{1,2,3\},\quad \Simp{2}\{1,2,4\},\quad \Simp{2}\{2,3,4\}.
\end{gather*}
It follows from Lemma~\ref{chap6:lem:f024} that
\[
  \Hom(\sk_3\Simp{4},\Gamma)\cong\Hom(S,\Gamma)\cong\Hom(\Horn{4}{1},\Gamma)\cong\Hom(\Horn{4}{3}, \Gamma).
\]
Recall that \(\Sp(4)\to \Horn{4}{1}\) is an inner collapsible extension. Since \(\Gamma\) satisfies lower inner Kan conditions and \(\Hom(\Sp(4),\Gamma)\) is representable, the presheaf \(\Hom(\Horn{4}{1},\Gamma)\) is representable. The remaining isomorphisms follow from Lemma~\ref{chap3:lem:kan(n+1,j)-implies-all}.
\end{proof}

Let \(\Gamma_4=\Hom(\sk_3\Simp{4},\Gamma)\). Then \(\Gamma_4\) is representable, and the conditions \(\Kan!(4,i)\) for \(0<i<4\) follow by the lemma above. Considering a colored version of Lemma~\ref{chap3:lem:kan(n+1,j)-implies-all}, we see that all other desired colored Kan condition \(\Kan!(4,0)[i,j]\) for \(i\ge 2\) and \(\Kan!(4,4)[i,j]\) for \(j\ge 2\) hold. If the bibundle \(E\) is right principal, then \(\Kan!(4, 0)\) holds.

\subsection{Higher dimensions}\label{chap6:ssec:higher}

When \(m>4\), suppose that \(\Gamma_{< m}\) are given and all desired Kan conditions hold. Since \(\sk_3\Simp{m}=\sk_3\Horn{m}{k}\) for \(0\le k\le m\), we have isomorphism
\[
\Hom(\sk_3\Simp{m},\Gamma)\cong \Hom(\sk_3\Horn{m}{k},\Gamma) \quad \text{for }0\le k\le m,
\]
which are defined to be \(\Gamma_m\).

Since \(\Sp(m)\to \Horn{m}{1}\) is an inner collapsible extension and \(\Hom(\Sp(m),\Gamma)\) is representable, we deduce that \(\Hom(\Horn{m}{1},\Gamma)=\Hom(\sk_3\Horn{m}{1},\Gamma)\) is representable. Therefore, \(\Gamma_m\) is representable. The conditions \(\Kan!(m, i)\) for \(0\le k\le m\) hold by construction.

This completes the construction of \(\Gamma\).

\section{From simplicial to categorification picture}

We now consider the other direction of Theorem~\ref{chap6:thm:bibundle-cat=simplicial}. Let the simplicial object \(\Gamma\) be a bibundle between \(X\) and \(Y\). We construct a categorified bibundle between the 2\nbdash{}groupoids \(G\rightrightarrows \trivial{M}\) and \(H\rightrightarrows \trivial{N}\). If, in addition, \(\Gamma\to \Simp{1}\) is a left Kan fibration satisfying appropriate Kan conditions, then the corresponding categorified bibundle is right principal.

Regarding \(\Gamma\) as an augmented  bisimplicial object with \(\Gamma_{-1, -1}\) being the terminal object, we have \(R'_{-1}\Gamma = Y\) and \(C'_{-1}\Gamma=X\).

\subsection{The groupoid of bigons}

We first construct the groupoid \(E\). Recall from Proposition~\ref{chap4:prop:TX-I-lifting-properties} that \(R'_{0} \Gamma \to Y\) is a 2\nbdash{}groupoid Kan fibration in \((\Cat,\covers)\). Thus its fiber is a 1-groupoid in \((\Cat, \covers)\), denoted~\(E\). Symmetrically, denote the fibre of \(C'_{0} \Gamma \to X\) by \(E'\). We shall show that \(E\) and \(E'\) are naturally isomorphic.

More concretely, \(E\) is the groupoid of bigons with one white vertex and two black vertices; that is, we have
\begin{equation}\label{chap6:eq:groupoid-E}
\begin{gathered}
E_0=\Gamma_{0,0},\quad E_1=\face_0^{-1}(\de_0 \Gamma_{-1,0})\subset \Gamma_{0,1},\\
\source=\face_1,\quad \target=\face_2, \quad \unit=\de_1,
\end{gathered}
\end{equation}
illustrated by Figure~\ref{chap6:fig:groupoid-E}. The composition is depicted by the right picture in Figure~\ref{chap6:fig:groupoid-E}, where the edges \((1,2)\), \((2,3)\) and the face \((1,2,3)\) are degenerate. The face \((0,1,3)\), given by \(\Kan!(3,2)[1,3]\), is the composite of \((0,1,2)\) and \((0,2,3)\). Similarly, we can use \(\Kan!(3,1)[1,3]\) and \(\Kan!(3,3)[1,3]\) to define the inverses. It is clear that \(\unit\) satisfies the unit identities. The associativity follows from \(\Kan!(4,2)[1,3]\) or \(\Kan!(4,3)[1,3]\). Therefore, with these structures, \(E\) is a groupoid.
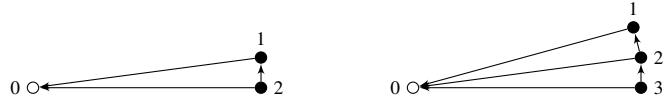
\begin{figure}[htbp]
  \centering
  \begin{tikzpicture}
    [>=latex',
    mydot/.style={draw,circle,inner sep=1.5pt},
    every label/.style={scale=0.6}]
    \node[mydot,label=-180:$0$] (p0) at (0,0) {};
    \node[mydot, fill, label=90:$1$] (p1) at (3, 0+0.4){};
    \node[mydot, fill, label=0:$2$] (p2) at (3, 0){};
    \draw[<-]
    (p0) edge (p1)
         edge (p2)
    (p1) edge (p2);
  \begin{scope}[xshift=5cm]
  \node[mydot,label=-180:$0$]      (q0) at (0,0) {};
  \node[mydot, fill, label=90:$1$] (q1) at (2.9, 0+0.8){};
  \node[mydot, fill, label=0:$2$]  (q2) at (3, 0.4){};
  \node[mydot, fill, label=0:$3$](q3) at (3, 0){};
  \draw[<-]
  (q0) edge (q1)
       edge (q2)
       edge (q3)
  (q1) edge (q2)
  (q2) edge (q3);
  \end{scope}
  \end{tikzpicture}
  \caption{The groupoid \(E\)}\label{chap6:fig:groupoid-E}
\end{figure}

The groupoid \(E'\) is the groupoid of bigons with two white vertices and one black vertex; that is, we have
\begin{gather*}
E'_0=\Gamma_{0,0},\quad E'_1=\face_2^{-1}(\de_0 \Gamma_{0,-1})\subset \Gamma_{1,0},\\
\source=\face_1,\quad \target=\face_0, \quad \unit=\de_0.
\end{gather*}
This groupoid is illustrated by Figure~\ref{chap6:fig:groupoid-E'}.
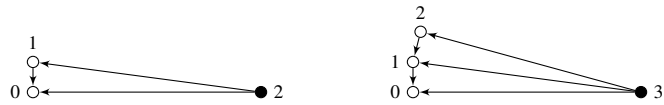
\begin{figure}[htbp]
  \centering
  \begin{tikzpicture}
    [>=latex',
    mydot/.style={draw,circle,inner sep=1.5pt},
    every label/.style={scale=0.6}]
    \node[mydot,label=-180:$0$] (p0) at (0,0) {};
    \node[mydot, label=90:$1$] (p1) at (0, 0+0.4){};
    \node[mydot, fill, label=0:$2$] (p2) at (3, 0){};
    \draw[<-]
    (p0) edge (p1)
         edge (p2)
    (p1) edge (p2);
  \begin{scope}[xshift=5cm]
  \node[mydot,label=-180:$0$]  (q0) at (0, 0) {};
  \node[mydot, label=-180:$1$] (q1) at (0, 0.4){};
  \node[mydot, label=90:$2$]   (q2) at (0.1, 0.8){};
  \node[mydot, fill, label=0:$3$](q3) at (3, 0){};
  \draw[<-]
  (q0) edge (q1)
       edge (q3)
  (q1) edge (q2)
       edge (q3)
  (q2) edge (q3);
  \end{scope}
  \end{tikzpicture}
  \caption{The groupoid \(E'\)}\label{chap6:fig:groupoid-E'}
\end{figure}

We construct a map \(\varphi\colon E_1\to E'_1\), illustrated by the top left picture in Figure~\ref{chap6:fig:groupoid-E=E'}. Given \((1,2,3) \in E_1\), let the faces \((0,1,3)\) and \((0,2,3)\) be degenerate. The condition \(\Kan!(3,3)[2,2]\) gives the face \((0,1,2)\in E'_1\). We can also use the top right picture to define a map \(E_1\to E'_1\). Considering a 4-simplex, we see that both maps are the same. Alternatively, we can use the bottom left or bottom right pictures to define a map \(\psi:E_1\to E'_1\). It is easy to see that \(\psi=\inv\circ \varphi\), where \(\inv\) is the inversion.
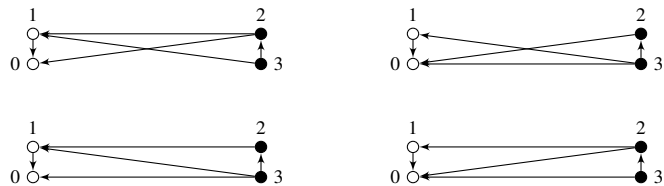
\begin{figure}[htbp]
  \centering
  \begin{tikzpicture}
    [>=latex',
    every label/.style={scale=0.6}]
    \begin{scope}[<-]
    \node[mydot,label=-180:$0$] (q0) at (0,0) {};
    \node[mydot,label=90:$1$] (q1) at (0, 0+0.4){};
    \node[mydot,fill,label=90:$2$] (q2) at (3, 0.4){};
    \node[mydot,fill,label=0:$3$] (q3) at (3, 0){};
    \draw
    (q0) edge (q1)
         edge (q2)
    (q1) edge (q2)
         edge (q3)
    (q2) edge (q3);
    \end{scope}
    \begin{scope}[xshift=5cm,<-]
    \node[mydot,label=-180:$0$] (q0) at (0,0) {};
    \node[mydot,label=90:$1$] (q1) at (0, 0+0.4){};
    \node[mydot,fill,label=90:$2$] (q2) at (3, 0.4){};
    \node[mydot,fill,label=0:$3$] (q3) at (3, 0){};
    \draw
    (q0) edge (q1)
         edge (q2)
         edge (q3)
    (q1) edge (q3)
    (q2) edge (q3);
    \end{scope}
    \begin{scope}[yshift=-1.5cm,<-]
    \node[mydot,label=-180:$0$] (p0) at (0,0) {};
    \node[mydot,label=90:$1$] (p1) at (0, 0+0.4){};
    \node[mydot,fill,label=90:$2$] (p2) at (3, 0.4){};
    \node[mydot,fill,label=0:$3$] (p3) at (3, 0){};
    \draw[<-]
    (p0) edge (p1)
         edge (p3)
    (p1) edge (p3)
         edge (p2)
    (p2) edge (p3);
    \end{scope}
    \begin{scope}[yshift=-1.5cm,xshift=5cm,<-]
    \node[mydot,label=-180:$0$] (q0) at (0,0) {};
    \node[mydot,label=90:$1$] (q1) at (0, 0+0.4){};
    \node[mydot,fill,label=90:$2$] (q2) at (3, 0.4){};
    \node[mydot,fill,label=0:$3$] (q3) at (3, 0){};
    \draw
    (q0) edge (q1)
         edge (q2)
         edge (q3)
    (q1) edge (q2)
    (q2) edge (q3);
    \end{scope}
\end{tikzpicture}
  \caption{The construction of maps \(E_1\to E'_1\)}\label{chap6:fig:groupoid-E=E'}
\end{figure}

To show that \(\varphi\) preserves compositions, we consider Figure~\ref{chap6:fig:groupoid-composition-E=E'}. Suppose that \((2,3,4)\) and \((2,4,5)\in E_1\) are given. Let the faces \((1,2,4)\), \((1,4,5)\) be degenerate. Set \((1,2,5)=\psi(2,4,5)\). Let \((1,3,4)\) be degenerate, and let \((1,2,3)\) be given by \(\Kan!(3,3)\{1,2,3,4\}\), that is, \((1,2,3)=\varphi(2,3,4)\). Let \((0,2,3)\) be degenerate. Let \((0,1,3)\) be given by applying \(\Kan!(3,1)[3,1]\) to \((0,1,2,3)\); that is, we have \((0,1,3)=\inv(1,2,3)\). Thus, \((0,1,3)=\psi(2,3,4)\). Let \((0,3,5)\) be degenerate and \((0,2,5)=\psi(2,3,5)\). Consider the 4\nbdash{}simplex \((1,2,3,4,5)\), where \((1,3,4,5)\) is degenerate. The condition \(\Kan!(4,3)\) implies that \((1,2,3,5)\) is compatible. Let \((0,1,3,5)=\de_2(0,1,3)\). Hence \((0,1,5)=(0,1,3)=\psi(2,3,4)\). Applying \(\Kan!(4, 3)\) to the 4-simplex \((0,1,2,3,5)\) proves that \((0,1,2,5)\) is compatible; that is,
\[
\psi(2,4,5)\circ\psi(2,3,4)=\psi(2,3,5),
\]
and this proves that \(\varphi\) preserves compositions.

\begin{figure}[htbp]
  \centering
  \begin{tikzpicture}
    [>=latex',
    mydot/.style={draw,circle,inner sep=1.5pt},
    every label/.style={scale=0.6}]
    \node[mydot,label=-180:$0$] (q0) at (0, 0) {};
    \node[mydot,label=-180:$1$] (q1) at (0, 0.4){};
    \node[mydot,label=90:$2$]   (q2) at (0.1, 0.8){};
    \node[mydot,fill,label=90:$3$] (q3) at (2.9,0.8){};
    \node[mydot,fill,label=0:$4$] (q4)  at (3, 0.4){};
    \node[mydot,fill,label=0:$5$] (q5)  at (3, 0){};
    \foreach \i/\j in {0/1,1/2,2/3,3/4,4/5,0/5,1/5,2/5,2/4}{
    \draw[<-] (q\i) --(q\j);
    }
  \end{tikzpicture}
    \caption{The groupoid isomorphism \(E\cong E'\)}\label{chap6:fig:groupoid-composition-E=E'}
\end{figure}
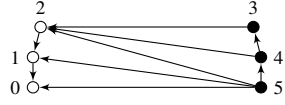

\subsection{The action morphisms}

We construct an HS morphism \(p\colon E\times_{\trivial{N}} H\to E\) from the HS bibundle \(\Gamma_{0,1}\). By symmetry, there is an HS morphism \(q\colon G\times_{\trivial{M}}E\to E\) from the HS bibundle \(\Gamma_{1,0}\).

By the definition of \(E\), the map \(\face_0\colon E_0\to N\) gives a groupoid functor \(E\to \trivial{N}\), which we take to be the right moment morphism \(J_r\). Similarly, the moment morphism \(J_l\colon E \to \trivial{M}\) is given by \(\face_1\colon E_0\to M\).

\begin{lemma}\label{chap6:lem:a1-bimodule}
  The space \(\Gamma_{0,1}\) is an HS bibundle \(E\times_{J_r,\trivial{N},\target} H\to E\), giving an HS morphism~\(p\). The space \(\Gamma_{1,0}\) is an HS bibundle \(G\times_{\source,\trivial{M},J_l} E\to E\), giving an HS morphism~\(q\).
\end{lemma}

\begin{proof}
By symmetry, it suffices to show the first statement. There is a left action of the groupoid \(E\) on \(\Gamma_{0,1}\) along the map \(\face_2\colon \Gamma_{0,1}\to \Gamma_{0,0}=E_0\) given by the composite map
  \[
  E_1\times_{\source, E_0, \face_2} \Gamma_{0,1} \to \Hom(\Horn{3}{2}[1,3],\Gamma)
   \cong \Gamma_{0,2} \xrightarrow{\face_2} \Gamma_{0,1},
  \]
where the map \(E_1\times_{\source, E_0, \face_2}\Gamma_{0,1} \to \Hom(\Horn{3}{2}[1,3],\Gamma)\) is given by \((\theta,a)\mapsto (\de_0 \face_0 (a) ,a, \theta)\). First, \(\Kan!(3,2)[1,3]\) implies the unit identity. Second, consider Figure~\ref{chap6:fig:E-action-A1}, where \(0<1''<1'<1\) and \((1'',1')\) and \((1',1)\) are degenerate. Applying \(\Kan!(4,2)[1,4]\) to \((0,1'',1',1,2)\), we obtain the associativity of the action. Similarly, we can also use \(E'\) to define the action, and the result is the same.

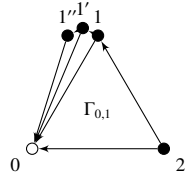
\begin{figure}[htbp]
  \centering
  \begin{tikzpicture}
    [>=latex', every label/.style={scale=0.6}]
    \node[scale=0.6] at (0,0) (a) {\(\Gamma_{0,1}\)};
    \node[mydot,label=210:\(0\)]              at (-0.866,-0.5)  (a0) {};
    \node[mydot,fill=black,label=90:\(1\)]    at (0,1)          (a1) {};
    \node[mydot,fill=black,label=90:\(1'\)]    at (0-0.2,1+0.1)      (a11) {};
    \node[mydot,fill=black,label=90:\(1''\)]    at (0-0.4,1)      (a12) {};
    \node[mydot,fill=black,label=-30:\(2\)]   at (0.866,-0.5)   (a2) {};
    \path[<-]
    (a0) edge (a1)
         edge (a11)
         edge (a12)
         edge (a2)
    (a12) edge (a11)
    (a11) edge (a1)
    (a1) edge (a2);
  \end{tikzpicture}
  \caption{Action of \(E\) on~\(\Gamma_{0,1}\)}
  \label{chap6:fig:E-action-A1}
\end{figure}

Similarly, there are a left action of \(H\) on \(\Gamma_{0,1}\) along the map \(\face_0\colon \Gamma_{0,1}\to \Gamma_{-1,1}=H_0\), and a right action of \(E\) on \(\Gamma_{0,1}\) along the map \(\face_1\colon \Gamma_{0,1} \to \Gamma_{0,0}=E_0\).

An argument similar to the above proof for associativity shows that these three actions mutually commute. Since \(\face_0\circ\face_2=\face_1\circ\face_0\colon \Gamma_{0,1}\to Y_0\), the groupoid \(E\times_{J_r,\trivial{N},\target} H\) acts on~\(\Gamma_{0,1}\). Thus, \(\Gamma_{0,1}\) is an \(E\times_{J_r,\trivial{N},\target} H\)-\(E\) bibundle; see Figure~\ref{chap6:fig:bimodule_A1}.
  \begin{figure}[htbp]
    \centering
    \begin{tikzpicture}
      [>=latex', every label/.style={scale=0.6},scale=0.75]
      \node[scale=0.6] at (0,0) (a) {\(\Gamma_{0,1}\)};
      \node[mydot,label=210:\(0\)]              at (-0.866,-0.5)  (a0) {};
      \node[mydot,fill=black,label=90:\(1\)]    at (0,1)          (a1) {};
      \node[mydot,fill=black,label=-30:\(2\)]   at (0.866,-0.5)   (a2) {};
      \path[<-]
      (a0) edge (a1)
           edge (a2)
      (a1) edge (a2);
      \begin{scope}[xshift=-3cm,<-]
        \node[mydot,label=210:\(0\)]              at (-0.866,-0.5)  (b0) {};
        \node[mydot,fill=black,label=90:\(1\)]    at (0,1)          (b1) {};
        \node[mydot,fill=black,label=-30:\(2\)]   at (0.866,-0.5)   (b2) {};
        \path
        (b0) edge[bend right=15] (b1)
             edge[bend left=15] node[scale=0.6,above left]{\(E\)}  (b1)
        (b1) edge[bend right=15] (b2)
        edge[bend left=15] node[scale=0.6,above right]{\(H\)}(b2);
      \end{scope}
      \begin{scope}[xshift=+3cm, <-]
        \node[mydot,label=210:\(0\)]              at (-0.866,-0.5)  (c0) {};
        \node[mydot,fill=black,label=-30:\(2\)]   at (0.866,-0.5)   (c2) {};
        \path
        (c0) edge[bend right=15] node[scale=0.6,below]{\(E\)} (c2)
             edge[bend left=15] (c2);
      \end{scope}
    \end{tikzpicture}
    \caption{The bibundle structure on~\(\Gamma_{0,1}\)}
    \label{chap6:fig:bimodule_A1}
  \end{figure}

Moreover, condition \(\Kan(2,1)[1,2]\) implies that the left moment map \(\Gamma_{0,1} \to E_0 \times_{M} H_0\) is a cover in \((\Cat,\covers)\), and \(\Kan!(3,1)[1,3]\) implies that the right action of \(E\) is principal. Therefore, \(\Gamma_{0,1}\) is an HS bibundle, giving an HS morphism \(p\colon E\times_{J_r,\trivial{N},\target} H \to E\).
\end{proof}

Since \(\face_1\circ\face_2=\face_1\circ \face_1\colon \Gamma_{0,1}\to M\), Lemma~\ref{chap5:lem:bundles-triangle-to-M} implies that  \(J_l\circ p=J_l\circ\pr_1\). Similarly, \(\face_0\circ \face_1=\face_1\circ\face_0:\Gamma_{0,1}\to N\) implies that \(J_r\circ p=\source\circ\pr_2\). A similar argument verifies the compatibility of \(q\) with moment morphisms.

To show that the HS morphisms \(p\) and \(q\) give a categorified bibundle structure on \(E\), we must construct associators and unitors.

\subsection{The associators}

We now construct the associators for the action morphisms \(p\) and \(q\).

The HS morphism \(p\circ (p \times \id)\colon E\times_{\trivial{N}} G \times_{\trivial{M}} G\to E\) is given by the HS bibundle
  \[
  ( (\Gamma_{0,1} \times_{N} H_1)\times_{E_0\times_{N} H_0} \Gamma_{0,1} )/( E \times_{\trivial{N}} H).
  \]
Since the actions of \(E\) and \(H\) commute and \(H\) acts on \(H_1\) by multiplication, we have
  \[
  ( (\Gamma_{0,1} \times_{N} H_1)\times_{E_0\times_{N} H_0} \Gamma_{0,1} )/( E \times_{\trivial{N}} H)
  \cong (\Gamma_{0,1} \times_{E_0} \Gamma_{0,1} )/ E.
  \]
Similarly, the HS morphism \(p\circ(\id \times \mult)\) is represented by the HS bibundle
  \[
  ( \Gamma_{0,1} \times_{H_0} \Gamma_{-1,2} ) / H,
  \]
where \(\mult\) is the multiplication of the 2\nbdash{}groupoid \(H \rightrightarrows \trivial{N}\).

In particular, the action of~\(E\) on \(\Gamma_{0,1} \times_{E_0} \Gamma_{0,1}\) and the action of \(H\) on \( \Gamma_{0,1} \times_{H_0} \Gamma_{-1,2}\) are principal. We now construct a one-to-one correspondence as follows
\begin{equation}\label{chap6:eq:associator-alpha}
  \begin{gathered}
  \alpha\colon (\Gamma_{0,1} \times_{E_0} \Gamma_{0,1} )/ E \to ( \Gamma_{0,1} \times_{H_0} \Gamma_{-1,2} ) / H, \\
   [(a,b)] \mapsto [(c,d)]\quad\text{if } (d,b,c,a) \in  \Gamma_{0,2},
  \end{gathered}
\end{equation}
which is illustrated by Figure~\ref{chap6:fig:associator}.
  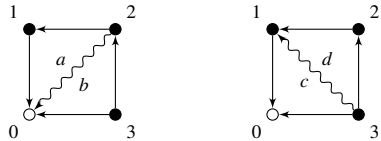
\begin{figure}[htbp]
    \centering
    \begin{tikzpicture}
      [>=latex',
      scale=0.8,
      mydot/.style={draw,circle,inner sep=1.5pt},
      every label/.style={scale=0.6}
      ]
      \node[mydot,label=225:\(0\)] at (225:1cm)(a0){};
      \foreach \i in {1,2,3}{
        \node[mydot,fill=black,label=225-\i*90:\(\i\)]  at (225-\i*90:1cm) (a\i){};
      }
      \path[<-]
      (a0)  edge        (a1)
      edge        (a3)
      edge[snake back]
      node[scale=0.6,above left]  {\(a\)}
      node[scale=0.6,below right] {\(b\)} (a2)
      (a1)  edge        (a2)
      (a2)  edge        (a3);
      \begin{scope}[xshift=+4cm]
        \node[mydot,label=225:\(0\)] at (225:1cm)(b0){};
        \foreach \i in {1,2,3}{
          \node[mydot,fill=black,label=225-\i*90:\(\i\)]  at (225-\i*90:1cm) (b\i){};
        }
        \path[<-]
        (b0)  edge        (b1)
        edge        (b3)
        (b1)  edge        (b2)
        edge[snake back]
        node[scale=0.6,above right] {\(d\)}
        node[scale=0.6,below left]  {\(c\)} (b3)
        (b2)  edge        (b3);
      \end{scope}
    \end{tikzpicture}
    \caption{The associator}
    \label{chap6:fig:associator}
  \end{figure}

We saw that we can compose a bigon and a triangle that share one side in~\(\Gamma\). We shall show that this composition lifts to tetrahedra. More precisely, we have the following lemma as a special case. It is not hard to formulate and prove the corresponding statements for other cases.

\begin{lemma}\label{chap6:lem:bigon-triangle}
  Let \(\gamma\in \Gamma_{0,2}\) be a \(3\)-simplex and let \(\theta\in \Gamma_{0,1}\) be a bigon along \(\gamma_{01}\). In Figure~\ref{chap6:fug:bigon-triangle}, assume that \(0<1'<1\) and that \((1',1)\) is degenerate. Suppose that \(\gamma\) is given by \((0,1,2,3)\), and \(\theta\) is given by \((0,1',1)\). Denote by \(\gamma'_{013}\) the composite of \(\gamma_{013}\) and~\(\theta\), and by \(\gamma'_{012}\) the composite of \(\gamma_{012}\) and~\(\theta\). Then the triangles \(\gamma_{123}, \gamma_{023}, \gamma'_{013}\), and \(\gamma'_{012}\) form a \(3\)-simplex in~\(\Gamma_{0,2}\).
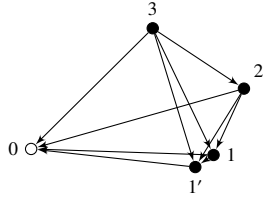
\begin{figure}[htbp]
\centering
  \begin{tikzpicture}%
  [>=latex', every label/.style={scale=0.6},scale=0.8]
  \node[mydot,label=180:\(0\)] at (0,0) (g0){};
  \node[mydot,fill,label=0:\(1\)] at (3,-0.1)   (g1){};
  \node[mydot,fill,label=60:\(2\)] at (3.5,1)(g2){};
  \node[mydot,fill,label=90:\(3\)] at (2,2)  (g3){};
  \node[mydot,fill,label=-90:\(1'\)] at (3-0.3,-0.3) (a){};
  \foreach \i/\j in
  {g1/g0,g2/g0,g3/g0,g2/g1,g3/g1,g3/g2,
  a/g0,g1/a,g2/a,g3/a}{
  \path[->] (\i) edge (\j);
  }
  \end{tikzpicture}
  \caption{Composition of bigons and triangles lifts to tetrahedra}\label{chap6:fug:bigon-triangle}
\end{figure}
\end{lemma}
\begin{proof}
  Suppose that \((1',1,2,3)=\de_0(1,2,3)\) is degenerate (see Figure~\ref{chap6:fug:bigon-triangle}). Applying \(\Kan!(3,2)[1,2]\) to \((0,1',1,2)\) yields a triangle \((0,1',2)\), which is \(\gamma'_{012}\). Similarly, the triangle \((0,1',3)\) gives \(\gamma'_{013}\). Applying \(\Kan!(4,2)[1,4]\) to \((0,1',1,2,3)\), we deduce that \((0,1',2,3)\) is a 3-simplex in~\(\Gamma_{0,2}\), and this proves the lemma.
\end{proof}
\begin{remark}\label{chap6:lem:inverse-bigon-triangle}
  We can also compose an inverse bigon and a triangle that share one side. Similarly, this composition can also be lifted to tetrahedra.
\end{remark}

Applying a suitable variant of the above lemma proves that~\eqref{chap6:eq:associator-alpha} does not depend on the choice of \((a,b)\) and \((c, d)\). Furthermore, the map \(\alpha\) is equivariant with respect to the right action of \(E\), and equivariant with respect to the left action of \(E\) and both actions of~\(H\).

Finally, since \(\Gamma_{0,1}\to \Hom(\Horn{2}{1}[1,2], \Gamma)\) is a cover, the map
  \[
  \Hom(\Horn{3}{2}[1,3], \Gamma)=\Hom(\Horn{3}{1}[1,3], \Gamma) \to \Gamma_{0,1}\times_{G_0} \Gamma_{-1,2}
  \]
is a cover. This map is invariant under the principal action of \(E\) and equivariant under the principal action of \( H\). Therefore, by Lemmas~\ref{chap1:lem:invariant-cover-descent} and~\ref{chap1:lem:action-over-principal-is-principal}, it descends to a cover, which is~\(\alpha\) by definition.

In conclusion, we have an isomorphism of HS morphisms \(E\times_{\trivial{N}}H\times_{\trivial{N}}H\to E\),
\[
\alpha\colon p\circ (p \times \id)\Rightarrow p\circ (\id\times \mult).
\]
Similarly, considering a different coloring for Figure~\ref{chap6:fig:associator}, we can construct isomorphisms of HS morphisms \(G\times_{\trivial{M}}G\times_{\trivial{M}} E \to E\),
\[
\alpha\colon p\circ (p \times \id)\Rightarrow p\circ (\id\times \mult),
\]
and of HS morphisms \( G\times_{\trivial{M}} E\times_{\trivial{N}} H\to E \),
\[
\alpha\colon q\circ (p \times \id)\Rightarrow p\circ (\id \times q).
\]

The pentagon condition for \(\alpha\) follows by a similar argument as in~\cite[Section 4.1]{Zhu:ngpd}; we can simply revert the argument of Section~\ref{chap6:ssec:level4}.

\subsection{The unitors}

We construct unitors for these two actions and then verify the triangle conditions. By symmetry, it suffices to consider the action of \(H\rightrightarrows \trivial{N}\) on \(E\).

The unit morphism \(\unit \colon \trivial{N} \to H\) is given by \(\de_0\colon N\to H_0\). Therefore, \(p\circ(\id\times  \unit)\) is given by the HS bibundle \(\face_0^{-1}\circ \de_0(N) \subset \Gamma_{0,1}\), which is \(E_1\) by construction. Thus we take the unitor \(\rho\colon p\circ(\id, \unit\circ J_r) \Rightarrow \id_E\) to be this isomorphism.

The HS morphism corresponding to \((e\cdot 1)\cdot  h\) is given by the HS bibundle \((E_1 \times_{\source,E_0,\face_2} \Gamma_{0,1} )/E\), and the one corresponding to \(e\cdot(1\cdot h)\) is given by the HS bibundle
\((\Gamma_{0,1}\times_{\face_0,H_0,\target}  H_1 )/H\). They are both isomorphic to \(\Gamma_{0,1}\) via the isomorphisms \(\rho\) and \(\lambda_H\). Moreover, applying \(\Kan!(3,1)[1,3]\) to a triple \((\de_0\face_0\eta,\eta,\de_1\face_2\eta)\), where \(\eta\in \Gamma_{0,1}\), we deduce that these isomorphisms are compatible with the associator~\(\alpha\). This proves the triangle condition for the unitor \(\rho\).

\subsection{Principality}

If \(\Gamma\to \Simp{1}\) also satisfies the left Kan conditions \(\Kan(m,0)\) for \(m>0\) and \(\Kan!(m,0)\) for \(m>2\), we shall show that the bibundle \(E\) is right principal as in Definition~\ref{chap6:def:cat-prinbibd}. This is essentially proved in Section~\ref{chap5:sec:principal-2-bundles}.

First, \(\Kan(1,0)[1,1]\) implies that the left moment morphism is given by a cover \(E_0\to M\). The lemma below verifies the other condition.

\begin{lemma}
The shear morphism
\[
(\pr_1,p)\colon E\times_{\trivial{N}} H \to E\times_{\trivial{M}} E, \quad (e,g)\mapsto (e, e\cdot g),
\]
is a Morita equivalence.
\end{lemma}
\begin{proof}
In view of Lemma~\ref{chap5:lem:shear-morphism-ismorita=left-principal}, it suffices to show that
\[
\Gamma_{0,1}\to \Gamma_{0,0}\times_{\face_0,M,\face_0} \Gamma_{0,0}
\]
is a left principal \(H\)-bundle. The condition \(\Kan(2,0)[1,2]\) implies that the map \(\Gamma_{0,1}\to \Gamma_{0,0}\times_{\face_0,M,\face_0}  \Gamma_{0,0}\) is a cover. The condition \(\Kan!(3,0)[1,3]\) implies that the left action of \(H\) on \(\Gamma_{0,1}\) is principal and that
\[
\Gamma_{0,1}/H\cong \Gamma_{0,0}\times_{\face_0,M,\face_0} \Gamma_{0,0}.
\]
This proves the lemma.
\end{proof}

This completes the proof of Theorem~\ref{chap6:thm:bibundle-cat=simplicial}.

\section{Weak equivalences of 2-groupoids}

We define weak equivalences of 2-groupoids in this section. We show that they can be modeled by weak acyclic fibrations whose cographs are two-sided principal bibundles.

\begin{definition}
  A morphism \(f\colon X\to Y\) of 2-groupoids in \((\Cat,\covers)\) is a \emph{weak equivalence} if
\begin{itemize}
 \item \(f\) is essentially surjective; that is, \(\face_1\circ\pr_2\colon X_0\times_{f_0,Y_0,\face_0} Y_1\to Y_0 \) is a cover;
 \item the induced functor \(f\colon \arrow{X}\to \arrow{Y}\) on groupoids of bigons is a weak equivalence of groupoids in \((\Cat,\covers)\).\footnote{Weak equivalences of groupoids in \((\Cat,\covers)\) are defined like those of Lie groupoids (Definition~\ref{chap1:def:weak-equi-groupoid}); we simply replace surjective submersions by covers.}
\end{itemize}
\end{definition}

\begin{remark}
The first condition is equivalent to
\[
 \face_0\circ\pr_2\colon X_0\times_{f_0,Y_0,\face_1} Y_1\to Y_0
\]
being a cover. The second condition means that the induced morphism \(\arrow{X}\to \arrow{Y}\) is essentially surjective and fully faithful.
\end{remark}

\begin{lemma}\label{chap6:lem:acyclic->weak-equivalence}
 An acyclic fibration \(f\colon X\to Y\) of 2-groupoids in \((\Cat,\covers)\) is a weak equivalence.
\end{lemma}
\begin{proof}
 The condition \(\Acyc(0)\) implies that \(f_0\colon X_0\to Y_0\) is a cover, hence \(f\) is essentially surjective. By \(\Acyc(1)\), the induced map \(\arrow{X}_0\to \arrow{Y}_0\) is a cover. Applying \(\Acyc!(2)\) proves that the map
\[
X_2\to \Hom(\partial\Simp{2}\to \Simp{2}, X\to Y)
\]
is an isomorphism. Restricting to the subspace of bigons, we get an isomorphism that shows exactly that the induced groupoid functor \(\arrow{X}\to \arrow{Y}\) is fully faithful.
\end{proof}

It is clear from the proof that, as for groupoids, the notion of 2-groupoid acyclic fibrations is stronger than that of weak equivalences.

\subsection{Weak acyclic fibrations}

We now give a weaker version of acyclic fibrations of 2-groupoids and then show that it is equivalent to the notion of weak equivalences.

\begin{definition}\label{chap6:def:weak-acyclic}
A morphism \(f\colon X\to Y\) of 2-groupoids in \((\Cat,\covers)\) is a \emph{weak acyclic fibration} if
\[
 \Hom(\Simp{k}\to \Simp{k}\star\Simp{0}, X\to Y)\to  \Hom(\partial\Simp{k}\to \partial\Simp{k}\star\Simp{0}, X\to Y)
\]
is a cover for \(k\ge 0\) and an isomorphism for \(k\ge 2\). We denote these conditions by \(\Acyc'(k)\) and \(\Acyc'!(k)\), respectively.
\end{definition}

\begin{remark}
  After completing the thesis, we learn that weak acyclic fibrations defined above were studied by Behrend and Getzler under the name of equivalences; see \cite[p.~71]{Getzler}.
\end{remark}

\begin{remark}
Proposition~\ref{chap4:prop:cograph-acyclic-is-kan} shows that \(\Acyc(k)\) implies \(\Acyc'(k)\). We also learn from the proof that a morphism is a weak acyclic fibration if and only if its cograph gives a two-sided principal bibundle.
\end{remark}

It is clear that \(\Hom(\Simp{k}\to \Simp{k}\star\Simp{0}, X\to Y)\) is representable, but we do not know whether \(\Hom(\partial\Simp{k}\to \partial\Simp{k}\star\Simp{0}, X\to Y)\) is representable.
\begin{lemma}
 Let \(f\colon X\to Y\) be a morphism of 2-groupoids. If \(f\) satisfies \(\Acyc'(i)\) for \(0\le i\le k-1\), then \(\Hom(\partial\Simp{k}\to \partial\Simp{k}\star\Simp{0}, X\to Y)\) is representable.
\end{lemma}
\begin{proof}
The statement is true for \(k=0\). Let \(k>0\). Since \(\Horn{k}{k}\) and \(\Horn{k}{k}\star\Simp{0}\) are collapsible and \(\Horn{k}{k}\to \Horn{k}{k}\star\Simp{0}\) is a collapsible extension by Proposition~\ref{chap2:prop:join-pushout-collapsible}, we deduce that
\[
\Hom(\Horn{k}{k}\to \Horn{k}{k}\star\Simp{0},X\to Y)
\]
is representable.

Let \(A\) and \(B\) be two simplicial sets such that \(A=B\cup_{\partial\Simp{i}}\Simp{i}\) for \(0\le i\le k-1\). We have a pullback diagram
\[
\xymatrix{
\Hom(A\to A\star\Simp{0}, X\to Y) \ar[r]\ar[d]& \Hom(\Simp{i}\to\Simp{i}\star\Simp{0}, X\to Y)\ar[d]\\
\Hom(B\to B\star\Simp{0}, X\to Y) \ar[r]& \Hom(\partial\Simp{i}\to\partial\Simp{i}\star\Simp{0}, X\to Y) \rlap{\ .}
}
\]
Thus if \(\Hom(B\to B\star\Simp{0}, X\to Y) \) is representable, then so is \(\Hom(A\to A\star\Simp{0}, X\to Y)\). Take \(B=\Horn{k}{k}\) and \(A=\partial\Simp{k}=\Horn{k}{k}\cup_{\partial\Simp{k-1}}\Simp{k-1}\). This proves the claim.
\end{proof}

\begin{lemma}\label{chap6:lem:weak-acyc2->acyc2}
 Let \(f\colon X\to Y\) be a weak acyclic fibration of 2-groupoids. Then \(f\) satisfies \(\Acyc!(2)\).
\end{lemma}
\begin{proof}
The statement follows from the pullback diagram
\[
\xymatrix{
 \Hom(\Simp{2}\to \Simp{2}, X\to Y)\ar[r]\ar[d]               & \Hom(\partial\Simp{2}\to \Simp{2}, X\to Y)\ar[d] \\
 \Hom(\Simp{2}\to \Simp{2}\star\Simp{0}, X\to Y)\ar[r]^-{\cong} &\Hom(\partial\Simp{2}\to \partial\Simp{2}\star\Simp{0}, X\to Y)\rlap{\ .}
}
\]
The left vertical arrow is induced by \(\de_2\colon Y_2\to Y_3\), and the right vertical arrow is induced by the composite
\[
Y_2\xrightarrow{\de_2} \Hom(\Simp{2}\star\Simp{0}, Y)\to \Hom(\partial\Simp{2}\star\Simp{0}, Y). \qedhere
\]
\end{proof}

Recall from~\cite{Zhu:ngpd} that if a morphism of 2-groupoids \(f\colon X\to Y\) satisfies \(\Acyc!(2)\), then it satisfies \(\Acyc!(k)\) for \(k\ge 2\).

\begin{lemma}
  A weak acyclic fibration of 2-groupoids \(f\colon X\to Y\) is a weak equivalence.
\end{lemma}
\begin{proof}
 The condition \(\Acyc'(0)\) implies that \(f\) is essentially surjective. Consider the pullback square
\[
 \xymatrix{
 X_1\times_{Y_1}\arrow{Y}_1\ar[d] \ar[r]& \Hom(\partial\Simp{1}\to \Simp{1}, X\to Y)\ar[d]\\
 \Hom(\Simp{1}\to \Simp{1}\star\Simp{0}, X\to Y)\ar[r]& \Hom(\partial\Simp{1}\to \partial\Simp{1}\star\Simp{0}, X\to Y)\rlap{\ ,}
 }
\]
where the right vertical arrow is induced by the composite
\[
Y_1\xrightarrow{\de_1} \Hom(\Simp{1}\star\Simp{0}, Y)\to \Hom(\partial\Simp{1}\star\Simp{0}, Y).
\]
The condition \(\Acyc'(1)\) implies that the bottom arrow is a cover. Thus so is the top arrow. This shows that the induced groupoid functor \(\arrow{X}\to\arrow{Y}\) is essentially surjective.

Lemma~\ref{chap6:lem:weak-acyc2->acyc2} shows that \(\Acyc'!(2)\) implies \(\Acyc!(2)\). Therefore, the induced groupoid functor \(\arrow{X}\to\arrow{Y}\) is fully faithful by Lemma~\ref{chap6:lem:acyclic->weak-equivalence}, and we are done.
\end{proof}

\begin{lemma}
 A weak equivalence of 2-groupoids \(f\colon X\to Y\) is a weak acyclic fibration.
\end{lemma}

\begin{proof}
First, the essential surjectivity of \(f\) implies \(\Acyc'(0)\).

\begin{figure}[htbp]
 \centering
 \begin{tikzpicture}%
  [>=latex', hdot/.style={mydot,fill=lightgray},every label/.style={scale=0.6}]
  \node[mydot,label=90:\(0\)] at (0,1) (g0){};
  \node[mydot,label=90:\(1\)] at (2,1) (g1){};
  \node[hdot,label=-90:\(0\)] at (0,0) (h0){};
  \node[hdot,label=-90:\(1\)] at (2,0) (h1){};
  \path[->]
   (g1) edge[out=150,in=30] (g0)
   (h1) edge[out=150,in=30] (h0)
        edge[out=-150,in=-30] (h0);
  \begin{scope}[xshift=5cm]
   \node[mydot,label=90:\(0\)] at (0,1) (g0){};
  \node[mydot,label=90:\(1\)] at (2,1) (g1){};
  \node[hdot,label=-90:\(0\)] at (0,0) (h0){};
  \node[hdot,label=-90:\(1\)] at (2,0) (h1){};
  \path[->]
   (h1) edge[out=-150,in=-30] (h0);
  \end{scope}
 \end{tikzpicture}
 \caption{\(\arrow{X}\to\arrow{Y}\) is essentially surjective}\label{chap6:fig:local-ess-surj}
\end{figure}
Second, the induced groupoid functor \(\arrow{X}\to\arrow{Y}\) is essentially surjective, thus the map from the space given by left picture in Figure~\ref{chap6:fig:local-ess-surj} to the space given by the right picture is a cover. In this figure, white dots are objects in \(X\) and gray dots are objects in~\(Y\), and \(f\) sends a white dot to the gray dot below. The condition \(\Acyc'(0)\) implies that the space on the right is representable. The groupoid of bigons \(\arrow{Y}\) acts on these two spaces. Composing these two actions with the principal action of \(\arrow{Y}\) on \(Y_2\) along \(\face_2\colon Y_2\to Y_1\), respectively, we obtain a \(\arrow{Y}\) invariant cover. Taking the quotients, we obtain \(\Acyc'(1)\), illustrated by Figure~\ref{chap6:fig:weak-acyc-1}.
\begin{figure}[htbp]
 \centering
 \begin{tikzpicture}%
  [>=latex', hdot/.style={mydot,fill=lightgray},every label/.style={scale=0.6}]
  \node[mydot,label=90:\(0\)] at (0,1) (g0){};
  \node[mydot,label=90:\(1\)] at (2,1) (g1){};
  \node[hdot,label=-90:\(0\)] at (0,0) (h0){};
  \node[hdot,label=-90:\(1\)] at (2,0) (h1){};
  \node[hdot,label=-90:\(2\)] at (1,-1)(h2){};
  \path[->]
   (g1) edge[out=150,in=30] (g0)
   (h1) edge[out=150,in=30] (h0)
        edge[snake arrow, out=-150,in=-30] (h0)
   (h2) edge (h1)
        edge (h0);
  \begin{scope}[xshift=5cm]
  \node[mydot,label=90:\(0\)] at (0,1) (g0){};
  \node[mydot,label=90:\(1\)] at (2,1) (g1){};
  \node[hdot,label=-90:\(0\)] at (0,0) (h0){};
  \node[hdot,label=-90:\(1\)] at (2,0) (h1){};
  \node[hdot,label=-90:\(2\)] at (1,-1)(h2){};
  \path[->]
   (h1) edge[snake arrow, out=-150,in=-30] (h0)
   (h2) edge (h1)
        edge (h0);
  \end{scope}
 \end{tikzpicture}
 \caption{\(\Acyc'(1)\) for \(X\to Y\)}\label{chap6:fig:weak-acyc-1}
\end{figure}

It remains to consider \(\Acyc'!(2)\); we show instead \(\Acyc!(2)\). Since the induced groupoid functor \(\arrow{X}\to\arrow{Y}\) is fully faithful, the spaces given by the two pictures in Figure~\ref{chap6:fig:local-fully-faithful} are isomorphic, where \(\#\) indicates that the 2-cell on that bigon is removed.
\begin{figure}[htbp]
 \centering
 \begin{tikzpicture}%
  [>=latex', hdot/.style={mydot,fill=lightgray},every label/.style={scale=0.6}]
  \node[mydot,label=90:\(0\)] at (0,1) (g0){};
  \node[mydot,label=90:\(1\)] at (2,1) (g1){};
  \node at (1,1)(g00) {\(\#\)};
  \node[hdot,label=-90:\(0\)] at (0,0) (h0){};
  \node[hdot,label=-90:\(1\)] at (2,0) (h1){};
  \path[->]
   (g1) edge[out=150,in=30] (g0)
        edge[out=-150,in=-30] (g0)
   (h1) edge[out=150,in=30] (h0)
        edge[out=-150,in=-30] (h0);
  \begin{scope}[xshift=5cm]
  \node[mydot,label=90:\(0\)] at (0,1) (g0){};
  \node[mydot,label=90:\(1\)] at (2,1) (g1){};
  \path[->]
   (g1) edge[out=150,in=30] (g0)
        edge[out=-150,in=-30] (g0);
  \end{scope}
 \end{tikzpicture}
 \caption{\(\arrow{X}\to\arrow{Y}\) is fully faithful}\label{chap6:fig:local-fully-faithful}
\end{figure}
There are \(\arrow{X}\)-actions on these two spaces. Composing these two actions with the principal action of \(\arrow{Y}\) on \(Y_2\) along \(\face_2\colon Y_2\to Y_1\), respectively, we obtain a \(\arrow{X}\) invariant isomorphism. Taking quotients, we obtain \(\Acyc!(2)\), as illustrated by Figure~\ref{chap6:fig:acyclic-2}.
\begin{figure}[htbp]
 \centering
 \begin{tikzpicture}%
  [>=latex', hdot/.style={mydot,fill=lightgray},every label/.style={scale=0.6}]
  \node[mydot,label=90:\(0\)] at (0,2) (g0){};
  \node[mydot,label=90:\(1\)] at (2,2) (g1){};
  \node[mydot,label=90:\(2\)] at (1,1) (g2){};
  \node at (1,2)(g00) {\(\#\)};
  \node[hdot,label=-90:\(0\)] at (0,0) (h0){};
  \node[hdot,label=-90:\(1\)] at (2,0) (h1){};
  \node[hdot,label=90:\(2\)] at (1,-1) (h2){};
  \path[->]
   (g1) edge[out=150,in=30] (g0)
        edge[snake arrow, out=-150,in=-30] (g0)
   (g2) edge (g0)
        edge (g1)
   (h1) edge[out=150,in=30] (h0)
        edge[snake arrow, out=-150,in=-30] (h0)
   (h2) edge (h1)
        edge (h0);
  \begin{scope}[xshift=5cm]
   \node[mydot,label=90:\(0\)] at (0,2) (g0){};
  \node[mydot,label=90:\(1\)] at (2,2) (g1){};
  \node[mydot,label=90:\(2\)] at (1,1) (g2){};
  \path[->]
   (g1) edge[out=150,in=30] (g0)
        edge[snake arrow, out=-150,in=-30] (g0)
   (g2) edge (g0)
        edge (g1);
  \end{scope}
 \end{tikzpicture}
 \caption{\(\Acyc!(2)\) for \(X\to Y\)}\label{chap6:fig:acyclic-2}
\end{figure}
\end{proof}

The following proposition concludes this section.
\begin{proposition}\label{chap6:prop:weak-equi=weak-acyc}
  A 2-groupoid morphism is a weak equivalence if and only if it is a weak acyclic fibration.
\end{proposition}

\section{Examples}\label{chap6:sec:example}

We present some examples of 2\nbdash{}groupoid bundles and bibundles.

\begin{example}
Suppose that both \(G\rightrightarrows \trivial{M}\) and \(H\rightrightarrows \trivial{N}\) are 1\nbdash{}groupoids; that is, \(G\) and \(H\) are 0\nbdash{}groupoids. If we assume that \(E\) is also a 0\nbdash{}groupoid, then a bibundle between \(G\rightrightarrows \trivial{M}\) and \(H\rightrightarrows \trivial{N}\) is simply a bibundle of 1\nbdash{}groupoids, and a right principal bibundle is a HS bibundle of 1-groupoids. Hence the notion of bibundles of 2\nbdash{}groupoids subsumes that of bibundles of 1-groupoids.

We assume further that these two 2\nbdash{}groupoids arise from 0\nbdash{}groupoids; that is \(G=\trivial{M}\) and \(H=\trivial{N}\). Every groupoid \(E\) with trivial actions is a bibundle between these two 2\nbdash{}groupoids. Hence the notion of bibundles between 1-groupoids viewed as 2\nbdash{}groupoids may be broader than that of bibundles of 1\nbdash{}groupoids.
\end{example}

\begin{example}
Let \(G\rightrightarrows \trivial{M}\) be a categorified groupoid and let \(X\) be the corresponding simplicial object. The groupoid \(G\) is a two-sided principal bibundle between \(G\rightrightarrows \trivial{M}\) and \(G\rightrightarrows \trivial{M}\). The actions are both given by the horizontal multiplication, while unitors and associators are given by those of the 2\nbdash{}groupoid. The simplicial object of this bibundle is the total décalage \(\Dec X\to \Simp{1}\).
\end{example}

\begin{example}
Let \(G\rightrightarrows \trivial{M}\) and \(H\rightrightarrows \trivial{N}\) be categorified groupoids, and let \(X\) and \(Y\) be the corresponding simplicial objects. Given a simplicial map \(f\colon X\to Y\), it induces a map \(f_0\colon M\to N\) and a groupoid functor \(f_1\colon G\to H\) that preserve all structures. We construct a bibundle structure on \(M\times_{N, \target} H\) as follows. The right action \(p\) is given by the right action of \(H\rightrightarrows \trivial{N}\) on \(H\). The left action \(q\) is given by the morphism \(f_1\) composed with the left action of \(H \rightrightarrows \trivial{N}\) on \(H\). In formulas, the moment morphisms are given by
\begin{gather*}
J_r=\source \circ\pr_2\colon  \trivial{M}\times_{\trivial{N}, \target} H \to \trivial{N},\\
J_l=\pr_1\colon               \trivial{M}\times_{\trivial{N}, \target} H \to \trivial{M},
\end{gather*}
the right action morphism is given by
\[
p=\id_\trivial{M}\times \mult \colon \trivial{M}\times_{\trivial{N}, \target} H \times_{\trivial{N}} H\to \trivial{M}\times_{\trivial{N}, \target} H,
\]
and the left action morphism is given by
\begin{gather*}
G\times_{f_0\circ \source, \trivial{N},\target} H=G\times_{\source, \trivial{M}}  \trivial{M}\times_{\trivial{N}, \target} H\to \trivial{M}\times_{\trivial{N}, \target} H,\\
q=(\target\circ\pr_1, \mult\circ((f_1\circ \pr_1)\times \pr_2)).
\end{gather*}
It is easy to see that this defines a bibundle. This bibundle is right principal because~\(H\) is a right principal bibundle of \(H\rightrightarrows \trivial{N}\). By construction, the corresponding simplicial map \(\Gamma\to \Simp{1}\) is the one given by the cograph of \(f\colon X\to Y\). We call this bibundle the \emph{bundlisation} of a morphism.
\end{example}

\begin{example}
  Let \(*\) be the trivial 2\nbdash{}groupoid corresponding to the simplicial object \(\Simp{0}\). An action of \(*\) on a groupoid \(E\) consist of a HS morphism \(\mu\colon E\to E\), two 2\nbdash{}morphisms \(\alpha\colon \mu\circ \mu\Rightarrow \mu\) and \(\rho\colon \mu\Rightarrow\id\), which come from the action morphism, the associator, and the unitor, respectively. The coherence conditions imply that the quadruple \((E,\mu,\alpha,\rho^{-1})\) is a monad in the 2\nbdash{}category \(\GPD\). Since all the morphisms are invertible, this monad is isomorphic to the trivial monad. Similarly, a bibundle between \(*\) and \(*\) is given by a pair of monads on a groupoid satisfying a distributive law. For monads in a 2\nbdash{}category and distributive laws between monads, see~\cite{Lack}.
\end{example}

\begin{example}\label{chap6:exa:action-as-bibundle}
  Given a right categorified action of \(G\rightrightarrows \trivial{M}\) on \(E\), then \(E\) is naturally a bibundle between \(*\) and \(G\rightrightarrows \trivial{M}\), where \(*\) is the trivial 2\nbdash{}groupoid and the left action is the obvious trivial action. Similarly, a right categorified bundle \(E\to \trivial{N}\) of \(G\rightrightarrows \trivial{M}\) gives a bibundle between \(\trivial{N}\) and \(G\rightrightarrows \trivial{M}\).
\end{example}

\begin{example}\label{chap6:exa:opposite-action}
  Let \((\action, J)\) be a categorified right action of \(G\rightrightarrows \trivial{M}\) on \(E\) corresponding to a 2\nbdash{}groupoid Kan fibration \(A\to X\). The Kan fibration \(A\to X\) also gives a categorified left action \((\action', J)\) of \(G\rightrightarrows \trivial{M}\) on \(E\). Similar to Figure~\ref{chap5:fig:associator}, the spaces given by the two pictures in Figure~\ref{chap6:fig:opposite-action} are isomorphic. Let \((1,3)\) be degenerate, we obtain an HS bibundle isomorphism \(A_1\cong (A_1\times_{G_1} E_\inv)/G\), where \(E_\inv\) is the space of inverse bigons. It follows that the categorified left action is given by the categorified right action via the horizontal inverse of the 2\nbdash{}groupoid; that is, \(\action'\) is isomorphic to the composite
  \[
  G\times_{\source, \trivial{M}}E\xrightarrow{\cong} E\times_{\trivial{M},\source}G\xrightarrow{\id\times\inv}E\times_{\trivial{M},\target}G \xrightarrow{\action} E.
  \]
  (Using the coherent theory for 2\nbdash{}groupoids~\cite{Ulbrich,Laplaza,Baez-Lauda}, we can directly construct a left categorified action out of a right categorified action.) Consider the opposite bibundle in Example~\ref{chap4:rem:opposite-bibundle} associated with a bibundle given by simplicial map \(\Gamma\to \Simp{1}\). We deduce that, given a categorified bibundle between \(G\rightrightarrows \trivial{M}\) and \(H\rightrightarrows \trivial{N}\), the above construction yields a categorified bibundle between \(H\rightrightarrows \trivial{N}\) and \(G\rightrightarrows \trivial{M}\).
\begin{figure}[htbp]
    \centering
    \begin{tikzpicture}
      [>=latex',
      scale=0.8,
      every label/.style={scale=0.6}
      ]
      \node[mydot,label=225:\(0\)] at (225:1cm)(a0){};
      \node[mydot,fill,label=135:\(1\)] at (135:1cm)(a1){};
      \node[mydot,fill,label=45:\(3\)] at (45:1cm)(a3){};
      \node[mydot,fill,label=-45:\(2\)] at (-45:1cm)(a2){};
      \path[<-]
      (a0)  edge        (a1)
      edge        (a2)
      edge[snake back]
      node[scale=0.6,above left]  {\(A_1\)}
      node[scale=0.6,below right] {\(A_1\)} (a3)
      (a1)  edge        (a3)
      (a2)  edge        (a3);
      \begin{scope}[xshift=+4cm]
      \node[mydot,label=225:\(0\)] at (225:1cm)(b0){};
      \node[mydot,fill,label=135:\(1\)] at (135:1cm)(b1){};
      \node[mydot,fill,label=45:\(3\)] at (45:1cm)(b3){};
      \node[mydot,fill,label=-45:\(2\)] at (-45:1cm)(b2){};
        \path[<-]
        (b0)  edge        (b1)
        edge        (b2)
        (b1)  edge        (b3)
        edge[snake back]
        node[scale=0.6,above right] {\(G_2\)}
        node[scale=0.6,below left]  {\(A_1\)} (b2)
        (b2)  edge        (b3);
      \end{scope}
    \end{tikzpicture}
    \caption{An isomorphism of HS bibundles}\label{chap6:fig:opposite-action}
  \end{figure}
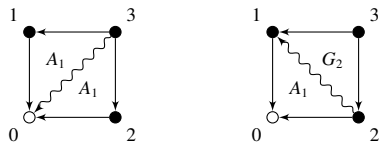
\end{example}

\begin{example}
  Let \(K(S^1, 2)\) be the 2-group associated with the crossed module \(S^1\to *\), where \(S^1\) is the circle group. Then a principal \(K(S^1, 2)\) bundle over a manifold \(M\) is the same as a bundle gerbe. To see the equivalence, let \(Y\to K(S^1, 2)\) be a 2-groupoid Kan fibration and let \(Y\to \sk_0 M\) be acyclic. Then \(Y_0\to M\) is a cover and \(Y_1\) is a principal \(S^1\)-bundle over \(Y_0\times_M Y_0\). The condition \(\Kan!(2,1)\) gives an isomorphism of \(S^1\)-bundles on \(Y_0\times_M Y_0\times_M Y_0 \), and \(\Kan!(3,1)\) gives the coherence condition. The equivalence between bundle gerbes and categorified bundles was explained in~\cite{Nikolaus-Waldorf:Four_gerbes}.
\end{example}

\chapter{Composition of 2-Groupoid Bibundles}\label{chap6:composition}
\thispagestyle{empty}

Let \(X, Y\) and \(Z\) be 2-groupoids in \((\Cat,\covers)\). Given an \(X\)-\(Y\) bibundle \(\Gamma\) and a \(Y\)-\(Z\) bibundle \(\Xi\) such that both are right principal, we shall construct their composite, which should be a right principal \(X\)-\(Z\) bibundle.

This is an analogue of the composition of HS bibundles of groupoids. The construction, however, is more complicated.

Our construction is motivated by 2-coends, which we will first explain briefly. We then introduce the pasting theorem of 2-categories and the notion of a fundamental groupoid for a 2-groupoid. The pasting theorem will simplify verification of higher Kan conditions.

We then consider the fundamental groupoid of the action 2\nbdash{}groupoid \(R'_0\Gamma\times_Y C'_0\Xi\). This is the analogue of the quotient construction for the composition of HS bibundles. The desired composite \(X\)-\(Z\) bibundle will be a bibundle on this fundamental groupoid.

We shall show three properties of the composition: up to a certain weak equivalence, the bundlisation preserves compositions; the total décalage is a unit for the composition; and the composition is associative.

Our construction is purely combinatorial. Since the construction gets complicated, the argument will rely heavily on pictures.

\section{Preparations on 2-categories}

We first explain the idea behind the construction. Then we introduce the pasting theorem of 2-categories and fundamental groupoids of 2-groupoids; they are useful tools for our construction.

\subsection{Remarks on the construction and 2-coends}

Before proceeding to the complicated construction, we shall explain where the ideas come from. Our construction is motivated by the notion of a 2-coend, which is also known as a cartesian coend or a lax coend. Along with 2-ends, 2-coends are studied by Bozapalides in his thesis~\cite{Bozapalides:thesis}. Unfortunately, the author is unable to find a copy of this thesis; see instead~\cite{Bozapalides77:fin-cart-gen,Bozapalides80:lax-presheaf}.

The 2-coend of a bifunctor \(F\colon \Cat[A]^\op\times \Cat[A]\to\Cats\) is a category \(\Cat[E]\) with a 2-categorical universal extranatural transformations from \(F\) to \(\Cat[E]\); see~\cite{Bozapalides77:fin-cart-gen} for 2-ends, the dual of 2-coends. Intuitively, we may think of the 2-coend as a 2-categorical quotient with respect to the diagonal action of~\(\Cat[A]\).

The composite of 2-profunctors \(\varphi\colon \Cat[B]^\op\times\Cat[A]\to \Cats\) and \(\psi\colon \Cat[C]^\op\times\Cat[B]\to\Cats\) is given by a 2-coend~\cite{Bozapalides77:2-dist}. The formula is almost the same as that of the composition of profunctors. Again, by the properties of 2-coends, the composition is unital and associative up to equivalences.

We cannot apply the results above directly to 2-groupoid bibundles. We will first construct a groupoid which can be viewed as a 2-coend. We then construct two actions on this groupoid. To this purpose we need to verify a lot of diagrams; for 2-profunctors this follows from the functoriality of limits. The unitality and associativity are also proved by verifying that some diagrams commute.

A similar construction in the context of stacky groupoids is given in~\cite{Bursztyn-Noseda-Zhu}. Modules of bicategories enriched in a monoidal bicategory are studied by Gerner and Shulman~\cite{Garner-Shulman}, where the tensor product (composition) is essentially a codescent object~\cite{Street87}.

It is expected that there should be a 3\nbdash{}category of 2-categories with bimodules as 1-morphisms. Some efforts toward this direction have been made in~\cite{Greenough10,Garner-Shulman}.

\subsection{The pasting theorem of 2-categories}

We recall the pasting theorem of 2-categories. Then we carry it over to the simplicial setting.

We saw in the previous chapters that we can use Kan conditions to compose triangles (and bigons) in a 2-groupoid or 2-groupoid bibundle. Given a collection of triangles, we often need to show that two different ways to obtain a certain triangle give the same result. This is done by applying Kan conditions, and the proof is usually tedious and long.

Recall the relation between a 2-category and its geometric nerve described in Section~\ref{chap5:ssec:geometric-nerve}. Using Kan conditions to compose triangles in the nerve corresponds to the pasting operation in the 2-category, which is the most general composition of 2-morphisms.

A \emph{plane graph with source and sink} is a directed plane graph with distinct vertices \(s\) and \(t\) in the exterior face such that for each other vertex \(v\) there are directed paths from \(s\) to~\(v\) and from \(v\) to~\(t\).

\begin{definition}
  A \emph{pasting scheme} is a plane graph with source and sink that contains no directed cycle.
\end{definition}

A \emph{labelling of a pasting scheme} in a strict 2-category is a labelling of each vertex by an object, each edge by a 1-morphism, and each face by a 2-morphism, all with appropriate sources and targets. For a weak 2-category, the labelling is more complicated because we need to remember identities and bracketings of edges; see~\cite{Verity} for details.

\begin{theorem}
  Every labelling of a pasting scheme has a unique composite.
\end{theorem}

This theorem is proved by Power~\cite{Power} for strict 2-categories and by Verity~\cite{Verity} for weak 2-categories. Figure~\ref{chap6:fig:pasting-kan-3-1} is a pasting scheme corresponding to the horn~\(\Horn{3}{1}\). The pasting composition corresponds to using \(\Kan!(3,1)\) to obtain the exterior face \((0,2,3)\).
\begin{figure}[htbp]
  \centering
  \begin{tikzpicture}%
  [>=latex',every label/.style={scale=0.6},scale=1.3]
  \node[mydot,label=-45:\(1\)] at (0,0) (a1){};
  \node[mydot,label=-90:\(2\)] at (0,-1) (a2){};
  \node[mydot,label=0:\(3\)] at (0.87,0.5)(a3){};
  \node[mydot,label=180:\(0\)] at (-0.87,0.5)(a0){};
  \path[->]
  (a1) edge (a0)
  (a2) edge (a0)
       edge (a1)
  (a3) edge (a0)
       edge (a1)
       edge (a2);
  \end{tikzpicture}
  \caption{A pasting scheme corresponding to \(\Horn{3}{1}\)}\label{chap6:fig:pasting-kan-3-1}
\end{figure}

Using the equivalence between 2-categories and their geometric nerves, we obtain the following corollary.

\begin{corollary}
  Let \(X\) be an inner Kan complex satisfying unique horn filling conditions above dimension \(2\). Given a collection of 2-simplices in \(X\) such that it is a labelling of a pasting scheme, then using inner Kan conditions we can uniquely obtain a composite.
\end{corollary}

This corollary implies \(\Kan!(4, k)\) for \(0<k<4\). Since we only have a pasting theorem for 2-categories, we will carefully avoid using outer Kan conditions.

\subsection{Fundamental groupoids of 2-groupoids}

Given a Kan complex, there is an associated fundamental groupoid (also called homotopy groupoid or Poincaré groupoid). We apply this construction to 2\nbdash{}groupoids.

Let \(X\) be a 2-groupoid in \((\Cat,\covers)\). The groupoid of bigons \(\arrow{X}\) induces an equivalence relation \(\sim\) on \(\arrow{X}_0=X_1\). More precisely, we have \(\gamma_1\sim \gamma_2 \in X_1\) if there is \(\theta\in X_2\) such that \(\face_1\theta=\gamma_1\), \(\face_2\theta=\gamma_2\), and \(\face_0\theta=\de_0\face_0\gamma_1\).

\begin{proposition}\label{chap6:prop:fundamental-groupoid}
  Let \(\theta\in X_2\) with \(\face_i\theta=\gamma_i\) for \(0\le i\le 2\). Define a composite
  \[
  [\gamma_2]\circ [\gamma_0]=[\gamma_1].
  \]
  With this composition, there is a \textup(set-theoretical\textup) groupoid \(\tau(X)\) with
  \begin{gather*}
  \tau(X)_0=X_0,\quad \tau(X)_1=X_1/{\sim}, \\
  \source([\gamma])=\face_0\gamma,\quad \target([\gamma])=\face_1\gamma,\quad \unit (x)=[s_0 x].
  \end{gather*}
  We call \(\tau(X)\) the \emph{fundamental groupoid} of~\(X\).
\end{proposition}
\begin{proof}
  This is a classical result; see~\cite[Chapter IV]{Gabriel-Zisman} or~\cite{Joyal2002}.
\end{proof}

The fundamental groupoid is a higher analogue of the quotient space of a groupoid. The fundamental groupoid \(\tau(X)\) of a 2-groupoid in \((\Cat,\covers)\) need not be a groupoid in \((\Cat,\covers)\).

\begin{lemma}\label{chap6:lem:fundamental-is-good}
  Let \(X\) be 2-groupoid in \((\Cat,\covers)\). Suppose that \(\arrow{X}\) acts on \(X_1\) principally \textup(hence \(\tau(X)_1\) is an object in \(\Cat\)\textup). Then \(\tau(X)\) is a groupoid in \((\Cat,\covers)\).
\end{lemma}
\begin{proof}
  Since \(\arrow{X}\) acts on \(X_1\) principally and the cover \(\face_0\colon X_1\to X_0\) is \(\arrow{X}\) invariant, it descends to a cover \(\target\colon \tau(X)_1=X_1/\arrow{X}\to X_0\) by Lemma~\ref{chap1:lem:invariant-cover-descent}. Similarly, the source map \(\source\) is also a cover. Consider the cover
  \[
  \face_1\colon X_2\to X_1.
  \]
  There is a principal \(\arrow{X}\)-actions on \(X_2\) with moment map given by \(\face_1\), and a principal \(\arrow{X}\)-action on \(X_1\). The map \(\face_1\) is equivariant with respect to these two actions.  Lemma~\ref{chap1:lem:action-over-principal-is-principal} implies that the induced map
  \begin{equation}\label{chap6:eq:compostion-induced-map}
   \Hom(\Horn{2}{1}, X)=X_2/\arrow{X}\to X_1/\arrow{X}=\tau(X)_1
  \end{equation}
   is a cover. The assumption implies that the two \(\arrow{X}\)-actions on \(\Hom(\Horn{2}{1}, X)\) are principal. Since the map~\eqref{chap6:eq:compostion-induced-map} is an invariant cover, it descends to a cover. This is the composition of the groupoid \(\tau(X)\). Thus \(\tau(X)\) is a groupoid in \((\Cat,\covers)\).
\end{proof}

We now apply the fundamental groupoid construction to 2\nbdash{}groupoid Kan fibrations or actions of 2\nbdash{}groupoids.
\begin{lemma}\label{chap6:lem:A-X:eqi-A=X-on-A1}
  Let \(\pi\colon A\to X\) be a 2-groupoid Kan fibration in \((\Cat,\covers)\). The equivalence relation \(\sim\) on \(A_1\) is the same as the equivalence relation given by the action of \(\arrow{X}\) on \(A_1\).
\end{lemma}
\begin{proof}
Consider the commutative diagram
\[
\xymatrix{
\arrow{A}_1 \ar[r]\ar[d]& A_1\times_{\pi,X_1,\source} \arrow{X}_1\ar[d]\ar[r] & A_0\ar[d]^{\de_0}\\
A_2\ar[r]^-{\cong}& \Hom(\horn{2}{1},A\to X)\ar[r]& A_1\rlap{\ ,}
}
\]
where \(\Hom(\horn{2}{1},A\to X)\to A_1\) is induced by \(\Simp{1}\{1,2\}\to \Horn{2}{1}\). The right square and the whole square are pullback diagrams, hence so is the left square. The condition \(\Kan!(2,1)\) implies that \( A_2\to \Hom(\horn{2}{1},A\to X)\) is an isomorphism. Thus \(\arrow{A}_1\to A_1\times_{\pi,X_1,\source} \arrow{X}_1\) is an isomorphism, and this proves the lemma.
\end{proof}

\begin{corollary}
Let \(\pi\colon A\to X\) be a 2-groupoid Kan fibration in \((\Cat,\covers)\) and let \(A\to \sk_0 N\) be acyclic. Then \(\tau(A)\) is a groupoid in \((\Cat,\covers)\).\footnote{It is straightforward to show the following statement: if \(A\to \sk_0 N\) is acyclic, then \(\tau(A)\) is isomorphic to the \v{C}ech groupoid of \(A_0\to N\).}
\end{corollary}
\begin{proof}
  Lemma~\ref{chap5:lem:shear-morphism-ismorita=left-principal} shows that the \(\arrow{X}\)-action on \(A_1\) is principal. The statement then follows from Lemmas~\ref{chap6:lem:A-X:eqi-A=X-on-A1} and~\ref{chap6:lem:fundamental-is-good}.
\end{proof}

\begin{corollary}\label{chap6:cor:princial-1-implies-pullback}
Let \(f\colon A\to X\) and \(g\colon B\to X\) be 2\nbdash{}groupoid Kan fibrations in \((\Cat,\covers)\). Suppose that \(g_0\) is a cover \textup(then Proposition~\ref{chap3:prop:pullback-Kan-fibration} implies that \(A\times_{X} B \to X\) is a 2\nbdash{}groupoid Kan fibration\textup) and that there is an acyclic fibration \(A\to \sk_0 N\). Then \(\tau(A\times_{X} B)\) is a groupoid in \((\Cat,\covers)\).
\end{corollary}
\begin{proof}
First, since \(g_0\) is a cover and the condition \(\Kan(1,0)\) holds, the composite \(B_1\to B_0\times_{f_0,X_0,\face_0} X_1\to X_1\) is a cover. Since \(A\) gives a principal bundle, \(\arrow{X}\) acts on \(A_1\) principally by Lemma~\ref{chap5:lem:shear-morphism-ismorita=left-principal}. Corollary~\ref{chap1:cor:G-bundles-P-times-Q-pullback} implies that \(\arrow{X}\) acts on \(A_1\times_{X_1} B_1\) principally, and the statement follows from Lemmas~\ref{chap6:lem:A-X:eqi-A=X-on-A1} and~\ref{chap6:lem:fundamental-is-good}.
\end{proof}

In particular, given a right principal \(X\)-\(Y\) bibundle \(\Gamma\) and a right principal \(Y\)-\(Z\) bibundle~\(\Xi\), the fundamental groupoid of \(R'_0\Gamma\times_Y C'_0\Xi\) is a groupoid in \((\Cat,\covers)\). This groupoid can be regarded as a geometric version of a 2-coend.

\section{The composition of 2-groupoid bibundles}

In this section, we compose two right principal bibundles of 2-groupoids in \((\Cat,\covers)\). Suppose that the right principal \(X\)-\(Y\) bibundle \(\Gamma\) is a simplicial object colored by white and gray, and that the right principal \(Y\)-\(Z\) bibundle \(\Xi\) is a simplicial object colored by gray and black. We will construct the desired \(X\)-\(Z\) bibundle as a simplicial object \(\Sigma\) colored by white and black. We first construct \(\Sigma_{0,0}\), \(\Sigma_{0,1}\), and \(\Sigma_{1,0}\). Then we use these data to construct the higher dimensions.

\subsection{Dimension 1}
First, the space \(\Sigma_{0,0}\) is defined by
\[
\Sigma_{0,0}\coloneqq \Gamma_{0,0}\times_{\face_1,Y_0,\face_0} \Xi_{0,0},
\]
illustrated by the left picture in Figure~\ref{chap6:fig:bibd-comp-Kan(1,0)-1-1}, where white dots are in~\(X\), gray dots are in~\(Y\), and black dots are in~\(Z\); a gray dot in the middle of white \(i\) and black \(j\) is labeled by~\(i/j\).

Since \(\Xi\) is right principal, the map \(\face_0\colon \Xi_{0,0}\to Y_0\) is a cover. Thus \(\Sigma_{0,0}\) is representable. The face map \(\face_0\colon \Sigma_{0,0}\to X_0\) is induced by \(\face_0\colon \Gamma_{0,0}\to X_0\). The face map \(\face_1\colon \Sigma_{0,0}\to Z_0\) is defined similarly. Since \(\face_0\colon\Sigma_{0,0}\to X_0\) is the composite of the two covers illustrated by the three pictures in Figure~\ref{chap6:fig:bibd-comp-Kan(1,0)-1-1}, the condition \(\Kan(1,0)[1,1]\) for \(\Sigma\) holds.
\begin{figure}[htbp]
  \centering
  \begin{tikzpicture}%
  [>=latex', hdot/.style={mydot,fill=lightgray},every label/.style={scale=0.6},scale=1.4]
  \node[mydot,label=180:\(0\)] at (-1,0) (g){};
  \node[hdot,label=90:\(0/1\)] at (0,0) (h){};
  \node[mydot,fill,label=0:\(1\)] at (1,0) (k){};
  \path[->]
  (h) edge (g)
  (k) edge (h);
  \begin{scope}[xshift=3.5cm]
  \node[mydot,label=180:\(0\)] at (-1,0) (g){};
  \node[hdot,label=90:\(0/1\)] at (0,0) (h){};
  \path[->]
  (h) edge (g);
  \end{scope}
  \begin{scope}[xshift=6cm]
  \node[mydot,label=180:\(0\)] at (-1,0) (g){};
  \end{scope}
  \end{tikzpicture}
  \caption{The Kan condition \(\Kan(1,0)[1,1]\)}\label{chap6:fig:bibd-comp-Kan(1,0)-1-1}
\end{figure}
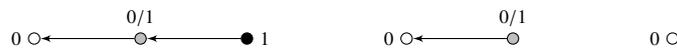

\subsection{Dimension 2}
We are going to construct the space \(\Sigma_{0,1  }\); the space \(\Sigma_{1,0}\) is defined similarly.

\begin{definition}
The space \(\Sigma_{0,1}\) is defined to be the space given by any the pictures in Figure~\ref{chap6:fig:bibd-comp-Sigma01}, where every snake line stands for the quotient by an action of an appropriate groupoid of bigons.
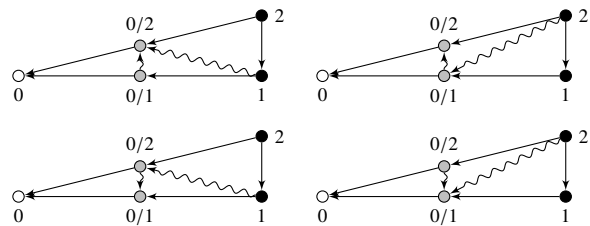
\begin{figure}[htbp]
  \centering
  \begin{tikzpicture}%
  [>=latex', hdot/.style={mydot,fill=lightgray},every label/.style={scale=0.6},scale=0.8]
  \node[mydot,label=-90:\(0\)] at (-2,0) (g0){};
  \node[hdot,label=-90:\(0/1\)] at (0,0) (h1){};
  \node[hdot,label=90:\(0/2\)] at (0,0.5) (h2){};
  \node[mydot,fill,label=-90:\(1\)] at (2,0) (k1){};
  \node[mydot,fill,label=0:\(2\)] at (2,1) (k2){};
  \path[->]
  (h1) edge (g0)
  (h2) edge (g0)
       edge[snake back] (h1)
  (k1) edge (h1)
       edge[snake arrow] (h2)
  (k2) edge (h2)
       edge (k1);
  \begin{scope}[xshift=5cm]
  \node[mydot,label=-90:\(0\)] at (-2,0) (g0){};
  \node[hdot,label=-90:\(0/1\)] at (0,0) (h1){};
  \node[hdot,label=90:\(0/2\)] at (0,0.5) (h2){};
  \node[mydot,fill,label=-90:\(1\)] at (2,0) (k1){};
  \node[mydot,fill,label=0:\(2\)] at (2,1) (k2){};
  \path[->]
  (h1) edge (g0)
  (h2) edge (g0)
       edge[snake back] (h1)
  (k1) edge (h1)
  (k2) edge (h2)
       edge (k1)
       edge[snake arrow] (h1);
  \end{scope}
  \begin{scope}[yshift=-2cm]
  \node[mydot,label=-90:\(0\)] at (-2,0) (g0){};
  \node[hdot,label=-90:\(0/1\)] at (0,0) (h1){};
  \node[hdot,label=90:\(0/2\)] at (0,0.5) (h2){};
  \node[mydot,fill,label=-90:\(1\)] at (2,0) (k1){};
  \node[mydot,fill,label=0:\(2\)] at (2,1) (k2){};
  \path[->]
  (h1) edge (g0)
  (h2) edge (g0)
       edge[snake arrow] (h1)
  (k1) edge (h1)
       edge[snake arrow] (h2)
  (k2) edge (h2)
       edge (k1);
  \end{scope}
  \begin{scope}[xshift=5cm,yshift=-2cm]
  \node[mydot,label=-90:\(0\)] at (-2,0) (g0){};
  \node[hdot,label=-90:\(0/1\)] at (0,0) (h1){};
  \node[hdot,label=90:\(0/2\)] at (0,0.5) (h2){};
  \node[mydot,fill,label=-90:\(1\)] at (2,0) (k1){};
  \node[mydot,fill,label=0:\(2\)] at (2,1) (k2){};
  \path[->]
  (h1) edge (g0)
  (h2) edge (g0)
       edge[snake arrow] (h1)
  (k1) edge (h1)
  (k2) edge (h2)
       edge (k1)
       edge[snake arrow] (h1);
  \end{scope}
  \end{tikzpicture}
  \caption{The constructions of \(\Sigma_{0,1}\)}\label{chap6:fig:bibd-comp-Sigma01}
\end{figure}
\end{definition}

Let us denote the top left picture by TL and the corresponding space by \(\Sigma_{01}^{TL}\). The symbols TR, \(\Sigma_{0,1}^{TR}\), BL, \(\Sigma_{0,1}^{BL}\), and BR, \(\Sigma_{0,1}^{BR}\) have a similar meaning.

For every picture, we have two actions given by the two snake arrows. Since these two actions commute, the order of taking quotients does not matter.

\begin{lemma}
  The space given by every picture in Figure~\ref{chap6:fig:bibd-comp-Sigma01} is corepresentable.
\end{lemma}
\begin{proof}
We prove the statement for TL; similar arguments work for the other cases. The space can be constructed in two steps illustrated by Figure~\ref{chap6:fig:bibd-comp-Sigma01-rep}. Each time we have two actions, one of which is principal and the other has moment map being a cover. Corollary~\ref{chap1:cor:G-bundles-P-times-Q-pullback} implies that the quotient spaces are corepresentable.
\begin{figure}[htbp]
  \centering
  \begin{tikzpicture}%
  [>=latex', hdot/.style={mydot,fill=lightgray},every label/.style={scale=0.6},scale=0.8]
  \node[hdot,label=-90:\(0/1\)] at (0,0) (h1){};
  \node[hdot,label=90:\(0/2\)] at (0,0.5) (h2){};
  \node[mydot,fill,label=-90:\(1\)] at (2,0) (k1){};
  \node[mydot,fill,label=0:\(2\)] at (2,1) (k2){};
  \path[->]
  (h2) 
       edge[<-] (h1)
  (k1) edge (h1)
       edge[snake arrow] (h2)
  (k2) edge (h2)
       edge (k1);
  \begin{scope}[xshift=5cm]
  \node[mydot,label=-90:\(0\)] at (-2,0) (g0){};
  \node[hdot,label=-90:\(0/1\)] at (0,0) (h1){};
  \node[hdot,label=90:\(0/2\)] at (0,0.5) (h2){};
  \node[mydot,fill,label=-90:\(1\)] at (2,0) (k1){};
  \node[mydot,fill,label=0:\(2\)] at (2,1) (k2){};
  \path[->]
  (h1) edge (g0)
  (h2) edge (g0)
       edge[snake back] (h1)
  (k1) edge (h1)
       edge[snake arrow] (h2)
  (k2) edge (h2)
       edge (k1);
  \end{scope}
  \end{tikzpicture}
  \caption{Corepresentability of \(\Sigma_{0,1}\)}\label{chap6:fig:bibd-comp-Sigma01-rep}
\end{figure}
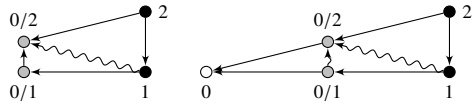
\end{proof}

The lemma below verifies that \(\Sigma_{0,1}\) is well-defined. In the sequel, we will mostly use TL unless otherwise specified.

\begin{lemma}\label{chap6:lem:4-iso-Sigma-0-1}
There is a commutative square
\[
\xymatrix{
\Sigma_{0,1}^{TL}\ar[r]\ar[d] & \Sigma_{0,1}^{TR}\ar[d]\\
\Sigma_{0,1}^{BL}\ar[r] & \Sigma_{0,1}^{BR}\rlap{\ ,}
}
\]
where every morphism is an isomorphism.
\end{lemma}
\begin{proof}
Recall from Lemma~\ref{chap6:lem:HS-bibundle-iso-alpha} that the spaces given by the pictures in Figure~\ref{chap6:fig:Xi01-Xi10-quotients-iso} are isomorphic. Thus the spaces given by TL and TR are isomorphic. Similarly, we have an isomorphism \(\Sigma_{0,1}^{BL}\cong\Sigma_{0,1}^{BR}\).
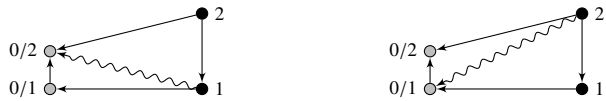
\begin{figure}[htbp]
  \centering
  \begin{tikzpicture}%
  [>=latex', hdot/.style={mydot,fill=lightgray},every label/.style={scale=0.6}]
  \begin{scope}
  \node[hdot,label=180:\(0/1\)] at (0,0) (h1){};
  \node[hdot,label=180:\(0/2\)] at (0,0.5) (h2){};
  \node[mydot,fill,label=0:\(1\)] at (2,0) (k1){};
  \node[mydot,fill,label=0:\(2\)] at (2,1) (k2){};
  \path[->]
  (h2) edge[<-] (h1)
  (k1) edge (h1)
       edge[snake arrow] (h2)
  (k2) edge (h2)
       edge (k1);
  \end{scope}
  \begin{scope}[xshift=5cm]
  \node[hdot,label=180:\(0/1\)] at (0,0) (h1){};
  \node[hdot,label=180:\(0/2\)] at (0,0.5) (h2){};
  \node[mydot,fill,label=0:\(1\)] at (2,0) (k1){};
  \node[mydot,fill,label=0:\(2\)] at (2,1) (k2){};
  \path[->]
  (h2) edge[<-] (h1)
  (k1) edge (h1)
  (k2) edge (h2)
       edge (k1)
       edge[snake arrow] (h1);
  \end{scope}
  \end{tikzpicture}
  \caption{Two isomorphic spaces}\label{chap6:fig:Xi01-Xi10-quotients-iso}
\end{figure}

Example~\ref{chap6:exa:opposite-action} implies that the spaces given by the three pictures in Figure~\ref{chap6:fig:Sigma-bigons-iso} are isomorphic. This gives an isomorphism \(\Sigma_{0,1}^{TL}\cong \Sigma_{0,1}^{BL}\). The isomorphism \(\Sigma_{0,1}^{TR}\cong\Sigma_{0,1}^{BR}\) follows similarly.
\begin{figure}[htbp]
  \centering
  \begin{tikzpicture}%
  [>=latex', hdot/.style={mydot,fill=lightgray},every label/.style={scale=0.6},scale=0.8]
  \begin{scope}[xshift=-5cm,yshift=-2cm]
  \node[mydot,label=90:\(0\)] at (-2,0) (g0){};
  \node[hdot,label=-90:\(0/1\)] at (0,0) (h1){};
  \node[hdot,label=90:\(0/2\)] at (0,0.5) (h2){};
  \node[mydot,fill,label=0:\(1\)] at (2,0) (k1){};
  \path[->]
  (h1) edge (g0)
  (h2) edge (g0)
       edge[snake back] (h1)
  (k1) edge (h1)
       edge (h2);
  \end{scope}
  \begin{scope}[xshift=+5cm, yshift=-2cm]
  \node[mydot,label=90:\(0\)] at (-2,0) (g0){};
  \node[hdot,label=-90:\(0/1\)] at (0,0) (h1){};
  \node[hdot,label=90:\(0/2\)] at (0,0.5) (h2){};
  \node[mydot,fill,label=0:\(1\)] at (2,0) (k1){};
  \path[->]
  (h1) edge (g0)
  (h2) edge (g0)
       edge[snake arrow] (h1)
  (k1) edge (h1)
       edge (h2);
  \end{scope}
  \begin{scope}[yshift=-2cm]
  \node[mydot,label=90:\(0\)] at (-2,0) (g0){};
  \node[hdot,label=-90:\(0/1\)] at (0,0) (h1){};
  \node[hdot,label=90:\(0/2\)] at (0,0.5) (h2){};
  \node[mydot,fill,label=0:\(1\)] at (2,0) (k1){};
  \path[->]
  (h1) edge (g0)
  (h2) edge (g0)
       edge[snake arrow,out=-60, in=60] (h1)
       edge[snake back,out=-120, in=120](h1)
  (k1) edge (h1)
       edge (h2);
  \end{scope}
  \end{tikzpicture}
  \caption{Three isomorphic spaces}\label{chap6:fig:Sigma-bigons-iso}
\end{figure}
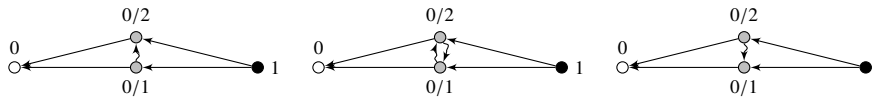

It remains to show that these four isomorphisms form a commutative square. Replace BL and BR by the pictures in Figure~\ref{chap6:fig:BL'-BR'} denoted by BL' and BR', respectively. We now choose representatives for the pictures TR, TL, BR', BL' such that they respect the isomorphisms
\[
  \Sigma_{0,1}^{TL} \cong \Sigma_{0,1}^{BL},\quad \Sigma_{0,1}^{BL} \cong \Sigma_{0,1}^{BR},\quad \Sigma_{0,1}^{BR} \cong \Sigma_{0,1}^{TR}.
\]
We need to show that they respect the remaining isomorphism \(\Sigma_{0,1}^{TL} \cong \Sigma_{0,1}^{TR}\). This is a variant of Lemma~\ref{chap6:lem:bigon-triangle}, and we are done.
\begin{figure}[htbp]
  \centering
  \begin{tikzpicture}%
  [>=latex', hdot/.style={mydot,fill=lightgray},every label/.style={scale=0.6},scale=0.8]
  \node[mydot,label=-90:\(0\)] at (-2,0) (g0){};
  \node[hdot,label=-90:\(0/1\)] at (0,0) (h1){};
  \node[hdot,label=90:\(0/2\)] at (0,0.5) (h2){};
  \node[mydot,fill,label=-90:\(1\)] at (2,0) (k1){};
  \node[mydot,fill,label=0:\(2\)] at (2,1) (k2){};
  \path[->]
  (h1) edge (g0)
  (h2) edge (g0)
       edge[snake arrow,out=-60, in=60] (h1)
       edge[snake back,out=-120, in=120](h1)
  (k1) edge (h1)
       edge[snake arrow]  (h2)
  (k2) edge (k1)
       edge (h2);
  \begin{scope}[xshift=5cm]
  \node[mydot,label=-90:\(0\)] at (-2,0) (g0){};
  \node[hdot,label=-90:\(0/1\)] at (0,0) (h1){};
  \node[hdot,label=90:\(0/2\)] at (0,0.5) (h2){};
  \node[mydot,fill,label=-90:\(1\)] at (2,0) (k1){};
  \node[mydot,fill,label=0:\(2\)] at (2,1) (k2){};
  \path[->]
  (h1) edge (g0)
  (h2) edge (g0)
       edge[snake arrow,out=-60, in=60] (h1)
       edge[snake back,out=-120, in=120](h1)
  (k1) edge (h1)
  (k2) edge (k1)
       edge (h2)
       edge[snake arrow]  (h1);
  \end{scope}
  \end{tikzpicture}
  \caption{BL' and BR'}\label{chap6:fig:BL'-BR'}
\end{figure}
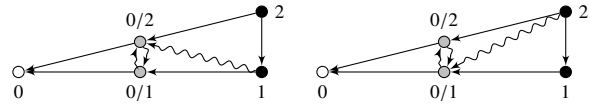
\end{proof}

\begin{remark}
  The construction of \(\Sigma_{0,1}\) is inspired by the composition of HS bibundles of groupoids. A triangle in \(\Sigma_{0,1}\) can be thought of as a formal composite of three smaller triangles in \(\Gamma\) and \(\Xi\). In view of the pasting theorem, the composite should be invariant under the actions of groupoids of bigons.
\end{remark}

The face maps \(\face_1,\face_2\colon \Sigma_{1,0}\to \Sigma_{0,0}\) and \(\face_0\colon \Sigma_{1,0}\to Z_1\) are natural to define. The degeneracy map \(\de_1\colon\Sigma_{0,0}\to \Sigma_{0,1}\) is illustrated by Figure~\ref{chap6:fig:bibd-comp-s1}. We first send the left picture \(\Sigma_{0,0}\) to the right picture using three degeneracy maps (the two shorter vertical arrows are degenerate). Composing with the quotient maps, we obtain \(\de_1\). The degeneracy map \(\de_0\colon \Sigma_{0,0}\to \Sigma_{1,0}\) is defined similarly.
\begin{figure}[htbp]
  \centering
  \begin{tikzpicture}%
  [>=latex',hdot/.style={mydot,fill=lightgray},every label/.style={scale=0.6},scale=0.8]
  \node[mydot,label=180:\(0\)] at (-2,0) (g0){};
  \node[hdot,label=-90:\(0/1\)] at (0,0) (h1){};
  \node[mydot,fill,label=0:\(1\)] at (2,0) (k1){};
  \path[->]
  (h1) edge (g0)
  (k1) edge (h1);
  \begin{scope}[xshift=5cm]
  \node[mydot,label=180:\(0\)] at (-2,0) (g0){};
  \node[hdot,label=-90:\(0/1\)] at (0,0) (h1){};
  \node[hdot,,label=90:\(0/2\)] at (0,0.2) (h2){};
  \node[mydot,fill,label=0:\(1\)] at (2,0) (k1){};
  \node[mydot,fill,label=0:\(2\)] at (2,0.4) (k2){};
  \path[->]
  (h1) edge (g0)
  (h2) edge (g0)
       edge[<-] (h1)
  (k1) edge (h1)
       edge (h2)
  (k2) edge (h2)
       edge (k1);
  \end{scope}
  \end{tikzpicture}
  \caption{The degeneracy map \(\de_1\colon\Sigma_{0,0}\to \Sigma_{0,1}\)}\label{chap6:fig:bibd-comp-s1}
\end{figure}

These face and degeneracy maps satisfy the desired simplicial identities. We can also use any other picture in Figure~\ref{chap6:fig:bibd-comp-Sigma01} to define face and degeneracy maps, and the results are compatible with the isomorphisms in Lemma~\ref{chap6:lem:4-iso-Sigma-0-1}.

\begin{lemma}\label{chap6:lem:Sigma-Kan(2,x)}
The Kan conditions \(\Kan(2,1)[1,2], \Kan(2,2)[1, 2]\) and \(\Kan(2,0)[1,2]\) for~\(\Sigma\) hold. Similarly, \(\Kan(2,1)[2,1]\) and \(\Kan(2,2)[2,1]\) hold.
\end{lemma}
\begin{proof}
  To show \(\Kan(2,1)[1, 2]\), we consider the four pictures in Figure~\ref{chap6:fig:bibd-comp-Kan(2,1)-1-2}. Let \(\tilde{\Sigma}_{0,1}\) be the space of three triangles fitting together as the left-most picture. Using the Kan conditions for \(\Gamma\) and \(\Xi\), we see that the canonical map from one picture to its right neighbour is a cover. Hence the map \(\tilde{\Sigma}_{0,1}\to \Hom(\Horn{2}{1}[1,2],\Sigma)\) is a cover. Lemma~\ref{chap1:lem:invariant-cover-descent} implies that it descends to a cover after taking quotients. This proves \(\Kan(2,1)[1,2]\).
  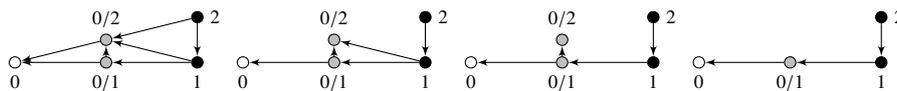
\begin{figure}[htbp]
  \centering
  \begin{tikzpicture}%
  [scale=0.6,>=latex',hdot/.style={mydot,fill=lightgray},every label/.style={scale=0.6}]
  \node[mydot,label=-90:\(0\)] at (-2,0) (g0){};
  \node[hdot,label=-90:\(0/1\)] at (0,0) (h1){};
  \node[hdot,label=90:\(0/2\)] at (0,0.5) (h2){};
  \node[mydot,fill,label=-90:\(1\)] at (2,0) (k1){};
  \node[mydot,fill,label=0:\(2\)] at (2,1) (k2){};
  \path[->]
  (h1) edge (g0)
  (h2) edge (g0)
       edge[<-] (h1)
  (k1) edge (h1)
       edge (h2)
  (k2) edge (h2)
       edge (k1);
  \begin{scope}[xshift=5cm]
  \node[mydot,label=-90:\(0\)] at (-2,0) (g0){};
  \node[hdot,label=-90:\(0/1\)] at (0,0) (h1){};
  \node[hdot,label=90:\(0/2\)] at (0,0.5) (h2){};
  \node[mydot,fill,label=-90:\(1\)] at (2,0) (k1){};
  \node[mydot,fill,label=0:\(2\)] at (2,1) (k2){};
  \path[->]
  (h1) edge (g0)
  (h2) edge[<-] (h1)
  (k1) edge (h1)
       edge (h2)
  (k2) edge (k1);
  \end{scope}
  \begin{scope}[xshift=10cm]
  \node[mydot,label=-90:\(0\)] at (-2,0) (g0){};
  \node[hdot,label=-90:\(0/1\)] at (0,0) (h1){};
  \node[hdot,label=90:\(0/2\)] at (0,0.5) (h2){};
  \node[mydot,fill,label=-90:\(1\)] at (2,0) (k1){};
  \node[mydot,fill,label=0:\(2\)] at (2,1) (k2){};
  \path[->]
  (h1) edge (g0)
  (h2) edge[<-] (h1)
  (k1) edge (h1)
  (k2) edge (k1);
  \end{scope}
  \begin{scope}[xshift=15cm]
  \node[mydot,label=-90:\(0\)] at (-2,0) (g0){};
  \node[hdot,label=-90:\(0/1\)] at (0,0) (h1){};
  \node[mydot,fill,label=-90:\(1\)] at (2,0) (k1){};
  \node[mydot,fill,label=0:\(2\)] at (2,1) (k2){};
  \path[->]
  (h1) edge (g0)
  (k1) edge (h1)
  (k2) edge (k1);
  \end{scope}
  \end{tikzpicture}
  \caption{Proof of \(\Kan(2,1)[1,2]\)}\label{chap6:fig:bibd-comp-Kan(2,1)-1-2}
\end{figure}

The other Kan conditions can be proved similarly. For instance, to prove \(\Kan(2,0)[1,2]\) we consider Figure~\ref{chap6:fig:bibd-comp-Kan(2,0)-1-2}.
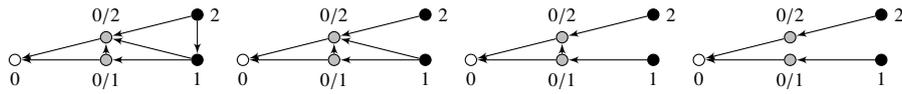
\begin{figure}[htbp]
  \centering
  \begin{tikzpicture}%
  [scale=0.6,>=latex',hdot/.style={mydot,fill=lightgray},every label/.style={scale=0.6}]
  \node[mydot,label=-90:\(0\)] at (-2,0) (g0){};
  \node[hdot,label=-90:\(0/1\)] at (0,0) (h1){};
  \node[hdot,label=90:\(0/2\)] at (0,0.5) (h2){};
  \node[mydot,fill,label=-90:\(1\)] at (2,0) (k1){};
  \node[mydot,fill,label=0:\(2\)] at (2,1) (k2){};
  \path[->]
  (h1) edge (g0)
  (h2) edge (g0)
       edge[<-] (h1)
  (k1) edge (h1)
       edge (h2)
  (k2) edge (h2)
       edge (k1);
  \begin{scope}[xshift=5cm]
  \node[mydot,label=-90:\(0\)] at (-2,0) (g0){};
  \node[hdot,label=-90:\(0/1\)] at (0,0) (h1){};
  \node[hdot,label=90:\(0/2\)] at (0,0.5) (h2){};
  \node[mydot,fill,label=-90:\(1\)] at (2,0) (k1){};
  \node[mydot,fill,label=0:\(2\)] at (2,1) (k2){};
  \path[->]
  (h1) edge (g0)
  (h2) edge (g0)
       edge[<-] (h1)
  (k1) edge (h1)
       edge (h2)
  (k2) edge (h2);
  \end{scope}
  \begin{scope}[xshift=10cm]
  \node[mydot,label=-90:\(0\)] at (-2,0) (g0){};
  \node[hdot,label=-90:\(0/1\)] at (0,0) (h1){};
  \node[hdot,label=90:\(0/2\)] at (0,0.5) (h2){};
  \node[mydot,fill,label=-90:\(1\)] at (2,0) (k1){};
  \node[mydot,fill,label=0:\(2\)] at (2,1) (k2){};
  \path[->]
  (h1) edge (g0)
  (h2) edge (g0)
       edge[<-] (h1)
  (k1) edge (h1)
  (k2) edge (h2);
  \end{scope}
  \begin{scope}[xshift=15cm]
  \node[mydot,label=-90:\(0\)] at (-2,0) (g0){};
  \node[hdot,label=-90:\(0/1\)] at (0,0) (h1){};
  \node[hdot,label=90:\(0/2\)] at (0,0.5) (h2){};
  \node[mydot,fill,label=-90:\(1\)] at (2,0) (k1){};
  \node[mydot,fill,label=0:\(2\)] at (2,1) (k2){};
  \path[->]
  (h1) edge (g0)
  (h2) edge (g0)
  (k1) edge (h1)
  (k2) edge (h2);
  \end{scope}
  \end{tikzpicture}
  \caption{Proof of \(\Kan(2,0)[1,2]\)}\label{chap6:fig:bibd-comp-Kan(2,0)-1-2}
\end{figure}
\end{proof}

\subsection{Dimension 3}

We first construct \(\Sigma_{0,2}\), and verify various simplicial identities and Kan conditions. A symmetric argument works for \(\Sigma_{2,0}\). We then consider \(\Sigma_{1,1}\), which is slightly different.

\begin{definition}
Let \(\tilde{\Sigma}_{0,2}\) be the space of all configurations of four tetrahedra fitting together as in Figure~\ref{chap6:fig:pre-Sigma-0-2}.
\begin{figure}[htbp]
  \centering
  \begin{tikzpicture}%
  [>=latex',
  hdot/.style={mydot,fill=lightgray}, kdot/.style={mydot,fill},
  every label/.style={scale=0.6}]
  \node[mydot,label=180:\(0\)] at (-2,0) (g0){};
  \node[hdot,label=-90:\(0/1\)] at (0,0) (h1){};
  \node[hdot,label=60:\(0/2\)] at (0.5,0.5) (h2){};
  \node[hdot,label=120:\(0/3\)] at (0,1) (h3){};
  \node[kdot,label=-90:\(1\)] at (2,0) (k1){};
  \node[kdot,label=60:\(2\)] at (3,1) (k2){};
  \node[kdot,label=120:\(3\)] at (2,2) (k3){};
  \foreach \i/\j in
  {h1/g0,h2/g0,h3/g0,k1/h1,k2/h2,k3/h3,
  k2/k1,k3/k1,k3/k2,h1/h2,h2/h3,h1/h3,
  k1/h2,k2/h3,k1/h3}{
  \path[->](\i) edge (\j);
  }
  \end{tikzpicture}
  \caption{The construction of \(\tilde{\Sigma}_{0,2}\)}\label{chap6:fig:pre-Sigma-0-2}
\end{figure}
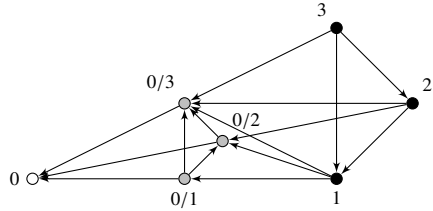
\end{definition}

The space \(\Sigma_{0,2}\) will be a quotient space of \(\tilde{\Sigma}_{0,2}\).

\begin{lemma}\label{chap6:lem:pre-Sigma-0-2-rep}
The space \(\tilde{\Sigma}_{0,2}\) is representable.
\end{lemma}
\begin{proof}
  Remove from \(\tilde{\Sigma}_{0,2}\) the triangle \((0,0/3,0/1)\) together with the unique horn associated with it. In this manner, we remove
  \[
  (0,0/3,0/1),\quad (0/3, 0/1, 1),\quad (0/3,1,3),\quad (0/3,1,2).
  \]
 We denote the resulting space by \(\tilde{\Sigma}_{0,2}'\), which consists of only triangles. Further removing the triangles \((0/3,0/2,0/1)\) and \((0/3, 0/2, 1)\) from \(\tilde{\Sigma}_{0,2}'\), we denote the resulting space by~\(\tilde{\Sigma}_{0,2}''\). The space \(\tilde{\Sigma}_{0,2}''\) is a un-quotiented ``big horn''. The proof of Lemma~\ref{chap6:lem:Sigma-Kan(2,x)} implies that \(\tilde{\Sigma}_{0,2}''\) is representable. It is routine to verify that, for each horn we removed from \(\tilde{\Sigma}_{0,2}\), the corresponding Kan condition holds. Therefore, \(\tilde{\Sigma}_{0,2}\) is representable and \(\tilde{\Sigma}_{0,2}\to \tilde{\Sigma}_{0,2}''\) is a cover.
\end{proof}

\begin{definition}
  Let \(\Sigma_{0,2}\) be the quotient space of \(\tilde{\Sigma}_{0,2}\) illustrated by Figure~\ref{chap6:fig:Sigma-0-2}, where every snake line stands for a quotient by an action of a groupoid of bigons (it is well-defined by Lemma~\ref{chap6:lem:bigon-triangle}).
\begin{figure}[htbp]
  \centering
  \begin{tikzpicture}%
  [>=latex',
  hdot/.style={mydot,fill=lightgray}, kdot/.style={mydot,fill},
  every label/.style={scale=0.6}]
  \node[mydot,label=180:\(0\)] at (-2,0) (g0){};
  \node[hdot,label=-90:\(0/1\)] at (0,0) (h1){};
  \node[hdot,label=60:\(0/2\)] at (0.5,0.5) (h2){};
  \node[hdot,label=120:\(0/3\)] at (0,1) (h3){};
  \node[kdot,label=-90:\(1\)] at (2,0) (k1){};
  \node[kdot,label=60:\(2\)] at (3,1) (k2){};
  \node[kdot,label=120:\(3\)] at (2,2) (k3){};
  \foreach \i/\j in
  {h1/g0,h2/g0,h3/g0,k1/h1,k2/h2,k3/h3,
  k2/k1,k3/k1,k3/k2}{
  \path[->](\i) edge (\j);
  }
  \foreach \i/\j in
  {h1/h2,h2/h3,h1/h3,k1/h2,k2/h3,k1/h3}{
  \path[->] (\i) edge[snake arrow] (\j);
  }
  \end{tikzpicture}
  \caption{The construction of \(\Sigma_{0,2}\)}\label{chap6:fig:Sigma-0-2}
\end{figure}
\end{definition}

\begin{lemma}\label{chap6:lem:Sigma-0-2-rep-Kan}
  The space \(\Sigma_{0,2}\) is corepresentable, and \(\Kan!(3,i)[1,3]\) holds for \(0\le i\le 3\).
\end{lemma}
\begin{proof}
  We show the condition \(\Kan!(3,2)[1,3]\); the other cases are similar. Consider the canonical maps in the proof of Lemma~\ref{chap6:lem:pre-Sigma-0-2-rep}:
  \[
  \tilde{\Sigma}_{0,2}\xrightarrow{\cong} \tilde{\Sigma}_{0,2}'\to \tilde{\Sigma}_{0,2}''.
  \]
  The last map is an equivariant cover, so it descends to a cover by Lemma~\ref{chap1:lem:action-over-principal-is-principal}. Observe that \(\tilde{\Sigma}_{0,2}/{\sim}\cong\tilde{\Sigma}_{0,2}'/{\sim} \cong \tilde{\Sigma}_{0,2}''/{\sim}\) and \(\tilde{\Sigma}_{0,2}''/{\sim}\cong \Hom(\Horn{3}{2}[1,3], \Sigma)\). It follows that \(\Sigma_{0,2}\) is corepresentable and that \(\Kan!(3,2)[1,3]\) holds.
\end{proof}

\begin{remark}
  The proof tells us that, given a horn \(\Horn{3}{2}[1,3]\) in \(\Sigma\), we can choose any representative and add arbitrary auxiliary simplices \((0/1,0/2,0/3)\) and \((0/2,0/3,1)\) to obtain the missing faces \((0,0/1,0/3)\), \((0/1,0/3,1)\), and \((0/3,1,3)\). After taking quotients, the result is independent of the choices.
\end{remark}

The face and degeneracy maps \(\de_1, \de_2\colon \Sigma_{0,1}\to \Sigma_{0,2}\) are natural to define. For instance, to define \(\de_2\), we first define a map \(\tilde{\de}_2\colon\tilde{\Sigma}_{0,1}\to \tilde{\Sigma}_{0,2}\) by degeneracy maps as in Figure~\ref{chap6:fig:Sigma-0-1-s2-Sigma-0-2}. This map is equivariant with respect to the actions, so it descends to a map \(\de_{2}\colon \Sigma_{0,1}\to \Sigma_{0,2}\). The simplicial identities hold by construction.
\begin{figure}[htbp]
  \centering
  \begin{tikzpicture}%
  [>=latex',
  hdot/.style={mydot,fill=lightgray}, kdot/.style={mydot,fill},
  every label/.style={scale=0.6}]
  \node[mydot,label=180:\(0\)] at (-2,0) (g0){};
  \node[hdot,label=-90:\(0/1\)] at (0,0) (h1){};
  \node[hdot,label=-60:\(0/2\)] at (0.5,0.5) (h2){};
  \node[hdot,label=120:\(0/3\)] at (0.5,0.7) (h3){};
  \node[kdot,label=-90:\(1\)] at (2,0) (k1){};
  \node[kdot,label=60:\(2\)] at (3,1) (k2){};
  \node[kdot,label=120:\(3\)] at (3,1.4) (k3){};
  \foreach \i/\j in
  {h1/g0,h2/g0,h3/g0,k1/h1,k2/h2,k3/h3,
  k2/k1,k3/k1,k3/k2,h1/h2,h2/h3,h1/h3,
  k1/h2,k2/h3,k1/h3}{
  \path[->](\i) edge (\j);
  }
  \end{tikzpicture}
  \caption{``Degeneracy'' map \(\tilde{\de}_2\colon\tilde{\Sigma}_{0,1}\to \tilde{\Sigma}_{0,2}\)}\label{chap6:fig:Sigma-0-1-s2-Sigma-0-2}
\end{figure}

Like \(\Sigma_{0,1}\), we may use other decompositions to define~\(\tilde{\Sigma}_{0,2}\) hence~\(\Sigma_{0,2}\). It is routine to verify that, for every other decomposition, we also have \(\Kan!(3,i)[1,3]\) for \(0\le i<1\). We claim that the result is independent of the choice.

First, for arrows in \(Y\) we may choose other direction. In fact, applying a suitable variant of Lemma~\ref{chap6:lem:bigon-triangle}, we can invert arrows in \(Y\).

Consider a map \(\partial\Simp{3}[1,3]\to \Sigma\), consisting of four faces, in two different decompositions. Each pair of faces may be decomposed differently but respects the isomorphisms in Lemma~\ref{chap6:lem:4-iso-Sigma-0-1}. Suppose that for certain decomposition these four faces are compatible that they respect the Kan conditions in Lemma~\ref{chap6:lem:Sigma-0-2-rep-Kan}, then we shall show that they are compatible for the other decomposition. Choosing appropriate representatives, we may arrange that the tetrahedra \((0,0/1,0/2,0/3)\) in \(\Gamma\) are the same, so we need only consider the remaining part in \(\Xi\). Without loss of generality, we may suppose that the two decompositions only differ on one face. We show below for one case; the other cases can be proved similarly.

\begin{lemma}\label{chap6:lem:Sigma-0-2-var}
Consider the two sets of eight triangles in \(\Xi\) as in Figure~\ref{chap6:fig:pre-Sigma-0-2-rep-var}. The left picture uses the standard decomposition, and the right differs only by one face. Suppose that the two decompositions are related by isomorphisms as in Lemma~\ref{chap6:lem:4-iso-Sigma-0-1}. If the eight triangles in the standard decomposition respect the Kan conditions in Lemma~\ref{chap6:lem:Sigma-0-2-rep-Kan}, then so do the eight triangles in the other decomposition.
\begin{figure}[htbp]
  \centering
  \begin{tikzpicture}%
  [>=latex',
  hdot/.style={mydot,fill=lightgray}, kdot/.style={mydot,fill},
  every label/.style={scale=0.6}]
  \node[hdot,label=-90:\(0/1\)] at (0,0) (h1){};
  \node[hdot,label=60:\(0/2\)] at (0.5,0.5) (h2){};
  \node[hdot,label=120:\(0/3\)] at (0,1) (h3){};
  \node[kdot,label=-90:\(1\)] at (2,0) (k1){};
  \node[kdot,label=60:\(2\)] at (3,1) (k2){};
  \node[kdot,label=120:\(3\)] at (2,2) (k3){};
  \foreach \i/\j in
  {
  k1/h1,k2/h2,k3/h3,
  k2/k1,k3/k1,k3/k2,h1/h2,h2/h3,h1/h3,
  k1/h2,k2/h3,k1/h3}{
  \path[->](\i) edge (\j);
  }
  \begin{scope}[xshift=5cm]
  \node[hdot,label=-90:\(0/1\)] at (0,0) (h1){};
  \node[hdot,label=60:\(0/2\)] at (0.5,0.5) (h2){};
  \node[hdot,label=120:\(0/3\)] at (0,1) (h3){};
  \node[kdot,label=-90:\(1\)] at (2,0) (k1){};
  \node[kdot,label=60:\(2\)] at (3,1) (k2){};
  \node[kdot,label=120:\(3\)] at (2,2) (k3){};
  \foreach \i/\j in
  {
  k1/h1,k2/h2,k3/h3,
  k2/k1,k3/k1,k3/k2,h1/h2,h2/h3,h1/h3,
  k2/h1,k2/h3,k1/h3}{
  \path[->](\i) edge (\j);
  }
  \end{scope}
  \end{tikzpicture}
  \caption{Two different decompositions}\label{chap6:fig:pre-Sigma-0-2-rep-var}
\end{figure}
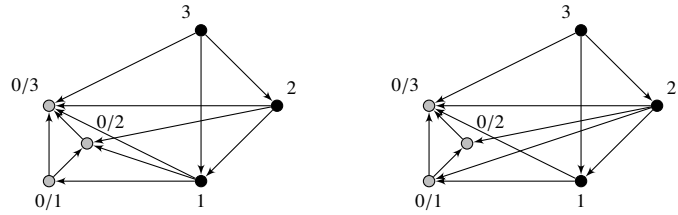
\end{lemma}
\begin{proof}
Two pairs of triangles on a face are related by isomorphisms as in Lemma~\ref{chap6:lem:4-iso-Sigma-0-1} if and only if their pasting composites are equal. The two pictures in Figure~\ref{chap6:fig:pre-Sigma-0-2-rep-var} can be interpreted as expressing \((0/2,0/3,2,3)\) as the pasting composites of other faces on the boundary. The claim then follows from the pasting theorem.
\end{proof}

We now construct \(\Sigma_{1,1}\); the idea is the same.

\begin{definition}
 The space \(\Sigma_{1,1}\) is defined by Figure~\ref{chap6:fig:Sigma-1-1}, which has six tetrahedra fitting together.
\begin{figure}[htbp]
  \centering
  \begin{tikzpicture}%
  [>=latex',
  hdot/.style={mydot,fill=lightgray}, kdot/.style={mydot,fill},
  every label/.style={scale=0.6}]
  \node[mydot,label=180:\(0\)]  at (-3,0)(g0){};
  \node[mydot,label=-120:\(1\)] at (1,-1)(g1){};
  \node[kdot,label=-60:\(2\)]   at (3,0) (k2){};
  \node[kdot,label=0:\(3\)]     at (2,3) (k3){};
  \node[hdot,label=120:\(0/3\)] at (-0.5,1.5) (h03){};
  \node[hdot,label=-120:\(0/2\)] at (0,0) (h02){};
  \node[hdot,label=-60:\(1/2\)] at (2,-0.5)(h12){};
  \node[hdot,label=30:\(1/3\)] at (1.5,1) (h13){};
  \foreach \i/\j in
  {g1/g0,h02/g0,k2/h02,h03/g0,k3/h03,
  h12/g1,k2/h12,h13/g1,k3/h13,k3/k2}{
  \path[->] (\i) edge (\j);
  }
  \foreach \i/\j in
  {h02/g1,h12/h02,h03/g1,h13/h03,
  h12/h13,k2/h13,h02/h03,k2/h03,h13/h02}{
  \path[->] (\i) edge[snake arrow] (\j);
  }
\end{tikzpicture}
\caption{The construction of \(\Sigma_{1,1}\)}\label{chap6:fig:Sigma-1-1}
\end{figure}
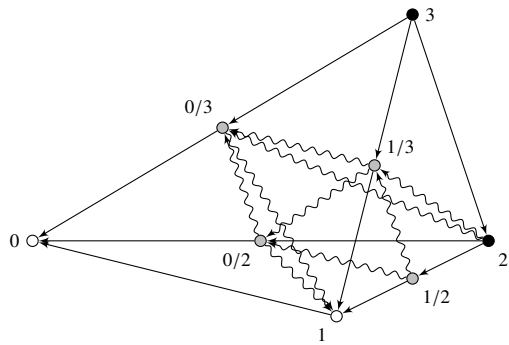
\end{definition}

The lemma below is similar to Lemma~\ref{chap6:lem:Sigma-0-2-rep-Kan}.
\begin{lemma}
The space \(\Sigma_{1,1}\) is corepresentable and \(\Kan(3,i)[2,2]\) holds for \(0\le i\le 3\).
\end{lemma}
\begin{proof}
Denote the un-quotiented space by \(\tilde{\Sigma}_{1,1}\). Removing from \(\tilde{\Sigma}_{1,1}\) the following six triangles and the associated horns in order:
\begin{gather*}
 (0,1,0/3),\quad (1,0/3,1/3),\quad (1,0/2,1/2),\\
 (0/3,1/3,3),\quad (0/3,1/3,2), \quad (0/2,1/2,1/3),
\end{gather*}
we denote the resulting space by \(\tilde{\Sigma}_{1,1}'\). Removing from \(\tilde{\Sigma}_{1,1}'\) the following triangles
\[
  (1,0/2,0/3), \quad (0/2,0/3,1/3),\quad (0/2,1/2,1/3),
\]
we denote the resulting space by \(\tilde{\Sigma}_{1,1}''\). It is routine to verify that \(\tilde{\Sigma}_{1,1}\to\tilde{\Sigma}_{1,1}''\) is a cover, and we get isomorphisms in \(\Cat\)
\[
\tilde{\Sigma}_{1,1}/{\sim}\cong \tilde{\Sigma}_{1,1}''/{\sim}\cong\Hom(\Horn{3}{2}[2,2],\Sigma).
\]
This shows that \(\Sigma_{1,1}=\tilde{\Sigma}_{1,1}/{\sim}\) is corepresentable and that \(\Kan(3,2)[2,2]\) holds. The other Kan conditions follow similarly.
\end{proof}

The face maps \(\face_0,\face_1\colon \Sigma_{1,1}\to \Sigma_{0,1}\) and  \(\face_2,\face_3\colon \Sigma_{1,1}\to \Sigma_{1,0}\) and the degeneracy maps \(\de_{2}\colon \Sigma_{1,0}\to \Sigma_{1,1}\) and \(\de_{0}\colon\Sigma_{0,1}\to \Sigma_{1,1}\) are natural to define. The simplicial identities hold by construction.

\begin{remark}
  We can also use different decompositions for \(\tilde{\Sigma}_{1,1}\) and \(\Sigma_{1,1}\). Four faces can be decomposed differently, and we can also connect \((1/2,0/3)\) instead of \((0/2,1/3)\) in Figure~\ref{chap6:fig:Sigma-1-1}. We claim that the results are the same for different decompositions. First, the direction of the arrows in \(Y\) does not matter. For two different decompositions, we choose appropriate representatives such that the four faces of \((0/2,0/3,1/2,1/3)\) are compatible or the same, then we divide the bigger tetrahedron into two parts. It suffices to consider the two parts separately. This is almost the same as Lemma~\ref{chap6:lem:Sigma-0-2-var}, except that the vertices are labeled differently, and the proof is the similar.
\end{remark}

\begin{proposition}
  The groupoid of bigons in \(\Sigma\) \textup(with color white and black\textup) is isomorphic to the fundamental groupoid \(\tau(R'_0\Gamma\times_{Y} C'_0\Xi)\).
\end{proposition}
\begin{proof}
  First, these two groupoids have the same spaces of objects. Considering Figure~\ref{chap6:fig:iso-gpd-bigon-is-tau}, we deduce that
  \[
  ((\Gamma_{0,1}\times_{Y_1} \Xi_{1,0})/\arrow{X}\times_{\Xi_{0,0}} \arrow{\Xi}_1/\arrow{\Xi}\cong (\Gamma_{0,1}\times_{Y_1} \Xi_{1,0})/\arrow{X},
  \]
  which proves that the spaces of arrows are the same. It is also clear that they have the same unit, source, and target.
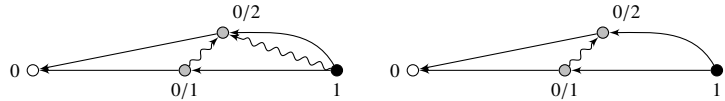
\begin{figure}[htbp]
  \centering
  \begin{tikzpicture}%
  [>=latex',
  hdot/.style={mydot,fill=lightgray}, kdot/.style={mydot,fill},
  every label/.style={scale=0.6}]
  \node[mydot,label=180:\(0\)] at (-2,0) (g0){};
  \node[hdot,label=-90:\(0/1\)] at (0,0) (h1){};
  \node[hdot,label=60:\(0/2\)] at (0.5,0.5) (h2){};
  \node[kdot,label=-90:\(1\)] at (2,0) (k1){};
  \foreach \i/\j in {h1/g0,h2/g0,k1/h1}{
  \path[->] (\i) edge (\j);
  }
  \path[->] (k1) edge[out=120,in=0] (h2);
  \path
  (h1) edge[snake arrow] (h2)
  (k1) edge[snake arrow] (h2);
  \begin{scope}[xshift=5cm]
  \node[mydot,label=180:\(0\)] at (-2,0) (g0){};
  \node[hdot,label=-90:\(0/1\)] at (0,0) (h1){};
  \node[hdot,label=60:\(0/2\)] at (0.5,0.5) (h2){};
  \node[kdot,label=-90:\(1\)] at (2,0) (k1){};
  \foreach \i/\j in {h1/g0,h2/g0,k1/h1}{
  \path[->] (\i) edge (\j);
  }
  \path[->] (k1) edge[out=120,in=0] (h2);
  \path (h1) edge[snake arrow] (h2);
  \end{scope}
  \end{tikzpicture}
  \caption{Two expressions for the space of bigons}\label{chap6:fig:iso-gpd-bigon-is-tau}
\end{figure}

It remains to show that the compositions are the same. Suppose that vertices \(1, 2, 3\) in Figure~\ref{chap6:fig:pre-Sigma-0-2} are identical and that \((0/3,0/2,1,2)=\de_2 (0/3,0/2,1)\) and \((0/3,1,2,3)=\de_2(0/3,1,2)\) are degenerate. Using this picture we compute the composition in the groupoids of bigons. Ignoring the vertices \(2,3\) from the picture, we get the composition in \(\tau(R'_0\Gamma\times_{Y} C'_0\Xi)\). This proves that the two compositions are the same.
\end{proof}

Indeed, we first figured out what is the groupoid of bigons for the desire composite bibundle, then we realise how to construct the actions.

\subsection{Dimension 4}

To complete the construction of \(\Sigma\), we need to verify all desired conditions \(\Kan!(4,k)[i,j]\) for
\(0\le k<4, i, j\ge 0\), and \(k=4, j\ge 2 \). For fixed \((i,j)\), considering a colored version of Lemma~\ref{chap3:lem:kan(n+1,j)-implies-all}, we need only verify \(\Kan!(4,k)[i,j]\) for some~\(k\).

\begin{lemma}
The condition \(\Kan!(4,3)[1,4]\) holds for\/ \(\Sigma\); the condition \(\Kan!(4,1)[4,1]\) holds by symmetry.
\end{lemma}

\begin{proof}
Given an element in \(\Hom(\Horn{4}{3}[1,4], \Sigma)\), which is four 3-simplices in \(\Sigma\) fitting together, we can choose representatives of these 3-simplices such that they agree on common 2-simplices. In other words, these representatives are 3\nbdash{}simplices in \(\Gamma\) and \(\Xi\) fitting together. Suppose that these representatives are given by
\begin{gather*}
(0,0/2,0/3,0/4,2,3,4)\in \tilde{\Sigma}_{0,2},\\
(0,0/1,0/3,0/4,1,3,4)\in \tilde{\Sigma}_{0,2},\\
(0,0/1,0/2,0/4,1,2,4)\in \tilde{\Sigma}_{0,2},\\
(1,2,3,4) \in Z_3
\end{gather*}
fitting together. We need to show that it determines \((0,0/1,0/2,0/4,1,2,4)\). The condition \(\Kan!(4,3)[1,4]\) for \(\Gamma\) implies that the four faces of the 3-simplex \((0,0/1,0/2,0/4)\) are compatible. It remains to consider the part in \(\Xi\). This follows by the pasting theorem.
\end{proof}

\begin{lemma}
The condition \(\Kan!(4,1)[2,3]\) holds for \(\Sigma\); by symmetry, \(\Kan!(4,3)[3,2]\) holds.
\end{lemma}

\begin{proof}
Given an element in \(\Hom(\Horn{4}{1}[2,3],\Sigma)\), as above we choose representatives such that it is a collection of 3-simplices in \(\Gamma\) and \(\Xi\) fitting together. Suppose that these representatives are given by
\begin{gather*}
(1,1/2,1/3,1/4,2,3,4) \in \tilde{\Sigma}_{0,2},\\
(0,1,0/1,0/3,1/3,1/4,3,4)\in \tilde{\Sigma}_{1,1},\\
(0,1,0/1,0/4,1/2,1/3,2,4)\in \tilde{\Sigma}_{1,1},\\
(0,1,0/2,0/3,1/2,1/3,2,3)\in \tilde{\Sigma}_{1,1}
\end{gather*}
fitting together. We need to show that they determine \((0,0/2,0/3,0/4,2,3,4)\) in \(\tilde{\Sigma}_{0,2}\). Fill the 3-dimensional horns
\begin{gather*}
  (0/2,\overline{1/2},1/3,1/4), \\
  (0/2,0/3,\overline{1/3},1/4), \\
  (0/2,0/3,0/4,\overline{1/4})
\end{gather*}
by Kan conditions. We break \((0,0/2,0/3,0/4,2,3,4)\) into two parts, which lie in \(\Gamma\) and \(\Xi\), respectively. It suffices to show that each part is compatible. This follows by applying the pasting theorem.
\end{proof}

We now complete the construction of the \(X\)-\(Z\) bibundle \(\Sigma\). This bibundle is called the composite of \(\Gamma\) and \(\Xi\), and we denote it by \(\Gamma\otimes \Xi\).

\section{Functoriality of the bundlisation}

Given a morphism of \(2\)\nbdash{}groupoids \(f\colon X\to Y\), we recall from Section~\ref{chap4:sec:cograph-morphism-higher-gpd} that the cograph (bundlisation) \(\Gamma_{i,j}=X_i\times_{Y_i} Y_{i+j}\) gives a right principal \(X\)-\(Y\) bibundle. We show that this construction respects compositions up to a weak equivalence.

\begin{proposition}\label{chap6:prop:bibundle-functorial}
Let \(f\colon X\to Y\) and \(g\colon Y\to Z\) be morphisms of \(2\)\nbdash{}groupoids in \((\Cat,\covers)\). Denote by \(\Gamma,\Xi,\Sigma\) the bibundles given by the cograph of \(f, g, g\circ f\). Then there is a morphism of \(X\)-\(Z\) bibundles \(\Sigma\to \Gamma\otimes \Xi\) which induces a weak equivalence on the groupoid of bigons, \(\arrow{\Sigma}\to \arrow{\Gamma\otimes\Xi}\).
\end{proposition}

We divide the proof into several lemmas.

\begin{lemma}
 There is a morphism of \(X\)-\(Z\) bibundles \(\Sigma\to \Gamma\otimes \Xi\).
\end{lemma}

\begin{proof}
Denote \(\Gamma\otimes \Xi\) by \(\Pi\). Since \(\Sigma\) and \(\Pi\) are both 3-coskeletal, it suffices to construct a morphism \(c\colon \Sigma\to \Pi\) on dimensions \(1,2,3\).

First, there is a natural inclusion
\[
\Sigma_{0,0}\to \Pi_{0,0},\quad (x_0, y_1)\mapsto \big((x_0,\de_0 f_0(x_0)),(f_0(x_0),y_1)\big),
\]
which is defined to be~\(c_{0,0}\). This map is illustrated by Figure~\ref{chap6:fig:bibundle-c00}, where white dots are in~\(X\), gray dots in \(Y\), and black dots in \(Z\). The left picture stands for \(\Sigma_{0,0}\); it has three parts in \(X, Y, Z\) matched by \(f\) and \(g\). We omit all labels here. The right picture stands for the image of \(\Sigma_{0,0}\) in \(\Pi_{0,0}\), where two gray dots are identical and the 1-simplex between them is degenerate.
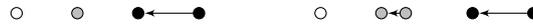
\begin{figure}[htbp]
  \centering
  \begin{tikzpicture}[>=latex',scale=0.8]
    \node[mydot] at (0,0) (x0) {};
    \node[mydot,fill=lightgray] at (1,0)(y0) {};
    \node[mydot,fill] at (2,0) (z0) {};
    \node[mydot,fill] at (3,0) (z1) {};
    \path[->]
    (z1) edge (z0);
    \begin{scope}[xshift=5cm]
    \node[mydot] at (0,0) (x0) {};
    \node[mydot,fill=lightgray] at (1,0)(y0) {};
    \node[mydot,fill=lightgray] at (1.4,0)(y1) {};
    \node[mydot,fill] at (2.5,0) (z0) {};
    \node[mydot,fill] at (3.5,0) (z1) {};
    \path[->]
    (z1) edge (z0)
    (y1) edge (y0);
    \end{scope}
  \end{tikzpicture}
  \caption{The construction of \(c_{0,0}\)}\label{chap6:fig:bibundle-c00}
\end{figure}

In the same spirit, we define \(c_{0,1}\) by Figure~\ref{chap6:fig:bibundle-c01}. We first send the left picture \(\Sigma_{0,1}\) to the right picture by adding two degenerate 2\nbdash{}simplices in \(Y\) and \(Z\). Composing with the quotient map, we obtain \(c_{0,1}\).
\begin{figure}[htbp]
  \centering
  \begin{tikzpicture}[>=latex',scale=0.8]
    \node[mydot] at (0,0) (x0) {};
    \node[mydot,fill=lightgray] at (1,0)(y0) {};
    \node[mydot,fill] at (2,0) (z0) {};
    \node[mydot,fill] at (3,0) (z1) {};
    \node[mydot,fill] at (3,1) (z2) {};
    \path[->]
    (z2) edge (z0)
         edge (z1)
    (z1) edge (z0);
   \begin{scope}[xshift=5cm]
    \node[mydot] at (0,0) (x0) {};
    \node[mydot,fill=lightgray] at (1,0)(y0) {};
    \node[mydot,fill=lightgray] at (1.4,0)(y1){};
    \node[mydot,fill=lightgray] at (1.2,0.4)(y2){};
    \node[mydot,fill] at (2.5,0) (z0) {};
    \node[mydot,fill] at (3.5,0) (z1) {};
    \node[mydot,fill] at (3.5,1) (z2) {};
    \node[mydot,fill] at (2.5,0.4)(zy2){};
    \path[->]
    (z2) edge (zy2)
         edge (z1)
    (z1) edge (z0)
         edge (zy2)
    (zy2)edge (z0)
    (y2) edge (y0)
         edge (y1)
    (y1) edge (y0);
    \end{scope}
  \end{tikzpicture}
  \caption{The construction of \(c_{0,1}\)}\label{chap6:fig:bibundle-c01}
\end{figure}
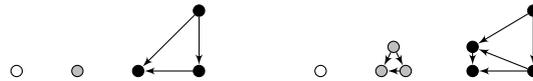

Similarly, the map \(c_{1,0}\) is given by Figure~\ref{chap6:fig:bibundle-c10}.
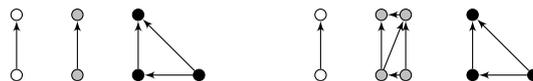
\begin{figure}[htbp]
  \centering
  \begin{tikzpicture}[>=latex',scale=0.8]
    \node[mydot] at (0,1) (x0) {};
    \node[mydot] at (0,0) (x1) {};
    \node[mydot,fill=lightgray] at (1,1)(y0) {};
    \node[mydot,fill=lightgray] at (1,0)(y1) {};
    \node[mydot,fill] at (2,1) (z0) {};
    \node[mydot,fill] at (2,0) (z1) {};
    \node[mydot,fill] at (3,0) (z2) {};
   \foreach \i/\j in
   {x1/x0,y1/y0,z1/z0,z2/z0,z2/z1}{
   \path[->](\i) edge (\j);}
   \begin{scope}[xshift=5cm]
    \node[mydot] at (0,1) (x0) {};
    \node[mydot] at (0,0) (x1) {};
    \node[mydot,fill=lightgray] at (1,1)(y0) {};
    \node[mydot,fill=lightgray] at (1,0)(y1) {};
    \node[mydot,fill=lightgray] at (1.4,1)(y00){};
    \node[mydot,fill=lightgray] at (1.4,0)(y11){};
    \node[mydot,fill] at (2.5,1) (z0) {};
    \node[mydot,fill] at (2.5,0) (z1) {};
    \node[mydot,fill] at (3.5,0) (z2) {};
    \foreach \i/\j in
   {x1/x0,y1/y0,y1/y00,y00/y0,y11/y1,y11/y00,z1/z0,z2/z0,z2/z1}{
   \path[->](\i) edge (\j);}
    \end{scope}
  \end{tikzpicture}
  \caption{The construction of \(c_{1,0}\)}\label{chap6:fig:bibundle-c10}
\end{figure}

Finally, \(c_{0,2},c_{1,1}\), and \(c_{2,0}\) are given by the three pictures in Figure~\ref{chap6:fig:bibundle-c02c11c20}, respectively. These maps are constructed first by adding appropriate degenerate 3\nbdash{}simplices, then composing with the quotient maps.
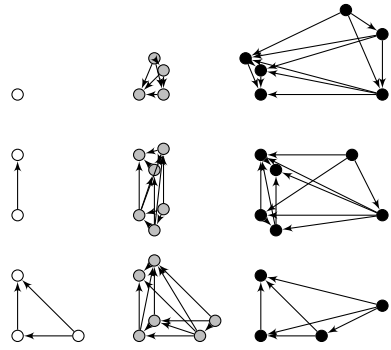
\begin{figure}[htbp]
  \centering
  \begin{tikzpicture}[>=latex',scale=0.8]
    \node[mydot] at (0,0) (x0) {};
    \node[mydot,fill=lightgray] at (2,0)(y0) {};
    \node[mydot,fill=lightgray] at (2.4,0)(y1){};
    \node[mydot,fill=lightgray] at (2.4,0.4)(y2){};
    \node[mydot,fill=lightgray] at (2.25,0.6)(y3){};
    \node[mydot,fill] at (4,0) (z11) {};
    \node[mydot,fill] at (6,0) (z1) {};
    \node[mydot,fill] at (6,1) (z2) {};
    \node[mydot,fill] at (5.4,1.4)(z3){};
    \node[mydot,fill] at (4,0.4)(z22){};
    \node[mydot,fill] at (4-0.25,0.6)(z33){};
   \foreach \i/\j in
   {y3/y0,y3/y1,y3/y2,y2/y0,y2/y1,y1/y0,
   z33/z22,z33/z11,z22/z11,z3/z2,z3/z1,z2/z1,z3/z33,z2/z22,z2/z33,z1/z11,z1/z22,z1/z33}{
   \path[->](\i) edge (\j);}
  \begin{scope}[yshift=-2cm]
   \node[mydot] at (0,1) (x0){};
   \node[mydot] at (0,0) (x1){};
   \node[mydot,fill=lightgray] at (2,1) (y0){};
   \node[mydot,fill=lightgray] at (2,0) (y1){};
   \node[mydot,fill=lightgray] at (2+0.25,1-0.25) (y02){};
   \node[mydot,fill=lightgray] at (2+0.25,0-0.25) (y12){};
   \node[mydot,fill=lightgray] at (2+0.4,1+0.1) (y03){};
   \node[mydot,fill=lightgray] at (2+0.4,0+0.1) (y13){};
   \node[mydot,fill] at (4,1) (z0){};
   \node[mydot,fill] at (4,0) (z1){};
   \node[mydot,fill] at (4+0.25,1-0.25) (z00){};
   \node[mydot,fill] at (4+0.25,0-0.25) (z11){};
   \node[mydot,fill] at (6,0) (z2){};
   \node[mydot,fill] at (5.5,1) (z3){};
   \foreach \i/\j in
   {x1/x0,y1/y0,y1/y02,y1/y03, y12/y02,y12/y1,y12/y03,y13/y03,y13/y12,y13/y1,y02/y0,y03/y02,y03/y0,
   z1/z0,z2/z1,z2/z0,z3/z2,z3/z1,z3/z0,z0/z00,z1/z11,z2/z00,z2/z11,z11/z00,z11/z0}{
   \path[->](\i) edge (\j);}
  \end{scope}
  \begin{scope}[yshift=-4cm]
   \node[mydot] at (0,1) (x0) {};
   \node[mydot] at (0,0) (x1) {};
   \node[mydot] at (1,0) (x2) {};
   \node[mydot,fill=lightgray] at (2,1) (y0) {};
   \node[mydot,fill=lightgray] at (2,0) (y1) {};
   \node[mydot,fill=lightgray] at (3,0) (y2) {};
   \node[mydot,fill=lightgray] at (2+0.25,1+0.25) (y00) {};
   \node[mydot,fill=lightgray] at (2+0.25,0+0.25) (y11) {};
   \node[mydot,fill=lightgray] at (3+0.25,0+0.25) (y22) {};
   \node[mydot,fill] at (4,1) (z0) {};
   \node[mydot,fill] at (4,0) (z1) {};
   \node[mydot,fill] at (5,0) (z2) {};
   \node[mydot,fill] at (6,0.5)(z3){};
   \foreach \i/\j in
   {x2/x1,x2/x0,x1/x0,
   y2/y1,y2/y0,y1/y0,y22/y11,y22/y00,y11/y00,y22/y2,y11/y1,y00/y0,y2/y00,y2/y11,y1/y00,
   z2/z1,z2/z0,z1/z0,z3/z2,z3/z1,z3/z0}{
   \path[->](\i) edge (\j);}
  \end{scope}
  \end{tikzpicture}
  \caption{The construction of \(c_{0,2},c_{1,1}\), and \(c_{2,0}\)}\label{chap6:fig:bibundle-c02c11c20}
\end{figure}

It is clear from the construction that the constructed maps \(c_{0,0}\), \(c_{0,1}\), \(c_{1,0}\), \(c_{2,0}\), \(c_{1,1}\), and~\(c_{2,0}\) are compatible with the face and degeneracy maps. This is the bibundle morphism \(c\colon \Sigma\to \Pi\) we are looking for.
\end{proof}

We now explain that the bibundle morphism \(c\) constructed above induces a weak equivalence between the groupoids of bigons, \(\arrow{\Sigma}\to \arrow{\Pi}\).

\begin{lemma}
  Let \(c\colon \Sigma\to \Pi\) be given as above. Then
\begin{multline*}
  \Hom(\Simp{1}[1,1]\to\Simp{1}\star \Simp{0}[1,2], \Sigma\to \Pi) \to \\
  \Hom(\partial\Simp{1}[1,1]\to\partial\Simp{1}\star\Simp{0}[1,2], \Sigma\to \Pi)
\end{multline*}
is a cover.
\end{lemma}
\begin{proof}
We need to show that the map from the left picture to the right picture in Figure~\ref{chap6:fig:bibundle-c-acyclic-1-right} is a cover. The Kan conditions for \(Z\) and \(Y\) imply that the two spaces given by the middle picture and the right picture are the same. The map from the un-quotiented version of the left picture to that of the middle picture is a cover, which descends to a cover after taking quotients. This proves the lemma.
\begin{figure}[htbp]
  \centering
  \begin{tikzpicture}[>=latex',scale=0.8]
    \node[mydot] at (0,0) (x0) {};
    \node[mydot,fill=lightgray] at (1,0)(y0) {};
    \node[mydot,fill=lightgray] at (1.4,0)(y1){};
    \node[mydot,fill=lightgray] at (1.5,1)(y2){};
    \node[mydot,fill] at (2.5,0) (z0) {};
    \node[mydot,fill] at (3.5,0) (z1) {};
    \node[mydot,fill] at (3.5,1.5) (z2) {};
    \node[mydot,fill] at (2.5,1)(zy2){};
    \path[->]
    (z2) edge (zy2)
         edge (z1)
    (z1) edge (z0)
         edge[snake arrow] (zy2)
    (zy2)edge[snake arrow] (z0)
    (y2) edge (y0)
         edge[snake arrow] (y1)
    (y1) edge (y0);
    \begin{scope}[xshift=5cm]
    \node[mydot] at (0,0) (x0) {};
    \node[mydot,fill=lightgray] at (1,0)(y0) {};
    \node[mydot,fill=lightgray] at (1.4,0)(y1){};
    \node[mydot,fill=lightgray] at (1.5,1)(y2){};
    \node[mydot,fill] at (2.5,0) (z0) {};
    \node[mydot,fill] at (3.5,0) (z1) {};
    \node[mydot,fill] at (3.5,1.5) (z2) {};
    \node[mydot,fill] at (2.5,1)(zy2){};
    \path[->]
    (z2) edge (zy2)
         edge (z1)
    (z1) edge[snake arrow] (zy2)
    (zy2)edge[snake arrow] (z0)
    (y2) edge (y0)
         edge[snake arrow] (y1)
    (y1) edge (y0);
    \end{scope}
    \begin{scope}[xshift=10cm]
    \node[mydot] at (0,0) (x0) {};
    \node[mydot,fill=lightgray] at (1,0)(y0) {};
    \node[mydot,fill=lightgray] at (1.5,1)(y2){};
    \node[mydot,fill] at (3.5,0) (z1) {};
    \node[mydot,fill] at (3.5,1.5) (z2) {};
    \node[mydot,fill] at (2.5,1)(zy2){};
    \path[->]
    (z2) edge (zy2)
         edge (z1)
    (y2) edge (y0);
    \end{scope}
  \end{tikzpicture}
  \caption{A pictorial proof}\label{chap6:fig:bibundle-c-acyclic-1-right}
\end{figure}
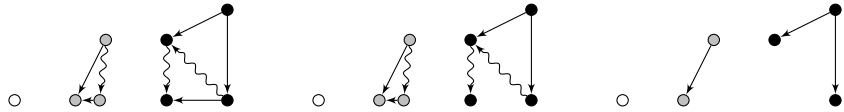
\end{proof}

\begin{lemma}
  The morphism \(c\colon \Sigma\to \Pi\) satisfies \(\Acyc!(2)\).
\end{lemma}
\begin{proof}
The lemma consists of two non-trivial statements. Consider Figure~\ref{chap6:fig:map-bibundle-c-acyclic2-2-1}. The first statement says that \(c\) induces an isomorphism between \(\Hom(\partial\Simp{3}[2,1]\to \Simp{3}[2,1], \Sigma\to \Pi)\), given by the left picture, and \(\Hom(\Simp{3}[1,2], \Sigma)\), given by the right picture. The claim follows because the space given by the middle picture is isomorphic to the spaces given by the remaining two pictures.
\begin{figure}[htbp]
  \centering
  \begin{tikzpicture}[>=latex',scale=0.8]
    \node[mydot] at (0,0) (x0) {};
    \node[mydot,fill=lightgray] at (1,0)(y0) {};
    \node[mydot,fill=lightgray] at (1.4,0)(y1){};
    \node[mydot,fill=lightgray] at (1.2,0.4)(y2){};
    \node[mydot,fill] at (2.5,0) (z0) {};
    \node[mydot,fill] at (3.5,0) (z1) {};
    \node[mydot,fill] at (3.5,1) (z2) {};
    \node[mydot,fill] at (2.5,0.4)(zy2){};
    \path[->]
    (z2) edge (zy2)
         edge (z1)
    (z1) edge (z0)
         edge[snake arrow] (zy2)
    (zy2)edge[out=-20,in=60,snake arrow] (z0)
    (y2) edge (y0)
         edge[out=-20,in=60,snake arrow] (y1)
    (y1) edge (y0);
    \begin{scope}[xshift=5cm]
        \node[mydot] at (0,0) (x0) {};
    \node[mydot,fill=lightgray] at (1,0)(y0) {};
    \node[mydot,fill=lightgray] at (1.4,0)(y1){};
    \node[mydot,fill=lightgray] at (1.2,0.4)(y2){};
    \node[mydot,fill] at (2.5,0) (z0) {};
    \node[mydot,fill] at (3.5,0) (z1) {};
    \node[mydot,fill] at (3.5,1) (z2) {};
    \node[mydot,fill] at (2.5,0.4)(zy2){};
    \node[mydot,fill] at (2.1,0) (zy0){};
    \path[->]
    (z2) edge (zy2)
         edge (z1)
    (z1) edge (z0)
         edge[snake arrow] (zy2)
    (zy2)edge[out=-20,in=60,snake arrow] (z0)
    (z0) edge (zy0)
    (zy2)edge (zy0)
    (y2) edge (y0)
    (y1) edge (y0);
    \end{scope}
    \begin{scope}[xshift=10cm]
    \node[mydot] at (0,0) (x0) {};
    \node[mydot,fill=lightgray] at (1,0)(y0) {};
    \node[mydot,fill] at (2,0) (z0) {};
    \node[mydot,fill] at (3,0) (z1) {};
    \node[mydot,fill] at (3,1) (z2) {};
    \path[->]
    (z2) edge (z0)
         edge (z1)
    (z1) edge (z0);
    \end{scope}
  \end{tikzpicture}
  \caption{Proof of \(\Acyc!(2)[2,1]\)}\label{chap6:fig:map-bibundle-c-acyclic2-2-1}
\end{figure}

The other statement is
\[
\Hom(\Simp{3}[1,2], \Sigma)\cong \Hom(\partial\Simp{3}[1,2]\to \Simp{3}[1,2], \Sigma\to \Pi),
\]
which follows by considering Figure~\ref{chap6:fig:map-bibundle-c-acyclic2-1-2}.
\begin{figure}[htbp]
  \centering
  \begin{tikzpicture}%
  [>=latex',every label/.style={scale=0.6},scale=0.8]
    \node[mydot] at (0,1) (x0) {};
    \node[mydot] at (0,0) (x1) {};
    \node[mydot,fill=lightgray] at (1,1)(y0) {};
    \node[mydot,fill=lightgray] at (1,0)(y1) {};
    \node[mydot,fill=lightgray] at (1.4,1)(y00){};
    \node[mydot,fill=lightgray] at (1.4,0)(y11){};
    \node[mydot,fill] at (2.5,1) (z0) {};
    \node[mydot,fill] at (2.5,0) (z1) {};
    \node[mydot,fill] at (3.5,0) (z2) {};
    \foreach \i/\j in
   {x1/x0,y1/y0,y00/y0,y11/y1,z2/z0,z2/z1}{
   \path[->](\i) edge (\j);}
   \path[->]
   (y1) edge[snake arrow] (y00)
   (y11)edge[snake arrow] (y00)
   (z1) edge[snake arrow] (z0);
   \begin{scope}[xshift=5cm]
   \node[mydot] at (0,1) (x0) {};
   \node[mydot] at (0,0) (x1) {};
   \node[mydot,fill=lightgray] at (1,1)(y0) {};
   \node[mydot,fill=lightgray] at (1,0)(y1) {};
   \node[mydot,fill] at (2,1) (z0) {};
   \node[mydot,fill] at (2,0) (z1) {};
   \node[mydot,fill] at (3,0) (z2) {};
   \foreach \i/\j in
   {x1/x0,y1/y0,z2/z0,z2/z1}{
   \path[->](\i) edge (\j);}
   \path[->]
   (y1) edge[out=60,in=-60,snake arrow] (y0)
   (z1) edge[snake arrow] (z0)
        edge[out=120,in=-120](z0);
   \end{scope}
   \begin{scope}[xshift=10cm]
   \node[mydot] at (0,1) (x0) {};
   \node[mydot] at (0,0) (x1) {};
   \node[mydot,fill=lightgray] at (1,1)(y0) {};
   \node[mydot,fill=lightgray] at (1,0)(y1) {};
   \node[mydot,fill] at (2,1) (z0) {};
   \node[mydot,fill] at (2,0) (z1) {};
   \node[mydot,fill] at (3,0) (z2) {};
   \foreach \i/\j in
   {x1/x0,y1/y0,z1/z0,z2/z0,z2/z1}{
   \path[->](\i) edge (\j);}
   \end{scope}
  \end{tikzpicture}
  \caption{Proof of \(\Acyc!(2)[1,2]\)}\label{chap6:fig:map-bibundle-c-acyclic2-1-2}
\end{figure}
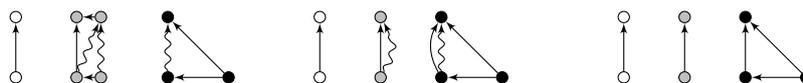
\end{proof}

The two lemmas above show that \(c\colon \Sigma\to\Pi\) is a colored version of a weak acyclic fibration (Definition~\ref{chap6:def:weak-acyclic}). It follows that the induced functor on groupoids of bigons, \(\arrow{\Sigma}\to \arrow{\Pi}\), is a weak equivalence, and this proves Proposition~\ref{chap6:prop:bibundle-functorial}.

\section{Unitality of the composition}

We recall from Section~\ref{chap4:sec:cograph-morphism-higher-gpd} that the total décalage \(\Dec(X)\) is a two-sided principal \(X\)-\(X\) bibundle. Our goal in this section is to prove the following result.

\begin{proposition}\label{chap6:prop:2gpd-bibundle-comp-unit}
Let \(\Gamma\) be a right principal \(X\)-\(Y\) bibundle. There is a canonical morphism of \(X\)-\(Y\) bibundles \(\Gamma \to \Dec (X)\otimes \Gamma\) which induces a weak equivalence on \(\arrow{\Gamma}\to \arrow{\Dec (X)\otimes \Gamma}\). A similar statement holds for \(\Gamma\otimes\Dec (X)\).
\end{proposition}

We divided the proof into several lemmas.

\begin{lemma}
  There are morphisms of \(X\)-\(Y\) bibundles \(\Gamma \to \Dec (X)\otimes \Gamma\) and \(\Gamma\to \Gamma\otimes \Dec(X)\).
\end{lemma}
\begin{proof}
  We construct the morphism of bibundles \(\Gamma\to \Dec (X)\otimes \Gamma\); the other morphism can be constructed by symmetry.

  Denote \(\Dec (X)\otimes \Gamma\) by \(\Sigma\). We need a morphism \(f\colon \Gamma\to \Sigma\) on the dimensions \(n\le 3\). First, on \(X\) and \(Y\), \(f\) is the identity. Recall that \(\Sigma_{0,0}=X_1\times_{\face_1,X_0,\face_0} \Gamma_{0,0}\). We define
  \[
  f_{0,0}\colon \Gamma_{0,0}\to \Sigma_{0,0}, \quad f_{0,0}=(\de_0\circ \face_0, \id),
  \]
  illustrated by Figure~\ref{chap6:fig:may-unit-bibudle-f00}, where the left picture is sent to the right picture by adding a degenerate 1-simplex (and then relabeling the vertices).
\begin{figure}[htbp]
  \centering
  \begin{tikzpicture}%
  [>=latex', hdot/.style={mydot,fill=lightgray},every label/.style={scale=0.6}]
  \node[mydot,label=180:\(0\)] at (0,0) (g0){};
  \node[hdot,label=0:\(1\)] at (2,0) (h1){};
  \path[->] (h1) edge (g0);
  \begin{scope}[xshift=5cm]
  \node[mydot,label=90:\(0/1\)] at (0,0) (g01){};
  \node[mydot,label=180:\(0\)] at (-1,0) (g0){};
  \node[hdot,label=0:\(1\)] at (2,0) (h1){};
  \path[->]
  (h1) edge (g01)
  (g01) edge (g0);
  \end{scope}
  \end{tikzpicture}
  \caption{The construction of \(f_{0,0}\)}\label{chap6:fig:may-unit-bibudle-f00}
\end{figure}
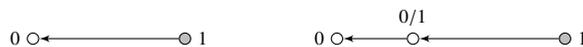

In the same spirit, we define \(f_{0,1}\colon \Gamma_{0,1}\to \Sigma_{0,1}\) by Figure~\ref{chap6:fig:may-unit-bibudle-f01}. We first send \((0,1,2)\) on the left to \((0/2,1,2)\) on the right. Then add degenerate simplices, \((0/1,0/2,1)=\de_0(0/2,1)\) and \((0,0/2,0/1)=\de_0(0/2,0/1)\). Composing with the quotient map, we get~\(f_{0,1}\).
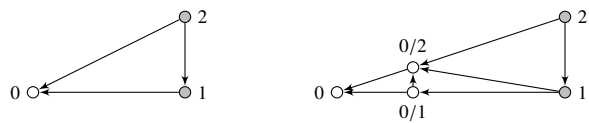
\begin{figure}[htbp]
  \centering
  \begin{tikzpicture}%
  [>=latex', hdot/.style={mydot,fill=lightgray},every label/.style={scale=0.6}]
  \node[mydot,label=180:\(0\)] at (0,0) (g0){};
  \node[hdot,label=0:\(1\)] at (2,0) (h1){};
  \node[hdot,label=0:\(2\)] at (2,1) (h2){};
  \foreach \i/\j in
  {h1/g0,h2/g0,h2/h1}{
  \path[->] (\i) edge (\j);}
  \begin{scope}[xshift=5cm]
  \node[mydot,label=180:\(0\)] at (-1,0) (g0){};
  \node[mydot,label=-90:\(0/1\)] at (0,0) (g01){};
  \node[mydot,label=90:\(0/2\)] at (0,0.33) (g02){};
  \node[hdot,label=0:\(1\)] at (2,0) (h1){};
  \node[hdot,label=0:\(2\)] at (2,1) (h2){};
  \foreach \i/\j in
  {g01/g0,g02/g0,g01/g02,
  h1/g01,h2/g02,h2/h1,h1/g02}{
  \path[->] (\i) edge (\j);}
  \end{scope}
  \end{tikzpicture}
  \caption{The construction of \(f_{0,1}\)}\label{chap6:fig:may-unit-bibudle-f01}
\end{figure}

Similarly, the map \(f_{1,0}\colon \Gamma_{1,0}\to \Sigma_{1,0}\) is given by Figure~\ref{chap6:fig:may-unit-bibudle-f10}. It sends \((0,1,2)\) on the left to \((0/2,1/2,2)\) on the right. The other 2\nbdash{}simplices of the right picture are degenerate, \((0,0/2,1/2)=\de_0(0/2,1/2)\) and \((0,1,0/1)=\de_1(0,1/2)\). The decomposition of \(\Sigma_{1,0}\) that we use here is different from the standard one.
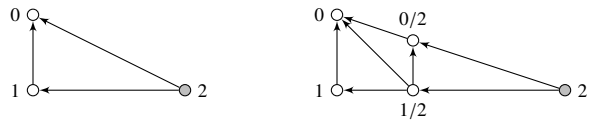
\begin{figure}[htbp]
  \centering
  \begin{tikzpicture}%
  [>=latex', hdot/.style={mydot,fill=lightgray},every label/.style={scale=0.6}]
  \node[mydot,label=180:\(0\)] at (0,1) (g0){};
  \node[mydot,label=180:\(1\)] at (0,0) (g1){};
  \node[hdot,label=0:\(2\)] at (2,0) (h2){};
  \foreach \i/\j in
  {g1/g0,h2/g0,h2/g1}{
  \path[->] (\i) edge (\j);}
  \begin{scope}[xshift=5cm]
  \node[mydot,label=180:\(0\)] at (-1,1) (g0){};
  \node[mydot,label=180:\(1\)] at (-1,0) (g1){};
  \node[mydot,label=90:\(0/2\)] at (0,0.66) (g02){};
  \node[mydot,label=-90:\(1/2\)] at (0,0) (g12){};
  \node[hdot,label=0:\(2\)] at (2,0) (h2){};
  \foreach \i/\j in
  {g1/g0,g02/g0,g12/g1,g12/g02,g12/g0,
  h2/g02,h2/g12}{
  \path[->] (\i) edge (\j);}
  \end{scope}
  \end{tikzpicture}
  \caption{The construction of \(f_{1,0}\)}\label{chap6:fig:may-unit-bibudle-f10}
\end{figure}

Finally, \(f_{2,0}, f_{1,1}\), and \(f_{2,0}\) are given by the three pictures in Figure~\ref{chap6:fig:map-unit-bibundle-f3s}, respectively. In these pictures, the 3-simplices \((0/2,1,2,3)\), \((0/3,1/3,2,3)\), and \((0/3,1/3,2/3,3)\) come from \(\Gamma\). The remaining 3-simplices are suitable degenerate simplices.
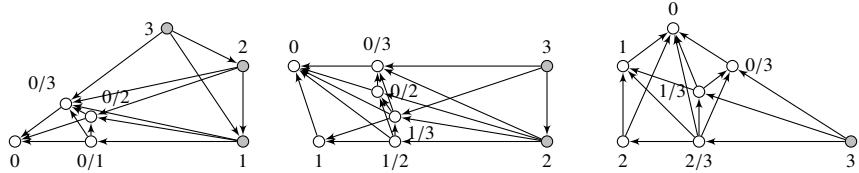
\begin{figure}[htbp]
  \centering
  \begin{tikzpicture}%
  [>=latex', hdot/.style={mydot,fill=lightgray},every label/.style={scale=0.6}]
  \node[mydot,label=-90:\(0\)] at (-1,0) (g0){};
  \node[mydot,label=-90:\(0/1\)] at (0,0) (g01){};
  \node[mydot,label=45:\(0/2\)] at (0,0.33) (g02){};
  \node[mydot,label=120:\(0/3\)] at (-0.33,0.5) (g03){};
  \node[hdot,label=-90:\(1\)] at (2,0) (h1){};
  \node[hdot,label=90:\(2\)] at (2,1) (h2){};
  \node[hdot,label=180:\(3\)] at (1,1.5) (h3){};
  \foreach \i/\j in
  {g01/g0,g02/g0,g03/g0,g01/g02,g02/g03,g01/g03,
  h1/g01,h2/g02,h3/g03,h2/h1,h3/h1,h3/h2,h1/g02,h2/g03,h1/g03}{
  \path[->] (\i) edge (\j);}
  \begin{scope}[xshift=3cm]
  \node[mydot,label=90:\(0\)] at (-0.33,1) (g0){};
  \node[mydot,label=-90:\(1\)] at (0,0) (g1){};
  \node[hdot,label=-90:\(2\)] at (3,0) (h2){};
  \node[hdot,label=90:\(3\)] at (3,1) (h3){};
  \node[mydot,label=0:\(0/2\)] at (0.77,0.66) (g02){};
  \node[mydot,label=90:\(0/3\)] at (0.77, 1) (g03){};
  \node[mydot,label=-90:\(1/2\)] at (1,0) (g12){};
  \node[mydot,label=-30:\(1/3\)] at (1,0.33) (g13){};
  \foreach \i/\j in
  {g1/g0,g02/g0,g03/g0,g12/g1,g13/g1,
  h3/g03,h2/g02,h3/g13,h2/g12,h3/h2,
  g02/g03,h2/g03,g12/g13,h2/g13,
  g12/g02,g12/g0,g13/g03,g13/g0,g13/g02}{
  \path[->] (\i) edge (\j);}
  \end{scope}
  \begin{scope}[xshift=7cm]
  \node[mydot,label=90:\(0\)] at (0.66,1.5)(g0){};
  \node[mydot,label=90:\(1\)] at (0,1)  (g1){};
  \node[mydot,label=-90:\(2\)] at (0,0) (g2){};
  \node[hdot,label=-90:\(3\)] at (3,0) (h3){};
  \node[mydot,label=0:\(0/3\)] at (1.44,1) (g03){};
  \node[mydot,label=180:\(1/3\)] at (1,0.66) (g13){};
  \node[mydot,label=-90:\(2/3\)] at (1,0) (g23){};
  \foreach \i/\j in
  {g1/g0,g2/g0,g2/g1,g03/g0,g13/g1,g23/g2,h3/g03,h3/g13,h3/g23,
  g13/g03,g13/g0,g23/g13,g23/g1,g23/g03,g23/g0}{
  \path[->] (\i) edge (\j);}
  \end{scope}
  \end{tikzpicture}
  \caption{The construction of \(f_{2,0}, f_{1,1}\), and \(f_{2,0}\)}\label{chap6:fig:map-unit-bibundle-f3s}
\end{figure}

It is clear that the maps \(f_{0,0}\), \(f_{0,1}\), \(f_{1,0}\), \(f_{2,0}\), \(f_{1,1}\), and \(f_{2,0}\) are compatible with face and degeneracy maps. This defines a bibundle morphism \(f\colon \Sigma\to \Gamma\).
\end{proof}

We now explain that the bibundle morphism \(f\) constructed above induces a weak equivalence on the groupoid of bigons, \(\arrow{\Sigma}\to \arrow{\Gamma}\).

\begin{lemma}
  Let \(f\colon \Sigma\to \Gamma\) be given as above. Then
\begin{multline*}
  \Hom(\Simp{1}[1,1]\to\Simp{1}\star \Simp{0}[1,2], \Sigma\to \Gamma) \to \\
  \Hom(\partial\Simp{1}[1,1]\to\partial\Simp{1}\star\Simp{0}[1,2], \Sigma\to \Gamma)
\end{multline*}
is a cover.
\end{lemma}
\begin{proof}
We need to show that the canonical map from the left picture to the right picture in Figure~\ref{chap6:fig:map-unit-bibundle-f-acyclic-1-right} is a cover. In the left picture, \((0, 0/1)\) is degenerate. Since the middle picture and the left picture represent the same spaces, it suffices to consider the map from the middle picture to the right picture. The Kan conditions for \(\Gamma\) implies that the map from the un-quotiented version of the middle picture to the right picture is a cover, so it descends to a cover. This proves the lemma.
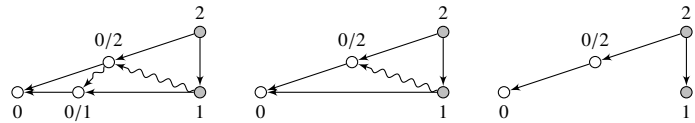
\begin{figure}[htbp]
  \centering
  \begin{tikzpicture}%
  [>=latex', hdot/.style={mydot,fill=lightgray},every label/.style={scale=0.6},scale=0.8]
  \node[mydot,label=-90:\(0\)] at (0,0) (g0){};
  \node[hdot,label=-90:\(1\)] at (3,0) (h1){};
  \node[hdot,label=90:\(2\)] at (3,1) (h2){};
  \node[mydot,label=-90:\(0/1\)] at (1,0)(g01){};
  \node[mydot,label=90:\(0/2\)] at (1.5,0.5)(g02){};
  \foreach \i/\j in
  {g01/g0,g02/g0,h1/g01,h2/g02,h2/h1}{
  \path[->](\i) edge (\j);}
  \path[->] (h1) edge[snake arrow] (g02)
  (g02) edge[snake arrow] (g01);
  \begin{scope}[xshift=4cm]
  \node[mydot,label=-90:\(0\)] at (0,0) (g0){};
  \node[hdot,label=-90:\(1\)] at (3,0) (h1){};
  \node[hdot,label=90:\(2\)] at (3,1) (h2){};
  \node[mydot,label=90:\(0/2\)] at (1.5,0.5)(g02){};
  \foreach \i/\j in
  {g02/g0,h1/g0,h2/g02,h2/h1}{
  \path[->](\i) edge (\j);}
  \path[->] (h1) edge[snake arrow] (g02);
  \end{scope}
  \begin{scope}[xshift=8cm]
  \node[mydot,label=-90:\(0\)] at (0,0) (g0){};
  \node[hdot,label=-90:\(1\)] at (3,0) (h1){};
  \node[hdot,label=90:\(2\)] at (3,1) (h2){};
  \node[mydot,label=90:\(0/2\)] at (1.5,0.5)(g02){};
  \foreach \i/\j in
  {g02/g0,h2/g02,h2/h1}{
  \path[->](\i) edge (\j);}
  \end{scope}
  \end{tikzpicture}
  \caption{A pictorial proof}\label{chap6:fig:map-unit-bibundle-f-acyclic-1-right}
\end{figure}
\end{proof}

\begin{lemma}
  The morphism \(f\colon \Sigma\to \Gamma\) satisfies \(\Acyc!(2)\).
\end{lemma}
\begin{proof}
The lemma consists of two non-trivial statements. Consider Figure~\ref{chap6:fig:map-unit-bibundle-f-acyclic2-2-1}. The first statement says that \(f\) induces an isomorphism between \(\Hom(\partial\Simp{3}[2,1]\to \Simp{3}[2,1], \Gamma\to \Sigma)\), given by the left picture, and \(\Hom(\Simp{3}[1,2], \Gamma)\), given by the right picture. In the left picture, \((0,0/2)\) and \((1,1/2)\) are both degenerate. The statement follows because the space given by the middle picture is isomorphic to the spaces given by the other two pictures.
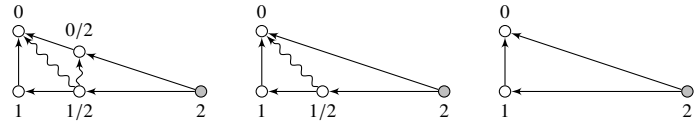
\begin{figure}[htbp]
  \centering
  \begin{tikzpicture}%
  [>=latex', hdot/.style={mydot,fill=lightgray},every label/.style={scale=0.6},scale=0.8]
  \node[mydot,label=90:\(0\)] at (0,1) (g0){};
  \node[mydot,label=-90:\(1\)] at (0,0) (g1){};
  \node[hdot,label=-90:\(2\)] at (3,0) (h2){};
  \node[mydot,label=-90:\(1/2\)] at (1,0)(g12){};
  \node[mydot,label=90:\(0/2\)] at (1,0.66)(g02){};
  \foreach \i/\j in
  {g1/g0,g12/g1,g02/g0,h2/g12,h2/g02}{
  \path[->](\i) edge (\j);}
  \path[->] (g12) edge[snake arrow] (g0)
  (g12) edge[snake arrow] (g02);
  \begin{scope}[xshift=4cm]
  \node[mydot,label=90:\(0\)] at (0,1) (g0){};
  \node[mydot,label=-90:\(1\)] at (0,0) (g1){};
  \node[hdot,label=-90:\(2\)] at (3,0) (h2){};
  \node[mydot,label=-90:\(1/2\)] at (1,0)(g12){};
  \foreach \i/\j in
  {g1/g0,g12/g1,h2/g0,h2/g12}{
  \path[->](\i) edge (\j);}
  \path[->] (g12) edge[snake arrow] (g0);
  \end{scope}
  \begin{scope}[xshift=8cm]
  \node[mydot,label=90:\(0\)] at (0,1) (g0){};
  \node[mydot,label=-90:\(1\)] at (0,0) (g1){};
  \node[hdot,label=-90:\(2\)] at (3,0) (h2){};
  \foreach \i/\j in
  {g1/g0,h2/g1,h2/g0}{
  \path[->](\i) edge (\j);}
  \end{scope}
  \end{tikzpicture}
  \caption{Proof of \(\Acyc!(2)[2,1]\)}\label{chap6:fig:map-unit-bibundle-f-acyclic2-2-1}
\end{figure}

The other statement is that
\[
\Hom(\Simp{3}[1,2], \Gamma)\cong \Hom(\partial\Simp{3}[1,2]\to \Simp{3}[1,2], \Gamma\to \Sigma),
\]
which follows by considering Figure~\ref{chap6:fig:map-unit-bibundle-f-acyclic2-1-2}.
\begin{figure}[htbp]
  \centering
  \begin{tikzpicture}%
  [>=latex', hdot/.style={mydot,fill=lightgray},every label/.style={scale=0.6},scale=0.8]
  \node[mydot,label=-90:\(0\)] at (0,0) (g0){};
  \node[hdot,label=-90:\(1\)] at (3,0) (h1){};
  \node[hdot,label=90:\(2\)] at (3,1) (h2){};
  \node[mydot,label=-90:\(0/1\)] at (1,0)(g01){};
  \node[mydot,label=90:\(0/2\)] at (1,0.33)(g02){};
  \foreach \i/\j in
  {g01/g0,g02/g0,h1/g01,h2/g02,h2/h1}{
  \path[->](\i) edge (\j);}
  \path[->] (h2) edge[snake arrow] (g01)
  (g02) edge[snake arrow] (g01);
  \begin{scope}[xshift=4cm]
  \node[mydot,label=-90:\(0\)] at (0,0) (g0){};
  \node[hdot,label=-90:\(1\)] at (3,0) (h1){};
  \node[hdot,label=90:\(2\)] at (3,1) (h2){};
  \node[mydot,label=-90:\(0/1\)] at (1,0)(g01){};
  \foreach \i/\j in
  {g01/g0,h2/g0,h1/g01,h2/h1}{
  \path[->](\i) edge (\j);}
  \path[->] (h2) edge[snake arrow] (g01);
  \end{scope}
  \begin{scope}[xshift=8cm]
  \node[mydot,label=-90:\(0\)] at (0,0) (g0){};
  \node[hdot,label=-90:\(1\)] at (3,0) (h1){};
  \node[hdot,label=90:\(2\)] at (3,1) (h2){};
  \foreach \i/\j in
  {h1/g0,h2/g0,h2/h1}{
  \path[->](\i) edge (\j);}
  \end{scope}
  \end{tikzpicture}
  \caption{Proof of \(\Acyc!(2)[1,2]\)}\label{chap6:fig:map-unit-bibundle-f-acyclic2-1-2}
\end{figure}
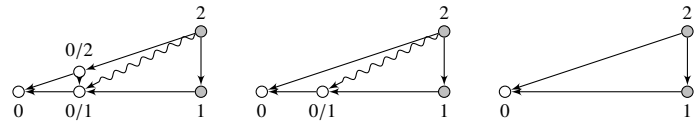
\end{proof}

The two lemmas above show that \(f\) is a colored version of a weak acyclic fibration. Hence the induced morphism on the groupoids of bigons, \(\arrow{\Sigma}\to \arrow{\Gamma}\), is a weak equivalence, and this finishes the proof of Proposition~\ref{chap6:prop:2gpd-bibundle-comp-unit}.

\section{Associativity of the composition}

In this section, we show that the composition of 2\nbdash{}groupoid bibundles is associative up to an isomorphism (not only a weak equivalence). Let \(X, Y, Z\), and  \(W\) be 2-groupoids in \((\Cat,\covers)\) in this section.

\begin{proposition}\label{chap6:prop:2gpd-bibundle-comp-ass}
  Let \(\Gamma, \Xi\), and \(\Omega\) be \(X\)\nbdash{}\(Y\), \(Y\)\nbdash{}\(Z\), and \(Z\)\nbdash{}\(W\) right principal bibundles, respectively. There is a canonical isomorphism \((\Gamma\otimes \Xi)\otimes\Omega \to \Gamma\otimes (\Xi\otimes \Omega)\).
\end{proposition}

Denote \((\Gamma\otimes \Xi)\otimes\Omega\) by \(\Sigma\) and \(\Gamma\otimes (\Xi\otimes \Omega)\) by \(\Pi\). It suffices to construct the isomorphism \(f\colon \Sigma\to \Pi\) on dimensions \(1, 2\), and \(3\). We divide the construction into several lemmas. First, let~\(f_{0,0}\) be the canonical isomorphism
\[
(\Gamma_{0,0}\times_{Y_0} \Xi_{0,0})\times_{Z_0}\Omega_{0,0}\to \Gamma_{0,0}\times_{Y_0} (\Xi_{0,0}\times_{Z_0}\Omega_{0,0})
\]
illustrated by Figure~\ref{chap6:fig:Sigma00-Pi00}, where white dots are in \(X\), gray dots in \(Y\), gray squares in \(Z\), and black dots in \(W\).
\begin{figure}[htbp]
  \centering
  \begin{tikzpicture}%
  [>=latex', hdot/.style={mydot,fill=lightgray},
  kdot/.style={inner sep=1.8pt,rectangle,draw,fill=lightgray},
  ldot/.style={mydot,fill},every label/.style={scale=0.6},scale=0.6]
  \node[mydot,label=180:\(0\)] at(0,0) (g0){};
  \node[hdot] at (1,0) (h0){};
  \node[kdot] at (2,0) (k0){};
  \node[ldot,label=0:\(1\)] at (4,0) (l0){};
  \foreach \i/\j in
  {h0/g0,k0/h0,l0/k0}{
  \path[->] (\i) edge (\j);
  }
  \begin{scope}[xshift=6cm]
  \node[mydot,label=180:\(0\)] at(0,0) (g0){};
  \node[hdot] at (2,0) (h0){};
  \node[kdot] at (3,0) (k0){};
  \node[ldot,label=0:\(1\)] at (4,0) (l0){};
  \foreach \i/\j in
  {h0/g0,k0/h0,l0/k0}{
  \path[->] (\i) edge (\j);
  }
  \end{scope}
  \end{tikzpicture}
  \caption{\(\Sigma_{0,0}\) and \(\Pi_{0,0}\)}\label{chap6:fig:Sigma00-Pi00}
\end{figure}
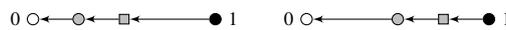

\begin{lemma}
  There are isomorphisms \(f_{0,1}\colon \Sigma_{0,1}\to \Pi_{0,1}\) and \(f_{1,0}\colon \Sigma_{1,0}\to \Pi_{1,0}\).
\end{lemma}
\begin{proof}
The spaces \(\Sigma_{0,1}\) and \(\Pi_{0,1}\) are given by the two pictures in the first row in Figure~\ref{chap6:fig:Sigma01-Pi01}, respectively. We use various decompositions to construct the compositions.
\begin{figure}[htbp]
  \centering
  \begin{tikzpicture}%
  [>=latex', hdot/.style={mydot,fill=lightgray},
  kdot/.style={inner sep=1.8pt,rectangle,draw,fill=lightgray},
  ldot/.style={mydot,fill},every label/.style={scale=0.6},scale=0.6]
  \node[mydot,label=180:\(0\)] at (0,0) (g0){};
  \node[ldot,label=0:\(1\)] at (4,0) (l1){};
  \node[ldot,label=0:\(2\)] at (4,2) (l2){};
  \node[kdot] at (2,0) (k01){};
  \node[kdot] at (2,1) (k02){};
  \node[hdot] at (1,0) (h01){};
  \node[hdot] at (1,0.5)(h02){};
  \foreach \i/\j in
  {h01/g0,h02/g0,k01/h01,k02/h02,l1/k01,l2/k02,l2/l1}{
  \path[->] (\i) edge (\j);}
  \foreach \i/\j in
  {h02/h01,k02/k01,l1/k02,k01/h02}{
  \path[->] (\i) edge[snake arrow] (\j);}
  \begin{scope}[xshift=6cm]
  \node[mydot,label=180:\(0\)] at (0,0) (g0){};
  \node[ldot,label=0:\(1\)] at (4,0) (l1){};
  \node[ldot,label=0:\(2\)] at (4,2) (l2){};
  \node[hdot] at (2,0) (h01){};
  \node[hdot] at (2,1)(h02){};
  \node[kdot] at (3,0) (k01){};
  \node[kdot] at (3,1.5) (k02){};
  \node[kdot] at (3,0.5)(k12){};
  \foreach \i/\j in
  {h01/g0,h02/g0,k01/h01,k02/h02,l1/k01,l2/k02,l2/l1}{
  \path[->] (\i) edge (\j);}
  \foreach \i/\j in
  {h02/h01,k02/k12,k12/k01,l1/k12,k12/h02,l1/k02,k01/h02}{
  \path[->] (\i) edge[snake arrow] (\j);}
  \end{scope}
  \begin{scope}[yshift=-2cm]
  \node[hdot] at (0,0)(h02){};
  \node[kdot] at (2,-1) (k01){};
  \node[kdot] at (2,1) (k02){};
  \node[ldot] at (4,0) (l1){};
  \foreach \i/\j in
  {k02/h02,l1/k01,k01/h02,l1/k02}{
  \path[->] (\i) edge (\j);}
  \path[->] (k02) edge[snake arrow] (k01);
  \end{scope}
  \begin{scope}[xshift=6cm, yshift=-2cm]
  \node[hdot] at (0,0)(h02){};
  \node[kdot] at (2,-1) (k01){};
  \node[kdot] at (2,1) (k02){};
  \node[ldot] at (4,0) (l1){};
  \node[kdot] at (2,0) (k12){};
  \foreach \i/\j in
  {k02/h02,l1/k01,k01/h02,l1/k02}{
  \path[->] (\i) edge (\j);}
  \foreach \i/\j in
  {k02/k12,k12/k01,l1/k12,k12/h02}{
  \path[->] (\i) edge[snake arrow] (\j);}
  \end{scope}
  \end{tikzpicture}
  \caption{\(\Sigma_{0,1}\) and \(\Pi_{0,1}\)}\label{chap6:fig:Sigma01-Pi01}
\end{figure}
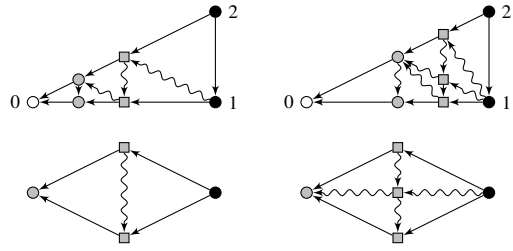

It suffices to prove that the two spaces in the second row are isomorphic. The horizontal two snake arrows in the bottom right picture stand for an action of \(\arrow{\Xi\otimes\Omega}\). Since \(\Theta_1\cong (\Theta_1\times_{\source,\Theta_0,\target}\Theta_1)/\Theta\) for a groupoid \(\Theta\) in \((\Cat,\covers)\), there is an induced isomorphism \(\Sigma_{0,1}\to \Pi_{0,1}\), denoted by \(f_{0,1}\).

The isomorphism \(f_{1,0}\colon \Sigma_{1,0}\to \Pi_{1,0}\) is given by a symmetric construction by considering the two pictures in Figure~\ref{chap6:fig:Sigma10-Pi10}.
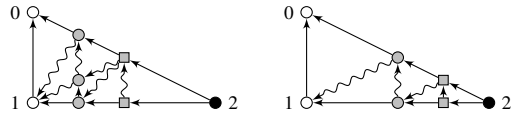
\begin{figure}[htbp]
  \centering
  \begin{tikzpicture}%
  [>=latex', hdot/.style={mydot,fill=lightgray},
  kdot/.style={inner sep=1.8pt,rectangle,draw,fill=lightgray},
  ldot/.style={mydot,fill},every label/.style={scale=0.6},scale=0.6]
  \node[mydot,label=180:\(0\)] at (0,2) (g0){};
  \node[mydot,label=180:\(1\)] at (0,0) (g1){};
  \node[ldot,label=0:\(2\)] at (4,0) (l2){};
  \node[kdot] at (2,1) (k02){};
  \node[kdot] at (2,0) (k12){};
  \node[hdot] at (1,1.5) (h02){};
  \node[hdot] at (1,0) (h12){};
  \node[hdot] at (1,0.5) (h01){};
  \foreach \i/\j in
  {g1/g0,h02/g0,h12/g1,k02/h02,k12/h12,l2/k02,l2/k12}{
  \path[->] (\i) edge (\j);}
  \foreach \i/\j in
  {k12/k02,k02/h12,k02/h01,h01/g1,h02/g1,h12/h01,h01/h02}{
  \path[->] (\i) edge[snake arrow] (\j);}
  \begin{scope}[xshift=6cm]
  \node[mydot,label=180:\(0\)] at (0,2) (g0){};
  \node[mydot,label=180:\(1\)] at (0,0) (g1){};
  \node[ldot,label=0:\(2\)] at (4,0) (l2){};
  \node[hdot] at (2,1) (h02){};
  \node[hdot] at (2,0) (h12){};
  \node[kdot] at (3,0.5) (k02){};
  \node[kdot] at (3,0) (k12){};
  \foreach \i/\j in
  {g1/g0,h02/g0,h12/g1,k02/h02,k12/h12,l2/k02,l2/k12}{
  \path[->] (\i) edge (\j);}
  \foreach \i/\j in
  {h12/h02,h02/g1,k12/k02,k02/h12}{
  \path[->] (\i) edge[snake arrow] (\j);}
  \end{scope}
  \end{tikzpicture}
  \caption{\(\Sigma_{1,0}\) and \(\Pi_{1,0}\)}\label{chap6:fig:Sigma10-Pi10}
\end{figure}
\end{proof}

It is easy to see that the maps \(f_{0,0}, f_{0,1}\), and \(f_{1,0}\) constructed above are compatible with face and degeneracy maps.

\begin{remark}\label{chap6:rem:map-ass-f-01-f10-reps}
The isomorphism \(f_{0,1}\colon \Sigma_{0,1}\to \Pi_{0,1}\) can be described on representatives as follows. For an element in \(\Sigma_{0,1}\) given by representatives as in the left picture in Figure~\ref{chap6:fig:map-ass-f-0-1}, we send it to \(\Pi\) by adding two degenerate 2-simplices illustrated by the right picture in Figure~\ref{chap6:fig:map-ass-f-0-1}, where the two thick arrows indicate suitable degenerate 2-simplices. There is a symmetric construction for \(f_{1,0}^{-1}\colon \Pi_{0,1}\to \Sigma_{1,0}\).
\begin{figure}[htbp]
  \centering
  \begin{tikzpicture}%
  [>=latex', hdot/.style={mydot,fill=lightgray},
  kdot/.style={inner sep=1.8pt,rectangle,draw,fill=lightgray},
  ldot/.style={mydot,fill},every label/.style={scale=0.6},scale=0.6]
  \node[mydot,label=180:\(0\)] at (0,0) (g0){};
  \node[ldot,label=0:\(1\)] at (4,0) (l1){};
  \node[ldot,label=0:\(2\)] at (4,2) (l2){};
  \node[kdot] at (2,0) (k01){};
  \node[kdot] at (2,1) (k02){};
  \node[hdot] at (1,0) (h01){};
  \node[hdot] at (1,0.5)(h02){};
  \foreach \i/\j in
  {h01/g0,h02/g0,k01/h01,k02/h02,l1/k01,l2/k02,l2/l1}{
  \path[->] (\i) edge (\j);}
  \foreach \i/\j in
  {h02/h01,k02/k01,l1/k02,k01/h02}{
  \path[->] (\i) edge (\j);}
  \begin{scope}[xshift=6cm]
  \node[mydot,label=180:\(0\)] at (0,0) (g0){};
  \node[ldot,label=0:\(1\)] at (4,0) (l1){};
  \node[ldot,label=0:\(2\)] at (4,2) (l2){};
  \node[hdot] at (2,0) (h01){};
  \node[hdot] at (2,1)(h02){};
  \node[kdot] at (3,0) (k01){};
  \node[kdot] at (3,1.5) (k02){};
  \foreach \i/\j in
  {h01/g0,h02/g0,k01/h01,k02/h02,l2/k02,l2/l1}{
  \path[->] (\i) edge (\j);}
  \foreach \i/\j in 
  {h02/h01,l1/k02,k02/k01}{
  \path[->] (\i) edge (\j);}
  \path[->,very thick]
    (l1)  edge (k01)
    (k01) edge (h02);
  \end{scope}
  \end{tikzpicture}
  \caption{A simple construction of \(f_{0,1}\colon \Sigma_{0,1}\to \Pi_{0,1}\)}\label{chap6:fig:map-ass-f-0-1}
\end{figure}
\end{remark}

\begin{lemma}
There are isomorphisms \(f_{0,2}\colon \Sigma_{0, 2}\to \Pi_{0,2}\) and \(f_{0,2}\colon\Sigma_{2,0}\to \Pi_{2,0}\). Moreover, these isomorphisms are compatible with \(f_{0,0}\), \(f_{0,1}\), \(f_{1,0}\), face, and degeneracy maps.
\end{lemma}
\begin{proof}
  We only consider \(\Sigma_{0, 2}\to \Pi_{0,2}\); the other case is given by symmetry. We already know that \(\Hom(\Horn{3}{2}[1,3],\Sigma)\) and \(\Hom(\Horn{3}{2}[1,3], \Pi)\) are isomorphic. Since both \(\Sigma\) and \(\Pi\) satisfy \(\Kan!(3,2)[1,3]\), it suffices to show that, given two horns \(\Horn{2}{3}[1,3]\) in \(\Sigma\) and in \(\Pi\) that are related by \(f_{0,1}\), the missing triangles determined by \(\Kan!(3,2)[1,3]\) are related by \(f_{0,1}\). We can then define \(f_{0,2}\) to be the isomorphism induced by
  \[
  \Sigma_{0,2}=\Hom(\Horn{3}{2}[1,3],\Sigma)\xrightarrow{\cong} \Hom(\Horn{3}{2}[1,3], \Pi)=\Pi_{0,2}.
  \]
  To this end, we choose representatives such that each pair of 2\nbdash{}simplices is related as in Remark~\ref{chap6:rem:map-ass-f-01-f10-reps}, then the claim will become obvious. The left picture in Figure~\ref{chap6:fig:Sigma02-Pi02} illustrates how to obtain \((0,1,3)\) with \((0,1,2), (0,2,3)\), and \((1,2,3)\) given in~\(\Sigma\). Adding some degenerate 2-simplices to the left picture, we obtain a horn in \(\Pi\) as shown in the right picture. Now we add some appropriate degenerate 3-simplices as well, then the same picture can be reinterpreted as how to obtain \((0,1,3)\) in~\(\Pi\). This proves the claim.
\begin{figure}[htbp]
  \centering
  \begin{tikzpicture}%
  [>=latex', hdot/.style={mydot,fill=lightgray},
  kdot/.style={inner sep=1.8pt,rectangle,draw,fill=lightgray},
  ldot/.style={mydot,fill},every label/.style={scale=0.6},scale=0.6]
  \node[mydot,label=180:\(0\)] at (0,0) (g0){};
  \node[ldot,label=0:\(1\)] at (4,0) (l1){};
  \node[ldot,label=0:\(2\)] at (4,1.2) (l2){};
  \node[kdot] at (2,0) (k01){};
  \node[kdot] at (2,0.6) (k02){};
  \node[hdot] at (1,0) (h01){};
  \node[hdot] at (1,0.3)(h02){};
  \foreach \i/\j in
  {h01/g0,h02/g0,k01/h01,k02/h02,l1/k01,l2/k02,l2/l1,h02/h01,k02/k01,l1/k02,k01/h02}{
  \path[->] (\i) edge (\j);}
  \node[ldot,label=180:\(3\)] at (3,3)(l3){};
  \node[hdot] at (0.75,0.75) (h03){};
  \node[kdot] at (1.5,1.5) (k03){};
  \foreach \i/\j in
  {h03/g0,k03/h03,l3/k03,l3/l2,l3/l1,h03/h02,k03/k02,l2/k03,k02/h03,h03/h01,k03/k01,l1/k03,k01/h03}{
  \path[->] (\i) edge (\j);}
  \begin{scope}[xshift=6cm]
  \node[mydot,label=180:\(0\)] at (0,0) (g0){};
  \node[ldot,label=0:\(1\)] at (4,0) (l1){};
  \node[ldot,label=0:\(2\)] at (4,1.2) (l2){};
  \node[kdot] at (3,0) (k01){};
  \node[kdot] at (3,0.9) (k02){};
  \node[hdot] at (2,0) (h01){};
  \node[hdot] at (2,0.6)(h02){};
  \foreach \i/\j in
  {h01/g0,h02/g0,k01/h01,k02/h02,l1/k01,l2/k02,l2/l1,h02/h01,k02/k01,l1/k02,k01/h02}{
  \path[->] (\i) edge (\j);}
  \node[ldot,label=180:\(3\)] at (3,3)(l3){};
  \node[hdot] at (1.5,1.5) (h03){};
  \node[kdot] at (2.25,2.25) (k03){};
  \foreach \i/\j in
  {h03/g0,k03/h03,l3/k03,l3/l2,l3/l1,h03/h02,k03/k02,l2/k03,k02/h03,h03/h01,k03/k01,l1/k03,k01/h03}{
  \path[->] (\i) edge (\j);}
  \foreach \i/\j in
  {l1/k01,k01/h02,l2/k02,k02/h03,k01/h03}{
  \path[->,very thick] (\i) edge (\j);}
  \end{scope}
  \end{tikzpicture}
  \caption{The construction of \(f_{0,2}\colon \Sigma_{0, 2}\to \Pi_{0,2}\)}\label{chap6:fig:Sigma02-Pi02}
\end{figure}
\end{proof}

In the same spirit, it is not hard to show the following lemma.

\begin{lemma}
There is an isomorphism \(f_{1,1}\colon \Sigma_{1, 1}\to \Pi_{1,1}\). Moreover, this isomorphism is compatible with \(f_{0,0}\), \(f_{0,1}\), \(f_{1,0}\), face, and degeneracy maps.
\end{lemma}

We now complete the proof of Proposition~\ref{chap6:prop:2gpd-bibundle-comp-ass}. With the maps \(f_{0,0}\), \(f_{0,1}, f_{1,0}\), \(f_{0,2}, f_{1,1}\), and \(f_{2,0}\) constructed above, we obtain an isomorphism \(f\colon\Sigma\to \Pi\).

\chapter{Differentiation of Higher Lie Groupoids}\label{chap7}
\thispagestyle{empty}

In this chapter, we study the differentiation of higher Lie groupoids. Our main aim is to prove that \v{S}evera's 1-jet functor of a higher Lie groupoid is representable. Thus the 1\nbdash{}jet functor assigns an NQ-manifold (higher Lie algebroid) to a higher Lie groupoid. We also show explicitly that the 1\nbdash{}jet functor coincides with the usual notion of differentiation for three special cases: Lie groups, crossed modules, and Lie groupoids.

\section{Basics on supermanifolds}

We start by collecting some basic facts about supermanifolds. Our main references are~\cite{Carmeli-Caston-Fioresi, Leites, Varadarajan}. The ground field is assumed to be the real numbers.

\subsection{Basic definitions}
Recall that a \emph{super-vector space} \(V\) is a \(\bZ_2\)-graded vector space, that is, \(V=V^\even\oplus V^\odd\). The elements in \(V^\even\) are even, and those in \(V^\odd\) are odd. An even or odd element \(v\) is called a \emph{homogeneous element}, its parity in \(\bZ_2\) is denoted by \(\parity{v}\). A parity-preserving (-reversing) linear map of super-vector spaces is called even (odd).

A \emph{superalgebra} \(A\) is an algebra on a super-vector space such that \(A^i A^j\subset A^{i+j}\). A superalgebra \(A\) is \emph{commutative} if\footnote{When we write formulas involving the parity of certain objects, we implicitly assume that these objects are homogeneous.}
\[
uv=(-1)^{\parity{u}\parity{v}}vu, \quad u, v\in A.
\]
In the following, all superalgebras are assumed to be commutative. Homomorphisms of superalgebras are even algebra homomorphisms. Denote the category of superalgebras by \(\SAlg\).

The \emph{tensor product} of two superalgebras \(A\) and \(B\) is a superalgebra \(A \otimes B\) with the multiplication determined by
\[
(a_1\otimes b_1)(a_2\otimes b_2) = (-1)^{\parity{b_1}\parity{a_2}}(a_1a_2\otimes b_1b_2).
\]

\begin{example}
The following two examples are important for us.
\begin{enumerate}
  \item A superalgebra \(A\) with \(A^\odd=\{0\}\) is a usual algebra; for instance, the algebra of smooth functions on an open subset \(U\subset \bR^p\);
  \item The \emph{Grassmann algebra} (exterior algebra) on \(q\) odd generators \(\theta^1,\dots,\theta^q\), denoted by~\(\Lambda(q)\), is a superalgebra.
\end{enumerate}
\end{example}

\begin{definition}
  Let \(f_0\colon A\to B\) be a homomorphism of superalgebras. An even (odd) \emph{derivation} with respect to \(f_0\) is an even (odd) linear map \(f_1\colon A\to B\) satisfying
  \[
  f_1(uv)=f_1(u)f_0(v)+(-1)^{\parity{f_1}\parity{u}}f_0(u)f_1(v), \quad u,v \in A.
  \]
\end{definition}

The \emph{standard supermanifold} \(\bR^{p|q}\) is the topological space \(\bR^p\) endowed with the sheaf \(C^\infty(\bR^p)\otimes \Lambda(q)\), where \(C^\infty(\bR^p)\) is the sheaf of smooth functions on~\(\bR^p\). We call \(\bR^{0|1}\) the \emph{odd line} and denote it by~\(D\).

\begin{definition}
A \emph{supermanifold} of dimension \(p|q\) is a ringed space \(M=(|M|, C_M)\), where \(|M|\) is a second countable Hausdorff space and \(C_M\) is a sheaf of superalgebras on~\(|M|\) that is locally isomorphic to \(\bR^{p|q}\).

A \emph{morphism between supermanifolds} is a morphism of the corresponding ringed spaces. Denote by \(\SMfd\) the category of supermanifolds.
\end{definition}

Explicitly, a morphism of supermanifolds \(f\colon M\to N \) consists of a continuous map \(|f|\colon |M|\to |N|\) and a morphism of sheaves \(f^*\colon C_N \to f_*C_M\), where \(f_*C_M\) is the direct image sheaf
\[
|U|\mapsto C_M(|f|^{-1}(|U|)),\quad |U|\subset |N|.
\]
Local sections of \(C_M\) are called local \emph{functions} on \(M\) and \(C_M\) is called the \emph{structure sheaf} of functions. The set of \emph{global sections} \(C_M(|M|)\) is denoted by \(C(M)\).

Odd functions generate a sheaf of nilpotent ideals, denoted by \(\Nil\). Then \((|M|,C_M/\Nil)\) is manifold of dimension \(p\), which is called the \emph{reduced manifold} and denoted by \(\reduce{M}\). This amounts to a functor \(\tilde{\phantom{r}}\colon\SMfd\to\Mfd\). For a local function \(f\) around \(x\in |U|\subset |M|\) its \emph{value at \(x\)} is the number \(\tilde{f}(x)\). This gives a homomorphism \(\ev_x\colon C_M(|U|)\to \bR\).

\begin{example}\label{chap7:exa:smfds}
We need the following examples in the future.
\begin{enumerate}
  \item Let \(M\) be a supermanifold. An open subset \(|U|\subset |M|\) determines an \emph{open subsupermanifold} \(U=(|U|,  C_M\big|_{|U|})\). Open subsupermanifolds of \(\bR^{p|q}\) are called \emph{superdomains}.
  \item Let \(E\to M\) be a vector bundle. The \emph{shifted vector bundle} \(\Pi E\) is a supermanifold given by the sheaf of sections \( \Gamma \Lambda^\bullet E^*\) on~\(M\). The supermanifold \(\Pi TM\) is called the \emph{odd tangent bundle} of~\(M\).
\end{enumerate}
\end{example}

If \((x^1,\dots,x^p)\) are global coordinates of \(\bR^p\) and \( (\theta^1,\dots,\theta^q) \) are generators of \(\Lambda(q)\), we call \((x^1,\dots,x^p, \theta^1,\dots,\theta^q)\) a system of \emph{global coordinates} for the supermanifold \(\bR^{p|q}\). We can speak of \emph{local coordinates} for a general supermanifold. Generally, we denote even variables by Latin letters and odd variables by Greek letters.

The following is a generalisation of Hadamard's Lemma; see~\cite{Leites} for the precise formulation of the statement.
\begin{lemma}\label{chap7:lem:taylor}
A function on a superdomain can be expanded uniquely as a Taylor series in coordinates with remainder term in an ideal.
\end{lemma}

\subsection{Functor of points}

The language of functor of points is a very useful tool for studying supermanifolds.

\begin{definition}
The \emph{functor of points} \(h_M\) of a supermanifold \(M\) is the presheaf on \(\SMfd\) represented by~\(M\).
An element \(f\) of \(h_M(T)\) is called a \emph{\(T\)\nbdash{}point} of \(M\), which is also denoted by \(f\in_T  M\). More generally, for a presheaf \(F\) on \(\SMfd\), elements in \(F(T)\) are called \(T\)\nbdash{}points of \(F\).
\end{definition}

The Yoneda Lemma (Lemma~\ref{chap2:lem:Yoneda}) allows us to replace a supermanifold by its functor of points. A morphism of two supermanifolds \(f\colon M\to N\) is equivalent to a functorial family of morphisms \(h_M(T)\to h_N(T)\) for every supermanifold~\(T\).

The following two theorems are very useful to study morphisms of supermanifolds.

\begin{theorem}[\cite{Leites}]\label{chap7:thm:morphism}
Let \(U^{p|q}\) be a  superdomain with coordinates \((x, \theta)=(x^i,\theta^j)\). Let \(M\) be a supermanifold. There is a bijection between morphisms \(M\to U^{p|q}\) and the set of pullbacks of \((x, \theta)\), which are \((p+q)\)-tuples \((u^i,\zeta^j)\), where \(u^i\in C(M)^\even\) and \(\zeta^j\in C(M)^\odd\), such that \((\reduce{u}^i)\) gives a morphism of manifolds \(\reduce{M}\to U^p\).
\end{theorem}

A morphism of supermanifolds is then a compatible family of such local morphisms. On the other hand, a morphism of supermanifolds is determined by global sections.

\begin{theorem}[\cite{Leites}]\label{chap7:thm:global}
For supermanifolds \(M\) and \(N\), there is a natural bijection
\[
\SMfd(M,N)=\SAlg(C(N),C(M)).
\]
\end{theorem}

\begin{example}\label{chap7:exa:T-points-of-R-p-q}
Let \((x,\theta)=(x^i,\theta^j)\) be coordinates on \(\bR^{p|q}\). Theorem~\ref{chap7:thm:morphism} implies that we can identify \(T\)\nbdash{}points of \(\bR^{p|q}\) with the set of all \((p+q)\)-tuples \((u,\zeta)=(u^i,\zeta^j)\), where \(u^i\in C(T)^\even\) and \(\zeta^j\in C(T)^\odd\). This yields an isomorphism by comparing \(T\)\nbdash{}points
\[
\bR^{p|q} \times \bR^{s|t}\cong \bR^{p+s|q+t}.
\]
Similarly, for a supermanifold \(M\) there is an isomorphism of sheaves of superalgebras on~\(|M|\)
\[
C_{M\times \bR^{0|q}}\cong C_M\otimes \Lambda(q).
\]
\end{example}

\begin{definition}\label{chap7:def:surj-subm}
A morphism of supermanifolds \(f\colon M\to N\) is a \emph{submersion} if it is locally isomorphic to the projection \(\bR^{p|q} \times \bR^{s|t}\to \bR^{p|q} \). If, in addition, the reduced morphism \(\reduce{f}\) is a surjection then \(f\) is a \emph{surjective submersion}.
\end{definition}

\begin{definition}\label{chap7:def:parity}
The \emph{parity involution} of a superalgebra \(A\) is the automorphism given by \(a\mapsto (-1)^{\parity{a}}a \). This defines an endofunctor of \(\SAlg\). It induces the parity involution on supermanifolds and further on presheaves on~\(\SMfd\).
\end{definition}

\subsection{Symbolic notation}
We can use local coordinates for supermanifolds just as we did for manifolds. Let \(M\) and \(N\) be supermanifolds with local coordinates \((x,\theta)=(x^i,\theta^j)\) and \((y,\eta)=(y^k,\eta^l)\), respectively. Locally, we can think of a morphism \(\psi\colon M\to N\) symbolically as
\[
(x,\theta) \mapsto (y,\eta), \qquad y=y(x, \theta),\quad \eta=\eta(x,\theta),
\]
which is an abuse of notation for the map \(\psi^*\)
\[
\psi^*y^k=y^k(x,\theta)\in C(M)^\even,\quad \psi^*\eta^l=\eta^l(x,\theta)\in C(M)^\odd.
\]
This notation is convenient since it resembles the usual function notation. We will specify explicitly when we use this notation.

The following example demonstrates how this works and how it relates to the functor of points.
\begin{example}\label{chap7:exa:T-points-coordinates}
Let us continue Example~\ref{chap7:exa:T-points-of-R-p-q}. The supermanifold \(\bR^{p|q}\) is a super Lie group with addition in symbolic notation
\begin{equation}\label{chap7:eq:Rpq-m}
\begin{gathered}
m\colon \bR^{p|q}\times \bR^{p|q} \xrightarrow{+} \bR^{p|q}\\
 m((x_1,\theta_1), (x_2,\theta_2))= (x_1+x_2, \theta_1+\theta_2),
\end{gathered}
\end{equation}
which really means in pullback notation
\[
m^*(x,\theta)=(x_1+x_2,\theta_1+\theta_2),
\]
where low indices indicate functions on different copies. Let \(f_1,f_2\colon T\to \bR^{p|q}\) be \(T\)-points of \(\bR^{p|q}\) such that \(f_i^* (x,\theta)=(u_i,\zeta_i)\). Their addition \(f=m\circ (f_1,f_2)\) is a \(T\)-point of \(\bR^{p|q}\) such that \(f^*(x,\theta)=(u_1+u_2,\zeta_1+\zeta_2)\). Thus the addition in \(T\)\nbdash{}points is
\begin{equation}\label{chap7:eq:Rpq-m-p}
\begin{gathered}
m\colon \hom(T, \bR^{p|q})\times \hom(T, \bR^{p|q}) \xrightarrow{+} \hom(T,\bR^{p|q})\\
 m((u_1,\zeta_1), (u_2,\zeta_2))= (u_1+u_2, \zeta_1+\zeta_2).
\end{gathered}
\end{equation}
This gives a group structure on the set of \(T\)\nbdash{}points that is functorial in~\(T\).
\end{example}

Observe that~\eqref{chap7:eq:Rpq-m} and~\eqref{chap7:eq:Rpq-m-p} have the same pattern. The formula in \(T\)\nbdash{}points is just the pullback of that in coordinates to the parameter space~\(T\).\footnote{Some authors use the same letters to denote both coordinates and their pullbacks on~\(T\).} It is often the case that the formula in \(T\)\nbdash{}points is easier to obtain. ``Thus the informal or symbolic way of thinking and writing about supermanifolds is essentially the same as the mode of operation furnished by the language of the functor of points''~\cite[p. 151]{Varadarajan}.

\begin{remark}\label{chap7:rem:RD-morphisms}
Similarly, we have morphisms of supermanifolds in \(T\)\nbdash{}points:
\begin{align*}
  m\colon &D\times D \xrightarrow{\cdot} \bR,\quad m(\ul{\delta}_1,\ul{\delta}_2)=\ul{\delta}_1 \ul{\delta}_2,\\
  m\colon & \bR\times D\xrightarrow{\cdot} D,\quad m(\ul{r},\ul{\delta})=\ul{r}\,\ul{\delta},
\end{align*}
where \(r\) is a coordinate on \(\bR\) and \(\ul{r} \in C(T)\) represents a \(T\)\nbdash{}point of \(\bR\) in this coordinate, \(\delta\) is a coordinate on \(D\) and \(\ul{\delta}\in C(T)\) represents a \(T\)\nbdash{}point of \(D\) in this coordinate. The morphism \(\bR\times D\xrightarrow{\cdot} D\) gives an action of \((\bR,\times)\) on~\(D\). The semidirect product \(\bR\ltimes D\cong \bR^{1|1}\) has multiplication in \(T\)-points
\[
\bR^{1|1}\times \bR^{1|1} \to \bR^{1|1},\quad
(\ul{r}_1,\ul{\delta}_1)\cdot (\ul{r}_2,\ul{\delta}_2)=( \ul{r}_1\ul{r}_2,\ul{r}_2\ul{\delta}_1+\ul{\delta}_2).
\]
\end{remark}

\subsection{Tangent vectors}

Like manifolds, tangent vectors on supermanifolds are defined by derivations.

\begin{definition}
Let \(M\) be a supermanifold. An even (odd) \emph{tangent vector} \(v \) at \(x\in |M|\) is an even (odd) derivation of the stalk, that is, an even (odd) linear map \(v\colon C_{M,x}\to \bR\) such that \(v(fg)=v(f)\ev_x(g)+(-1)^{\parity{v}\parity{f}}\ev_x(f) v(g)\).
\end{definition}

Let \(v\) be a tangent vector at \(x\). Let \(U\) be a neighbourhood of \(x\). Composing with the restriction map \(C_M(U)\to C_{M,x}\) yields a derivation \(C_M(U)\to \bR\),denoted also by \(v\).

\begin{definition}
   A \emph{vector filed} \(V\) on \(M\) is a derivation of \(C_M\), that is, a family of derivations \(V_U\colon C_M(U)\to C_M(U)\) which is compatible with restrictions.
\end{definition}

Let \((x^i,\theta^j)\) be coordinates on an open subsupermanifold \(U\subset M\). One can show that a derivation is determined by its value on coordinates~\cite[Lemma 4.4.4]{Carmeli-Caston-Fioresi}. The coordinate vector fields \(\partial_{x^i}\) and \(\partial_{\theta^j}\) are derivations determined by
\[
\partial_{x^i} x^{i}=1,\quad \partial_{x^i}x^k=0 \text{ for }k \neq i, \quad \partial_{x^i}\theta^j=0,
\]
and
\[
\partial_{\theta^j}x^i=0,\quad \partial_{\theta^j}\theta^j=1, \quad \partial_{\theta^j}\theta^k \text{ for }k\neq j.
\]
Coordinate vectors at \(x\) are derivations given by
\[
\partial_{x^i}|_x=\ev_x\circ \partial_{x^i},\quad \partial_{\theta^j}|_x=\ev_x\circ \partial_{\theta^j}.
\]
Coordinate vectors form a basis of the tangent space at~\(x\).

We will need the infinitesimal action of a super Lie group action. The following is~\cite[Definition 8.2.1]{Carmeli-Caston-Fioresi}; see~\cite{Carmeli-Caston-Fioresi} for more details.

\begin{definition}\label{chap7:def:induced-vector}
Let \(G\) be a super Lie group, \(M\) a supermanifold, and let \(\alpha\colon G\times M\to M \) be an action. Given \(v\in T_e G\), the composite
\[
\rho_\alpha(v)\colon C_M(U) \xrightarrow{\alpha^*} C_{G\times M}(U_e\times U) \xrightarrow{v_M} C_M(U)
\]
is a derivation of \(C_M(U)\) for every open \(U\subset |M|\) and \(U_e \subset |G|\), where \(v_M\) is the unique \(C_M\)-linear derivation with \(v_M(a\otimes b)=v(a)b\). The derivation~\(\rho_\alpha(v)\) is called the \emph{induced vector field}.
\end{definition}

\subsection{A criterion for representability}

We recall a criterion for representability of presheaves on \(\SMfd\). This is a special case of a general statement in algebraic geometry; see~\cite{Carmeli-Caston-Fioresi} for more details.

Let us introduce some terminology. Let \(M\) be a supermanifold. If \(\{|U_i|\}\) is an open cover of \(|M|\), then we call the open subsupermanifolds determined by \(\{|U_i|\}\) an \emph{open cover} of~\(M\). With the notion of open covers, we can speak of sheaves on \(\SMfd\). The notion of open covers extends to presheaves. Let \(U\to F\) be a morphism of presheaves on \(\SMfd\). We call \(U\) a \emph{subfunctor} of \(F\) if \(U(T)\to F(T)\) is a subset for every~\(T\). We call a subfunctor \(U\to F\) \emph{open} if for every representable \(V\) with a morphism \(V\to F\) the pullback \(V\times_F U\) is representable and \(V\times _F U\to V\) is an open supermanifold. A family of open subfunctors \(\{U_i\to F\}\) is an \emph{open cover} of \(F\) if for every representable \(V\) with a presheaf morphism \(V\to F\), the pullbacks \(\{V\times_F U_i\}\) form an open cover of \(V\). Furthermore, a representable open subfunctor is called an \emph{open supermanifold subfunctor}.

\begin{theorem}[{\cite[Theorem~4.3]{Carmeli-Caston-Fioresi}}]
A presheaf on \(\SMfd\) is representable if and only if it is a sheaf and covered by a family of open supermanifold subfunctors.
\end{theorem}

In practice, we use the following direct corollary.

\begin{corollary}\label{chap7:cor:criterion-rep}
Let \(F \to B\) be a morphism of sheaves on \(\SMfd\). Suppose that \(B\) is representable and covered by a family of open supermanifold subfunctors \(\{U_i\}\). If \(U_i\times_B F\) is representable for every \(i\), then \(F\) is representable and covered by \(\{U_i\times_B F\}\).
\end{corollary}

\begin{example}
  The product \(M\times N\) is representable for all supermanifolds \(M\) and \(N\). Applying Corollary~\ref{chap7:cor:criterion-rep} to \(M\times N\to M\), we are reduced to the case of \(M=\bR^{p|q}\); it is further reduced to the case of \(N=\bR^{s|t}\), which is given by Example~\ref{chap7:exa:T-points-of-R-p-q}.
\end{example}

\begin{definition}\label{chap7:def:internal-hom}
Let \(L\) and \(M\) be supermanifolds. The \emph{internal hom} is the presheaf on \(\SMfd\) defined by
\[
 \bhom(L, M)\colon T\mapsto \hom(L\times T, M), \quad T\in \SMfd.
\]
We also denote the internal hom by \(M^L\). For \(f\colon L\times T\to M\), let \(\hat{f}\) denote the corresponding morphism \( T\to M^L\).
\end{definition}

The \emph{evaluation} is the morphism defined in \(T\)\nbdash{}points by
\[
\ev\colon \bhom(L, M)\times L\to M, \quad(f,g)\mapsto f\circ(g,\id_T),
\]
where \(f\colon L\times T\to M\) and \(g\colon T\to L\). The \emph{composition} is the morphism defined by
\[
\mu\colon \bhom(L, M)\times \bhom(M, N)\to \bhom(L, N),\quad (g, f)\mapsto f\circ (g,\pr_T),
\]
where \(f\colon M\times T\to N\) and \(g\colon L\times T\to M\).

It is not hard to see that the internal hom is a sheaf on \(\SMfd\). We claim that \(\bhom(D^n, M)\) is representable for every supermanifold \(M\). Applying Corollary~\ref{chap7:cor:criterion-rep} to \(\bhom(D^n, M)\to M\), we are reduced to the case of \(M=\bR^{p|q}\). There are isomorphisms
\begin{align*}
\bhom(L, M_1\times M_2)&\cong \bhom(L, M_1)\times \bhom(L, M_2),\\
\bhom(L_2, \bhom(L_1, M))&\cong \bhom(L_1\times L_2, M).
\end{align*}
Thus it suffices to consider \(\bhom(D, D)\) and \(\bhom(D, \bR)\), which will be done in the next section.

\section{The odd tangent bundles and NQ-manifolds}

We show that the internal hom \(M^D\) can be identified with the odd tangent bundle and that the natural action of \(D^D\) on \(M^D\) gives the de~Rham complex structure. This old observation appeared in many places, for instance, \cite{Vaintrob}; see also~\cite{Kontsevich, Kochan-Severa, Severa06}. We then introduce the notion of NQ-manifolds.

\subsection{The odd tangent bundles}

Recall from Example~\ref{chap7:exa:smfds} that the odd tangent bundle \(\Pi TM\) of a manifold \(M\) is a supermanifold given by the sheaf of differential forms.

\begin{proposition}
Let \(M\) be a manifold. The internal hom \(M^D\) is representable and naturally isomorphic to \(\Pi TM\).
\end{proposition}
\begin{proof}
We prove the claim by comparing the \(T\)\nbdash{}points.

Let \(f\colon T \to \Pi TM\) be a \(T\)\nbdash{}point of \(\Pi TM\). Theorem~\ref{chap7:thm:global} implies that~\(f\) is equivalent to a superalgebra homomorphism \(f^*\colon \Omega(M)\to C(T)\). Denote by \(\pi\) the projection \(\Pi TM \to M\). Define \(f_0=\pi \circ f\colon T\to M\) and \(f_1=f^*\circ d\colon C(M)\to C(T)\), where \(d\) is the de~Rham differential. It is easy to check that \(f_1\) is an odd derivation with respect to \(f^*_0\). Conversely, a supermanifold morphism \(f_0\colon T\to M\) and an odd derivation with respect to~\(f^*_0\) determine~\(f\).

Let \(g\colon D\times T\to M \) be a \(T\)\nbdash{}point of \(M^D\). This gives a superalgebra homomorphism \(g^*\colon C(M) \to C(D\times T)\). Let \(\iota\colon T \to D\times T\) be the natural inclusion induced by \(\{0\} \xrightarrow{0} D\). Define \(g_0=g\circ \iota\colon T \to M\). Example~\ref{chap7:exa:T-points-of-R-p-q} implies that each \(y\in C(D\times T)\) can be uniquely written as \(y=v+\delta \eta\), where \(\delta\) is a fixed coordinate on \(D\) and \(v, \eta \in C(T)\). It is clear that \(\iota^*y=v\). Define \(p\colon C(D\times T)\to C(T)\) by \(y\mapsto \eta\). It is easy to verify that \(p\) is an odd derivation with respect to~\(\iota^*\), hence \(g_1=p\circ g^*\colon C(M) \to C(T)\) is an odd derivation with respect to \(g_0^*\). Conversely, a supermanifold morphism \(g_0\colon T\to M\) and an odd derivation with respect to \(g_0^*\) determine~\(g\).

It is clear that the argument above is functorial in \(T\). We conclude that \(M^D\) can be naturally identified with \(\Pi TM\).
\end{proof}

Let us write the above isomorphism in local coordinates. Let \((x^i)\) be local coordinates on~\(M\). We can take \((x^i,dx^i)\) as local coordinates on \(\Pi TM\), where \(d\) is the de~Rham differential. A morphism \(f\colon T\to \Pi TM\) is locally determined by \(f^*x^i\in  C(T)^\even\) and \(f^*d x^i\in C(T)^\odd\). It is clear that \(\{f^*x^i\}\) give the morphism \(f_0\colon T\to M\) and \(\{f^*d x^i\}\) give the odd derivation~\(f_1\). Conversely, let \(g\colon D\times T \to M \) be a morphism corresponding to \(\hat{g}\colon T\to \Pi TM\). Suppose that \(g^*x^i=u^i+\delta\zeta^i\). Then \(\{u^i\}\) give a morphism \(g_0\colon T\to M\) and \(\{\zeta^i\}\) give an odd derivation. Therefore, we can identify \((u^i,\zeta^i)\) with \((\hat{g}^*x^i,\hat{g}^*d x^i)\).

\begin{proposition}
The presheaf \(D^D\) is representable and naturally isomorphic to \(\bR\times D \).
\end{proposition}
\begin{proof}
Let \(\delta\) be a coordinate on \(D\) and let \(r\) be a coordinate on \(\bR\). A \(T\)\nbdash{}point of \(\bR\times D \), \(f\colon T\to \bR\times D \), is determined by \(f^*\delta\in C(T)^\odd\) and \(f^*r\in C(T)^\even\). On the other hand, a \(T\)\nbdash{}point of \(D^D\), \(g\colon D\times T\to D\), is given by \(g^*\delta\in C(D\times T)^\odd\). Decompose \(g^*\delta=\eta+\delta y\) with \(\eta \in C(T)^\odd\) and \(y \in C(T)^\even\). Thus~\(g\) is also determined by an odd element and an even element of \(C(T)\), and this proves the claim.
\end{proof}

Our next goal is to describe the monoid structure of \(D^D\) and the natural action of \(D^D\) on \(M^D\). As above, let \(g\colon D\times T\to D\) be a \(T\)-point of \(D^D\). Let \(\delta\) be a coordinate on \(D\) and let \((r,\delta)\) be coordinates on \(D^D\cong \bR\times D\). Suppose that \(g^*\delta=\eta+\delta y\) with \(\eta \in C(T)^\odd\) and \(y \in C(T)^\even\). Hence \(\hat{g}^*(r,\delta)=(y,\eta)\). Let \(f\) be a \(T\)-point of \(D\) such that \(f^*\delta=\theta\) with \(\theta\in C(T)^\odd\). Then \(g\circ(f,\id_T)\) is a \(T\)-point of \(D\) such that \((g\circ(f,\id_T))^*\delta=\eta+\theta y\). This gives the evaluation map in \(T\)-points
\[
 D^D\times D\xrightarrow{\ev}D, \quad \ev((y,\eta),\theta)=\eta+\theta y.
\]
Therefore, in coordinates we have (recall the rule in Example~\ref{chap7:exa:T-points-coordinates})
\[
\ev^*\delta'=\delta + \delta'r,
\]
where \(\delta'\) being the same as \(\delta\) is renamed.

Similarly, let \(g_1,g_2\colon D\times T\to D\) be \(T\)-points of \(D^D\) such that \(g_i^*\delta=\eta_i+\delta y_i\). Then \(g_2\circ (g_1,\pr_T)\) is a \(T\)-point of \(D^D\) such that \((g_2\circ (g_1,\pr_T))^*\delta=\eta_2+\eta_1 y_2+\delta y_1y_2\). This gives the composition in \(T\)\nbdash{}points
\[
\mu\colon D^D \times D^D \to D^D,\quad \mu((y_1,\eta_1),(y_2,\eta_2))=(y_1y_2,\eta_2+\eta_1 y_2).
\]
Therefore, in coordinates \((r,\delta)\) we have
\[
\mu^*(r,\delta)=(r_1 r_2,\delta_2+\delta_1 r_2).
\]
The monoid \(D^D\) is isomorphic to the semidirect product \(\bR \ltimes D\) in Remark~\ref{chap7:rem:RD-morphisms}.

\begin{proposition}
  Let \(M\) be a manifold. The infinitesimal action of \(\mu\colon D^D\times M^D \to M^D\) gives the de~Rham complex structure on differential forms of~\(M\).
\end{proposition}
\begin{proof}
We first write~\(\mu\) in local coordinates. Let \(x=(x^i)\) be local coordinates on \(M\). Let \(f\colon D\times T\to M\) be a \(T\)\nbdash{}point of \(M^D\) such that \(f^*x^i=u^i+\delta\zeta^i\), where \(u^i\in C(T)^\even,\zeta^i\in C(T)^\odd\). Let \(g\colon D\times T \to D\) be a \(T\)\nbdash{}point of \(D^D\) such that \(g^*\delta=\eta+\delta y\), where \(y\in C(T)^\even,\eta\in C(T)^\odd\). Then \(f\circ (g,\pr_T)\colon D\times T \to M\) is a \(T\)\nbdash{}point of \(M^D\) such that
\[
(f\circ(g,\pr_T))^*x^i=u^i+(g^*\delta)\eta^i=(u^i+\eta\zeta^i)+\delta(y\zeta^i),
\]
which describes the action \(\mu\) in \(T\)\nbdash{}points. Take local coordinates \((x,\xi)=(x^i,\xi^i)\) on~\(M^D\) and \((r,\delta)\) on \(D^D\cong \bR\ltimes D\) such that \(\hat{f}^*(x,\xi)=(u,\zeta)\) and \(\hat{g}^*(r,\delta)=(y,\eta)\). The action in local coordinates \((x,\xi)\) and \((r,\delta)\) is
\begin{equation}\label{chap7:eq:action}
\mu^* x=x+ \delta\xi, \quad \mu^* \xi=r \xi.
\end{equation}

By Definition~\ref{chap7:def:induced-vector}, the coordinate vectors \(\partial_r\big|_{r=1}\) and \(\partial_\delta\big|_{r=1}\) on \(D^D\) induce vector fields on \(M^D\) by the following formulas:
\begin{align*}
\partial_r\big|_{r=1} x   &= \partial_r\big|_{r=1}\mu^*x  =\partial_r\big|_{r=1}( x+ \delta\xi)=0,\\
\partial_r\big|_{r=1} \xi &= \partial_r\big|_{r=1}\mu^*\xi=\partial_r\big|_{r=1} (r \xi)=\xi,\\
\partial_\delta\big|_{r=1} x&=\partial_\delta\big|_{r=1}\mu^*x=\partial_\delta\big|_{r=1}( x+ \delta\xi)=\xi,\\
\partial_\delta\big|_{r=1} \xi&=\partial_\delta\big|_{r=1}\mu^*\xi=\partial_\delta\big|_{r=1} (r \xi)=0.
\end{align*}
Since the coordinates \(\xi\) can be identified with the coordinates \(dx\), we deduce that the induced infinitesimal action of \(\partial_r\big|_{r=1}\) is the degree and that of \(\partial_\delta\big|_{r=1}\) is the de~Rham differential.
\end{proof}

\subsection{NQ-manifolds}

The notion of Q\nbdash{}manifolds is introduced in~\cite{Vaintrob:normal,Vaintrob:lie,AKSZ:Qmfd}; the following definition of NQ\nbdash{}manifolds, due to \v{S}evera~\cite{Severa05}, generalises that of both \(L_\infty\)\nbdash{}algebras and Lie algebroids. Vaintrob~\cite{Vaintrob:lie} first observed that Lie algebroids are NQ-manifolds. NQ-manifolds can and should be viewed as higher Lie algebroids~\cite{Voronov,Sheng-Zhu,Bonavolonta-Poncin}; see also~\cite{Voronov,Sheng-Zhu,Bonavolonta-Poncin} for more examples of NQ-manifold.

\begin{definition}
An \emph{N\nbdash{}manifold} is a supermanifold equipped with an action of the monoid \((\bR, \times)\) such that \(-1\) acts as the parity involution (Definition~\ref{chap7:def:parity}). An \emph{NQ\nbdash{}manifold} is a supermanifold equipped with an action of \(D^D\) such that \(-1\) acts as the parity involution.
\end{definition}
\begin{remark}
In more elementary terms~\cite{Severa05}, N\nbdash{}manifolds are \(\bN\)-graded manifolds, and NQ\nbdash{}manifolds are \(\bN\)-graded manifolds equipped with a degree-1-vector field \(Q\) such that \([Q, Q]=0\). The induced infinitesimal action of \(\bR\subset D^D\) gives the degree, and the induced infinitesimal action of \(D\subset D^D\) gives the vector field \(Q\). The parity condition ensures that the induced degree is compatible with the supermanifold parity.
\end{remark}

\begin{example}
 First examples of NQ\nbdash{}manifolds include odd tangent bundles of manifolds and (shifted) Lie algebras; their NQ\nbdash{}manifold structures are given by the de~Rham cochain complex and the Chevalley--Eilenberg cochain complex, respectively.
\end{example}

\section{Differentiation of higher Lie groupoids}

In this section, we prove that \v{S}evera's 1-jet of a higher Lie groupoid is representable, hence we obtain an NQ\nbdash{}manifold (higher Lie algebroid) associated with a higher Lie groupoid.

\subsection{The 1-jets of higher Lie groupoids}

Let \(0\le n<\infty\). Recall that a Lie \(n\)\nbdash{}groupoid~\(X\) is a simplicial manifold such that the natural map
\[
\horn{m}{k}\colon X_m \to X_{m,k}\coloneqq \Hom(\Horn{m}{k}, X), \quad 0\le k \le m
\]
is a surjective submersion for \(m\ge 1\) and a diffeomorphism for \(m>n\).

The following definition is due to \v{S}evera~\cite{Severa06}. Let \(P\) denote the nerve of the pair groupoid (in the category \(\SMfd\)) of the odd line \(D\); hence \(P_k=D^{k+1}\).

\begin{definition}
Let \(X\) be a Lie \(n\)\nbdash{}groupoid. Its \emph{1\nbdash{}jet} \(\bhom(P,X)\) is the presheaf on \(\SMfd\) given by \( T\to \hom(P\times T, X)\) for \(T\in \SMfd\).
\end{definition}

This defines a functor from higher Lie groupoids to presheaves on \(\SMfd\). The 1\nbdash{}jet of a Lie \(n\)-groupoid is, \emph{a priori}, only a presheaf. It is clear that \(\bhom(P,X)\) carries a natural action of \(D^D\) and that \(-1\) acts as the parity involution. \v{S}evera claimed that

\begin{theorem}\label{chap7:thm:rep}
Let \(X\) be a Lie \(n\)\nbdash{}groupoid. The presheaf \(\bhom(P, X)\) is representable, hence an NQ-manifold.
\end{theorem}

The proof in~\cite[appendix]{Severa06} is incomplete and contains some mistakes. The rest of this section is devoted to a detailed proof of the above statement with gaps fixed. More generally, for any positive integer~\(k\), we can define a \(k\)\nbdash{}jet functor for super Lie \(n\)\nbdash{}groupoids. Our argument can be easily adapted to show that the \(k\)\nbdash{}jet of a super Lie \(n\)\nbdash{}groupoid is representable.

\subsection{A discrete version of the theorem}

As suggested by~\cite{Severa06}, we first prove a discrete version of the theorem. This will, in turn, guide us how to approach Theorem~\ref{chap7:thm:rep}. The argument presented in~\cite{Severa06} has mistakes; our proof is inspired by the idea therein.

\begin{theorem}
Let \(X\) be a Lie \(n\)\nbdash{}groupoid. Let \(S\) be a finite set, and denote by \(P(S)\) the nerve of the pair groupoid of \(S\). The presheaf\/ \(\bhom(P(S),X)=\Hom(P(S),X)\) is representable.
\end{theorem}

\begin{proof}
Fix an element \(*\) in \(S\). Let \(P^{(k)}(S)\) be the smallest simplicial subset of \(P(S)\) containing all \(k\)\nbdash{}simplices in \(P(S)\) with 0-th vertex \(*\). This yields a filtration
\[
P^{(0)}(S)\to P^{(1)}(S)\to \dots \to P^{(k)}(S)\to \dots \to P(S).
\]
We observe that the inclusion \(P^{(k-1)}(S)\to P^{(k)}(S)\) is a collapsible extension and that \(P^{(k)}(S)\) is obtained from \(P^{(k-1)}(S)\) by filling finite many \(k\)-dimensional horns.

Let \(G^{(k)}(S)=\Hom(P^{(k)}(S), X)\). Lemma~\ref{chap3:lem:covers_in_groupoid} shows that \(G^{(k)}(S)\) is representable for each \(k\) and that \(G^{(k)}(S)\to G^{(k-1)}(S)\) is a surjective submersion for each~\(k\) and an isomorphism for \(k>n\). Since \(X\) is a Lie \(n\)\nbdash{}groupoid, \(\Hom(P(S),X)\to G^{(k)}(S)\) is an isomorphism for \(k \ge n\). This proves that \(\Hom(P(S), X)\) is representable.
\end{proof}

We shall give a direct description of how to construct \(G^{(k)}(S)\) out of \(G^{(k-1)}(S)\). This will be useful for proving Theorem~\ref{chap7:thm:rep}. By the definition of \(P^{(k)}(S)\), a simplicial morphism \(g^{(k)}\colon P^{(k)}(S)\to X\) is determined by the map \(g_k\colon S^k\cong\{*\}\times S^k\subset P_k(S)\to X_k\). The space \(G^{(k)}(S)\) is thus a subspace of \(\hom(S^k, X_k)\).

First, \(P^{(0)}(S)=\Simp{0}\{*\}\) and \(G^{(0)}(S)=X_0\). Suppose that we know \(G^{(k-1)}(S)\) and let us describe \(G^{(k)}(S)\). For any \(g^{(k-1)}\in G^{(k-1)}(S)\) given by \(g_{k-1}\colon S^{k-1} \to X_{k-1}\), we shall describe all \(g^{(k)}\) sitting over \(g^{(k-1)}\). For any \((s_1,\dots,s_k) \in S^k\) the \(k\) elements \(g_{k-1}(s_1,\dots,\hat{s}_i,\dots,s_k)\in X_{k-1}\) for \(1 \le i\le k\) form an element of~\(X_{k,0}\). This gives a morphism
\[
\Gamma_{k-1} \colon G^{(k-1)}(S)\to \hom(S^{k+1}, X_{k,0}).
\]
On non-degenerate \(k\)-simplices of \(P^{(k)}(S)\), that is \(s_i\ne s_{i-1}\) for all \(1\le i\le k\) (suppose \(s_0=*\)), we choose an arbitrary element of \(X_k\) lying over \(X_{k,0}\); on degenerate \(k\)-simplices, \(g_k\) is determined by \(g_{k-1}\):
\begin{equation}\label{chap7:eq:gk-degenerate}
\begin{aligned}
  g_k(*,s_1,\dots,s_{k-1})&=\de^X_0 \circ g_{k-1}(s_1,\dots,s_{k-1}), \\
  g_k(s_1,\dots,s_i,s_i,\dots,s_{k-1})&=\de^X_i\circ g_{k-1}(s_1,\dots,s_i,\dots,s_{k-1}).
\end{aligned}
\end{equation}
This way we choose an element of \(X_k\) for each \(k\)-tuple \((s_1,\dots,s_k)\), and declare it to be \(g_k(s_1,\dots,s_k)\). When \(k> n\) the map \(g_k\) is uniquely determined by \(g_{k-1}\), since there is a diffeomorphism \(X_k\cong X_{k,0}\).

\begin{remark}
  It is worth noticing that \(g_k\) on degenerate \(k\)-simplices given by~\eqref{chap7:eq:gk-degenerate} is well-defined and  compatible with face maps and \(g_{k-1}\). We will see an analogy between these facts and Lemmas~\ref{chap7:lem:g(k)-g(k-1)-lower} and~\ref{chap7:lem:highest-term-constraint}.
\end{remark}

We summarise the discussion above by the following lemma.
\begin{lemma}\label{chap7:lem:Gk}
Let \(G^{(k)}(S)\) be defined as above. Denote by \(\de_i^{P(S)}\colon S^{k-1}\to S^k\) the restriction of the degeneracy map
\[
\{*\}\times S^{k-1}\subset P_{k-1}(S) \xrightarrow{\de_i^{P(S)}} P_k(S) \supset \{*\}\times S^k.
\]
Let \(G^{(k-1)}(S)\xrightarrow{\de_i^X} \hom(S^{k-1},X_k)\) be the composite
\[
G^{(k-1)}(S)\to \hom(S^{k-1},X_{k-1})\xrightarrow{\de_i^X} \hom(S^{k-1},X_k).
\]
Then \(G^{(k)}(S)\) is the limit of the following diagram
\begin{equation}\label{chap7:eq:diagram-Gk-S}
\begin{gathered}
\xymatrix{
  &\hom(S^k,X_k)\ar[ld]_{\horn{k}{0}}\ar[rd]^{(\de_i^{P(S)})^*}&\\
  \hom(S^k,X_{k,0})&&\hom(S^{k-1},X_k)\rlap{\, }\\
  &G^{(k-1)}(S)\ar[lu]^{\Gamma_{k-1}}\ar[ru]_{\de_i^X}\rlap{\ ,} &
}
\end{gathered}
\end{equation}
where \(i\) runs over\footnote{\v{S}evera considered only \(i=k-1\). The limit of that diagram is easy to compute.} \(0\le i\le k-1\) and for each \(i\) there is a copy of \(\hom(S^{k-1},X_k)\).
\end{lemma}

\subsection{Proof of the theorem}
We are now ready to prove the representability Theorem~\ref{chap7:thm:rep}.

Starting with \(P^{(0)}=\sk_0\{0\}\), we define \(P^{(k)}\) approximating \(P\) similar to \(P^{(k)}(S)\). Explicitly, we set
\[
P^{(k)}_i=P_i, \quad 0\le i\le k-1,\qquad P^{(k)}_k=\{0\}\times D^k\bigcup_{0\le j\le k-1} \de_j^P(P_{k-1}),
\]
and higher dimensions are degenerate \(P^{(k)}_{i+1}=\cup_{0\le j\le i} \de_j^P P^{(k)}_i\) for \(i\ge k\). Notice that \(P^{(k)}_i\) for \(i\ge k\) may not be a supermanifold but a union of supermanifolds. Thus \(G^{(k)}\coloneqq \bhom(P^{(k)},X)\) is a limit of sheaves on \(\SMfd\), hence also a sheaf. It is clear that Lemma~\ref{chap7:lem:Gk} carries over to this case; that is, \(G^{(k)}\) is the limit of the following diagram
\begin{equation}\label{chap7:eq:diagram-Gk}
\begin{gathered}
\xymatrix{
  &\hom(D^k,X_k)\ar[ld]_{\horn{k}{0}}\ar[rd]^{(\de_i^{P})^*}&\\
  \hom(D^k,X_{k,0})&&\hom(D^{k-1},X_k)\rlap{\, }\\
  &G^{(k-1)}\ar[lu]^{\Gamma_{k-1}}\ar[ru]_{\de_i^X}\rlap{\ ,} &
}
\end{gathered}
\end{equation}
where all maps are defined as above. We will prove inductively that \(G^{(k)}\) is representable for \(k\ge 0\). Corollary~\ref{chap7:cor:criterion-rep} for \(G^{(k)}\to G^{(0)}=X_0\) shows that it suffices to consider \(G^{(k)}\) over a local piece of~\(X_0\).

Let us first introduce some notation. Fix \(p_0\in X_0\). The unique map \([i]\to [0]\) induces a map \(X_0\to X_i\), which gives \(p_i\in X_i\) for \(i\ge 0\). Similarly, we obtain a point \(p_{i,0}\in X_{i,0}\). Let \(g\) be a \(T\)\nbdash{}point of \(\bhom(P,X)\) approximated by \(g^{(i)}\colon P^{(i)}\times T\to X\) which is given by \(g_i\colon D^i\times T\to X_i\) for \(i \ge 0\). Suppose that the image of the reduced morphism \(\reduce{g}_0\colon\reduce{T}\to X_0\) is contained in a neighbourhood \(U_0\) of \(p_0\). Let \((U_i,x_i)\) be local coordinates on \(X_i\) around~\(p_i\), where \(x_i\) is a \((\dim U_i)\)-tuple of functions on~\(U_i\). Suppose that \(\de_j^X(U_i)\subset U_{i+1}\) for all \(0\le j\le i\) and \(0\le i\le k-1\).

Let~\(\delta\) be a coordinate on \(D\) and let \((\delta_1, \dots, \delta_k)\) be the standard coordinates on~\(D^k\). Let us make a change of coordinates of \(D^k\)
\[
  \delta'_1 = \delta_1, \quad  \delta'_i = \delta_i-\delta_{i-1}, \quad 2\le i\le k.
\]
For an ordered set \(J=\{J_1,\dots,J_j\}\), denote \(\delta'_{J_1}\cdots\delta'_{J_j}\) by \(\delta'_{J}\); let \(\delta'_{J}=1\) if \(J=\emptyset\). Denote the ordered set \(\{1,\dots,k\}\) by \(\oset{k}\). Suppose that
\[
  g_k^*x_k=\sum_{J\subset \oset{k}} \delta'_J y_k^J,
\]
where \(J\) runs over all subset of \(\oset{k}\) and each term \(y_k^J\) is a \((\dim U_k)\)-tuple of functions on \(T\) of parity \( \card J \pmod{2}\).

The strategy is as follows: we compute the limit of the diagram~\eqref{chap7:eq:diagram-Gk} in two steps. Lemma~\ref{chap7:lem:g(k)-g(k-1)-lower} deals with the right part of the diagram. We show that \(y_k^{J}\) is determined by \(g_{k-1}\) if \(\card J<k\). Lemma~\ref{chap7:lem:highest-term-constraint} concerns the left part of the diagram. We deduce that \(y_k^{J}\) for \(\card J=k\) satisfies a constraint condition that produces some new degrees of freedom. To this end, we need to show two compatibility conditions for \(y_{k-1}^{J}\). The following two lemmas allow us to deduce the compatibility conditions for coordinates from the simplicial identities.

\begin{lemma}\label{chap7:lem:sub-surj}
There is a one-to-one correspondence between injections\footnote{Here and in what follows, injections and surjections between ordered sets are order-preserving maps.} \(\oset{l}\to \oset{k}\) and surjections \([k]\to [l]\) for \(0\le l\le k\). Moreover, this extends to a contravariant isomorphism of categories.
\end{lemma}
\begin{proof}
We construct the correspondence only, and the rest is routine. Given an injection \(A\colon \oset{l}\to \oset{k}\), we define a surjection \(\alpha\colon [k]\to [l]\) by \(j\mapsto i\) for \(A(i)\le j <A(i+1)\) and \(0\le i\le l\) (assume that \(A(0)=0\) and \(A(l+1)=k+1\)). Conversely, given a surjection \(\alpha\colon [k]\to [l]\), then \(A(i)=\min \alpha^{-1}\{i\}\) for \(1\le i\le l\) gives an injection \(A\colon \oset{l}\to \oset{k}\).
\end{proof}

Let the injection \(A\colon \oset{l}\to \oset{k}\) correspond to the surjection \(\alpha\colon [k]\to [l]\) as above. For a simplicial object \(Y\), denote \(Y(\alpha)\colon Y_l\to Y_k\) by~\(Y(A)\).

\begin{lemma}\label{chap7:lem:sA}
Let \(A\colon \oset{l}\to \oset{k}\) be an injection. Denote the restriction of \(P(A)\colon P_l \to P_k\) on \(D^l\cong\{0\}\times D^l\to \{0\}\times D^k\cong D^k\) and the induced map \(D^l\times T \to D^k\times T\) still by the same letter \(P(A)\). Let \(\sum_{J\subset \oset{k}}\delta'_J z^J\) be an element in  \(C(D^k\times T)\). Then we have
\begin{equation}\label{chap7:eq:sA}
  P(A)^* \big(\sum_{J\subset \oset{k}}\delta'_J z^J\big)=\sum_{J\subset \oset{l}} \delta'_J z^{A(J)}.
\end{equation}
\end{lemma}
\begin{proof}
Lemma~\ref{chap7:lem:sub-surj} implies that if the statement holds for injections \(A\colon \oset{l}\to \oset{k}\) and \(A'\colon \oset{m}\to \oset{l}\), then it holds for \(A\circ A'\). So it suffices to consider the case \(l=k-1\). Let \(A^i\) be the unique injection \(\oset{k-1}\to\oset{k}\) with \(i+1\not\in\im A^i\). Notice that \(\de_{i}^P=P(A^i)\colon P_{k-1}\to P_k\) for \(0\le i\le k-1\). In coordinates \((\delta'_1,\dots,\delta'_k)\), the restriction \(P(A^i)\colon D^{k-1}\to D^k\) is given by
\begin{equation*}
(\delta'_1,\dots,\delta'_{k-1})\mapsto
\begin{cases}
  (0,\delta'_1,\dots,\delta'_{k-1}),                              &\text{if \( i=0\),}\\
  (\delta'_1,\dots,\delta'_i,0,\delta'_{i+1},\dots,\delta'_{k-1}),&\text{otherwise.}
\end{cases}
\end{equation*}
It follows that
\[
  P(A^i)^*\big(\sum_{J\subset \oset{k}}\delta'_J z^J\big)
                           =\sum_{J\subset \oset{k-1}} \delta'_{J} z^{J^i},
\]
where \(J^i=\{J^i_1,\dots, J^i_j\}\) is given by
\[
J^i_u=
\begin{cases}
  J_u   &\text{if \(J_u<i+1\)},\\
  J_u+1 &\text{if \(J_u\ge i+1\)}.
\end{cases}
\]
It clear that \(J^i=A^i(J)\), and we are done.
\end{proof}

Let us back to the problem. We first consider the right part of diagram~\eqref{chap7:eq:diagram-Gk}.

\begin{lemma}\label{chap7:lem:g(k)-g(k-1)-lower}
Let \(A\colon \oset{l}\to \oset{k}\) be an injection. The following relation holds
\begin{equation}\label{chap7:eq:g(k)-g(k-1)-lower}
\sum_{J\subset \oset{l}} \delta'_J y_k^{A(J)} =X(A) \big(\sum_{J\subset \oset{l}} \delta'_J y_{l}^J\big),
\end{equation}
where the symbolic notation is used on the right hand side. It follows that if \(l<k\) and \(I\subset\im A\) then \(y_k^I\) is determined by \(g_l\) hence by \(g_{k-1}\). We may use different~\(A\) to determine~\(y_k^I\), but the results are the same.
\end{lemma}

\begin{proof}
Suppose that the lemma holds for \(0,\dots,k-1\). Let us consider the case \(k\). It suffices to consider \(l<k\). By the definition of \(g_k\), the diagram below commutes for \(0\le i\le k-1\):
\[\xymatrix{
  T\ar[d]_{\hat{g}_{k-1}}\ar[r]^-{\hat{g}_k} &\bhom(D^k, X_k)\ar[d]^{(\de_i^P){*}}\\
  G^{(k-1)}\ar[r]^-{\de_i^X}                & \bhom(D^{k-1}, X_k)\rlap{\ .}
}\]
Repeated application of the diagram above shows that
\[
X(A)\circ g_l=g_k\circ P(A)\colon D^l \times T \to X_k.
\]
Lemma~\ref{chap7:lem:sA} implies that
\[
 (g_k\circ P(A))^* x_k= P(A)^*\big(\sum_{J\subset \oset{k}} \delta'_J y_k^J\big)= \sum_{J\subset \oset{l}} \delta'_J y_k^{A(J)}.
\]
On the other hand, we have
\[
(X(A) \circ g_l)^* x_k=X(A) \big(\sum_{J\subset \oset{l}} \delta'_J y_l^J\big),
\]
establishing~\eqref{chap7:eq:g(k)-g(k-1)-lower}.

We now show that if \(\card I<k\) then \(y_k^I\) is uniquely determined by \(g_{k-1}\). Let \(C\colon \oset{m}\to\oset{k}\) be an injection such that \(\im C=I\). Then \(y_k^I\) is determined by~\(g_m\). For any injection \(A\colon \oset{l}\to \oset{k}\) with \(I\subset\im A\), we can also determine \(y_k^I\) by \(g_l\). Suppose that \(B\colon\oset{m}\to\oset{l}\) is an injection with \(C=A\circ B\). Comparing coefficients of \(\delta'_J\) for \(A(J)\subset I\) on both sides of~\eqref{chap7:eq:g(k)-g(k-1)-lower}, we need to show that
\[
\sum_{\mathclap{J\subset\oset{l},A(J)\subset I}} \delta'_J y_k^{A(J)} =X(A) \big(\sum_{\mathclap{J\subset\oset{l}, A(J)\subset{I}}}\delta'_J y_l^J\big),
\]
The condition \(A(J)\subset I\) implies that \(J\subset\im B\). Renaming the indices, we can rewrite the above identity as
\[
\sum_{J\subset \oset{m}} \delta'_J y_k^{A\circ B(J)}=X(A) \big(\sum_{J\subset \oset{m}} \delta'_J y_l^{B(J)}\big).
\]
Applying~\eqref{chap7:eq:g(k)-g(k-1)-lower} to both sides, we turn the equation above into
\[
X(A\circ B) \big(\sum_{J\subset \oset{m}} \delta'_J y_m^{J}\big)=X(A) \circ X(B) \big(\sum_{J \subset \oset{m}} \delta'_J y_m^J\big).
\]
This follows from Lemma~\ref{chap7:lem:sub-surj} and we are done.
\end{proof}

Hence only \(y_k^J\) for \(\card J=k\) is undetermined. We now consider the left part of diagram~\eqref{chap7:eq:diagram-Gk}.

\begin{lemma}\label{chap7:lem:highest-term-constraint}
Let \(g_{k,0}\colon D^k\times T\to X_{k,0}\) be given by
\[
\hat{g}_{k,0}\colon T\xrightarrow{\hat{g}_{k-1}} G^{(k-1)}\xrightarrow{\Gamma_{k-1}} \bhom(D^k, X_{k,0}).
\]
Let \((U_{k,0},x_{k,0})\) be local coordinates around \(p_{k,0}\in X_{k,0}\). The commutative diagram
\[
\xymatrix{
  T\ar[r]^{\hat{g}_{k-1}}\ar[d]_{\hat{g}_k}& G^{(k-1)}\ar[d]^{\Gamma_{k-1}}\\
  \bhom(D^k, X_k) \ar[r]^{\horn{k}{0}} & \bhom(D^{k}, X_{k,0})\\
}
\]
is expressed in coordinates as follows
\begin{equation}\label{chap7:eq:highest-term-constraint}
\horn{k}{0} (\sum_{J\subset \oset{k}}\delta'_J y_k^J)=g_{k,0}^* x_{k,0},
\end{equation}
where the symbolic notation is used on the left hand side. Compare coefficients of \(\delta'_J\) on two sides. We obtain a constraint condition for \(y_k^J\) with \(\card J=k\). We claim that if \(\card J<k\) then the coefficients of \(\delta'_J\) on two sides are automatically equal.
\end{lemma}
\begin{proof}
Suppose that the claim is true for \(0, \dots, k-1\), and let us consider the case~\(k\). By Lemma~\ref{chap7:lem:sA}, we need to show that
\[
P(A)^* \horn{k}{0} \big(\sum_{J\subset\oset{k}}\delta'_J y_k^J\big)= P(A)^* g_{k,0}^* x_{k,0}
\]
for every injection \(A\colon \oset{l}\to\oset{k}\) with \(l<k\), that is,
\[
\horn{k}{0} \circ g_k\circ P(A)=g_{k,0}\circ P(A).
\]
By the definition of \(g_{k,0}\), the above equation is equivalent to
\begin{equation}\label{chap7:eq:face-g=g-face}
\face_i^X \circ g_k\circ P(A)=g_{k-1}\circ \face_i^P\circ P(A), \quad \text{for \(1\le i\le k\)}.
\end{equation}
There are an injection \(A'\colon \oset{l-1}\to\oset{k-1}\) and some \(i'\) such that
\begin{equation}\label{chap7:eq:face-injection=injection-face}
\face_i^X \circ X(A)=X(A') \circ \face_{i'}^X,\quad \face_i^P \circ P(A)=P(A') \circ \face_{i'}^P.
\end{equation}
Lemma~\ref{chap7:lem:g(k)-g(k-1)-lower} implies that \(g_k\circ P(A)=X(A)\circ g_l\), hence
\[
  \face_i^X \circ g_k\circ P(A) =\face_i^X \circ X(A)\circ g_l =X(A') \circ \face_{i'}^X \circ g_l.
\]
The inductive hypothesis \(\face_{i'}^X \circ g_l=g_{l-1} \circ \face_{i'}^P\) implies that
\[
\face_i^X \circ g_k\circ P(A)=X(A') \circ g_{l-1} \circ \face_{i'}^P.
\]
By Lemma~\ref{chap7:lem:g(k)-g(k-1)-lower} and~\eqref{chap7:eq:face-injection=injection-face}, we deduce that
\[
\face_i^X \circ g_k\circ P(A)=g_{k-1}\circ P(A') \circ \face_{i'}^P=g_{k-1} \circ \face_i^P\circ P(A).
\]
This proves~\eqref{chap7:eq:face-g=g-face} and we are done.
\end{proof}

Combining the two lemmas above, we obtain the following result:
\begin{lemma}\label{chap7:lem:G-k}
The presheaf \(G^{(k)}\) is locally isomorphic to \(U_0\times \Pi_{1\le i\le k} \bR^{d_i [i]} \), where
\(d_i=\dim U_i-\dim U_{i,0}\), and \(\bR^{d_i[i]}=\bR^{d_i}\) if \(i\) is even or \(\bR^{0|d_i}\) if \(i\) is odd.
\end{lemma}
\begin{proof}
Suppose that the statement holds for \(0,\dots,k-1\). Let us consider the case~\(k\). Choose local coordinates \(x_{k+1}\) and \(x_{k,0}\) that give a local canonical form of the submersion \(\horn{k}{0}\colon U_k\to U_{k,0}\). Then the constraint condition~\eqref{chap7:eq:highest-term-constraint} for \(y_k^J\) with \(\card J=k\) produces \(d_k\) new free functions on \(T\) of parity \(k\pmod{2}\). Applying Theorem~\ref{chap7:thm:morphism}, we prove the lemma.
\end{proof}

\begin{proof}[Final step of the proof]
It is clear from the proof of Lemma~\ref{chap7:lem:G-k} that \(G^{(k)}\) is representable, and that \(G^{(k+1)}\to G^{(k)}\) is a surjective submersion (Definition~\ref{chap7:def:surj-subm}) for \(k \ge 0\) and a diffeomorphism for \(k\ge n\). Thus \(\bhom(P, X)\cong G^{(k)}\) for \(k\ge n\) is representable, and this completes the proof of Theorem~~\ref{chap7:thm:rep}.
\end{proof}

\begin{remark}\label{chap7:rem:degree}
The degrees of \(G^{(k)}\) can be lifted to \(\mathbb{N}\), giving the N-manifold structure on \(G^{(k)}\) (hence on \(\bhom(P,X)\)). In fact, since the \(\bR\)-actions are respected in the process from \(g^{(k)}\colon P^{(k)}\times T\to X\) to \(g_k\colon  D^k\times T\to X_k\), and the changing of coordinates from \((\delta_1,\dots,\delta_k)\) to \((\delta'_1,\dots,\delta'_k)\), the degree of \(\bR^{d_k[k]}\) can be lifted to \(k\).
\end{remark}

\section{Three special cases}

We work out the 1-jet in detail for three special cases, namely, Lie groups, crossed modules, and Lie groupoids. Since the differentiation for these three cases is well-known, we shall verify that the 1-jet functor yields the same results. The results are stated very concisely in~\cite{Severa06}. Moreover, Jur\v{c}o~\cite{Jurco} showed that the 1-jet for simplicial Lie groups produces the expected results. This justifies viewing the 1-jet as differentiation for higher Lie groupoids.

\subsection{Lie groups}

Let \(G\) be a Lie group. Let \(\inv\) be the inverse and \(\mult\) the multiplication. Denote by \(X\) be the nerve of the associated one-object groupoid; hence \(X_k=G^k\).

\subsubsection{Functor of points}

Suppose temporarily that \(D\) is a set with a distinguished point \(0\in D\). We can describe \(\hom(P, X)\) as follows. Since \(P\) and \(G\) are groupoids, an element of \(\hom(P, X)\) is given by a map \(\Psi\colon D\times D \to G \) such that \(\Psi(d,d)=e\) and \(\Psi(d_1,d_3)=\Psi(d_1,d_2)\Psi(d_2,d_3)\) for \(d, d_1,d_2,d_3 \in D\). This is equivalent to a map \(\psi\colon D \to G\) with \(\psi(0)=e\). The relation between \(\Psi\) and \(\psi\) is
\begin{align*}
  \Psi(d_1,d_2) &= \psi(d_1)^{-1}\psi(d_2), \\
  \psi(d)   &= \Psi(0,d).
\end{align*}
Let us say that \(\psi\) is the \emph{normalised form} of \(\Psi\).

Now assume that \(D\) is the odd line. The argument above carries over to the setting of supermanifolds if we use the language of \(T\)-points. First, the odd line has a distinguished \(T\)-point \(0\in_T D\). A \(T\)-point of \(\bhom(P, X)\) is given by \(\Psi\in\hom(D \times D\times T, G)\) such that \(\Psi(d,d, t)=e\) and \(\Psi(d_1,d_3, t)=\Psi(d_1,d_2, t)\Psi(d_2,d_3, t)\) for \(d, d_1,d_2, d_3 \in_S D\) and \(t\in_S T\), where \(S\) is a supermanifold. Such a \(T\)\nbdash{}point has a normalised form \(\psi\in\hom(D\times T, G)\) with \(\psi(0, t)=e\). The relation between \(\Psi\) and \(\psi\) is
\begin{equation}\label{chap7:eq:lie-gpd-Psi-and-psi}
\begin{aligned}
  \Psi(d_1,d_2, t) &= \psi^{-1}(d_1, t)\psi(d_2, t), \\
  \psi(d, t)   &= \Psi(0, d, t).
\end{aligned}
\end{equation}
Let \(s\) and \(d\) be \(T\)-points of \(\bR\) and \(D\). Then
\begin{equation}\label{chap7:eq:lie-group-action}
\Psi_s(d_1,d_2,t)=\Psi(sd_1,sd_2,t),\quad \Psi_d(d_1,d_2,t)=\Psi(d+d_1,d+d_2,t),
\end{equation}
describes the action of \(D^D=\bR\ltimes D\) on \(\bhom(P, X)\) in \(T\)-points. In diagrams
\begin{equation*}
\begin{gathered}
\Psi_s\colon D\times D\times T\to D\times D\times \bR\times T\to D\times D\times T\to G, \\
(d_1,d_2,t)\mapsto (d_1,d_2,s,t)\mapsto (sd_1,sd_2,t)\mapsto \Psi(sd_1,sd_2,t);\\
\Psi_d\colon D\times D\times T\to D\times D\times D\times T\to  D\times D\times T\to G,\\
(d_1,d_2,t)\mapsto (d_1,d_2,d,t)\mapsto (d+d_1,d+d_2,t)\mapsto \Psi(d+d_1,d+d_2,t).\\
\end{gathered}
\end{equation*}

\subsubsection{Local structure of Lie groups}

Let \(x=(x^i)\) be local coordinates around \(e\in G\) such that \(x(e)=0\). Applying Lemma~\ref{chap7:lem:taylor}, we suppose that \(\inv\) and \(\mult\) have the following local form near~\(e\):
\begin{equation}\label{chap7:eq:lie-grp-inv-mult}
\begin{aligned}
  \inv^*(x^i) &= -x^i +O_2(x),\\
  \mult^*(x^i) &= x_1^i+x_2^i+\frac{1}{2} b_{jk}^i x_1^j x_2^k +O_3(x_1, x_2),
\end{aligned}
\end{equation}
where \(x_1^i\) and \(x_2^i\) stand for the same functions \(x^i\) on different copies, and the remainders \(O_2(x)\), \(O_3(x_1,x_2)\) lie in appropriate ideals.

\begin{lemma}\label{chap7:lem:group:mult-formula}
The coordinate vectors \(\{\partial_i=\partial_{x^i}\big|_{x=0}\}\) form a basis of the Lie algebra \(\mathfrak{g}\).
Let~\(c_{jk}^i\) be the structure constants defined by \([\partial_j, \partial_k]=c_{jk}^i\partial_i\). Then we have
\[
c_{jk}^i=\frac{1}{2}(b_{jk}^i-b_{kj}^i).
\]
\end{lemma}
\begin{proof}
Let \(v_i\) be the left invariant vector field on \(G\) generated by \(\partial_i\). By definition we have
\[
v_j x^l=(\id\otimes \partial_j)\mult^* (x^l)=\delta^l_j+\frac{1}{2}b^l_{mj}x^m+O_2(x).
\]
The claim follows by comparing two sides of the following equation at \(x=0\)
\[
[v_j,v_k]x^l=c_{jk}^i v_i x^l. \qedhere
\]
\end{proof}

\subsubsection{The main calculation}
Let \(x=(x^i)\) be local coordinates on \(G\) as above. Let \(\delta\) be a coordinate of \(D\).

Recall that a \(T\)\nbdash{}point of \(\bhom(P,X)\) is given by \(\psi\colon D\times T\to G\) with \(\psi(0,t)=e\) for \(t\in_S T\). Suppose that
\[
\psi^*x^i=y^i+\delta\eta^i,
\]
where \(y^i\in C(T)^\even\) and \(\eta^i\in C(T)^\odd\). The condition \(\psi(0, t)=e\) implies \(y^i=0\), hence
\begin{equation}\label{chap7:eq:grp:psi}
\psi^*x^i=\delta\eta^i.
\end{equation}
This proves that the presheaf \(\bhom(P, X)\) is representable and isomorphic to \(\Pi\mathfrak{g}\) as a supermanifold. We equip \(\bhom(P, X)\) with odd coordinates \((\xi^i)\) such that their pullbacks to~\(T\) are \((\eta^i)\).

We now compute the action of \(D^D=\bR\times D\) on \(\bhom(P, X)\), that is, the NQ\nbdash{}manifold structure on \(\bhom(P, X)\). First, by Remark~\ref{chap7:rem:degree} we can write down the infinitesimal action of \(\bR\) directly:
\[
\partial_r\big|_{r=1} \xi^i = \xi^i,
\]
which gives the N\nbdash{}manifold structure. To compute the action of \(D\) we need to work with~\(\Psi\). By~\eqref{chap7:eq:lie-gpd-Psi-and-psi} and~\eqref{chap7:eq:lie-grp-inv-mult}, \eqref{chap7:eq:grp:psi}, we can write \(\Psi \in \hom( D\times D\times T, G)\) in local coordinates\footnote{Notice that the remainders in~\eqref{chap7:eq:lie-grp-inv-mult} is killed by the nilpotency of odd variables.}:
\begin{equation*}
\begin{aligned}
\Psi^*x^i&=-\delta_1\eta^i+\delta_2\eta^i-\frac{1}{2}b_{jk}^i (\delta_1\eta^j)(\delta_2\eta^k),\\
        &=-\delta_1\eta^i+\delta_2\eta^i+\frac{1}{2}c_{jk}^i \delta_1 \delta_2\eta^j\eta^k.
\end{aligned}
\end{equation*}
Let \(\varepsilon\) be a coordinate on the copy \(D\) which acts. For a \(T\)\nbdash{}point of~\(D\) given by \(\ul{\varepsilon}\in C(T)^\odd\), the action~\eqref{chap7:eq:lie-group-action} is expressed in coordinates by
\[
(\Psi_{\ul{\varepsilon}})^*x^i=- (\delta_1+\ul{\varepsilon})\eta^i+(\delta_2+\ul{\varepsilon})\eta^i+
\frac{1}{2}c_{jk}^i (\delta_1+\ul{\varepsilon})(\delta_2+\ul{\varepsilon})\eta^j\eta^k.
\]
In view of~\eqref{chap7:eq:lie-gpd-Psi-and-psi}, the action of \(\ul{\varepsilon}\) on \(\psi\) is
\[
(\psi_{\ul{\varepsilon}})^*x^i=\delta\eta^i-\frac{1}{2}c_{jk}^i\delta \ul{\varepsilon}\eta^j\eta^k,
\]
which gives the action of \(D\) on \(\bhom(P, X)\) in \(T\)\nbdash{}points. It follows that the action, denoted by \(\mu\), in coordinates is
\[
\mu^* \xi^i =\xi^i-\frac{1}{2}c_{jk}^i\varepsilon\xi^j\xi^k.
\]
The infinitesimal action is thus given by
\[
\partial_\varepsilon (\xi^i) =-\frac{1}{2}c_{jk}^i \xi^j\xi^k.
\]
This is the Chevalley--Eilenberg differential which is dual to the Lie bracket. We see that the 1-jet of a Lie group is an NQ\nbdash{}manifold that is equivalent to the associated Lie algebra.

\subsection{Crossed modules}

We calculate the 1-jet of crossed modules in this subsection. This is an analogue of the previous computation.

\subsubsection{Basic facts about crossed modules}

Let us recall some facts about crossed modules from~\cite[Section XII.8]{MacLane}. A crossed module of Lie groups \(X=(H, G)\) consists of the following data: a Lie group homomorphism \(\tau\colon H \to G\) and a left \(G\)-action on \(H\) as automorphisms, \(\alpha\colon G \times H\to H\), or denoted by \(\cdot\) for simplicity, such that
\[
  \tau(g \cdot h) = g\tau(h)g^{-1},\quad \tau (h_{1}) \cdot h_{2} = h_{1}h_{2}h_{1}^{-1}.
\]

Given a crossed module \(X=(H, G)\), there is an associated simplicial manifold \(X\) with
\[
  X_0=\pt,\quad  X_1=G,\quad  X_2=G^2\times H,\quad X_3=G^3\times H^3,\quad \dots.
\]
This simplicial manifold \(X\) is a Lie 2\nbdash{}groupoid with one object. Its 1\nbdash{}simplices are elements of \(G\), the composition of 1-simplices is the multiplication of \(G\); 2\nbdash{}simplices are triples \( (g_1, g_2, h) \), where \(g_1, g_2\in G\) and \(h\in H\), which can be regarded as triangles with boundary \((g_1,g_2,\tau(h)g_1g_2)\) and interior \(h\); and 3\nbdash{}simplices are three triangles fitting together (recall that the unique Kan conditions hold).

\subsubsection{Functor of points}

A \(T\)\nbdash{}point of \(\bhom(P,X)\) is determined by a pair of \(T\)-parameterized morphisms
\[
  \Psi\colon D \times D\times T\to G,\quad \Phi\colon D \times D \times D\times T\to H,
\]
such that~\cite[Example 13.5]{Severa06}
\begin{equation*}
\begin{split}
  \Psi(d, d, t)&=e_G,\\
  \Phi(d_1,d_1,d_2, t)&=\Phi(d_1,d_2,d_2, t)=e_H,\\
  \Psi(d_1,d_2, t)\Psi(d_1, d_2, t)&=\tau(\Phi(d_1,d_2,d_3, t))\Psi(d_1,d_3, t),\\
  \Phi(d_1,d_2,d_3, t)\Phi(d_1, d_3, d_4, t)&=(\Psi(d_1, d_2, t) \cdot \Phi(d_1, d_2, d_3, t)) \Phi(d_1, d_2, d_4, t),
\end{split}
\end{equation*}
where \(t\in_S T, d, d_1, d_2, d_3, d_4 \in_S D\) for every supermanifold \(S\).

Alternatively, a \(T\)-point of \(\bhom(P,X)\) can be given by a normalised form
\[
  \psi\colon D \times T \to G, \quad \varphi\colon D\times D \times T\to H
\]
such that
\begin{equation}\label{chap7:eq:psiphi}
  \psi(0, t)=e_G, \quad \varphi(d, d, t)=\varphi(0,d, t)=e_H.
\end{equation}
The relation between \((\Psi,\Phi)\) and \((\psi,\varphi)\) is
\begin{equation}\label{chap7:eq:xmod-relation}
\begin{aligned}
  \psi(d, t)&=\Psi(0,d, t),\\
  \varphi(d_1, d_2, t)&=\Phi(0, d_1, d_2, t),\\
  \Psi(d_1,d_2, t)&=\psi(d_1, t)^{-1}\tau(\varphi(d_1,d_2, t))\psi(d_2, t),\\
  \Phi(d_1,d_2,d_3, t)&=\psi(d_1, t)^{-1}\cdot(\varphi(d_1,d_2, t)\varphi(d_2,d_3, t)\varphi(d_1,d_3, t)^{-1}).
\end{aligned}
\end{equation}

\subsubsection{Local structure of crossed modules}
Let \(x=(x^i)\) be local coordinates around \(e_G \in G\) with \(x(e_G)=0\), and let \(y=(y^l)\) be local coordinates around \(e_H\in H\) with \(y(e_H)=0\). We reserve indices \(i,j,k\) for \(G\) and \(l,m\) for \(H\). Write \(\partial_i=\partial_{x^i}|_{x=0}\) and \(\partial_l=\partial_{y^l}|_{y=0}\).

Let \(b_{jk}^i\) be as in~\eqref{chap7:eq:lie-grp-inv-mult} and let \(c_{jk}^i\) be the structure constants of \(\mathfrak{g}\) as in Lemma~\ref{chap7:lem:group:mult-formula}. Let~\(\tau_l^i\) be the components of the induced homomorphism \(\bar{\tau}\colon \mathfrak{h}\to \mathfrak{g}\) defined by \(\bar{\tau} \partial_l=\tau_l^i\partial_i\). We have
\begin{equation}\label{chap7:eq:tau}
\tau^* x^i=\tau_l^i y^l + O_2(y).
\end{equation}
Let \(\alpha_{im}^l \) be the structure constants of the induced action of Lie algebras \(\bar{\alpha}\colon \mathfrak{g}\times \mathfrak{h}\to \mathfrak{h}\) defined by \(\bar{\alpha}(\partial_i)\partial_m=\alpha_{im}^l\partial_l\). We have
\begin{equation}\label{chap7:eq:alpha}
\alpha^* y^l=y^l+\alpha_{im}^l x^i y^m+ O_3(x, y).
\end{equation}

\subsubsection{The main calculation}
Let \((x^i)\) and \((y^l)\) be local coordinates on \(G\) and \(H\) as above. Let \(\delta\) be a coordinate of \(D\).

Let \(\psi\colon D\times T\to G\) and \(\varphi\colon D\times D\times T\to H\) give a \(T\)\nbdash{}point of \(\bhom(P, X)\). Suppose that
\[
\psi^*x^i=u^i+\delta\upsilon^i, \quad \varphi^* y^l=v^l+\delta_1\eta^l+\delta_2\gamma^l+ \delta_1 \delta_2 w^l,
\]
where \(u^i, v^l, w^l \in C(T)^\even\) and \(\upsilon^i, \eta^l, \gamma^l,\in C(T)^\odd\). The condition~\eqref{chap7:eq:psiphi} implies \(u^i=v^l=0\), and \(\eta^l=\gamma^l=0\), hence
\begin{equation}\label{chap7:eq:pullback-psiphi}
\psi^*x^i=\delta\upsilon^i, \quad \varphi^* y^l=\delta_1 \delta_2 w^l .
\end{equation}
This proves that \(\bhom(P, X)\) is representable and isomorphic to \(\Pi\mathfrak{g}\times \mathfrak{h}\) as a supermanifold. We equip \(\bhom(P, X)\) with odd coordinates \((\xi^i)\) and even coordinates \((z^l)\) such that their pullbacks to~\(T\) are \((\upsilon^i,w^l)\).

We now compute the NQ\nbdash{}manifold structure. By Remark~\ref{chap7:rem:degree}, we can write down the infinitesimal action of \(\bR\) or the N\nbdash{}manifold structure directly:
\[
  \partial_r\big|_{r=1} \xi^i =\xi^i,\quad \partial_r\big|_{r=1} z^l = 2z^l.
\]
To calculate the action of \(D\) we consider \((\Psi,\Phi)\). By~\eqref{chap7:eq:xmod-relation}, and~\eqref{chap7:eq:lie-grp-inv-mult}, \eqref{chap7:eq:tau}, \eqref{chap7:eq:pullback-psiphi}, we deduce that
\[
  \Psi^*x^i =-\delta_1\upsilon^i+\tau_l^i\delta_1\delta_2 w^l+\delta_2 \upsilon^i- \frac{1}{2}c_{jk}^i \delta_1\upsilon^j\delta_2\upsilon^k.
\]
Let \(\varepsilon\) be a coordinate on the copy \(D\) which acts. A \(T\)\nbdash{}point of \(D\) given by \(\ul{\varepsilon}\in C(T)^\odd\) acts on \(\Psi\) by
\[
  (\Psi_{\ul{\varepsilon}})^*x^i=-\delta_1\upsilon^i+\delta_2\upsilon^i+
  \tau_l^i(\delta_1+\ul{\varepsilon})(\delta_2+\ul{\varepsilon}) w^l+\frac{1}{2}c_{jk}^i (\delta_1+\ul{\varepsilon})(\delta_2+\ul{\varepsilon})\upsilon^j \upsilon^k.
\]
Applying~\eqref{chap7:eq:xmod-relation} gives the action of this \(T\)\nbdash{}point on \(\psi\):
\[
  (\psi_{\ul{\varepsilon}})^*x^i=\delta\upsilon^i-\tau_l^i\delta\ul{\varepsilon} w^l -\frac{1}{2}c_{jk}^i \delta\ul{\varepsilon}\upsilon^j \upsilon^k.
\]
Considering the corresponding expression in \((\xi, z)\), we get the infinitesimal action
\begin{equation}\label{chap7:eq:xi-action}
  \partial_\varepsilon \xi^i=-\tau_l^i z^l -\frac{1}{2}c_{jk}^i \xi^j \xi^k.
\end{equation}
By~\eqref{chap7:eq:xmod-relation} and~\eqref{chap7:eq:lie-grp-inv-mult}, \eqref{chap7:eq:alpha}, \eqref{chap7:eq:pullback-psiphi}, we deduce that
\[
  \Phi^* y^l=(\delta_1\delta_2+\delta_2\delta_3-\delta_1\delta_3)w^l-\alpha_{im}^l\delta_1 \upsilon^i \delta_2\delta_3 w^m.
\]
A \(T\)\nbdash{}point of \(D\) given by \(\ul{\varepsilon}\) acts on \(\Phi\) by
\begin{multline*}
  (\Phi_{\ul{\varepsilon}})^*y^l=
((\delta_1+\ul{\varepsilon})(\delta_2+\ul{\varepsilon})+(\delta_2+\ul{\varepsilon})(\delta_3+\ul{\varepsilon})
-(\delta_1+\ul{\varepsilon})(\delta_3+\ul{\varepsilon}))w^l \\
  -\alpha_{im}^l(\delta_1+\ul{\varepsilon})(\delta_2+\ul{\varepsilon})(\delta_3+\ul{\varepsilon})\upsilon^i w^m.
\end{multline*}
Relation~\eqref{chap7:eq:xmod-relation} implies that this \(T\)\nbdash{}point acts on \(\varphi\) by
\[
  (\varphi_{\ul{\varepsilon}})^*y^l=\delta_1\delta_2 w^l -\alpha_{im}^l \delta_1\delta_2\ul{\varepsilon} \upsilon^i w^m.
\]
Taking the infinitesimal action, we get
\begin{equation}\label{chap7:eq:w-action}
  \partial_\varepsilon z^l=-\alpha_{im}^l \xi^i z^m.
\end{equation}
Identities~\eqref{chap7:eq:xi-action} and~\eqref{chap7:eq:w-action} combined give the dual form of the associated crossed module of Lie algebras (see~\cite[p. 16]{Schreiber-Stasheff}). This proves that the 1\nbdash{}jet of a crossed module is an NQ\nbdash{}manifold which is equivalent to the usual differentiation.

\subsection{Lie groupoids}

First, we briefly recall below the Lie algebroid associated with a Lie groupoid; for more details see~\cite{Moerdijk-Mrcun2003}.

Let \(G=(G_0,G_1,\unit,\inv,\source,\target,\circ)\) be a Lie groupoid. The left action of \(G\) on \(G_1\) lifts to \(T^\target G_1\coloneqq\ker d\, \target\subset TG_1\). Denote the set of invariant sections by \(\mathfrak{A}\), which are uniquely determined by their restriction on \(\unit(G_0)\). Then \(\mathfrak{A}\) is a subalgebra of \(\mathfrak{X}(G_1)\), and the source map induces a Lie algebra homomorphism \(d\,\source\colon \mathfrak{A}\to \mathfrak{X}(G_0)\).

\subsubsection{Local structure of Lie groupoids}
Fix \(g_0\in G_0\). Let \((U, (x^i))\) be local coordinates on \(G_0\) around \(g_0\) with \(x(g_0)=0\). Assume that \(U\) is small enough. We can choose local coordinates \((U\times V, (x^i, y^l) )\) on \(G_1\) around \(\unit(g_0)\) such that
\[
y(\unit(g_0))=0, \quad\target(a, b)=a, \quad \unit(a)=(a,0), \quad a\in U, b\in V.
\]

The following two lemmas due to~\cite{Ramazan} are analogous to the Baker--Campbell--Hausdorff formula (or Lemma \ref{chap7:lem:group:mult-formula}). They determine the local structure of the associated Lie algebroid.

\begin{lemma}\label{chap7:lem:groupoid-local}
Let \(p\colon U\times V\times V\to V\) be given by \((a,b_1)\circ(\source(a,b_1),b_2)=(a, p(a,b_1,b_2))\), where \(a\in U, b_1,b_2\in V\). We have
\[
p^*y=y_1+y_2+B(y_1,y_2)+O_3(y_1, y_2),
\]
where \(B(y_1, y_2)\) depends on \(x\) and is bilinear in \(y_1, y_2\). Let \(\tau\colon U\times V\to V\) be given by \(\inv(a, b)=(\source(a,b), \tau(a,b))\). Then we have
\begin{gather*}
\tau^*y=-y+B(y,y)+O_3(y),\\
\source^*y=x+A(y)+O_2(y),
\end{gather*}
where \(A(y)\) depends on \(x\) and is linear in \(y\).
\end{lemma}

\begin{corollary}\label{chap7:cor:gpd-q}
Suppose that \( \inv(a,b_1)\circ (a, b_2)=(\source(a,b_1), q(a,b_1,b_2))\), where \(a\in U\), \(b_1, b_2\in V\) and \(q\colon U\times V\times V\to V \). Then we have
\[
q^* y=-y_1+y_2-B(y_1, y_2)+O_2(y_1)+O_3(y_1, y_2).
\]
\end{corollary}

\begin{lemma}\label{chap7:lem:lie-algbroid-local}
Let \(e_l=\partial_{y^l}\big|_{y=0}\). Then \(\{e_l\}\) forms a local frame of the Lie algebroid associated with~\(G\). Suppose that the Lie algebroid has local structure
\begin{equation}\label{chap7:eq:lie-algbroid}
a(e_l)=a_l^i \partial_{x^i}, \quad [e_m, e_n]=c_{mn}^l e_l.
\end{equation}
Then the structure constants are given by \(A\) and \(B\)
\[
a_l^i=A_l^i, \quad c_{mn}^l= B_{mn}^l-B_{nm}^l.
\]
\end{lemma}

\subsubsection{The main calculation}
Denote the nerve of \(G\) by the same letter. A \(T\)\nbdash{}point of \(\bhom(P, G)\) is given by a pair
\[
\psi\colon T\to G_0, \quad \varphi\colon D\times T \to G_1
\]
such that
\begin{equation}\label{chap7:eq:gpd:g0g1}
\varphi(0, t)=\unit\circ \psi (t), \quad \target\circ \varphi (d, t)= \psi (t),
\end{equation}
or, equivalently, by a pair (satisfying some relation that we do not need)
\[
\Psi \colon D\times T\to G_0,\quad \Phi \colon D\times D\times T \to G_1.
\]
The relation between \((\psi, \varphi)\) and \((\Psi, \Phi)\) is
\begin{equation}\label{chap7:eq:gpd:relation}
\begin{aligned}
\Psi(d, t)&=\source\circ \varphi(d, t),& \Phi(d_1, d_2, t) &=\varphi(d_1, t)^{-1} \varphi (d_2, t),\\
\psi(t)&=\Psi(0, t),&       \varphi(d, t)&=\Phi(0, d, t),
\end{aligned}
\end{equation}
where \(d, d_1, d_2\in_S D\), and \(t\in_S T\).

Let \((U, (x^i))\) be local coordinates on \(G_0\), and let \( (U\times V, (x^i, y^l)) \) be local coordinates on~\(G_1\) as above. Let \(\delta\) be a coordinate of~\(D\). We consider \(\bhom(P, G)\) over \(U\subset G_0\). Suppose that
\[
\psi^* x^i= u^i, \quad \varphi^*x^i=v^i+\delta \gamma^i, \quad \varphi^* y^l =w^l +\delta \zeta^l,
\]
where \(u^i, v^i, w^l\in C(T)^\even\), and \(\gamma^i, \zeta^l \in C(T)^\odd\). Equation~\eqref{chap7:eq:gpd:g0g1} implies   \(u^i=v^i\), \(\gamma^i=0\), and \(w^l=0\). Thus a \(T\)\nbdash{}point of \(\bhom(P, G)\) is given by functions \(\{u^i\}\) and \(\{\zeta^l\}\). Therefore, the presheaf \(\bhom(P, G)\) is representable. We can take \((x^i, \eta^l)\) to be local coordinates on \(\bhom(P, G)\) such that their pullbacks to \(T\) are \((u^i, \zeta^l)\).

Locally, \(\bhom(P, G)\) consists of all odd tangent vectors of \(G_1\) on \(G_0\) that are tangent to \(\target\)\nbdash{}fibres. Since this description is independent of local coordinates, the supermanifold \(\bhom(P, G)\) globally consists of all odd tangent vectors of \(G_1\) on \(G_0\) that are tangent to \(\target\)\nbdash{}fibres. As a supermanifold \(\bhom(P, G)\) is isomorphic to the shifted vector bundle associated with the Lie algebroid.

We now compute the NQ\nbdash{}manifold structure. As above the infinitesimal action of \(\bR\) can be written down directly:
\[
\partial_r\big|_{r=1} x^i= 0, \quad \partial_r\big|_{r=1} \eta^l =\eta^l.
\]
To see the action of \(D\) we consider \((\Psi,\Phi)\). Equation~\eqref{chap7:eq:gpd:relation} together with Lemma~\ref{chap7:lem:groupoid-local} and Corollary~\ref{chap7:cor:gpd-q} imply that
\[
\Psi^*x^i= \Phi^* x^i= u^i+\delta A^i_l \zeta^l, \quad
\Phi^* y^l= (-\delta_1+\delta_2) \zeta^l+\delta_1\delta_2 B_{mn}^l \zeta^m \zeta^n.
\]
Let \(\varepsilon\) be a coordinate on the copy \(D\) which acts. A \(T\)-point of \(D\) given by \(\ul{\varepsilon} \in C(T)\) acts on \((\Psi, \Phi)\) by
\begin{gather*}
  \Psi_{\ul{\varepsilon}}^* x^i= \Phi_{\ul{\varepsilon}}^* x^i=u^i+(\ul{\varepsilon}+\delta) A^i_l \zeta^l, \\
  \Phi_{\ul{\varepsilon}}^* y^l= (-\delta_1+\delta_2) \zeta^l+(\ul{\varepsilon}+\delta_1)(\ul{\varepsilon}+\delta_2) B_{mn}^l \zeta^m \zeta^n.
\end{gather*}
In view of~\eqref{chap7:eq:gpd:relation}, the induced action on \((\psi, \varphi)\) is
\begin{gather*}
  \psi_{\ul{\varepsilon}}^* x^i= \varphi_{\ul{\varepsilon}}^* x^i=u^i+\ul{\varepsilon} A^i_l \zeta^l, \\
  \varphi_{\ul{\varepsilon}}^* y^l= \delta \zeta^l-\delta\ul{\varepsilon} B_{mn}^l \zeta^m \zeta^n.
\end{gather*}
Considering the expressions in coordinates \((x,\eta)\) and \(\varepsilon\), we obtain the infinitesimal action
\[
\partial_\varepsilon x^i=A^i_l \eta^l, \quad \partial_\varepsilon \eta^l =-B_{mn}^l \eta^m \eta^n=-\frac{1}{2}c_{mn}^l \eta^m \eta^n,
\]
which is the dual form of identities~\eqref{chap7:eq:lie-algbroid}. Hence the 1\nbdash{}jet of a Lie groupoid is an NQ\nbdash{}manifold that is equivalent to the associated Lie algebroid.

\appendix

\chapter{Appendix}\label{app}
\thispagestyle{empty}
\section{Colored outer Kan conditions}\label{app:sec:colored-out-kan}

This section is dedicated to colored outer Kan conditions.

\subsection{Statement of the result}
Let \((\Cat, \covers)\) be an extensive category with a pretopology satisfying Assumption~\ref{chap4:asmp:extensive-pretopology}. Our aim is to prove the following result.

\begin{theorem}\label{chap4:thm:enriched}
Let \(Y\) be a simplicial object in \(\Cat\) with a morphism \(\pi\colon Y\to j(\Simp{1})\) such that the two ends are \(n\)\nbdash{}groupoids in \((\Cat, \covers)\). Assume that \(\pi\) satisfies \(\Kan(m, k)\) for \(m\ge 2, 0<k<m \) and \(\Kan!(m, k)\) for \(m>n, 0<k<m\). Suppose that Assumption~\ref{chap4:asmp:ct} is fulfilled.
\begin{enumerate}
  \item\label{chap4:enm:it:i}
  The colored outer Kan conditions \(\Kan(m,0)[i,j]\) for \(i \ge 2\) and \(\Kan(m,m)[i,j]\) for \(j \ge 2\) hold.
  \item\label{chap4:enm:it:ii}
 For \(n=2\) and \((\Sets,\covers_\surj)\) and \((\Mfd, \covers_\subm)\) in Example~\ref{chap3:exa:singleton-pretopoloty}, the corresponding unique outer Kan conditions hold for \(m> 2\) .
\end{enumerate}
\end{theorem}

\begin{assumption}\label{chap4:asmp:ct}
  Let \(g=g_2\circ g_1\) be morphisms in \(\Cat\). If \(g\) is a cover in \(\covers\), then so is \(g_2\).
\end{assumption}

This assumption is satisfied by \((\Sets,\covers_\surj)\).\footnote{This assumption is not satisfied by \((\Mfd, \covers_\subm)\). A weaker assumption as (i) in Lemma~\ref{chap1:lem:surj-sub-2-prop} will ensure that Theorem~\ref{chap4:thm:enriched} hold for \((\Mfd, \covers_\subm)\).}

\subsection{Basic idea}
When \((\Cat,\covers)=(\Sets, \covers_\surj)\), statement~\ref{chap4:enm:it:i} follows from a theorem of Joyal; see also~\cite{Lurie}.

\begin{theorem}[\cite{Joyal2002}]\label{chap4:thm:joyal}
Let \(X\) be an inner Kan complex and \(\varphi:\Simp{1}\to X\) a quasi-invertible arrow~\footnote{An arrow \(\varphi\in X_1\) is quasi-invertible if it is invertible in the fundamental category of \(X\).} in \(X\). Let \(m\ge 2\). For every map \(f\colon\Horn{m}{0}\to X\) with \(f|_{\Simp{1}\{0,1\}} =\varphi\), there is a lift of \(f\) to \(\Simp{m}\to X\).
\end{theorem}

This theorem states that certain special outer horns can be filled in an inner Kan complex. Given a map \(f\colon \Horn{m}{0}[i,j]\to Y\) with \(i\ge 2\), the arrow \(f|_{\{0,1\}}\) is quasi-invertible since it lives in the higher groupoid \(\pi^{-1}(0)\). Theorem~\ref{chap4:thm:joyal} implies \(\Kan(m,0)[i,j]\) if \((\Cat,\covers)=(\Sets, \covers_\surj)\).

Let us first sketch Joyal's proof of Theorem~\ref{chap4:thm:joyal}. Roughly, since~\(\varphi\) is quasi-invertible, we can turn the outer horn to an inner horn.

Unwinding the definitions, to find the desired lift of \(f\) we must find the diagonal lift of
\begin{equation*}
\begin{gathered}
  \xymatrix{
  \{0\} \ar[r] \ar@{^{(}->}[d] & X_{/ \Simp{m-2} } \ar[d]^{q}\\
  \Simp{1} \ar[r]^{\varphi'} \ar@{-->}[ru]^? & X_{/ \partial\Simp{m-2} }\rlap{\ ,}
}
\end{gathered}
\end{equation*}
where \(\Simp{m-2}\) is \(\Simp{m-2}\{2,\cdots,m\}\) and the slice simplicial sets are given by \(f\) (see~\cite{Joyal2002} for the definition). The proof is divided into three steps:
\begin{itemize}
  \item The projections of the slice simplicial sets \(q\colon X_{/\Simp{m-2}} \to X_{/\partial\Simp{m-2}}\) and \(p\colon X_{/\partial\Simp{m-2}} \to X \) are right fibrations of inner Kan complexes~\cite[Corollary 3.9]{Joyal2002}.
 \item The arrow \(\varphi'\) as a preimage of \(\varphi\) under the right fibration~\(p\) is quasi-invertible~\cite[Proposition 2.7]{Joyal2002}.
 \item The quasi-invertible arrow \(\varphi'\) can be lifted under the right fibration~\(q\), and this gives the desired lift~\cite[Proposition 2.7]{Joyal2002}.
\end{itemize}

Although these steps use slice simplicial sets, they can be understood on the level of \(X\). For instance, the first step follows from the fact that filling a right horn of \(X_{/\Simp{m-2}} \to X_{/\partial\Simp{m-2}}\) or \(X_{/\partial\Simp{m-2}} \to X \) corresponds to filling an inner horn in \(X\). For the next two steps, we add some auxiliary diagrams, then fill certain horns by known Kan conditions. In the end, we forget some auxiliary diagrams. This motivates the following definitions.

\begin{definition}
  Let \(s\colon S\to T\) be a collapsible extension of colored simplicial sets. We say that \(s\) is \emph{special} if every horn added is either an inner horn or its vertices are of the same color.
\end{definition}

It is clear that for special collapsible extension of colored simplicial sets \(s\colon S\to T\), the map below is a cover:
\[
\Hom_{\Simp{1}}(T,Y)\to \Hom_{\Simp{1}}(S, Y).
\]

\begin{definition}
  An inclusion of colored simplicial sets \(a\colon R\to S\) is called \emph{admissible} if there is an inclusion of colored simplicial sets \(i\colon S\to T\) such that \(i\circ a\colon R\to T\) is special.
\end{definition}

If \(a\colon R\to S\) is an admissible inclusion of colored simplicial sets, then Assumption~\ref{chap4:asmp:ct} implies that the map below is a cover:
\[
\Hom_{\Simp{1}}(S,Y)\to \Hom_{\Simp{1}}(R, Y).
\]

\subsection{Proof of the first statement}

Guided by Joyal's argument, we prove statement~\ref{chap4:enm:it:i} in Theorem~\ref{chap4:thm:enriched}. Our goal is to show that \(\Horn{m}{0}[i,j]\to \Simp{m}[i,j]\) is admissible for \(i\ge 2\); the other cases follow by symmetry.

We need to construct a colored simplicial set \(T\) containing \(\Simp{m}[i, j]\) such that \(\Horn{m}{0}[i,j]\to T\) is a special collapsible extension.

First, add a vertex \(1^+\) with color \(0\) to \(\Horn{m}{0}[i,j]\) such that \(1<1^+<2\) and \((0,1^{+})\) is degenerate. Fill \(\Simp{m}\{0,1^+,2,\cdots,m\}\) by the degenerate simplex \(\de_0(\Simp{m-1}\{0,2,\cdots,m\})\).

Next we fill \(\Simp{2} \{ 0,1,1^+ \} \star\partial \Simp{m-2}\{2,\cdots,m\} \). If \(m=2\), we may fill the horn \(\Horn{2}{0}\{ 0,1,1^+ \}\) since they are of the same color. When \(m>2\), we add a vertex \(1^{++}\) with color \(0\) such that \(1^+<1^{++}<2\) and \((1,1^{++})\) is degenerate. Fill \(\Simp{2}\{0,1,1^{++}\} \star \partial\Simp{m-2}\{2,\cdots,m\}\) and \(\Simp{2}\{0,1^+,1^{++}\} \star \partial\Simp{m-2}\{2,\cdots,m\}\) by degenerate simplices. Since the vertices \((0,1,1^+,1^{++})\) are of the some color, we can fill \(\Simp{2}\{0,1,1^{+}\}\) and then \(\Simp{3}\{0,1,1^+,1^{++}\}\). Since \(\Horn{2}{2}\star\partial\Simp{m-2} \cup \Simp{2} \subset \Simp{2} \star\partial \Simp{m-2}\) is inner collapsible by Proposition~\ref{chap2:prop:join-pushout-collapsible}, we can fill \(\Simp{2} \{ 1,1^+,1^{++} \} \star\partial \Simp{m-2} \{2,\cdots,m\} \) by filling several inner horns. Since \(\Horn{3}{3}\star\partial\Simp{m-2} \cup \Simp{3} \subset \Simp{3} \star\partial \Simp{m-2}\) is inner collapsible, we can fill \(\Simp{3}\{0,1,1^+,1^{++}\}\star \partial \Simp{m-2}\{2,\cdots,m\}\) by filling several inner horns.

Next, we fill the horn \(\Horn{m}{1}\{1,1^+, 2,\cdots,m\}\), and then \(\Horn{m+1}{2} \{ 0,1,1^+, 2,\cdots,m\} \). The resulting simplicial set is the simplicial set \(T\) we are looking for, and we are done.

\subsection{Uniqueness}

We now prove statement~\ref{chap4:enm:it:ii} in Theorem~\ref{chap4:thm:enriched}. In these cases, we observe that a bijective cover is an isomorphism. It then remains to show that
\begin{equation}\label{eq:kanmap}
 \hom(\Simp{m}[i,j], Y ) \to \hom(\Horn{m}{0}[i,j], Y), \quad i \ge 2
\end{equation}
is injective. If \(m>3\), then \(\Horn{m}{0}\) contains all 2-simplices in \(\Simp{m}\). Since \(\Kan!(m,k)\) holds for \(m>2\) and \( 0<k<m \), there is at most one way to lift \(f\colon \Horn{m}{0}\to Y\) to \(\Simp{m}\to Y\), and the claim follows.

When \(m=3\), it suffices to show that for any \(f\colon \Horn{3}{0}\to Y\) with vertices \((0, 1)\) of color \(0\), there is a unique way to determine \((1,2,3)\).

In the previous subsection, we gave a way to obtain \((1,2,3)\). We claim that every qualified \((1,2,3)\) arises in this way. Indeed, given \((0,1,2,3)\) and a vertex \(1^+\) such that \(1<1^+<2\) and \((0,1^+,2,3)=\de_0 (0,2,3)\) is degenerate,  Theorem~\ref{chap4:thm:joyal} implies that we can get \((0,1,1^+,2,3)\). Thus \((1,2,3)\) arises from \((1,1^+,2)\) as we described. Similarly, \((1,1^+,2)\) arises in the way we described, and this proves the claim.

\begin{figure}[htbp]
\centering
\begin{tikzpicture}[scale=1.2,>=latex',every label/.style={scale=0.7}]
  \begin{scope}[<-]
  \node[mydot,label=180:$0$] at (0,0)           (p0) {};
  \node[mydot,label=-30:$1$] at (1.5,-0.6)      (p1) {};
  \node[mydot,fill,label=0:$2$] at (2,0)   (p2) {};
  \node[mydot,label=120:$1^{+}$]     at (0,0.3)  (p11){};
  \node[mydot,label=0:$1^{++}$]      at (1.8,-0.4)(p12){};
  \path
    (p0) edge (p1)
         edge (p2)
    (p1) edge (p2);
  \path[dashed]
     (p0) edge (p11)
     (p1) edge (p11)
     (p11)edge (p2)
          edge (p12)
     (p0) edge (p12)
     (p1) edge (p12)
     (p12)edge (p2);
  \end{scope}
  \begin{scope}[<-, xshift=3.5cm]
  \node[mydot,label=180:$0$] at (0,0)           (q0) {};
  \node[mydot,label=-30:$1$] at (1.5,-0.6)      (q1) {};
  \node[mydot,fill,label=0:$2$] at (2,0)   (q2) {};
  \node[mydot,label=150:$1^+$] at (0,0.3)       (q11){};
  \node[mydot,label=60:$1^{+'}$] at (0.4,0.4)   (q12){};
  \path
  (q0) edge (q1)
       edge (q2)
  (q1) edge (q2);
  \path
  (q0) edge (q11)
  (q1) edge (q11)
  (q11)edge (q2);
  \path
  (q0) edge (q12)
  (q1) edge (q12)
  (q12)edge (q2);
  \draw (q11) to (q12);
  \end{scope}
  \begin{scope}[<-, xshift=7cm]
  \node[mydot,label=180:$1^+$] at (0,0)           (r11){};
  \node[mydot,label=120:$1^{+'}$] at (0,0.3)      (r12){};
  \node[mydot,label=-30:$1$]         at (1.5,-0.6)(r1) {};
  \node[mydot,fill,label=0:$2$] at (2,0)     (r2) {};
  \node[mydot,fill,label=90:$3$]at (1,1.45)  (r3) {};
  \path
   (r1) edge (r2)
        edge (r3)
   (r2) edge (r3);
  \path
   (r1) edge (r11)
   (r11)edge (r2)
   (r11)edge (r3);
  \path
   (r1)  edge (r12)
   (r12) edge (r2)
         edge (r3);
  \draw (r11) to (r12);
  \end{scope}
\end{tikzpicture}
\caption{An illustration of the proof}\label{chap4:fig:inner-imply-outer}
\end{figure}
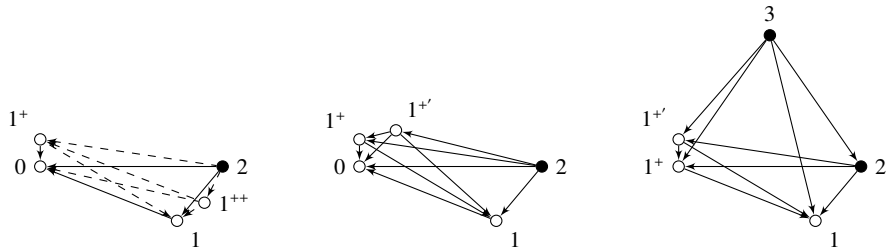

We now show that \((1,2,3)\) given by the previous subsection is unique. We claim (see Figure~\ref{chap4:fig:inner-imply-outer}):
\begin{enumerate}
 \item Given \((0,1,1^+)\), there is a unique way to determine \((1,1^+,2)\) and \((1,1^+,3)\);
 \item The ``difference'' of two given different \((0,1,1^+)\) equals the ``difference'' of two resulting \( (1,1^+,2) \) and \( (1,1^+,3) \);
 \item For different \((1,1^+,2)\) and \((1,1^+,3)\), we get the same \((1,2,3)\).
\end{enumerate}

If \(i=3\), the first claim is obvious. Let us consider \(i=2\). Recall how to obtain \((1,1^+,2)\). Add a vertex \(1^{++}\) with color \(0\) such that \((1,1^{++})\) is degenerate, and fill \(\Simp{3}\{0,1,1^{++},2\}\) and \(\Simp{3}\{0,1^+,1^{++},2\}\) with degenerate simplices. If \((0,1,1^+)\) is given, then there is a unique way to fill the whole diagrams, since each step is filling a horn \(\Horn{3}{i}\) or \(\Horn{4}{3}\). This shows that \((1,1^{+},2)\) is determined by \((0,1,1^+)\) uniquely.

Let \((0,1,1^+)\) and \((0,1,1^{+'})\) be two different choice for \((0,1,1^+)\). Their ``difference'' \((1,1^+,1^{+'})\) is given by \(\Kan!(3,0)\) for the horn \(\Horn{3}{0}\{0,1,1^{+},1^{+'}\}\). Accordingly, we have different faces \((1,1^+,2)\) and \((1,1^{+'},2)\). Let \(m_i\) be the 3-multiplications. We claim that
\[
 (1,1^{+'},2)=m_1((1,1^+,2),(1,1^+,1^{+'}),\de_0(1^+,2)).
\]
Since \((1,1^{+'},2)\) is unique, it suffices to show that
\[
m_1((1,1^+,2),(1,1^+,1^{+'}),\de_0(1^+,2))
\]
is compatible with \((0,1,1^{+'})\), \((0,1^{+'},2)\) and \((0,1,2)\). This follows by applying \(\Kan!(4,2)\) to the horn \(\Horn{4}{2}\{0,1,1^+,1^{+'},2\}\). Similarly, we have
\[
 (1,1^{+'},3)=m_1((1,1^+,3),(1,1^+,1^{+'}),\de_0(1^+,3)).
\]

Finally, applying \(\Kan!(4,2)\) to the horn \(\Horn{4}{2}\{1,1^+,1^{+'},2,3\}\), we conclude that for different \((1,1^+,2)\) and \((1,1^+,3)\), we get the same \((1,2,3)\). This completes the proof.

\backmatter
\printbibliography[heading=bibintoc]

\end{document}